\documentclass[11pt]{amsart}
\usepackage{graphicx}
\usepackage{amsmath,amsthm,amssymb,enumerate}
\usepackage{euscript,mathrsfs}
\usepackage{dsfont}
\usepackage[citecolor=blue,colorlinks=true]{hyperref}
\usepackage[left=3.2cm,right=3.2cm,top=4cm,bottom=4cm]{geometry}
\usepackage{color}
\catcode`\@=11 \@addtoreset{equation}{section}

\catcode`\@=12

\allowdisplaybreaks
\usepackage{amsfonts, color, MnSymbol}
\newcommand{\be}{\begin{eqnarray}}
\newcommand{\ee}{\end{eqnarray}}
\newcommand{\ce}{\begin{eqnarray*}}
\newcommand{\de}{\end{eqnarray*}}

\def\bt{\begin{theorem}}
	\def\et{\end{theorem}}
\def\bl{\begin{lemma}}
	\def\el{\end{lemma}}
\def\br{\begin{remark}}
	\def\er{\end{remark}}
\def\bx{\begin{Examples}}
	\def\ex{\end{Examples}}
\def\bd{\begin{definition}}
	\def\ed{\end{definition}}
\def\bp{\begin{proposition}}
	\def\ep{\end{proposition}}
\def\bc{\begin{corollary}}
	\def\ec{\end{corollary}}

\newtheorem{thm}{Theorem}[section]
\newtheorem{cor}[thm]{Corollary}
\newtheorem{lem}[thm]{Lemma}
\newtheorem{prp}[thm]{Proposition}

\theoremstyle{definition}

\newtheorem{rem}[thm]{Remark}

\theoremstyle{definition}
\newtheorem{defn}{Definition}[section]

\definecolor{wco}{rgb}{0.5,0.2,0.3}
\definecolor{greenn}{rgb}{.13,0.54,.13}

\numberwithin{equation}{section} \theoremstyle{remark}

\DeclareMathOperator*{\esssup}{esssup}

\newcommand{\rmb}[1]{\textcolor{black}{#1}}

\def\R{\mathbb{R}}

\def\[{{\Big[}}
\def\]{{\Big]}}

\def\({{\Big(}}
\def\){{\Big)}}

\newfam\msbfam
\font\tenmsb=msbm10 \textfont\msbfam=\tenmsb \font\sevenmsb=msbm7
\scriptfont\msbfam=\sevenmsb \font\fivemsb=msbm5
\scriptscriptfont\msbfam=\fivemsb

\numberwithin{equation}{section}

\def\R{{\mathbb R}}

\def\bc{\begin{center}}
\def\ec{\end{center}}
\def\no{\noindent}

\def\div{\textrm{div}}
\def\Id{\textrm{Id}}

\def\cB{{\mathcal B}}
\def\cC{{\mathcal C}}

\def\mB{{\mathbb B}}

\def\mN{{\mathbb N}}

\def\mP{{\mathbb P}}

\def\mR{{\mathbb R}}

\def\mT{{\mathbb T}}

\def\mX{{\mathbb X}}
\def\mY{{\mathbb Y}}

\def\1{{\mathbf{1}}}

\begin{document}

\title{Non-uniqueness in law of stochastic 3D Navier--Stokes equations}

\author{Martina Hofmanov\'a}
\address[M. Hofmanov\'a]{Fakult\"at f\"ur Mathematik, Universit\"at Bielefeld, D-33501 Bielefeld, Germany}
\email{hofmanova@math.uni-bielefeld.de}

\author{Rongchan Zhu}
\address[R. Zhu]{Department of Mathematics, Beijing Institute of Technology, Beijing 100081, China; Fakult\"at f\"ur Mathematik, Universit\"at Bielefeld, D-33501 Bielefeld, Germany}
\email{zhurongchan@126.com}

\author{Xiangchan Zhu}
\address[X. Zhu]{ Academy of Mathematics and Systems Science,
Chinese Academy of Sciences, Beijing 100190, China; Fakult\"at f\"ur Mathematik, Universit\"at Bielefeld, D-33501 Bielefeld, Germany}
\email{zhuxiangchan@126.com}
\thanks{This project received funding from the European Research
	Council (ERC) under the European Union’s Horizon 2020 research and innovation
	programme (grant agreement no. 949981). The financial support by the DFG
	through the CRC 1283 “Taming uncertainty and profiting from randomness and
	low regularity in analysis, stochastics and their applications” are greatly acknowledged.
	R.Z. is grateful for the financial support of the NSFC (no. 11922103).
	X.Z. is grateful to the financial support in part by National Key R\&D Program
	of China (no. 2020YFA0712700) and the NSFC (nos. 11771037, 12090014, and
	11688101) and the support by key Lab of Random Complex Structures and Data
	Science, Youth Innovation Promotion Association (2020003), Chinese Academy
	of Science.}

\begin{abstract}
We consider the  stochastic Navier--Stokes equations in three dimensions and prove that the law of analytically weak solutions is not unique. In particular, we focus on three examples of a stochastic perturbation:  an additive, a linear multiplicative and a nonlinear noise of cylindrical type, all  driven by a Wiener process. In these settings, we develop a stochastic counterpart of the convex integration method  introduced recently by Buckmaster and Vicol. This permits to construct probabilistically strong and analytically weak solutions defined up to a suitable stopping time. In addition, these solutions fail the corresponding energy inequality at a prescribed time with a prescribed probability. Then we introduce a general probabilistic construction used to extend the convex integration solutions beyond the stopping time and in particular to the whole time interval $[0,\infty)$. Finally, we show that their law is distinct from the law of   solutions obtained by Galerkin approximation. In particular, non-uniqueness in law holds on an arbitrary time interval $[0,T]$, $T>0$.

\end{abstract}


\date{\today}

\maketitle

\tableofcontents

\section{Introduction}

The fundamental problems in fluid dynamics remain largely open.
On the theoretical side, existence and smoothness of solutions to the three dimensional incompressible Navier--Stokes system  is one of the Millennium Prize Problems. An intimately related question is that of uniqueness of solutions. Intuitively, smooth solutions are unique whereas uniqueness for less regular solutions, such as weak solutions, is very challenging and  not even true for a number of models.

A revolutionary step was made through the method of convex integration by De Lellis and
Sz{\'e}kelyhidi~Jr.
\cite{DelSze2,DelSze3,DelSze13}. They were able to construct infinitely many weak solutions to the incompressible Euler system which dissipate energy and even satisfy various additional criteria such as a global or local energy inequality.
After this breakthrough, an avalanche of excitement and intriguing results followed, proving existence of solutions with often  rather pathological behavior. In particular, it is nowadays well understood that the compressible counterpart of the Euler system  is desperately ill-posed: even  certain smooth initial data give rise to infinitely many weak solutions satisfying an energy inequality, see Chiodaroli et al. \cite{ChKrMaSwI}. Very recently, the non-uniqueness of weak solutions to the incompressible Navier--Stokes equations was obtained by Buckmaster and Vicol \cite{BV19a}, see also Buckmaster, Colombo and Vicol \cite{BCV18}.

In view of these substantial  theoretical difficulties, it is natural to believe  that  a certain probabilistic description is  indispensable and may eventually help with the non-uniqueness issue.
 In particular, it is essential to
develop a suitable  probabilistic understanding of the deterministic
systems, in order to capture their chaotic and intrinsically random nature after the blow-up and  loss of uniqueness.
Moreover, there is evidence that a suitable stochastic perturbation may provide a regularizing effect on deterministically ill-posed problems, in particular those involving transport as shown, e.g., by Flandoli, Gubinelli and Priola \cite{FlaGubPri} and Flandoli and Luo \cite{FL19}. Also a linear multiplicative noise as treated in the present paper has a certain stabilizing effect on the three dimensional Navier--Stokes system, see R\"ockner, Zhu and Zhu \cite{RZZ14}.

On the other hand, an external stochastic forcing is often included in the system of governing equations, taking additional model uncertainties into account. Mathematically, this introduces new phenomena and raises basic questions of solvability of the system, i.e. existence and uniqueness of solutions, as well as their long time behavior. In particular, the question of uniqueness of the probability measures induced by solutions, the so-called uniqueness in law, has been a longstanding  open problem in the field.

\bigskip

In the present paper, we prove that non-uniqueness in law holds for the stochastic three dimensional Navier--Stokes system posed on a periodic domain in a class of analytically weak solutions. This system governs the time evolution of the velocity $u$ of a viscous incompressible fluid under stochastic perturbations. It reads as

\begin{equation}
\label{1}
\aligned
 d u-\nu\Delta u dt+\div(u\otimes u)dt+\nabla Pdt&=G(u)dB,
\\\div u&=0,
\endaligned
\end{equation}
where $G(u)dB$ represents a stochastic force acting on the fluid and $\nu>0$ is the kinematic viscosity.

We particularly focus on three examples of a stochastic forcing, namely,  an additive noise driven by a cylindrical Wiener process $B$ of trace class, i.e.,
\begin{equation}\label{eqad}
G(u)dB=G dB=\sum_{i=1}^\infty G^idB_i,\quad G^i=G^i(x),\quad {\mathrm{Tr}(GG^{*})<\infty,}
\end{equation}
and a linear multiplicative noise driven by a real-valued Wiener process $B_{1}$, i.e.,
\begin{equation}
\label{eqli}
 G(u)dB=udB_1,
\end{equation}
{and finally a nonlinear noise of cylindrical type
\begin{equation}\label{eq:nl}
G(u)dB=\left(\sum_{j=1}^mg_{ij}\big(\langle u,\varphi_1\rangle,\dots,\langle u,\varphi_{k_{ij}}\rangle\big)dB_j\right)_i, \quad g_{ij}\in C_b^3(\mT^3\times \mR^{k_{ij}};\mR),\ \varphi_{i}\in C^{\infty}(\mT^3),
\end{equation}
with  $B=(B_{j})$ being an $m$-dimensional Wiener process and $g_{\cdot j}$ is divergence free with respect to the spatial variable in $\mT^3$.}

In these three settings, we develop a stochastic counterpart of the convex integration method introduced by Buckmaster and Vicol \cite{BV19} and construct  analytically weak solutions with unexpected behavior defined up to suitable stopping times.
The striking feature of these solutions is that they are  probabilistically strong, i.e.,  adapted to the given Wiener process.
This  severely contradicts the general belief present within the SPDEs community, namely, that probabilistically strong
solutions and  uniqueness in law could help with the uniqueness problem for the Navier--Stokes system.

We say that uniqueness in law holds for a system of  SPDEs provided the probability law induced by the solutions is uniquely determined. On the other hand, we say that pathwise uniqueness holds true if two solutions coincide almost surely. There are explicit examples of stochastic differential equations (SDEs), where pathwise
uniqueness does not hold but uniqueness in law is valid.
Pathwise uniqueness for the stochastic Navier--Stokes system essentially poses the same
difficulties as uniqueness in the deterministic setting.
As a
consequence, there has been a clear hope that showing uniqueness in law for the Navier--Stokes system might be easier than proving pathwise uniqueness.
Furthermore,   Yamada--Watanabe--Engelbert's theorem
 states that, for a certain class of SDEs, pathwise uniqueness is equivalent to uniqueness in law and existence of a probabilistically strong solution,  see Kurtz \cite{K07}, Cherny \cite{C03}. This  suggests another possible way towards pathwise uniqueness, provided one could prove uniqueness in law.

Our  main result  proves the above hopes wrong, at least for a certain class of analytically weak solutions. However, the question of uniqueness of the so-called Leray solutions remains an outstanding open problem. In particular, we show that non-uniqueness in law for analytically weak solutions holds true on an arbitrary time interval $[0,T]$, $T>0$.  This trivially implies pathwise non-uniqueness. More precisely, we construct a deterministic divergence-free initial condition $u(0)\in L^{2}$ which gives rise to two solutions to the Navier--Stokes system \eqref{1} with distinct laws. One of the solutions is constructed by means of the convex integration method whereas the other one is a  solution obtained by a classical compactness argument  from a Galerkin approximation, see e.g. \cite{FG95}.

We note that the solutions obtained by Galerkin approximation are clearly more physical as they correspond to Leray solutions in the deterministic setting and satisfy the energy inequality. However, these solutions are not probabilistically strong as the adaptedness with respect to the given noise is lost within the stochastic compactness method.
On the other hand, the convex integration permits to construct adapted solutions up to a stopping time but they behave in an unphysical way with respect to the energy inequality. Moreover, the spatial regularity is worse as we can only prove that they belong to $H^{\gamma}$ for a certain $\gamma>0$ small.

\subsection{Main results}

Even though the main result, i.e., non-uniqueness in law, is the same in the three considered settings \eqref{eqad}, \eqref{eqli} and \eqref{eq:nl}, the proofs are different. The additive noise case is easier and we present a direct construction of two  solutions with different laws. This is not possible in the case of a linear multiplicative noise where the proof becomes more involved. {The nonlinear case is  even more challenging and requires  tools from the theory of rough paths.}
For notational simplicity,   we  suppose from now on  that $\nu=1$.

\subsubsection{Additive noise}

Consider the stochastic Navier--Stokes system driven by an additive noise on $\mathbb{T}^3$, which reads as
\begin{equation}
\label{ns}
\aligned
 d u-\Delta u dt+\div(u\otimes u)dt+\nabla Pdt&=dB,
\\\div u&=0,
\endaligned
\end{equation}
where $B$ is a $GG^*$-Wiener process on a probability space $(\Omega, \mathcal{F}, \mathbf{P})$ and $G$ is a Hilbert--Schmidt operator from $L^2$ to $L^2$.
Let
$(\mathcal{F}_t)_{t\geq0}$ denote the normal filtration generated by $B$, that is, the canonical right continuous filtration augmented by all the $\mathbf{P}$-negligible events.

Our first result in this setting is the existence of a probabilistically strong solution which is defined up to a stopping time and which violates the corresponding energy inequality.

\begin{thm}\label{Main results1}
Suppose that {$\mathrm{Tr}(GG^*) <\infty$}. Let $T>0$, $K>1$ and $\kappa\in (0,1)$ be given. Then there  exist $\gamma\in (0,1)$ and a $\mathbf{P}$-a.s. strictly positive stopping time $\mathfrak{t}$ satisfying $\mathbf{P}(\mathfrak{t}\geq T)>\kappa$ such that the following holds true: There exists
an $(\mathcal{F}_{t})_{t\geq0}$-adapted process $u$ which belongs to  $C([0,\mathfrak{t}];H^{\gamma})$ $\mathbf{P}$-a.s.
and is an analytically weak solution to \eqref{ns} with $u(0)$ deterministic. In addition,
 \begin{equation}\label{eq:vv}
 \esssup_{\omega\in\Omega}\sup_{t\in[0,\mathfrak{t}]}\|u(t)\|_{H^{\gamma}}<\infty,
 \end{equation}
and
\begin{equation}\label{eq:env}
\|u(T)\|_{L^{2}}> K \|u(0)\|_{L^{2}}+K \big(T\,\mathrm{Tr}(GG^{*})\big)^{1/2}
\end{equation}
on the set $\{\mathfrak{t}\geq T\}$.
\end{thm}

The proof of this result relies on a the convex integration method and the stopping time is employed in the construction in order to control the noise in various bounds. While this result readily implies non-uniqueness in law for solutions defined on the random time interval $[0,\mathfrak{t}]$, our main result is  more general: we prove non-uniqueness in law on an arbitrary time interval or more generally up to an arbitrary stopping time.

\begin{thm}\label{Main results2}
Suppose that {$\mathrm{Tr}(GG^*) <\infty$}. Then non-uniqueness in law holds for the Navier--Stokes system \eqref{ns} on $[0,\infty)$. Furthermore, for every given $T>0$, non-uniqueness in law holds for the Navier--Stokes system \eqref{ns} on $[0,T]$.

\end{thm}

In order to derive the result of  Theorem~\ref{Main results2} from Theorem~\ref{Main results1}, it is necessary to  extend the convex integration solutions to the whole time interval $[0,\infty)$. To this end, we present a general probabilistic construction which connects the law of solutions defined up to a stopping time to a law of a solution obtained by the classical compactness argument. The principle difficulty is to  allow for the concatenation of solutions at a random  time. Since  the stopping time $\mathfrak{t}$ is defined in terms of the solution $u$, we work with the notion of martingale solution which is defined as the law of a solution $u$. Consequently, we are able to obtain non-uniqueness in law, i.e., non-uniqueness of martingale solutions directly, as opposed to the case of a linear multiplicative noise.

\subsubsection{Linear multiplicative noise}

Consider the  stochastic Navier--Stokes equation driven by linear multiplicative noise on  $\mathbb{T}^3$, which reads as
\begin{equation}\label{ns1}
\aligned
 d u-\Delta u dt+\div(u\otimes u)dt+\nabla Pdt&=udB,
\\
\div u&=0,
\endaligned
\end{equation}
where $B$ is a real-valued Wiener process on a probability space $(\Omega, \mathcal{F}, \mathbf{P})$. Similarly to above, we denote by $(\mathcal{F}_t)_{t\geq0}$ the normal filtration generated by $B$.
The main results in this case are as follows.

\begin{thm}\label{Main results1 li}

Let $T>0$, $K>1$ and $\kappa\in (0,1)$ be given. Then there  exist  $\gamma\in (0,1)$ and a $\mathbf{P}$-a.s. strictly positive stopping time $\mathfrak{t}$ satisfying $\mathbf{P}(\mathfrak{t}\geq T)>\kappa$ and the following holds true: There exists an $(\mathcal{F}_{t})_{t\geq0}$-adapted process $u$ which belongs to  $C([0,\mathfrak{t}];H^{\gamma})$ $\mathbf{P}$-a.s. and is an  analytically weak solution to \eqref{ns1} with $u(0)$ deterministic. In addition,
 \begin{equation*}
 \esssup_{\omega\in\Omega}\sup_{t\in[0,\mathfrak{t}]}\|u(t)\|_{H^{\gamma}}<\infty,
 \end{equation*}
and
\begin{equation*}
\|u(T)\|_{L^{2}}> K e^{T/2} \|u(0)\|_{L^{2}}
\end{equation*}
on the set $\{\mathfrak{t}\geq T\}$.
\end{thm}

\begin{thm}\label{Main results2 li}
Non-uniqueness in law holds for the Navier--Stokes system \eqref{ns1} on $[0,\infty)$. Furthermore, for every given $T>0$, non-uniqueness in law holds for the Navier--Stokes system \eqref{ns1} on $[0,T]$.
\end{thm}

Contrary to the additive noise setting, the stopping time $\mathfrak{t}$ in  the case of the linear multiplicative noise is a function of $B$ and not a function of the solution $u$. As a consequence, we are forced to work with the notion of a probabilistically weak solution which governs the joint law of $(u,B)$. We extend our method of concatenation of two solutions to connect the probabilistically weak solution obtained through Theorem~\ref{Main results1 li} to a probabilistically weak solution obtained by compactness. Accordingly, we first only deduce joint non-uniqueness in law, i.e., non-uniqueness of probabilistically weak solutions. Finally, we prove that joint non-uniqueness in law implies non-uniqueness in law, concluding the proof of Theorem~\ref{Main results2 li}. This relies on a generalization of the result of Cherny \cite{C03} to the infinite dimensional setting which is interesting in its own right, see Appendix \ref{s:C}.

\subsubsection{A nonlinear  noise}
We consider the following Navier--Stokes equations
\begin{equation}
\label{ns g}
\aligned
d u-\Delta u dt+\div(u\otimes u)dt+\nabla Pdt&=G(u)dB,
\\\div u&=0,
\endaligned
\end{equation}
with $G(u)$ defined via \eqref{eq:nl} and $B$ an $m$-dimensional Brownian motion defined on a probability space $(\Omega, \mathcal{F},\mathbf{P})$ and we denote by  $(\mathcal{F}_{t})_{t\geq0}$ its normal filtration. In this setting, we apply convex integration in order to establish the following results.

\begin{thm}\label{Main resultsg1}
	Let $T>0$, $K>1$ and $\kappa\in (0,1)$ be given. Then there  exist $\gamma\in (0,1)$ and a $\mathbf{P}$-a.s. strictly positive stopping time $\mathfrak{t}$ satisfying $\mathbf{P}(\mathfrak{t}\geq T)>\kappa$ such that the following holds true: There exists
	an $(\mathcal{F}_{t})_{t\geq0}$-adapted process $u$ which belongs to  $C([0,\mathfrak{t}];L^2)\cap L^2([0,\mathfrak{t}];H^\gamma)$ $\mathbf{P}$-a.s.
	and is an analytically weak solution to \eqref{ns g} with $u(0)$ deterministic. In addition, for $q\in \mN$
\begin{equation}\label{energy}
	\mathbf{E}\left[\sup_{r\in [0,t\wedge \mathfrak{t}]}\|u(r)\|_{L^{2}}^{2q}+\int_0^{t\wedge \mathfrak{t}}\|u(r)\|_{H^{\gamma}}^2dr\right]\leq C_{t,q},
	\end{equation}
	for some constant $C_{t,q}$ and
	\begin{equation}\label{eq:env g}
	\mathbf{E}\left[1_{\mathfrak t\geq T}\|u(T)\|_{L^{2}}^2\right]> K \|u(0)\|_{L^{2}}^2+KTC_{G},
	\end{equation}
	with
	$$
	C_{G}=(2\pi)^{3}\sum_{i=1}^{3}\sum_{j=1}^{m}\|g_{i j}\|^{2}_{C^{0}}.
	$$
	 \end{thm}

\begin{thm}\label{Main resultsg2}
Non-uniqueness in law holds for the Navier--Stokes system \eqref{ns g} on $[0,\infty)$. Furthermore, for every given $T>0$, non-uniqueness in law holds for the Navier--Stokes system \eqref{ns g} on $[0,T]$.
	\end{thm}
	
This nonlinear case presents further challenges which did not appear in the previous settings of additive and linear multiplicative noise. First of all, there is no obvious transformation of the SPDEs into a PDE with random coefficients. Consequently, it is necessary to employ rough path theory in order to obtain a pathwise control of the stochastic integral in the convex integration scheme. This is the reason why we restricted ourselves to the cylindrical noise of the form \eqref{eq:nl}. Nevertheless, more general noise could be considered provided the corresponding rough path estimate is valid.

 Using rough path theory to control the stochastic integral requires the so-called iterated integral of $B$ against $B$ to be included in the path space. Accordingly, the stopping time $\mathfrak{t}$ is a function of $(B,\int B\otimes dB)$. Since we have to  define the
corresponding stopping time on the canonical path space, the difficulty
lies in how to define the iterated stochastic integral on the path space without
the use of any probability measure. Indeed, due to the low time regularity of the
Wiener process, the stochastic integral cannot be defined by purely analytical tools
and probability theory is required in a nontrivial way. We overcome this issue by introducing a notion of  generalized probabilistically weak solution which takes this issue into account.

{We note that in order to apply  rough path theory it is essential that the intermittent jets possess sufficient time regularity, namely, we require a complementary Young regularity to the Brownian motion, i.e. $\alpha_{0}=2/3+\kappa$ for $\kappa>0$ small. To this end, it is necessary to lower the spatial regularity and we derive new bounds for the intermittent jets in Lemma B.2. They lead to the convergence of $v_{q}$ in $C^{\alpha_{0}}([0,\mathfrak{t}];B_{1,1}^{-5-\delta})$, see \eqref{iteration g}.  This is the reason for restricting to the case of a cylindrical noise, i.e. noise which smoothens in the spatial variable. Other cases of spatially smoothing noise can be treated similarly.}

\begin{rem}
Let us  emphasize that if we directly tried to apply convex integration method without using stopping time,  we would  have to take expectation to control the stochastic integral. As convex integration scheme is an iteration procedure, we would have to include  $L^p$-moment estimates for arbitrary $p$ which is typically achieved by the Burkholder--Davis--Gundy inequality. However, as the implicit constant here depends on  $p$,  the estimates  would blow up during the iteration scheme.
\end{rem}

{\begin{rem}
Our convex integration schemes in the case of additive and nonlinear noise could be understood as follows. In addition to the principal part of the perturbation $w_{q+1}^{(p)}$, the incompressibility corrector $w_{q+1}^{(c)}$ and the temporal corrector $w^{(t)}_{q+1}$ as appearing in the deterministic literature, we introduce a \emph{stochastic corrector} $w^{(s)}_{q+1}$. Its role is to add noise scale by scale as one proceeds through the iteration. More precisely, for the original equation for $u$,
we  construct iterations $u_{q}$ given by
$$
\begin{aligned}
u_{q+1}=v_{q+1}+z_{q+1}&=v_{\ell}+w^{(p)}_{q+1}+w^{(c)}_{q+1}+w^{(t)}_{q+1}+z_{q+1}\\
&=(v_{\ell}+z_{\ell})+w^{(p)}_{q+1}+w^{(c)}_{q+1}+w^{(t)}_{q+1}+(z_{q+1}-z_{\ell})\\
&=u_{\ell}+w^{(p)}_{q+1}+w^{(c)}_{q+1}+w^{(t)}_{q+1}+w^{(s)}_{q+1},
\end{aligned}
$$
where $w^{(s)}_{q+1}=z_{q+1}-z_{\ell}$ is the stochastic corrector. In the case of additive noise, we set
 $z_{q+1}=\mathbb{P}_{\leq f(q+1)}z$ (i.e. suitable truncation in Fourier space) with
$$
d z -\Delta z d t = G d B,
$$
whereas in the case of nonlinear noise we define
$$
d z_{q+1} -\Delta z_{q+1} d t = G(v_{q}+z_{q+1})d B.
$$
Due to the dependence on $v_{q-1}$,  $z_q$ diverges in $C^1$ but converges in $L^2$. When we need to control the $C^1$-norm of $z_q$ in the estimates of the Reynolds stress, we can always  use a small constant from $v_q$ to absorb it.
\\
\indent Finally, we note that due to its particular structure, the linear multiplicative noise case is different in this respect.
 Here, the perturbations  are additionally randomized multiplicatively  by $e^{B}$ in the following way
$$
u_{q+1}=e^{B}v_{\ell}+e^{B}w^{(p)}_{q+1}+e^{B}w^{(c)}_{q+1}+e^{B}w^{(t)}_{q+1}.
$$
\end{rem}}

\color{black}

\subsection{Further relevant literature}

Stochastic Navier--Stokes  equations driven by a  trace-class noise, have been the subject of interest of a large number of works. The reader is referred e.g. to \cite{FG95, HM06, Del13} and the reference therein. In the two dimensional case, existence and uniqueness of strong solutions was obtained if the noisy forcing term is white in time and colored in space. In the three dimensional case, existence of martingale solutions was proved  in \cite{FR08, DD03, GRZ09}. Furthermore,  ergodicity was proved  if the system is driven by non-degenerate trace-class noise, see \cite{DD03, FR08,R08}. Navier--Stokes equations driven by space-time white noise are also considered in \cite{DD02} and \cite{ZZ15} and the system is studied in the context of rough paths theory in \cite{HLN19,HLN19b}.
The linear multiplicative noise (\ref{eqli}) can be seen as a damping term: it is shown in \cite{RZZ14} that it prevents the system from exploding  with a large probability.  In a more recent work, Flandoli and Luo  \cite{FL19} proved that
one kind of transport noise  improves the vorticity blow-up in 3D Navier-Stokes equations
with large probability.
In \cite{BR17}, a global solution starting from small initial data was constructed for 3D Navier--Stokes equations in vorticity formulation driven by linear multiplicative noise. However, the solutions are not adapted to the filtration generated by the noise and the stochastic integral should be understood in a rough path sense (see \cite{RZZ19} and \cite{MR19} for more general noise).  By the methods in \cite{BR17, MR19}, adapted solutions up to a stopping time can also be obtained. However, we note that  existence of globally defined probabilistically strong solutions to the stochastic Navier--Stokes system without any stopping time remains a challenging open problem. Finally, we note that the convex integration has already been applied in a stochastic setting, namely, to the isentropic Euler system in \cite{BFH17} and to the full Euler system in \cite{CFF19}.

{
\subsection{Relevant literature update}
In the first version of the present paper uploaded to  arXiv we established non-uniqueness in law only in the case of spatially regular additive noise (namely, $\mathrm{Tr}((-\Delta)^{3/2+2\sigma}GG^{*})<\infty$)  and a  linear multiplicative noise. Our method was then applied to several other fluid models driven by these noises, see \cite{HZZ20,HZZ21,RS21,Ya20a,Ya20b,Ya21,Ya21b,Ya21c}. In particular, in \cite{HZZ20} we studied the question of well-posedness for stochastic Euler equations from various perspectives. In \cite{HZZ21} we proved existence and non-uniqueness of  global-in-time probabilistically strong and Markov solutions to the stochastic Navier--Stokes system. In the present version of the manuscript  we are for the first time able  to prove non-uniqueness in law for the Navier--Stokes system with  a nonlinear stochastic perturbation.}

\subsection{Organization of the paper} In Section \ref{s:not}, we collect the  notations used throughout the paper. Section \ref{s:nonuniqueI}
 and Section \ref{s:1.1} are devoted to the proof of our first main result Theorem \ref{Main results2}, the non-uniqueness in law for the case of an additive noise. First, in Section \ref{s:nonuniqueI} we introduce the notion of martingale solution and present a general method of extending martingale solutions defined up to a stopping time to the whole time interval $[0,\infty)$. This is then applied to solutions obtained through the convex integration technique and the non-uniqueness in law is shown in Section \ref{s:1.2}. The convex integration solutions are constructed in Section \ref{s:1.1}, which proves Theorem \ref{Main results1}. A~similar structure can be found in Section \ref{s:nonuniquII} and Section \ref{s:1.3} devoted to the setting of a linear multiplicative noise. This relies on the notion of probabilistically weak solution and a general concatenation procedure presented in Section \ref{ss:gen1}. Application to the convex integration solutions together with the proof of Theorem \ref{Main results2 li} can be found in Section \ref{s:1.4}. The convex integration in this setting is applied in Section \ref{s:1.3}, where Theorem \ref{Main results1 li} is established. In Section \ref{s:nonuniquIII} and Section \ref{s:88}, we prove the results for the nonlinear noise. In Appendix~\ref{ap:A}, we collect several auxiliary results concerning stability of martingale, probabilistically weak as well as generalized probabilistically weak solutions. In Appendix \ref{s:B},  the construction of intermittent jets needed for the convex integration is recalled. In Appendix~\ref{s:C}, we show that non-uniqueness in law implies joint non-uniqueness in law in a general infinite dimensional SPDE setting. Finally, Appendix \ref{s:D} is devoted to the rough path analysis required in the nonlinear setting.

\section*{Acknowledgments}
We would like to thank Michael R\"{o}ckner for helpful suggestions.

\section{Notations}
\label{s:not}

\subsection{Function spaces}

  Throughout the paper, we use the notation $a\lesssim b$ if there exists a constant $c>0$ such that $a\leq cb$, and we write $a\simeq b$ if $a\lesssim b$ and $b\lesssim a$. Given a Banach space $E$ with a norm $\|\cdot\|_E$ and $T>0$, we write $C_TE=C([0,T];E)$ for the space of continuous functions from $[0,T]$ to $E$, equipped with the supremum norm $\|f\|_{C_TE}=\sup_{t\in[0,T]}\|f(t)\|_{E}$. We also use $CE$ or $C([0,\infty);E)$ to denote the space of continuous functions from $[0,\infty)$ to $E$. For $\alpha\in(0,1)$ we  define $C^\alpha_TE$ as the space of $\alpha$-H\"{o}lder continuous functions from $[0,T]$ to $E$, endowed with the seminorm $\|f\|_{C^\alpha_TE}=\sup_{s,t\in[0,T],s\neq t}\frac{\|f(s)-f(t)\|_E}{|t-s|^\alpha}.$ Here we use $C_T^\alpha$ to denote the case when $E=\mathbb{R}$.  We  also use $C_\mathrm{loc}^\alpha E$ to denote the space of functions from $[0,\infty)$ to $E$ satisfying $f|_{[0,T]}\in C_T^\alpha E$ for all $T>0$. For $p\in [1,\infty]$ we write $L^p_TE=L^p([0,T];E)$ for the space of $L^p$-integrable functions from $[0,T]$ to $E$, equipped with the usual $L^p$-norm. We also use $L^p_{\mathrm{loc}}([0,\infty);E)$ to denote the space of functions $f$ from $[0,\infty)$ to $E$ satisfying $f|_{[0,T]}\in L^p_T E$ for all $ T>0$.
    We use $L^p$ to denote the set of  standard $L^p$-integrable functions from $\mathbb{T}^3$ to $\mathbb{R}^3$. For $s>0$, $p>1$ we set $W^{s,p}:=\{f\in L^p; \|(I-\Delta)^{\frac{s}{2}}f\|_{L^p}<\infty\}$ with the norm  $\|f\|_{W^{s,p}}=\|(I-\Delta)^{\frac{s}{2}}f\|_{L^p}$. Set $L^{2}_{\sigma}=\{u\in L^2; \div u=0\}$. For $s>0$, $H^s:=W^{s,2}\cap L^2_\sigma$. For $s<0$ define $H^s$ to be the dual space of $H^{-s}$. We also use the Besov space $B^\beta_{p,q}, \beta\in \mR,$  defined by the closure of  smooth functions with respect to the $B^{\beta}_{p,q}$-norm
    $$
    \|f\|_{B^\beta_{p,q}}:=\Big(\sum_{j\geq -1} 2^{\beta j q}\|\Delta_j f\|_{L^p}^q\Big)^{1/q},
    $$
    with $\Delta_j,$ $j\in\mN_{0}\cup\{-1\}$, being the usual Littlewood-Paley blocks.

$\|f\|_{C^N_{t,x}}=\sum_{0\leq n+|\alpha|\leq N}\|\partial_t^n D^\alpha f\|_{L^\infty_t L^\infty}$. For a Polish space $H$ we also use $\mathcal{B}(H)$ to denote the $\sigma$-algebra of Borel sets in $H$.

\subsection{Probabilistic elements}

Let  $\Omega_0:=C([0,\infty),H^{-3})\cap L_{\textrm{loc}}^2([0,\infty),L^2_\sigma)$  and let $\mathscr{P}(\Omega_0)$ denote the set of all probability measures on $(\Omega_0,\mathcal{B})$ with $\mathcal{B}$ being the Borel $\sigma$-algebra coming from the topology of locally uniform convergence on $\Omega_0$. Let  $x:\Omega_0\rightarrow H^{-3}$ denote the canonical process on $\Omega_{0}$ given by
$$x_t(\omega)=\omega(t).$$
Similarly, for $t\geq 0$ we define $\Omega_{t}:=C([t,\infty),H^{-3})\cap L_{\textrm{loc}}^2([t,\infty),L^2_\sigma)$ equipped with its Borel $\sigma$-algebra $\mathcal{B}^{t}$ which coincides with $\sigma\{ x(s),s\geq t\}$.
Finally,  we define the canonical filtration  $\mathcal{B}_t^0:=\sigma\{ x(s),s\leq t\}$, $t\geq0$, as well as its right continuous version $\mathcal{B}_t:=\cap_{s>t}\mathcal{B}^0_s$, $t\geq 0$. For given probability measure $P$ we use $E^P$ to denote the expectation under $P$.

For a Hilbert space $U$, let $L_2(U,L^2_\sigma)$ be the space all Hilbert--Schmidt operators from $U$ to $L^2_\sigma$ with the norm $\|\cdot\|_{L_2(U,L^2_\sigma)}$. Let $G: L^2_\sigma\rightarrow L_2(U,L^2_\sigma)$ be $\mathcal{B}(L^2_\sigma)/\mathcal{B}(L_2(U,L^2_\sigma))$ measurable. In the following, we assume
$$\|G(x)\|_{L_2(U,L_{\sigma}^2)}\leq C(1+\|x\|_{L^2}),$$
for every $x\in C^\infty(\mathbb{T}^{3})\cap L^2_\sigma$ and if in addition $y_n\rightarrow y$ in $L^2$ then
$$\lim_{n\rightarrow \infty}\|G(y_n)^*x-G(y)^*x\|_U=0,$$
 where the asterisk denotes the adjoint operator.

Suppose there is another  Hilbert space $U_1$ such that the embedding  $U\subset  U_1$ is Hilbert--Schmidt. Let  $\bar{\Omega}:=C([0,\infty);H^{-3}\times U_1)\cap L^2_{\mathrm{loc}}([0,\infty);L^2_\sigma\times U_1)$  and let $\mathscr{P}(\bar{\Omega})$ denote the set of all probability measures on $(\bar{\Omega},\bar{\mathcal{B}})$ with $\bar{\mathcal{B}}$   being the Borel $\sigma$-algebra coming from the topology of locally uniform convergence on $\bar\Omega$. Let  $(x,y):\bar{\Omega}\rightarrow H^{-3}\times U_{1}$ denote the canonical process on $\bar{\Omega}$ given by
$$(x_t(\omega),y_t(\omega))=\omega(t).$$
For $t\geq 0$ we define   $\sigma$-algebra $\bar{\mathcal{B}}^{t}=\sigma\{ (x(s),y(s)),s\geq t\}$.
Finally,  we define the canonical filtration  $\bar{\mathcal{B}}_t^0:=\sigma\{ (x(s),y(s)),s\leq t\}$, $t\geq0$, as well as its right continuous version $\bar{\mathcal{B}}_t:=\cap_{s>t}\bar{\mathcal{B}}^0_s$, $t\geq 0$.

\section{Non-uniqueness in law I: the case of an additive noise}
\label{s:nonuniqueI}

\subsection{Martingale solutions}
\label{s:martsol}

Let us begin with a definition of martingale solution on $[0,\infty)$. In what follows, we fix $\gamma\in (0,1)$.

\begin{defn}\label{martingale solution}
Let $s\geq 0$ and $x_{0}\in L^{2}_{\sigma}$. A probability measure $P\in \mathscr{P}(\Omega_0)$ is  a martingale solution to the Navier--Stokes system (\ref{1})  with the initial value $x_0 $ at time $s$ provided

\no(M1) $P(x(t)=x_0, 0\leq t\leq s)=1$,
 and for any $n\in\mathbb{N}$
$$P\left\{x\in \Omega_0: \int_0^n\|G(x(r))\|_{L_2(U;L^2_\sigma)}^2dr<+\infty\right\}=1.$$

\no(M2) For every $e_i\in C^\infty(\mathbb{T}^3)\cap L^2_\sigma$, and for $t\geq s$ the process
$$M_{t,s}^{i}:=\langle x(t)-x(s),e_i\rangle+\int^t_s\langle \div(x(r)\otimes x(r))-\Delta x(r),e_i\rangle dr$$
is a continuous square integrable $(\mathcal{B}_t)_{t\geq s}$-martingale under $P$ with the quadratic variation process
given by
$\int_s^t\|G(x(r))^*e_i\|_{U}^2dr$, where the asterisk denotes the adjoint operator.

\no (M3) For any $q\in \mathbb{N}$ there exists a positive real function $t\mapsto C_{t,q}$ such that  for all $t\geq s$
$$E^P\left(\sup_{r\in [0,t]}\|x(r)\|_{L^2}^{2q}+\int_{s}^t\|x(r)\|^2_{H^{\gamma}}dr\right)\leq C_{t,q}(\|x_0\|_{L^2}^{2q}+1),$$
where $E^P$ denotes the expectation under $P$.
\end{defn}

In particular,  we observe that in the context of Definition \ref{martingale solution} for additive noise case, i.e. $G$ independent of $x$, if $\{e_i\}_{i\in\mathbb{N}}$ is an orthonormal basis of $L^2_\sigma$ consisting of eigenvectors of $GG^*$ then $M_{t,s}:=\sum_{i\in\mathbb{N}}M_{t,s}^{i}e_i$ is a  $GG^*$-Wiener process starting from $s$ with respect to the filtration $(\mathcal{B}_t)_{t\geq s}$  under $P$.

Similarly, we may define martingale solutions up to a stopping time $\tau:\Omega_{0}\to[0,\infty]$. To this end, we define the space of trajectories stopped at the time $\tau$ by
$$
\Omega_{0,\tau}:=\{\omega(\cdot\wedge\tau(\omega));\omega\in \Omega_{0}\}.
$$
We note that due to the Borel measurability of $\tau$, the set $\Omega_{0,\tau}=\{\omega: x(t,\omega)=x(t\wedge \tau(\omega),\omega), \forall t\geq0\}$ is a Borel subset of $\Omega_{0}$ hence $\mathscr{P}(\Omega_{0,\tau})\subset \mathscr{P}(\Omega_{0})$.

\begin{defn}\label{def:martsol}
Let $s\geq 0$ and $x_{0}\in L^{2}_{\sigma}$. Let $\tau\geq s$ be a $(\mathcal{B}_{t})_{t\geq s}$-stopping time. A probability measure $P\in\mathscr{P}(\Omega_{0,\tau})$  is  a martingale solution to the Navier--Stokes system (\ref{1}) on $[s,\tau]$ with the initial value $x_0$ at time $s$ provided

\no(M1) $P(x(t)=x_0, 0\leq t\leq s)=1$,  and for any $n\in\mathbb{N}$
$$P\left\{x\in \Omega_0: \int_0^{n\wedge \tau}\|G(x(r))\|_{L_2(U;L_2^\sigma)}^2dr<+\infty\right\}=1.$$

\no(M2) For every $e_i\in C^\infty(\mathbb{T}^3)\cap L^2_\sigma$, and for $t\geq s$ the process
$$M_{t\wedge\tau,s}^{i}:=\langle x(t\wedge\tau)-x_{0},e_i\rangle+\int^{t\wedge\tau}_s\langle \div(x(r)\otimes x(r))-\Delta x(r),e_i\rangle dr$$
is a continuous square integrable $(\mathcal{B}_t)_{t\geq s}$-martingale under $P$ with the quadratic variation process
given by
$\int_s^{t\wedge \tau}\|G(x(r))^*e_i\|_U^2dr.$

\no (M3) For any $q\in \mathbb{N}$ there exists a positive real function $t\mapsto C_{t,q}$ such that  for all $t\geq s$
$$E^P\left(\sup_{r\in [0,t\wedge\tau]}\|x(r)\|_{L^2}^{2q}+\int_{s}^{t\wedge\tau}\|x(r)\|^2_{H^{\gamma}}dr\right)\leq C_{t,q}(\|x_0\|_{L^2}^{2q}+1),$$
where $E^P$ denotes the expectation under $P$.
\end{defn}

The following result provides the existence of martingale solutions as well as a stability of the set of all martingale solutions. A similar result can be found in \cite{FR08,GRZ09} but in the present paper we require in addition stability with respect to the initial time. For completeness, we include the proof in Appendix \ref{ap:A}.

\begin{thm}\label{convergence}
  For every $(s,x_0)\in [0,\infty)\times L_{\sigma}^2$, there exists  $P\in\mathscr{P}(\Omega_0)$ which is a martingale solution to the Navier--Stokes system \eqref{1} starting at time $s$ from the initial condition $x_0$  in the sense of Definition \ref{martingale solution}. The set of all such martingale solutions with the same $C_{t,q}$ in \emph{(M3)} of Definition  \ref{martingale solution} is denoted by $\mathscr{C}(s,x_0,C_{t,q})$.

   Let $(s_n,x_n)\rightarrow (s,x_0)$ in $[0,\infty)\times L_{\sigma}^2$ as $n\rightarrow\infty$  and let $P_n\in \mathscr{C}(s_n,x_n,C_{t,q})$. Then there exists a subsequence $n_k$ such that the sequence $\{P_{n_k}\}_{k\in\mathbb{N}}$  converges weakly to some $P\in\mathscr{C}(s,x_0,C_{t,q})$.
\end{thm}

For completeness, let us recall the definition of uniqueness in law.

\begin{defn}\label{def:uniquelaw}
We say that uniqueness in law holds for  \eqref{1} if martingale solutions starting from the same initial distribution are unique.
\end{defn}

Now, we have all in hand to proceed with the proof of our first main result, Theorem  \ref{Main results2}. On the one hand, by classical arguments as in Theorem \ref{convergence} we obtain existence of a martingale solution to (\ref{1}) which satisfies the corresponding energy inequality. On the other hand, for the case of an additive noise,  Theorem \ref{Main results1} provides  a stopping time $\mathfrak{t}$ such that  there exists an $(\mathcal{F}_{t})_{t\geq 0}$-adapted analytically  weak solution $u\in C([0,\mathfrak{t}];H^{\gamma})$ to (\ref{ns}), which violates the energy inequality. The main idea is to construct  a martingale solution which is defined on the full interval $[0,\infty)$ and preserves the properties of the adapted solution on $[0,\mathfrak{t}]$, that is, the energy inequality is not satisfied in this random time interval. To this end, the essential point is to make use of adaptedness of  solutions obtained through Theorem \ref{Main results1} and connect them to ordinary martingale solutions obtained by Theorem \ref{convergence}. The difficulty is that the connection has to happen at a random  time, which only turns out to be a stopping time with respect the right continuous filtration $(\mathcal{B}_{t})_{t\geq0}$. Consequently, the classical martingale theory of Stroock and Varadhan \cite{SV79} does not apply and  we are facing a number of measurability issues which have to be carefully treated.

\subsection{General construction for martingale solutions}
\label{ss:gen}

First, we present an auxiliary  result which is then used in order  to extend martingale solutions defined up a stopping time $\tau$ to the whole interval $[0,\infty)$. To this end, we denote by $\mathcal{B}_{\tau}$ the $\sigma$-field associated to the stopping time $\tau$. The results of this section apply to a general form of a noise in \eqref{1}, the restriction to an additive noise is only required in  Section \ref{s:1.2} below in order to apply the result of Theorem \ref{Main results1}.

\begin{prp}\label{prop:1}
Let $\tau$ be a bounded $(\mathcal{B}_{t})_{t\geq0}$-stopping time.
Then for every $\omega\in \Omega_{0}$ there exists $Q_{\omega}\in\mathscr{P}(\Omega_{0})$ such that for $\omega\in \{x(\tau)\in L^2_\sigma\}$
\begin{equation}\label{qomega}
Q_\omega\big(\omega'\in\Omega_{0}; x(t,\omega')=\omega(t) \textrm{ for } 0\leq t\leq \tau(\omega)\big)=1,
\end{equation}
and
\begin{equation}\label{qomega2}
Q_\omega(A)=R_{\tau(\omega),x(\tau(\omega),\omega)}(A)\qquad\text{for all}\  A\in \mathcal{B}^{\tau(\omega)}.
\end{equation}
where $R_{\tau(\omega),x(\tau(\omega),\omega)}\in\mathscr{P}(\Omega_0)$ is a martingale solution to the Navier--Stokes system \eqref{1} starting at time $\tau(\omega)$ from the initial condition $x(\tau(\omega),\omega)$. Furthermore, for every $B\in\mathcal{B}$ the mapping $\omega\mapsto Q_{\omega}(B)$ is $\mathcal{B}_{\tau}$-measurable.
\end{prp}

\begin{proof}
It is necessary to select in a measurable way from the set of all martingale solutions.
To this end, we observe that as a consequence of the stability with respect to the initial time and the initial condition in  Theorem \ref{convergence}, for every $(s,x_0)\in [0,\infty)\times L_{\sigma}^2$ the set $\mathscr{C}(s,x_{0}, C_{t,q})$ of all associated martingale solutions to \eqref{1} with the same $C_{t,q}$  is compact with respect to the weak convergence of probability measures. Let $\mathrm{Comp}(\mathscr{P}(\Omega_{0}))$ denote the space of all compact subsets of $\mathscr{P}(\Omega_{0})$ equipped with the Hausdorff metric. Using the stability from Theorem \ref{convergence}  together with \cite[Lemma 12.1.8]{SV79} we obtain that the map
$$[0,\infty)\times L^{2}_{\sigma}\to\mathrm{Comp}(\mathscr{P}(\Omega_{0})),\qquad (s,x_{0})\mapsto \mathscr{C}(s,x_{0}, C_{t,q}),$$
is Borel measurable. Accordingly, \cite[Theorem 12.1.10]{SV79} gives the existence of a measurable selection. More precisely, there exists a Borel measurable map
$$[0,\infty)\times L^{2}_{\sigma}\to\mathscr{P}(\Omega_{0}),\qquad(s,x_{0})\mapsto R_{s,x_{0}},$$
such that $R_{s,x_{0}}\in \mathscr{C}(s,x_{0}, C_{t,q})$ for all $(s,x_{0})\in[0,\infty)\times L^{2}_{\sigma}$.

As the next step, we recall that the canonical process $x$ on $\Omega_{0}$ is continuous in $H^{-3}$, hence $x:[0,\infty)\times \Omega_0\rightarrow H^{-3}$ is progressively measurable with respect to the canonical filtration $(\mathcal{B}^{0}_{t})_{t\geq 0}$ and consequently it is also progressively measurable  with respect to the right continuous filtration $(\mathcal{B}_{t})_{t\geq 0}$. In addition, $\tau$ is a stopping time with respect to the same filtration $(\mathcal{B}_{t})_{t\geq 0}$. Therefore, it follows from \cite[Lemma 1.2.4]{SV79} that both $\tau$ and $x(\tau(\cdot),\cdot)$ are $\mathcal{B}_{\tau}$-measurable. Furthermore, $L^2_\sigma\subset H^{-3}$ continuously and densely, by Kuratowski's measurability theorem we know $L^2_\sigma\in \mathcal{B}(H^{-3})$ and $\mathcal{B}(L^2_\sigma)=\mathcal{B}(H^{-3})\cap L^2_\sigma$, which implies that  $1_{\{x(\tau)\in L^2_\sigma\}}\in\mathcal{B}_{\tau}$.   Therefore,  $x(\tau(\cdot),\cdot)1_{\{x(\tau)\in L^2_\sigma\}}:\Omega_0\to L^2_\sigma$ is $\mathcal{B}_{\tau}$-measurable, where $\mathcal{B}_{\tau}$ is the $\sigma$-algebra associated to $\tau$. Combining this fact with the measurability of the selection $(s,x_{0})\mapsto R_{s,x_{0}}$ constructed above, we deduce that
\begin{equation}\label{eq:RR}
\Omega_{0}\to \mathscr{P}(\Omega_{0}),\qquad\omega\mapsto R_{\tau(\omega),x(\tau(\omega),\omega)1_{\{x(\tau(\omega),\omega)\in L^2_\sigma\}}}
\end{equation}
is $\mathcal{B}_{\tau}$-measurable as a composition of $\mathcal{B}_{\tau}$-measurable mappings. Recall that for every $\omega\in\Omega_{0}\cap \{x(\tau)\in L^2_\sigma\}$ this mapping gives a martingale solution starting at the deterministic time $\tau(\omega)$ from the deterministic initial condition $x(\tau(\omega),\omega)$. Hence, for $\omega\in\{x(\tau)\in L^2_\sigma\}$
$$
R_{\tau(\omega),x(\tau(\omega),\omega)}\big(\omega'\in\Omega_{0}; x(\tau(\omega),\omega')=x(\tau(\omega),\omega)\big)=1.
$$

Now, we apply \cite[Lemma 6.1.1]{SV79} and deduce that for every $\omega\in \Omega_{0}\cap \{x(\tau)\in L^2_\sigma\}$ there is a unique probability measure
\begin{equation}\label{q1}
  \delta_\omega\otimes_{\tau(\omega)}R_{\tau(\omega),x(\tau(\omega),\omega)}\in\mathscr{P}(\Omega_{0}),
\end{equation}
such that for every $\omega\in \Omega_{0}\cap \{x(\tau)\in L^2_\sigma\}$ \eqref{qomega} and \eqref{qomega2} hold.
This permits to concatenate, at the deterministic time $\tau(\omega)$, the Dirac mass $\delta_{\omega}$ with the martingale solution $R_{\tau(\omega),x(\tau(\omega),\omega)}$. Define
\begin{align*}
Q_\omega=\begin{cases}
\delta_\omega\otimes_{\tau(\omega)}R_{\tau(\omega),x(\tau(\omega),\omega)},&\omega \in\{x(\tau)\in L^2_{\sigma}\} ,\\
\delta_{x(\cdot\wedge \tau(\omega))}, &\textrm{otherwise}.
\end{cases}
\end{align*}

In order to show that the mapping $\omega\mapsto Q_{\omega}(B)$  is $\mathcal{B}_{\tau}$-measurable for every $B\in \mathcal{B}$, it is enough to consider sets of the form $A=\{x(t_{1})\in \Gamma_{1},\dots, x(t_{n})\in\Gamma_{n}\}$ where $n\in\mathbb{N}$, $0\leq t_{1}<\cdots< t_{n}$, and $\Gamma_{1},\dots,\Gamma_{n}\in\mathcal{B}(H^{-3})$. Then by the definition of $Q_{\omega}$, we have
\begin{align*}
\begin{aligned}
\delta_\omega\otimes_{\tau(\omega)}R_{\tau(\omega),x(\tau(\omega),\omega)}(A)&={\bf 1}_{[0,t_{1})}(\tau(\omega))R_{\tau(\omega),x(\tau(\omega),\omega)}(A)\\
&\quad+\sum_{k=1}^{n-1}{\bf 1}_{[t_{k},t_{k+1})}(\tau(\omega)){\bf 1}_{\Gamma_{1}}(x(t_{1},\omega))\cdots {\bf 1}_{\Gamma_{k}}(x(t_{k},\omega))\\
&\qquad\quad\times R_{\tau(\omega),x(\tau(\omega),\omega)}\big(x(t_{k+1})\in \Gamma_{k+1},\dots, x(t_{n})\in\Gamma_{n}\big)\\
&\quad+{\bf 1}_{[t_{n},\infty)}(\tau(\omega)){\bf 1}_{\Gamma_{1}}(x(t_{1},\omega))\cdots {\bf 1}_{\Gamma_{n}}(x(t_{n},\omega)).
\end{aligned}
\end{align*}
Here the right hand side multiplied by $1_{\{x(\tau)\in L^2_{\sigma}\}}$ is $\mathcal{B}_{\tau}$-measurable as a consequence of the $\mathcal{B}_{\tau}$-measurability of \eqref{eq:RR} and $\tau$. Moreover, $\delta_{x(\cdot\wedge \tau(\omega))}$  is ${\mathcal{B}}_{\tau}$-measurable as a consequence of the ${\mathcal{B}}_{\tau}$-measurability of $x(\tau\wedge \cdot)$.
Thus the final result follows from $\{x(\tau)\in L^2_{\sigma}\}$ is ${\mathcal{B}}_{\tau}$-measurable.
\end{proof}

\begin{rem}
If  $P$ as a martingale solution up to a stopping time  $\tau$,  our ultimate goal is to make use of Proposition \ref{prop:1}, in order to define a probability measure
\begin{equation*}
P\otimes_{\tau}R(\cdot):=\int_{\Omega_{0}}Q_{\omega} (\cdot)\,P(d\omega)
\end{equation*}
and show that it is a martingale solution on $[0,\infty)$ in the sense of Definition \ref{martingale solution} which coincides with $P$ up to the time $\tau$. However,
  due to the fact that $\tau$ is only a stopping time with respect to the right continuous filtration $(\mathcal{B}_{t})_{t\geq0}$, \eqref{qomega} does not suffice to show that  $(Q_{\omega})_{\omega\in \Omega_{0}}$ is a  conditional probability distribution of $ P\otimes_{\tau}R$ given $\mathcal{B}_{\tau}$. More precisely, we cannot prove that
 for every $A\in\mathcal{B}_{\tau}$ and $B\in\mathcal{B}$
$$
P\otimes_{\tau}R(A\cap B)=\int_{A} Q_{\omega}(B)P(d\omega).
$$
This is the reason why the corresponding results of  \cite{SV79}, namely Theorem 6.1.2 and in particular Theorem 1.2.10 leading to the desired  martingale property (M2), cannot be applied. It will be seen below in Proposition \ref{prop:2} that an additional condition on $Q_{\omega}$, i.e., \eqref{Q1}, is necessary in order to guarantee (M1), (M2) and (M3).
To conclude this remark, we note that certain measurability of the mapping $\omega\mapsto Q_{\omega}(B)$ is only needed to define the integral in \eqref{eq:PR}. Since we do not show that $(Q_{\omega})_{\omega\in \Omega_{0}}$ is a  conditional probability distribution, the $\mathcal{B}_{\tau}$-measurability from Proposition \ref{prop:1} is actually not used in the sequel.
\end{rem}

\begin{prp}\label{prop:2}
 Let $x_{0}\in L^{2}_{\sigma}$.
 Let $P$ be a martingale solution to the Navier--Stokes system \eqref{1} on $[0,\tau]$ starting at the time $0$ from the initial condition $x_{0}$. In addition to the assumptions of Proposition \ref{prop:1}, suppose that there exists a Borel  set $\mathcal{N}\subset\Omega_{0,\tau}$ such that $P(\mathcal{N})=0$ and for every $\omega\in \mathcal{N}^{c}$ it holds
\begin{equation}\label{Q1}
\aligned
&Q_\omega\big(\omega'\in\Omega_{0}; \tau(\omega')=
\tau(\omega)\big)=1.
\endaligned\end{equation}
Then the  probability measure $ P\otimes_{\tau}R\in \mathscr{P}(\Omega_{0})$ defined by
\begin{equation}\label{eq:PR}
P\otimes_{\tau}R(\cdot):=\int_{\Omega_{0}}Q_{\omega} (\cdot)\,P(d\omega)
\end{equation}
satisfies $P\otimes_{\tau}R= P$  on the $\sigma$-algebra $\sigma(x(t\wedge \tau),t\geq0)$ and
it is a martingale solution to the Navier--Stokes system \eqref{1} on $[0,\infty)$ with initial condition $x_{0}$.
\end{prp}

\begin{proof}
First, we observe that due to \eqref{Q1} and \eqref{qomega}, it holds $P\otimes_{\tau}R(A)=P(A)$ for every Borel set $A\in\sigma(x(t\wedge \tau),t\geq0)$. It remains to verify that $P\otimes_{\tau}R$ satisfies {(M1)}, (M2) and {(M3)} in Definition~\ref{martingale solution} with $s=0$.
The first condition in (M1)  follows easily since by construction
$
P\otimes_{\tau}R(x(0)=x_{0})=P(x(0)=x_{0})=1.
$
The second one in (M1) follows from (M3) and the assumption on $G$.
In order to show (M3), we write
 \begin{equation*}
 \aligned
  &E^{P\otimes_{\tau}R}\left(\sup_{r\in[0,t]}\|x(r)\|_{L^2}^{2q}+\int_0^t\|x(r)\|_{H^\gamma}^2dr\right)\\
& \leq E^{P\otimes_{\tau}R}\left(\sup_{r\in[0,t\wedge \tau]}\|x(r)\|_{L^2}^{2q}+\int_0^{t\wedge\tau}\|x(r)\|_{H^\gamma}^2dr\right)
 +E^{P\otimes_{\tau}R}\left(\sup_{r\in[t\wedge\tau,t]}\|x(r)\|_{L^2}^{2q}+\int_{t\wedge\tau}^t\|x(r)\|_{H^\gamma}^2dr\right).
 \endaligned
 \end{equation*}
Here, the first term can be estimated due to the bound (M3) for $P$, whereas the second term can be bounded based on (M3) for $R$.
Then by (\ref{Q1})
 \begin{equation*}
 \aligned
&  E^{P\otimes_{\tau}R}\left(\sup_{r\in[0,t]}\|x(r)\|_{L^2}^{2q}+\int_0^t\|x(r)\|_{H^\gamma}^2dr\right)\\
&\quad \leq C(\|x_{0}\|_{L^{2}}^{2q}+1)+C(E^P\|x(\tau)\|_{L^2}^{2q}+1)\leq C(\|x_{0}\|_{L^{2}}^{2q}+1).
 \endaligned
 \end{equation*}
 In the last step,   we used  the fact  that $\tau$ is bounded together with (M3) for $P$.

Finally, we shall verify (M2). To this end, we recall that since $P$ is a martingale solution on $[0,\tau]$, the process
 $M_{t\wedge\tau,0}^{i}$ is a   continuous square integrable $(\mathcal{B}_t)_{t\geq 0}$-martingale under $P$ with  the quadratic variation process
given by
$\int_0^{t\wedge\tau}\|G(x(r))^*e_i\|_{U}^2dr.$
On the other hand, since for every $\omega\in \Omega_{0}$, the probability measure $R_{\tau(\omega),x(\tau(\omega),\omega)}$ is a martingale solution  starting at the time $\tau(\omega)$ from the  initial condition $x(\tau(\omega),\omega)$,  the process $M^{i}_{t,t\wedge \tau(\omega)}$ is a continuous square integrable $(\mathcal{B}_{t})_{t\geq \tau(\omega)}$-martingale under $R_{\tau(\omega),x(\tau(\omega),\omega)}$ with  the quadratic variation process
given by
$\int_{t\wedge\tau(\omega)}^t\|G(x(r))^*e_i\|_{U}^2dr$, $t\geq\tau(\omega)$. In other words, the process $ M^{i}_{t,0}-M^{i}_{t\wedge\tau(\omega),0}$ is a continuous square integrable $(\mathcal{B}_{t})_{t\geq 0}$-martingale under $R_{\tau(\omega),x(\tau(\omega),\omega)}$ with  the quadratic variation process
given by
$\int_{t\wedge\tau(\omega)}^t\|G(x(r))^*e_i\|_{U}^2dr$.

Next, we will show that $M_{t,0}^{i}$ is a  continuous square integrable $(\mathcal{B}_t)_{t\geq 0}$-martingale under $P\otimes_{\tau}R$ with  the quadratic variation process
given by
$\int_{0}^t\|G(x(r))^*e_i\|_{U}^2dr.$
To this end, let $s\leq t$ and $A\in\mathcal{B}_s$. We first prove that
\begin{equation}\label{ma}E^{Q_\omega}\left[M_{t,0}^{i} {\bf 1}_A\right]=E^{Q_\omega}\left[M_{(t\wedge\tau(\omega))\vee s,0}^{i}{\bf 1}_A\right].\end{equation}
In fact, it is enough to consider sets of the form $A=\{x(t_{1})\in \Gamma_{1},\dots, x(t_{n})\in\Gamma_{n}\}$ where $n\in\mathbb{N}$, $0\leq t_{1}<\cdots< t_{n}\leq s$, and $\Gamma_{1},\dots,\Gamma_{n}\in\mathcal{B}(H^{-3})$. For more general $A\in\mathcal{B}_s$ we could use the approximation and the continuity of $M_{\cdot,0}^i$ to conclude. Then by the definition of $Q_{\omega}$ and using the martingale property with respect to $R_{\tau(\omega),x(\tau(\omega),\omega)}$ which is valid  for $t\geq\tau(\omega)$, we have
\begin{align*}
\begin{aligned}
& E^{Q_{\omega}}\left[(M_{t,0}^i-M_{(t\wedge\tau(\omega))\vee s,0}^i){\bf 1}_A\right]\\&={\bf 1}_{[0,t_{1})}(\tau(\omega))E^{R_{\tau(\omega),x(\tau(\omega),\omega)}}\left[(M_{t,0}^i-M_{s,0}^i){\bf 1}_A\right]\\
&\quad+\sum_{k=1}^{n-1}{\bf 1}_{[t_{k},t_{k+1})}(\tau(\omega)){\bf 1}_{\Gamma_{1}}(x(t_{1},\omega))\cdots {\bf 1}_{\Gamma_{k}}(x(t_{k},\omega))\\
&\qquad\quad\times E^{R_{\tau(\omega),x(\tau(\omega),\omega)}}\big((M_{t,0}^i-M_{s,0}^i){\bf 1}_{x(t_{k+1})\in \Gamma_{k+1},\dots, x(t_{n})\in\Gamma_{n}}\big)\\
&\quad+{\bf 1}_{[t_{n},\infty)}(\tau(\omega)){\bf 1}_{\Gamma_{1}}(x(t_{1},\omega))\cdots {\bf 1}_{\Gamma_{n}}(x(t_{n},\omega))\times E^{R_{\tau(\omega),x(\tau(\omega),\omega)}}(M_{t,0}^i-M_{(t\wedge \tau(\omega))\vee s,0}^i)
\\&=0.
\end{aligned}
\end{align*}
Now (\ref{ma}) follows.

Then it follows from \eqref{eq:PR} and \eqref{q1} that
\begin{equation*}\aligned
E^{P\otimes_{\tau} R}\left[M_{t,0}^{i} {\bf 1}_A\right]
&=\int_{\Omega_{0}} E^{Q_\omega}\left[M_{t,0}^{i} {\bf 1}_A\right]P(d\omega)
\\
&=
\int_{\Omega_{0}} E^{\delta_\omega\otimes_{\tau(\omega)}R_{\tau(\omega),x(\tau(\omega),\omega)}}\left[M_{t,0}^{i} {\bf 1}_A\right]P(d\omega).
\endaligned\end{equation*}
According to (\ref{ma}) and then using the key assumption \eqref{Q1} we further deduce that
\begin{equation*}\aligned
E^{P\otimes_{\tau} R}\left[M_{t,0}^{i} {\bf 1}_A\right]
&=\int_{\Omega_{0}} E^{\delta_\omega\otimes_{\tau(\omega)}R_{\tau(\omega),x(\tau(\omega),\omega)}}\left[M_{(t\wedge\tau(\omega))\vee s,0}^{i}{\bf 1}_A\right]P(d\omega)
\\&=E^{P\otimes_{\tau} R}\left[M_{(t\wedge\tau)\vee s,0}^{i}{\bf 1}_A\right]
\\&=E^{P\otimes_{\tau} R}\left[M_{t\wedge\tau,0}^{i}{\bf 1}_{A\cap \{\tau>s\}}\right]+E^{P\otimes_{\tau} R}\left[M_{s,0}^{i}{\bf 1}_{A\cap \{\tau\leq s\}}\right].
\endaligned\end{equation*}
Finally, using the martingale property up to $\tau$ with respect to $P$, we get
\begin{equation*}\aligned
E^{P\otimes_{\tau} R}\left[M_{t,0}^{i} {\bf 1}_A\right]
&=E^{P\otimes_{\tau} R}\left[M_{s,0}^{i}{\bf 1}_{A\cap \{\tau>s\}}\right]+E^{P\otimes_{\tau} R}\left[M_{s,0}^{i}{\bf 1}_{A\cap \{\tau\leq s\}}\right]
\\&=E^{P\otimes_{\tau} R}\left[M_{s,0}^{i}{\bf 1}_{A}\right].
\endaligned\end{equation*}
Hence $M^{i}$ is a $(\mathcal{B}_{t})_{t\geq0}$-martingale with respect to $P\otimes_{\tau}R$. In order to identify its quadratic variation, we proceed similarly and write
\begin{equation*}\aligned&
E^{P\otimes_{\tau} R}\left[\left((M_{t,0}^{i})^2-\int_0^t\|G(x(r))^*e_i\|_{U}^2dr\right){\bf 1}_{A}\right]
\\&=\int_{\Omega_{0}}E^{Q_{\omega}}\left[\left((M_{t,0}^{i}-M_{t\wedge \tau(\omega),0}^{i})^2-\int_{t\wedge\tau(\omega)}^t\|G(x(r))^*e_i\|_{U}^2dr\right){\bf 1}_{A}\right]P(d\omega)
\\&\quad+\int_{\Omega_{0}}E^{Q_{\omega}}\left[\left((M_{t\wedge \tau(\omega),0}^{i})^2-\int_0^{t\wedge \tau(\omega)}\|G(x(r))^*e_i\|_{U}^2\right){\bf 1}_{A}\right]P(d\omega)
\\&\quad+2\int_{\Omega_{0}}E^{Q_{\omega}}\left[\left(M_{t\wedge \tau(\omega),0}^{i}(M_{t,0}^{i}-M_{t\wedge \tau(\omega),0}^{i})\right){\bf 1}_{A}\right]P(d\omega)
\\&=:J_1+J_2+J_3.
\endaligned\end{equation*}
Here, due to the martingale property with respect to $R$ and $P$ similar as in (\ref{ma}), we obtain
\begin{equation*}\aligned
J_1&=\int_{\Omega_{0}}E^{Q_{\omega}}\left[\left((M_{t\wedge \tau(\omega)\vee s,0}^{i}-M_{t\wedge \tau(\omega),0}^{i})^2-\int^{t\wedge \tau(\omega)\vee s}_{t\wedge \tau(\omega)}\|G(x(r))^*e_i\|_{U}^2dr\right){\bf 1}_{A}\right]P(d\omega),
\endaligned\end{equation*}
\begin{align*}
J_2&=\int_{\Omega_{0}}E^{Q_{\omega}}\left[\left((M_{s\wedge \tau(\omega),0}^{i})^2-\int_0^{s\wedge \tau(\omega)}\|G(x(r))^*e_i\|_{U}^2dr\right){\bf 1}_{A}\right]P(d\omega),
\end{align*}
\begin{equation*}\aligned J_3&=2\int_{\Omega_{0}}E^{Q_{\omega}}\left[M_{t\wedge \tau(\omega),0}^{i}\left(M_{t\wedge \tau(\omega)\vee s,0}^{i}-M_{t\wedge \tau(\omega),0}^{i}\right){\bf 1}_{A}\right]P(d\omega).
\endaligned\end{equation*}
Combining these calculations and using (\ref{Q1}) as above we finally deduce that
\begin{equation*}\aligned
&E^{P\otimes_{\tau} R}\left[\left((M_{t,0}^{i})^2-\int_0^t\|G(x(r))^*e_i\|_{U}^2dr\right){\bf 1}_{A}\right]
\\&=E^{P\otimes_{\tau} R}\left[\left((M_{s\wedge \tau,0}^{i})^2-\int_0^{s\wedge \tau}\|G(x(r))^*e_i\|_{U}^2dr\right){\bf 1}_{A}\right]\\
&\quad+E^{P\otimes_{\tau} R}\left[\left((M_{s,0}^{i}-M_{\tau,0}^{i})^2-\int_\tau^s\|G(x(r))^*e_i\|_{U}^2dr\right){\bf 1}_{A\cap \{\tau\leq s\}}\right]
\\&\quad+2E^{P\otimes_{\tau} R}\left[M_{\tau,0}^{i}\left(M_{s,0}^{i}-M_{\tau,0}^{i}\right){\bf 1}_{A\cap \{\tau\leq s\}}\right]
\\&=E^{P\otimes_{\tau} R}\left[\left((M_{s,0}^{i})^2-\int_0^s\|G(x(r))^*e_i\|_{U}^2dr\right){\bf 1}_{A}\right],
\endaligned\end{equation*}
which completes the proof of (M2).
\end{proof}

As the next step, we present an auxiliary result which allows to show that for weakly continuous stochastic processes, hitting times of  open sets are  stopping times with respect to the corresponding right continuous filtration. Here we want to emphasize that the filtration $(\mathcal{B}_t)_{t\geq0}$ used below is not the augmented one since we have to consider different probabilities. As a consequence, we have to be careful about making any conclusions about stopping times.

\begin{lem}\label{time}
Let $(\Omega,\mathcal{F},(\mathcal{F}_{t})_{t\geq 0},\mathbb{P})$ be a stochastic basis. Let $H_{1},H_{2}$ be separable Hilbert spaces such that the embedding $H_{1}\subset H_{2}$ is continuous. Suppose that there exists $\{h_k\}_{k\in\mathbb{N}}\subset H_2^{*}\subset H_1^{*}$ such that we have for $f\in H_1$
$$\|f\|_{H_{1}}=\sup_{k\in\mathbb{N}}h_k(f).$$
Suppose that $X$ is an $(\mathcal{F}_t)_{t\geq 0}$-adapted stochastic process with trajectories in $C([0,\infty);H_{2})$.
Let $L>0$ and $\alpha\in(0,1)$. Then
$$\tau_1:=\inf\{t\geq  0; \|X(t)\|_{H_{1}}>L\}\qquad\text{and}\qquad \tau_2:=\inf\{t\geq 0; \|X\|_{C_t^\alpha H_{1}}>L\}$$
are $(\mathcal{F}_{t+})_{t\geq0}$-stopping times where $\mathcal{F}_{t+}=\cap_{\varepsilon>0}\mathcal{F}_{t+\varepsilon}$.
\end{lem}

We note that in  the above result, the process $X$ a priori does not need to take values in $H_{1}$. In other words, without additional regularity of the trajectories of $X$, we simply have $\tau_{1}=\tau_{2}=0$. However, in the application of Lemma \ref{time} in the proof of Theorem \ref{Main results2} below, additional regularity will be known a.s. under a suitable probability measure.

\begin{proof}[Proof of Lemma \ref{time}]
In the proof we use $X^\omega(s)$ to denote $X(s,\omega)$. First, we  observe that the trajectories of $X$ are lower semicontinuous in $H_{1}$ in the following sense
\begin{equation}\label{lower}
\|X(t)\|_{H_{1}}=\sup_{k\in\mathbb{N}}h_k( X(t))
=\sup_{k\in\mathbb{N}}\lim_{s\to t}h_k( X(s)) \leq \liminf_{s\to t}\sup_{k\in\mathbb{N}}h_k(X(s))\leq\liminf_{s\to t}\|X(s)\|_{H_{1}},
\end{equation}
where $t\geq 0$.
Note  that since by assumption we only know that $X$ takes values in $ H_{2}\supset H_{1}$, the $H_{1}$-norms appearing in \eqref{lower} may be infinite.
Next, we have for $t>0$
 \begin{equation*}
 \aligned
 \{\tau_1\geq t\}=&\cap_{s\in[0,t]}\left\{\|X(s)\|_{H_{1}}\leq L\right\}
=\cap_{s\in[0,t]\cap \mathbb{Q}}\left\{\|X(s)\|_{H_{1}}\leq L\right\}\in\mathcal{F}_t.
\endaligned
\end{equation*}
Indeed, to show  the first equality, we observe that the right hand side is a subset of the left  one. For the converse inclusion, we know that
  $\{\tau_1> t\}$ is a subset of the right hand side.  Now, we consider $\omega\in\{\tau_1= t\}$. In this case, $\|X^{\omega}(s)\|_{H_{1}}\leq L$ for every $s\in [0,t)$. Thus, there exists  a sequence $t_k\uparrow t$ such that $\|X^\omega(t_k)\|_{H_{1}}\leq L$ and by the lower semicontinuity of $X$ it follows that $\|X^\omega(t)\|_{H_{1}}\leq L$.
The second equality is also a consequence of lower semicontinuity. Indeed, if $\omega$ belongs to the right hand side, then for $s\in[0,t]$, $s\notin\mathbb{Q}$, there is a sequence $(s_{k})_{k\in\mathbb{N}}\subset[0,t]\cap\mathbb{Q}$, $s_{k}\to s$, such that $\|X^{\omega}(s_{k})\|_{H_{1}}\leq L$. Hence $\|X^{\omega}(s)\|_{H_{1}}\leq L$ and $\omega$ belongs to the let hand side as well.
Therefore, we deduce that
$$\{\tau_{1}\leq t\}=\cap_{\varepsilon>0}\{\tau_{1}<t+\varepsilon\}\in\mathcal{F}_{t+},$$
which proves that $\tau_{1}$ is an $(\mathcal{F}_{t+})_{t\geq 0}$-stopping time.

We proceed similarly for $\tau_{2}$. By the same argument as in \eqref{lower} we obtain that also the time increments of $X$ are lower semicontinuous in $H_{1}$. More precisely, for $t_{1},t_{2}\geq 0$ we have
  $$\|X(t_1)-X(t_2)\|_{H_{1}}\leq \liminf_{{s_{1}\to t_{1},s_{2}\to t_{2}}}\|X(s_1)-X(s_2)\|_{H_{1}}$$
  and as a consequence if $t_{1}\neq t_{2}$ then
$$
  \frac{\|X(t_1)-X(t_2)\|_{H_{1}}}{|t_{1}-t_{2}|^{\alpha}}\leq \liminf_{\substack{s_{1}\to t_{1},s_{2}\to t_{2}\\s_{1}\neq s_{2}}}\frac{\|X(s_1)-X(s_2)\|_{H_{1}}}{|s_{1}-s_{2}|^{\alpha}}.
$$
  This implies  for $t>0$ that
 \begin{equation}\label{eq:1}
 \aligned
 \{\tau_2\geq t\}=&\left\{\|X\|_{C_t^{\alpha}H_{1}}\leq L\right\}
 =\cap_{s_1\neq s_2\in[0,t]\cap \mathbb{Q}}\left\{\frac{\|X(s_1)-X(s_2)\|_{H_{1}}}{|s_1-s_2|^{\alpha}}\leq L\right\}\in\mathcal{F}_t,
 \endaligned
 \end{equation}
Indeed, for  the first equality, we obtain immediately that the right hand side is a subset of the left  one, because  the process $t\mapsto \|X\|_{C^{\alpha}_{t}H_{1}}$ is nondecreasing. For the converse inclusion, we know that
  $\{\tau_2> t\}$ is a subset of the right hand side.  Let $\omega\in\{\tau_2= t\}$. Then there is  a sequence $t_k\uparrow t$ such that $\|X^{\omega}\|_{C_{t_k}^{\alpha}H}\leq L$ and we have
  \begin{align*}\sup_{s_1\neq s_2\in [0,t]}\frac{\|X^{\omega}(s_1)-X^{\omega}(s_2)\|_{H_{1}}}{|s_1-s_2|^\alpha}&\leq \sup_{s_1\neq s_2\in [0,t]}\liminf_{k\to\infty}\frac{\|X^{\omega}(s_1\wedge t_k)-X^{\omega}(s_2\wedge t_k)\|_{H_{1}}}{|s_1\wedge t_k-s_2\wedge t_k|^\alpha}\\
&  \leq \sup_{k\in\mathbb{N}} \sup_{s_1\neq s_2\in [0,t_k]}\frac{\|X^{\omega}(s_1)-X^{\omega}(s_2)\|_{H_{1}}}{|s_1-s_2|^\alpha}\leq L.
  \end{align*}
We deduce that $\|X^{\omega}\|_{C_t^{\alpha}H_{1}}\leq L$ hence $\omega$ also belongs to the set on the right hand side of the first equality in \eqref{eq:1}. The second equality in \eqref{eq:1} follows by a similar argument. Therefore, we conclude that $\tau_{2}$ is an $(\mathcal{F}_{t+})_{t\geq 0}$-stopping time.
\end{proof}

\subsection{Application to solutions obtained through Theorem \ref{Main results1}}
\label{s:1.2}

As the first step, we decompose the Navier--Stokes system (\ref{ns})  into two parts, one is linear and contains the stochastic integral, whereas the second one is a nonlinear but random PDE. More precisely, we consider
\begin{equation}\label{linear}
\aligned
 dz -\Delta z+\nabla P_1dt&=dB,
\\\div z&=0,
\\z(0)&=0,
\endaligned
\end{equation}
and
\begin{equation}\label{nonlinear}
\aligned
 \partial_tv -\Delta v+\div((v+z)\otimes (v+z))+\nabla P_{2}&=0,
\\\div v&=0,
\endaligned
\end{equation}
where by $P_{1}$ and $P_{2}$ we denoted the associated pressure terms. Note that the initial value for $v$ was not given in advance but it was part of the construction in Theorem~\ref{Main results1}.
This decomposition allows to separate the difficulties coming from the stochastic perturbation from those originating in the nonlinearity.

Now, we fix   a $GG^{*}$-Wiener process $B$ defined on a probability space $(\Omega, \mathcal{F},\mathbf{P})$ and we denote by  $(\mathcal{F}_{t})_{t\geq0}$ its normal filtration, i.e. the canonical filtration of $B$ augmented by all the $\mathbf{P}$-negligible sets. This filtration is right continuous.
We recall that using the  factorization method it is standard to derive  regularity of the stochastic convolution $z$ which solves the linear equation \eqref{linear} on $(\Omega, \mathcal{F},(\mathcal{F}_{t})_{t\geq0},\mathbf{P})$.  In particular, the following result follows from \rmb{\cite[Theorem 5.14, 5.16]{DPZ92}}  together with the Kolmogorov continuity criterion.

{\begin{prp}\label{fe z}
Suppose that $\mathrm{Tr}(GG^*)<\infty$. Then for all $\delta\in (0,1/2)$ and $T>0$
$$
E^{\mathbf{P}}\left[\|z\|_{C_{T}H^{1-\delta}}+\|z\|_{C_T^{\frac{1}{2}-\delta}L^2}\right]<\infty.
$$
\end{prp}}

As the next step, for every $\omega\in \Omega_0$ we define a process $M_{t,0}^\omega$ similarly to  Definition \ref{martingale solution}, that is,
\begin{equation}\label{eq:M}
M_{t,0}^\omega:= \omega(t)-\omega({0})+\int^t_0 [\mathbb{P} \div(\omega(r)\otimes \omega(r))-\Delta \omega(r) ]dr
\end{equation}
and for every $\omega\in \Omega_0$ we let
 \begin{equation}\label{eq:Z}
 Z^\omega(t):= M_{t,0}^\omega+\int_0^t\mathbb{P}\Delta e^{(t-r)\Delta} M_{r,0}^\omega dr.
 \end{equation}
The idea behind these definitions is as follows. The process $M$ is defined in terms of the canonical process $x$ and hence its definition makes sense for every $\omega\in \Omega_{0}$, i.e. without the reference to any probability measure.  Consequently, the same applies to $Z$. In addition, if $P$ is a martingale solution to the Navier--Stokes system \eqref{ns}, the process  $M$ is a $GG^{*}$-Wiener process under $P$. Hence we may apply an integration by parts formula to show that, the process $Z$ solves \eqref{linear} with $B$ replaced by $M$.  In other words, under $P$, $Z$ is almost surely equal to a stochastic convolution, i.e., we have
\begin{equation*}
Z(t)=\int_0^t\mathbb{P} e^{(t-r)\Delta} d M_{r,0}\qquad P\text{-a.s}.
\end{equation*}

In addition, by definition of $Z$ and $M$ together with the regularity of trajectories in $\Omega_{0}$, it follows that for every $\omega\in \Omega_0$, $Z^\omega\in C([0,\infty),H^{-3})$.
 For $n\in\mathbb{N},L>0$ and for  $\delta\in(0,1/12)$ to be determined below   we define
 \begin{equation*}
 \aligned\tau_L^n(\omega)&=\inf\left\{t\geq 0, \|Z^\omega(t)\|_{\rmb{H^{{1-\delta}}}}>\frac{(L-\frac{1}{n})^{1/4}}{C_S}\right\}\bigwedge \inf\left\{t>0,\|Z^\omega\|_{C_t^{\frac{1}{2}-2\delta}\rmb{L^2}}>\frac{(L-\frac{1}{n})^{1/2}}{C_S}\right\}\bigwedge L,
 \endaligned
 \end{equation*}
 where $C_S$ is the Sobolev constant for $\|f\|_{L^\infty}\leq C_S \|f\|_{H^{\frac{3+\sigma}{2}}}$ with $\sigma>0$.
We observe that the sequence $(\tau^{n}_{L})_{n\in\mathbb{N}}$ is nondecreasing and define
\begin{equation}\label{eq:tauL}
\tau_L:=\lim_{n\rightarrow\infty}\tau_L^n.
\end{equation}
Note that without additional regularity of the trajectory $\omega$, it holds true that $\tau^{n}_{L}(\omega)=0$.
However, under $P$ we may use the regularity assumption on $G$ to deduce that $Z\in {CH^{1-\delta}\cap C^{1/2-\delta}_{\mathrm{loc}}L^2}$ $P$-a.s.
By Lemma \ref{time} we obtain that $\tau_L^{n}$ is a $(\mathcal{B}_t)_{t\geq0}$-stopping time and consequently also  $\tau_L$ is a $(\mathcal{B}_t)_{t\geq 0}$-stopping time as an increasing limit of stopping times.  {We emphasize that we need to introduce the stopping time on the path space without using any probability. The introduction of $\tau_L^n$ is to approximate $\tau_L$ which coincides with the stopping time $T_L$ introduced in \eqref{stopping time} below under  the law of the convex integration solution. Moreover, $\tau_L^n$ is defined on the path space and is a stopping time by Lemma \ref{time}. We cannot directly prove that $T_L$ is stopping time on the path space without using continuity property of the Brownian motion.}

As the next step,  we apply Theorem~\ref{Main results1}  on the stochastic basis $(\Omega,\mathcal{F},(\mathcal{F}_{t})_{t\geq0},\mathbf{P})$. We note that the stopping time $\mathfrak{t}$ from the statement of Theorem~\ref{Main results1} is given by $T_{L}$ for a sufficiently large $L>1$,
 defined in \eqref{stopping time} below. We recall that $u$ is adapted  to $(\mathcal{F}_{t})_{t\geq0}$ which is an essential property employed in the sequel. We denote by $P$ the law of $u$ and prove the following result.

\begin{prp}\label{prop:ext}
The probability measure $P$ is a martingale solution to the Navier--Stokes system \eqref{ns} on $[0,\tau_{L}]$ in the sense of Definition \ref{def:martsol}, where $\tau_{L}$ was defined in \eqref{eq:tauL}.
\end{prp}

\begin{proof}
Recall that the stopping time $T_{L}$ was defined in \eqref{stopping time} in terms of the process $z$, the solution to the linear equation \eqref{linear}. Theorem~\ref{Main results1} yields the existence of a solution $u$ to the Navier--Stokes system \eqref{ns} on $[0,T_{L}]$  such that $u(\cdot\wedge T_{L})\in \Omega_{0}$ $\mathbf{P}$-a.s. We will now prove that
 \begin{equation}\label{eq}\tau_L(u)=T_L \quad \mathbf{P}\text{-a.s}.\end{equation}
To this end, we observe that due to the definition of $M$ in \eqref{eq:M} and $Z$ in \eqref{eq:Z} together with the fact that $u$ solves the Navier--Stokes system \eqref{ns} on $[0,T_{L}]$, we have
  \begin{equation}\label{eq1}
  Z^u(t)=z(t) \quad\textrm{ for } t\in[0,T_L] \ \mathbf{P}\text{-a.s.}
  \end{equation}
Since $z\in{CH^{1-\delta}\cap C^{1/2-\delta}_{\mathrm{loc}}L^2}$ $\mathbf{P}$-a.s. according to Proposition \ref{fe z}, the trajectories of the processes
$$t\mapsto \|z(t)\|_{{H^{1-\delta}}}\quad\text{and}\quad t\mapsto \|z\|_{C^{\frac12-2\delta}_{t}{L^2}}$$
are  $\mathbf{P}$-a.s. continuous.
 It follows from the definition of $T_L$ that one of the following three statements holds $\mathbf{P}$-a.s.:
$$
\text{either }\  T_L=L\ \text{ or }\  \|z(T_L)\|_{{H^{1-\delta}}}\geq L^{1/4}/C_S \ \text{ or }\  \|z\|_{C_{T_L}^{\frac{1}{2}-2\delta}{L^2}}\geq L^{1/2}/C_S.
$$
Therefore, as a consequence of (\ref{eq1}), we deduce that $\tau_L(u)\leq T_L$ $\mathbf{P}$-a.s. Suppose now  that $\tau_L(u)<T_L$ holds true on a set of positive probability $\mathbf{P}$. Then it holds on this set that
$$
\|z(\tau_L(u))\|_{{H^{1-\delta}}}=\|Z^u(\tau_L(u))\|_{{H^{1-\delta}}}\geq L^{1/4}/C_S\ \text{ or }\  \|Z^u\|_{C_{\tau_L(u)}^{\frac{1}{2}-2\delta}{L^2}}=\|z\|_{C_{\tau_L(u)}^{\frac{1}{2}-2\delta}{L^2}}\geq L^{1/2}/C_S,
$$
which however contradicts the definition of $T_L$. Hence we have proved  (\ref{eq}).

Recall that $\tau_{L}$ is a $(\mathcal{B}_{t})_{t\geq 0}$-stopping time. We intend to show that  $P$ is a martingale solution to the Navier--Stokes system \eqref{ns} on $[0,\tau_{L}]$ in the sense of Definition \ref{def:martsol}. First, we observe that it can be seen from the construction in Theorem \ref{Main results1} that the initial value $u(0)=v(0)+z(0)=v(0)$ is indeed deterministic.   Hence the condition (M1) follows. However, we note that the initial value $v(0)$ cannot be prescribed in advance. In other words, Theorem \ref{Main results1} does not yield a solution to the Cauchy problem, it only provides the existence of an initial condition for which a solution violating the energy inequality exists.
For an appropriate choice of the constant $C_{t,q}$ in Definition \ref{def:martsol}, which has to depend on the constant $C_{L}$ in \eqref{eq:vv} in Theorem \ref{Main results1}, the condition (M3) also follows.

Let us now  verify (M2). To this end, let $s\leq t$ and let  $g$ be a bounded and real valued $\mathcal{B}_s$-measurable and continuous function on $\Omega_0$. Since $u(\cdot\wedge T_{L})$   is an $(\mathcal{F}_t)_{t\geq0}$-adapted process and \eqref{eq} holds, we deduce that $u(\cdot\wedge \tau_{L}(u))$ is also $(\mathcal{F}_t)_{t\geq0}$-adapted.  Consequently, the composition $g(u(\cdot\wedge\tau_{L}(u)))$ is ${\mathcal{F}}_s$-measurable. On the other hand, we know that under $\mathbf{P}$, $M^{u,i}_{t\wedge \tau_L(u),0}=\langle B_{t\wedge \tau_L(u)},e_i\rangle$ is an $(\mathcal{F}_t)_{t\geq0}$-martingale. Its quadratic variation process is given by $\|Ge_{i}\|_{L^{2}}^{2}(t\wedge\tau_{L}(u))$. Therefore, we have
$$E^P[M^{i}_{t\wedge \tau_L,0}g]=E^{\mathbf{P}}[M^{u,i}_{t\wedge \tau_L(u),0}g(u)]
=E^{\mathbf{P}}[M^{u,i}_{s\wedge \tau_L(u),0}g(u)]=E^P[M^{i}_{s\wedge \tau_L,0}g]$$
and by similar arguments we also obtain that
$$E^P\left[\left((M^{i}_{t\wedge \tau_L,0})^2-(t\wedge\tau_L)\|Ge_{i}\|_{L^{2}}^{2}\right)g\right]=E^P\left[\left((M^{i}_{s\wedge \tau_L,0})^2-(s\wedge\tau_L)\|Ge_{i}\|_{L^{2}}^{2}\right)g\right].$$
Accordingly,  the process $M_{t\wedge\tau_L,0}^{i}$ is a   continuous square integrable $(\mathcal{B}_t)_{t\geq 0}$-martingale under $P$ with  the quadratic variation process
given by
$\|Ge_i\|_{L^2}^2(t\wedge\tau_L)$ and  (M2) in   Definition \ref{def:martsol} follows.
\end{proof}

At this point, we are already able to deduce that martingale solutions on $[0,\tau_{L}]$ in the sense of Definition \ref{def:martsol} are not unique. However, we aim at a stronger result, namely that globally defined martingale solutions on $[0,\infty)$ in the sense of Definition \ref{martingale solution} are not unique. Moreover, we will prove that
  for an arbitrary time interval $[0,T]$, the martingale solutions on $[0,T]$ are not unique.
To this end, we will  extend $P$ to a martingale solution on $[0,\infty)$ through the procedure developed in Section \ref{ss:gen}.
More precisely, as an immediate corollary of Proposition \ref{prop:ext} and the fact that $\tau_{L}$ is a $(\mathcal{B}_{t})_{t\geq0}$-stopping time, we may apply Proposition \ref{prop:1}. In particular, we construct $Q_{\omega}$ for all $\omega\in\Omega_{0}$.
In view of Proposition \ref{prop:2},  (M1-M3) follows once we verify the condition \eqref{Q1} for $Q_{\omega}$. This will be achieved in the following result.

\begin{prp}\label{prp:ext2}
The probability measure $P\otimes_{\tau_{L}}R$ is a martingale solution to the Navier--Stokes system \eqref{ns} on $[0,\infty)$ in the sense of Definition \ref{martingale solution}.
\end{prp}

\begin{proof}
In light of Proposition \ref{prop:1} and Proposition \ref{prop:2}, it only remains to establish \eqref{Q1}.
Due to \eqref{eq} and \eqref{eq1}, we know that
\begin{align*}
\begin{aligned}
P\left(\omega:Z^\omega(\cdot\wedge \tau_L(\omega))\in C{H^{1-\delta}}\cap C^{\frac{1}{2}-\delta}_{\mathrm{loc}}{L^2}\right)&=\mathbf{P}\left(Z^u(\cdot\wedge \tau_L(u))\in C{H^{1-\delta}}\cap C^{\frac{1}{2}-\delta}_{\mathrm{loc}}{L^2}\right)\\
&=\mathbf{P}\left(z(\cdot\wedge T_L)\in C{H^{1-\delta}}\cap C^{\frac{1}{2}-\delta}_{\mathrm{loc}}{L^2}\right)=1.
\end{aligned}
\end{align*}
This means that  there exists a $P$-measurable set $\mathcal{N}\subset \Omega_{0,\tau_L}$ such that $P(\mathcal{N})=0$ and for $\omega\in \mathcal{N}^c$
\begin{equation}\label{continuity1}
Z^\omega_{\cdot\wedge \tau_L(\omega)}\in C{H^{1-\delta}}\cap C^{\frac{1}{2}-\delta}_{\mathrm{loc}}{L^2}.
\end{equation}
On the other hand, it follows from \eqref{eq:Z} that for every $\omega'\in\Omega_0$
\begin{align*}Z^{\omega'}(t)-Z^{\omega'}(t\wedge\tau_L(\omega))&=M^{\omega'}_{t,0}-e^{(t-t\wedge\tau_L(\omega))\Delta}M^{\omega'}_{t\wedge \tau_L(\omega),0}+\int_{t\wedge\tau_L(\omega)}^t\mathbb{P}\Delta e^{(t-s)\Delta}M_{s,0}^{\omega'} ds
\\&\quad+(e^{(t-t\wedge\tau_L(\omega))\Delta}-I)\left[M^{\omega'}_{t\wedge\tau_L(\omega),0}+\int_0^{t\wedge \tau_L(\omega)}\mathbb{P}\Delta e^{(t\wedge\tau_L(\omega)-s)\Delta}M_{s,0}^{\omega'} ds\right]\\&=\mathbb{Z}_{\tau_L(\omega)}^{\omega'}(t)+(e^{(t-t\wedge\tau_L(\omega))\Delta}-I)Z^{\omega'}(t\wedge\tau_L(\omega)),\end{align*}
with
\begin{align*}
\mathbb{Z}^{\omega'}_{\tau_L(\omega)}(t)&=M^{\omega'}_{t,0}-e^{(t-t\wedge\tau_L(\omega))\Delta}M^{\omega'}_{t\wedge \tau_L(\omega),0}+\int_{t\wedge\tau_L(\omega)}^t\mathbb{P}\Delta e^{(t-s)\Delta}M_{s,0}^{\omega'} ds
\\&=M^{\omega'}_{t,0}-M^{\omega'}_{t\wedge \tau_L(\omega),0}+\int_{t\wedge\tau_L(\omega)}^t\mathbb{P}\Delta e^{(t-s)\Delta}(M_{s,0}^{\omega'}-M_{s\wedge\tau_L(\omega),0}^{\omega'}) ds.\end{align*}
Since $M_{\cdot,0}-M_{\cdot\wedge \tau_L(\omega),0}$ is $\mathcal{B}^{\tau_L(\omega)}$-measurable,
we know that $\mathbb{Z}^{\omega'}_{\tau_L(\omega)}$ is $\mathcal{B}^{\tau_L(\omega)}$-measurable.

Using  \eqref{qomega} and \eqref{qomega2} it holds that for all $\omega\in\Omega_0$
\begin{equation*}\aligned
& Q_\omega\left(\omega'\in\Omega_{0}; Z^{\omega'}_{\cdot}\in C{H^{1-\delta}}\cap C^{\frac{1}{2}-\delta}_{\mathrm{loc}}{L^2}\right)
\\&=Q_\omega\left(\omega'\in\Omega_{0}; Z^{\omega'}_{\cdot\wedge \tau_L(\omega)}\in C{H^{1-\delta}}\cap C^{\frac{1}{2}-\delta}_{\mathrm{loc}}{L^2}, \mathbb{Z}_{\tau_L(\omega)}^{\omega'}\in C{H^{1-\delta}}\cap C^{\frac{1}{2}-\delta}_{\mathrm{loc}}{L^2}\right)
\\&=\delta_\omega\left(\omega'\in\Omega_{0}; Z^{\omega'}_{\cdot\wedge \tau_L(\omega)}\in C{H^{1-\delta}}\cap C^{\frac{1}{2}-\delta}_{\mathrm{loc}}{L^2}\right) \\
&\qquad\times  R_{\tau_L(\omega),x(\tau_L(\omega),\omega)}\left(\omega'\in\Omega_{0}; \mathbb{Z}_{\tau_L(\omega)}^{\omega'}\in C{H^{1-\delta}}\cap C^{\frac{1}{2}-\delta}_{\mathrm{loc}}{L^2}\right).\endaligned\end{equation*}
Here the first factor on the right hand side equals to $1$ for all  $\omega\in \mathcal{N}^c$ due to \eqref{continuity1}.
Since for $\omega\in \{x(\tau)\in L^2_\sigma\}$, $R_{\tau_L(\omega),x(\tau_L(\omega),\omega)}$ is a martingale solution to the Navier--Stokes system \eqref{ns} starting at the deterministic time $\tau_{L}(\omega)$ from the deterministic initial condition $x(\tau_{L}(\omega),\omega)$, the process $\omega'\mapsto M_{\cdot,0}^{\omega'}-M^{\omega'}_{\cdot\wedge \tau_L(\omega),0}$ is a $GG^{*}$-Wiener process starting from $\tau_L(\omega)$ with respect to $(\mathcal{B}_{t})_{t\geq0}$ under the measure $R_{\tau_L(\omega),x(\tau_L(\omega),\omega)}$.
Due to the regularity of its covariance we deduce that also the second factor equals to $1$.
Indeed,  we have for $R_{\tau_L(\omega),x(\tau_L(\omega),\omega)}$-a.e. $\omega'$ that
$$\mathbb{Z}^{\omega'}_{\tau_L(\omega)}(t)=\int_{0}^t\mathbb{P}e^{(t-s)\Delta}d(M_{s,0}^{\omega'}-M_{s\wedge\tau_L(\omega),0}^{\omega'})$$
and the regularity of this stochastic convolution follows again from Proposition \ref{fe z}.
In particular, it holds for $R_{\tau_L(\omega),x(\tau_L(\omega),\omega)}$-a.e. $\omega'$ that
$${\mathbb{Z}}^{\omega'}_{\tau_L(\omega)}\in C{H^{1-\delta}}\cap C^{\frac{1}{2}-\delta}_{\mathrm{loc}}{L^2}.$$
To summarize, we have proved that  for all  $\omega\in \mathcal{N}^c\cap \{x(\tau)\in L^2_\sigma\}$
$$
Q_\omega\left(\omega'\in\Omega_{0}; Z^{\omega'}_{\cdot}\in C{H^{1-\delta}}\cap C^{\frac{1}{2}-\delta}_{\mathrm{loc}}{L^2}\right)=1.
$$

As a consequence, for all
$\omega\in \mathcal{N}^c\cap \{x(\tau)\in L^2_\sigma\}$ there exists a measurable set $N_\omega$ such that $Q_\omega(N_\omega)=0$ and for all $\omega'\in N_\omega^c$ the trajectory
$t\mapsto Z^{\omega'}(t)$ belongs to $C{H^{1-\delta}}\cap C^{\frac{1}{2}-\delta}_{\mathrm{loc}}{L^2}$. Therefore, by \eqref{eq:tauL} we obtain that
 $\tau_L(\omega')=\bar{\tau}_L(\omega')$ for all $\omega'\in N_\omega^c$ where
 \begin{equation*}\bar{\tau}_L(\omega'):=\inf\left\{t\geq 0, \|Z^{\omega'}(t)\|_{{H^{1-\delta}}}\geq L^{1/4}/C_{S}\right\}\bigwedge\inf\left\{t\geq 0,\|Z^{\omega'}\|_{C_t^{1/2-2\delta}{L^2}}\geq L^{1/2}/C_{S}\right\}\bigwedge L.\end{equation*}
This implies that for $t< L$
\begin{equation}\label{mea}\aligned
\left\{\omega'\in N_\omega^c,\tau_L(\omega')\leq t\right\}&=\left\{\omega'\in N_\omega^c, \sup_{s\in\mathbb{Q},s\leq t}
\|Z^{\omega'}(s)\|_{{H^{1-\delta}}}\geq L^{1/4}/C_S\right\}
\\&\qquad\bigcup\left\{\omega'\in N_\omega^c, \sup_{s_1\neq s_2\in \mathbb{Q}\cap [0,t]}\frac{\|Z^{\omega'}(s_1)-Z^{\omega'}(s_2)\|_{{L^2}}}{|s_1-s_2|^{\frac{1}{2}-2\delta}}\geq L^{1/2}/C_S\right\}
\\&:= N^c_\omega \cap A_t.\endaligned
\end{equation}
Finally, we deduce that for all $\omega\in\mathcal{N}^c\cap \{x(\tau)\in L^2_\sigma\}$ with $P(x(\tau)\in L^2_\sigma)=1$
\begin{equation}\label{Q}
\aligned
&Q_\omega\big(\omega'\in\Omega_{0}; \tau_L(\omega')=\tau_L(\omega)\big)=Q_\omega\big(\omega'\in N_\omega^c; \tau_L(\omega')=\tau_L(\omega)\big)
\\&\quad=Q_\omega\big(\omega'\in N_\omega^c; \omega'(s)=\omega(s), 0\leq s\leq \tau_L(\omega), \tau_L(\omega')=\tau_L(\omega)\big)=1,
\endaligned
\end{equation}
where we used \eqref{qomega} and the fact that (\ref{mea}) implies
$$\{\omega'\in N_\omega^c; \bar{\tau}_L(\omega')=\tau_L(\omega)\}=N_\omega^c\cap (A_{\tau_L(\omega)}\backslash (\cup_{n=1}^\infty A_{\tau_L(\omega)-\frac1n}))\in N_\omega^c\cap \mathcal{B}_{\tau_L(\omega)}^0,$$
and  $Q_\omega(A_{\tau_L(\omega)}\backslash (\cup_{n=1}^\infty A_{\tau_L(\omega)-\frac1n}))=1$.
This verifies the condition \eqref{Q1} in Proposition \ref{prop:2} and as a consequence $P\otimes_{\tau_{L}}R$ is a martingale solution to the Navier--Stokes system \eqref{ns} on $[0,\infty)$ in the sense of Definition \ref{martingale solution}.
\end{proof}

\begin{rem}
The property \eqref{Q} is essential for showing that the concatenated probability measure satisfies (M1-M3). This is the reason why we had to introduce $\bar{\tau}_{L}$ and make use of the continuity of $Z$ under the law of a martingale solution, which is different from the original regularity of $Z$ which follows merely from its  definition \eqref{eq:Z} together with the regularity of trajectories in $\Omega_{0}$. Without the improved regularity, we could only prove that $\tau_{L}$ is a stopping time with respect to the right continuous filtration $(\mathcal{B}_{t})_{t\geq 0}$ and the dependence on the right limit does not allow to establish \eqref{Q}.
\end{rem}

Finally, we have all in hand to conclude the proof of our main result, Theorem \ref{Main results2}.

\begin{proof}[Proof of Theorem  \ref{Main results2}]
Let $T>0$ be arbitrary, let $\kappa=1/2$ and $K=2$. Based on Theorem~\ref{Main results1} and Proposition~\ref{prp:ext2} there exists $L>1$ and a measure $P\otimes_{\tau_{L}}R$ which is a martingale solution to the Navier--Stokes system \eqref{ns} on $[0,\infty)$ and it coincides  on the random interval $[0,\tau_{L}]$ with the law of the solution constructed through Theorem~\ref{Main results1}.  The martingale solution $P\otimes_{\tau_{L}}R$ starts from certain deterministic initial value $x_{0}=v(0)\in L^{2}_{\sigma}$ dictated by the construction in  Theorem~\ref{Main results1}. The key result is the failure of the energy inequality at time $T$ formulated in \eqref{eq:env} on the set $\{T_{L}\geq T\}\subset \Omega$. In view of \eqref{eq:PR}, \eqref{Q} and \eqref{eq}, we obtain
%
by \eqref{eq:env} and the choice of $K=2$
$$
\aligned
E^{P\otimes_{\tau_{L}}R}\big[\|x(T)\|^{2}_{L^{2}}\big]&= E^{P\otimes_{\tau_{L}}R}\big[{\bf 1}_{\{\tau_{L}\geq T\}}\|x(T)\|^{2}_{L^{2}}\big] + E^{P\otimes_{\tau_{L}}R}\big[{\bf 1}_{\{\tau_{L}< T\}}\|x(T)\|^{2}_{L^{2}}\big]\\
&\geq \int_{\Omega_{0}}E^{Q_{\omega}}[{\bf 1}_{\{\tau_{L}\geq T\}}\|x(T)\|^{2}_{L^{2}}]P(d\omega)> 2\left( \|x_{0}\|^{2}_{L^{2}} +T\,\mathrm{Tr}(GG^{*})\right).
\endaligned
$$

On the other hand, by a classical compactness argument based on a Galerkin approximation we may  construct another martingale solution $\tilde P$ which starts from the same deterministic initial condition $x_{0}$ and which satisfies the energy inequality
$$
E^{\tilde P}\big[\|x(T)\|^{2}_{L^{2}}\big]\leq \|x_{0}\|^{2}_{L^{2}} +T\,\mathrm{Tr}(GG^{*}).
$$

Therefore, we can finally conclude that the two martingale solutions $P\otimes_{\tau_{L}}R$ and $\tilde P$ are distinct and non-uniqueness in law holds for the Navier--Stokes system \eqref{ns}.
\end{proof}

\section{Proof of Theorem \ref{Main results1}}
\label{s:1.1}

In this section we fix a probability space $(\Omega,\mathcal{F},\mathbf{P})$ and let $B$ be a $GG^*$-Wiener process on $(\Omega,\mathcal{F},\mathbf{P})$. We let $(\mathcal{F}_t)_{t\geq0}$ be the normal filtration generated by $B$, that is, the  canonical right continuous filtration augmented by all the $\mathbf{P}$-negligible sets.  In order to verify that the solution constructed in this section is a martingale solution before a suitable stopping time, it is essential that the solution is adapted to the filtration $(\mathcal{F}_t)_{t\geq0}$, which corresponds to a probabilistically strong solution. In the following, we construct a probabilistically strong solution before a stopping time. Furthermore, the solution does not satisfy the energy inequality.

We intend to develop an iteration procedure leading to the proof of Theorem \ref{Main results1}. More precisely, we apply the convex integration method to the nonlinear equation (\ref{nonlinear}). The iteration is indexed by a parameter $q\in\mathbb{N}_{0}$. We consider an increasing sequence $\{\lambda_q\}_{q\in\mathbb{N}}\subset \mathbb{N}$ which diverges to $\infty$, and a sequence $\{\delta_q\}_{q\in \mathbb{N}}\subset(0,1)$  which is decreasing to $0$. We choose $a\in\mathbb{N}$, $b\in\mathbb{N}$,  $\beta\in (0,1)$ and let
$$\lambda_q=a^{(b^q)}, \quad \delta_q=\lambda_q^{-2\beta},$$
where $\beta$ will be chosen sufficiently small and $a$ as well as $b$ will be chosen sufficiently large. At each step $q$, a pair $(v_q, \mathring{R}_q)$ is constructed solving the following system
\begin{equation}\label{induction}
\aligned
 \partial_tv_q-\Delta v_q +\div((v_q+{z_q})\otimes (v_q+{z_q}))+\nabla p_q&=\div \mathring{R}_q,
\\
\div v_q&=0,
\endaligned
\end{equation}
{with $z_q=\mP_{\leq f(q)}z$ for $f(q)=\lambda_{q+1}^{\alpha/8}$ and $\mP_{\leq f(q)}$ being the a Fourier multiplier operator, which projects a function onto its Fourier frequencies $\leq f(q)$ in absolute value. Hence, in addition to the conditions below, we need $b\alpha/8\in \mN$ such that $f(q)\in \mN$.}
By the Sobolev embedding we know $\|f\|_{L^\infty}\leq C_S\|f\|_{H^{\frac{3+\sigma}{2}}}$ for $\sigma>0$, where we choose $C_S\geq 1$.
Define for  $L>1$ and $0<\delta<1/12$
\begin{equation}\label{stopping time}
T_L:=\inf\{t\geq0, \|z(t)\|_{{H^{1-\delta}}}\geq L^{1/4}/C_S\}\wedge \inf\{t\geq0,\|z\|_{C_t^{1/2-2\delta}{L^2}}\geq L^{1/2}/C_S\}\wedge L.
\end{equation}
According to Proposition \ref{fe z}, the stopping time $T_{L}$ is $\mathbf{P}$-a.s. strictly positive and it holds that $T_{L}\uparrow\infty$ as $L\rightarrow\infty$ $\mathbf{P}$-a.s.
Moreover, for  $t\in[0, T_L]$
{\begin{equation}\label{z}
\aligned
\| z(t)\|_{H^{1-\delta}}\leq L^{1/4}/C_S, \quad& \|z\|_{C_t^{\frac{1}{2}-2\delta}L^2}\leq L^{1/2}/C_S,\quad \|z_q\|_{C_t^{\frac{1}{2}-2\delta}L^2}\leq L^{1/2}/C_S
\\
\| z_q(t)\|_{L^\infty}\leq L^{1/4}{\lambda_{q+1}^{\frac\alpha8}}, \quad&\|\nabla z_q(t)\|_{L^\infty}\leq L^{1/4}{\lambda_{q+1}^{\frac\alpha4}}, \quad \|z_q\|_{C_t^{\frac{1}{2}-2\delta}L^\infty}\leq {\lambda_{q+1}^{\frac\alpha4}}L^{1/2}.
\endaligned
\end{equation}}
Let $M_{0}(t)=L^{4}e^{4Lt}$. By induction on $q$ we  assume the following bounds for the iterations $(v_q,\mathring{R}_q)$: if $t\in[0, T_L]$ then
\begin{equation}\label{inductionv}
\aligned
\|v_q\|_{C_{t}L^2}&\leq M_0(t)^{1/2}(1+\sum_{1\leq r\leq q}\delta_{r}^{1/2})\leq 2M_0(t)^{1/2} ,
\\ \|v_q\|_{C^1_{t,x}}&\leq M_0(t)^{1/2}\lambda_q^4,
\\
\|\mathring{R}_q\|_{C_{t}L^1}&\leq M_0(t)c_R\delta_{q+1}.
\endaligned
\end{equation}
Here we defined $\sum_{1\leq r\leq 0}:=0$,  $c_R>0$ is a sufficiently small universal constant given in (\ref{estimate a}) and \eqref{estimate wqp} below. In addition, we used $\sum_{r\geq 1}\delta_{r}^{1/2}\leq \sum_{r\geq1}a^{-rb\beta}=\frac{a^{-\beta b}}{1-a^{-\beta b}}<1/2$ which boils down  to the requirement
\begin{equation}\label{aaa}
a^{\beta b}>3,
\end{equation}
which we assume from now on.
The iteration will be initiated through the following result which also establishes compatibility conditions between the parameters $L,a,\beta,b$ essential for  the sequel.

\begin{lem}\label{lem:v0}
For $L>1$ define
$$
v_0(t,x)=\frac{L^2e^{2Lt}}{(2\pi)^{\frac{3}{2}}}\left(\sin(x_3),0,0\right).
$$
Then the associated Reynolds stress is given by\footnote{We denote by $\mathring{\otimes}$  the trace-free part of the tensor product.}
\begin{equation}\label{eq:R0}
\mathring{R}_0(t,x)=\frac{(2L+1)L^{2}e^{2Lt}}{(2\pi)^{3/2}}\left(
\begin{array}{ccc} 0 & 0 &-\cos(x_3)
\\ 0 & 0 &0\\ -\cos(x_3) &0 &0
 \end{array}
\right)+v_0\mathring{\otimes}\rmb{ z_0+ z_0\mathring{\otimes} v_0+z_0\mathring{\otimes} z_0}.
\end{equation}
Moreover,   all the estimates in \eqref{inductionv} on the level $q=0$ for $(v_{0},\mathring{R}_{0})$ as well as \eqref{aaa} are valid  provided
\begin{equation}\label{c:a3}
45\cdot (2\pi)^{3/2}<5\cdot (2\pi)^{3/2}a^{2\beta b}\leq c_{R}L\leq c_{R}\left(\frac{(2\pi)^{3/2}a^{4}}{2}-1\right).
\end{equation}
In particular, we require
\begin{equation}\label{c:a2}
c_{R}L>45\cdot (2\pi)^{3/2}.
\end{equation}
Furthermore, the initial values $v_0(0,x)$ and $\mathring{R}_0(0,x)$ are deterministic.
\end{lem}

\begin{proof}
The first bound in \eqref{inductionv} follows immediately since
$$
\|v_{0}(t)\|_{L^{2}}= \frac{L^2e^{2Lt}}{\sqrt{2}}\leq M_{0}(t)^{1/2}.
$$
For the second bound, we have
$$
\|v_{0}\|_{C^{1}_{t,x}}\leq M_{0}(t)^{1/2}\frac{2(1+L)}{(2\pi)^{3/2}}\leq M_{0}(t)^{1/2}\lambda_{0}^{4}=M_{0}(t)^{1/2}a^{4}
$$
provided
\begin{equation}\label{c:2}
\frac{2(1+L)}{(2\pi)^{3/2}}\leq a^{4}.
\end{equation}
A direct computation implies that the corresponding Reynolds stress is given by \eqref{eq:R0} 
and we obtain
$$
\|\mathring{R}_0(t)\|_{L^{1}}\leq (2\pi)^{3/2} M_{0}(t)^{1/2}2(2L+1)+2M_{0}(t)^{1/2}L^{1/4}+L^{1/2}.
$$
Therefore, the desired third bound in \eqref{inductionv} holds  provided
$$
\|\mathring{R}_0(t)\|_{L^{1}}\leq 5\cdot (2\pi)^{3/2} M_0(t)/L\leq M_{0}(t)c_{R}\delta_{1}=M_{0}(t)c_{R}a^{-2\beta b},
$$
which requires $5\cdot (2\pi)^{3/2}L^{-1}\leq c_{R}a^{-2\beta b}$. Here we used \eqref{c:a2} in the first inequality.  Combining this condition with \eqref{c:2}, we obtain the requirement
\begin{equation}\label{c:a}
5\cdot (2\pi)^{3/2}a^{2\beta b}\leq c_{R}L\leq c_{R}\left(\frac{(2\pi)^{3/2}a^{4}}{2}-1\right).
\end{equation}
In particular, we require that
\begin{equation}\label{c:a1}
c_{R}L>5\cdot (2\pi)^{3/2},
\end{equation}
otherwise the left inequality in \eqref{c:a} cannot be fulfilled.  Under these conditions, all the estimates in \eqref{inductionv} are valid on the level $q=0$. Taking into account \eqref{aaa}, the conditions \eqref{c:a} and \eqref{c:a1} are strengthened to \eqref{c:a3} and \eqref{c:a2} from the statement of the lemma  and the proof is complete.
\end{proof}

The key result of this section  which is used to prove Theorem \ref{Main results1} is the following.

\begin{prp}\label{main iteration}
\emph{(Main iteration)}
Let  $L>1$ satisfying \eqref{c:a2} be given and let $(v_q,\mathring{R}_q)$ be an $(\mathcal{F}_t)_{t\geq0}$-adapted solution to \eqref{induction} satisfying \eqref{inductionv}. Then there exists a choice of parameters $a,b,\beta$ such that \eqref{c:a3} is fulfilled and  there exist $(\mathcal{F}_t)_{t\geq0}$-adapted processes
$(v_{q+1},\mathring{R}_{q+1})$ which solve \eqref{induction}, obey \eqref{inductionv} at level $q+1$ and  for $t\in[0,T_{L}]$ we have
\begin{equation}\label{iteration}
\|v_{q+1}(t)-v_q(t)\|_{L^2}\leq M_0(t)^{1/2}\delta_{q+1}^{1/2}.
\end{equation}
Furthermore, if $v_q(0), \mathring{R}_q(0)$ are deterministic, so are $v_{q+1}(0), \mathring{R}_{q+1}(0)$.
\end{prp}

The proof of  Proposition \ref{main iteration} is presented in Section \ref{ss:it}. At this point, we take  Proposition \ref{main iteration} for granted and  apply it in order to complete the proof of Theorem \ref{Main results1}.

\begin{proof}[Proof of Theorem \ref{Main results1}]
The proof relies on the above described iteration procedure. More precisely, our goal is to prove that for $L>1$ satisfying \eqref{c:a2}, Lemma \ref{lem:v0} and Proposition~\ref{main iteration} give rise to an $(\mathcal{F}_{t})_{t\geq0}$-adapted analytically weak solution  $v$ to the transformed problem \eqref{nonlinear}. By possibly increasing the value of $L$, the corresponding solution $v$ fails a suitable energy inequality at the given time $T$. Finally, again by  possibly  making $L$ bigger, we verify that $u:=v+z$ and $\mathfrak{t}:=T_{L}$ fulfill all the requirements in the statement of the theorem.

Starting from $(v_0,\mathring{R}_0)$ given in Lemma \ref{lem:v0}, the iteration Proposition \ref{main iteration} yields a sequence $(v_q, \mathring{R}_q)$ satisfying (\ref{inductionv}) and (\ref{iteration}).
 By interpolation we deduce that the following series is summable for $\gamma\in (0,\frac{\beta}{4+\beta})$, $t\in [0,T_L]$
 \begin{equation*}
 \sum_{q\geq0}\|v_{q+1}(t)-v_q(t)\|_{H^{\gamma}}\lesssim \sum_{q\geq0}\|v_{q+1}(t)-v_q(t)\|_{L^2}^{1-\gamma}\|v_{q+1}(t)-v_q(t)\|_{H^1}^{\gamma}\lesssim M_0(t)\sum_{q\geq0}\delta_{q+1}^{\frac{1-\gamma}{2}}\lambda_{q+1}^{4\gamma}\lesssim M_0(t).
 \end{equation*}
 Thus  we obtain a limiting solution $v=\lim_{q\rightarrow\infty}v_q$, which lies in $C([0,T_L],H^{\gamma})$.  Since $v_q$ is $(\mathcal{F}_t)_{t\geq0}$-adapted for every $q\geq0$, the limit
$v$ is $(\mathcal{F}_t)_{t\geq0}$-adapted as well.
   Furthermore,  $v$ is an analytically  weak solution to (\ref{nonlinear})  since it holds $\lim_{q\rightarrow\infty}\mathring{R}_q=0$ in $C([0,T_L];L^1)$ \rmb{and $\lim_{q\rightarrow\infty}z_q=z$ in $C([0,T_L];L^2)$}. In addition, there exists a deterministic constant $C_{L}$ such that
\begin{equation}\label{eq:vvv}
 \|v(t)\|_{H^{\gamma}}\leq C_{L}
\end{equation}
holds true for all $t\in[0,T_{L}]$.

Let us now show that the constructed solution $v$ fails the corresponding energy inequality at time $T$. Namely,  we will  show
\begin{equation}\label{eq:env1}
\|v(T)\|_{L^{2}}> (\|v(0)\|_{L^{2}}+L)e^{LT}.
\end{equation}
According to \eqref{iteration}, in view of $b^{q+1}\geq b(q+1)$ which holds if $b\geq 2$ and then applying \eqref{aaa}, we obtain for all $t\in[0,T_{L}]$
$$
\|v(t)-v_{0}(t)\|_{L^{2}}\leq\sum_{q\geq0}\|v_{q+1}(t)-v_{q}(t)\|_{L^{2}}\leq M_{0}(t)^{1/2}\sum_{q\geq0}\delta_{q+1}^{1/2}\leq M_{0}(t)^{1/2}\sum_{q\geq0} (a^{-\beta b})^{q+1}
$$
$$
=  M_{0}(t)^{1/2} \frac{a^{-\beta b}}{1-a^{-\beta b}} <  \frac12 M_{0}(t)^{1/2}.
$$
Consequently,
$$(\|v(0)\|_{L^{2}}+L)e^{LT}\leq (\|v_{0}(0)\|_{L^{2}}+\|v(0)-v_{0}(0)\|_{L^{2}}+L)e^{LT}
\leq\left(\frac32 M_{0}(0)^{1/2}+L\right)e^{LT},
$$
which we want to estimate (strictly) by
$$
\left(\frac{1}{\sqrt{2}}-\frac12\right)M_{0}(T)^{1/2}\leq \|v_{0}(T)\|_{L^{2}}-\|v(T)-v_{0}(T)\|_{L^{2}}\leq \|v(T)\|_{L^{2}}
$$
on the set $\{T_{L}\geq T\}\subset\Omega$.
In view of the definition of $M_{0}(t)$, this is indeed possible provided
\begin{equation}\label{c:7}
\left(\frac32+\frac{1}{L}\right)<\left(\frac{1}{\sqrt{2}}-\frac12\right)e^{LT}.
\end{equation}
In other words,  given $T>0$ and the universal constant $c_{R}>0$, we can choose $L=L(T,c_{R})>1$ large enough so that \eqref{c:a2} as well as \eqref{c:7} holds and consequently \eqref{eq:env1} is satisfied. Moreover, in view of Proposition \ref{fe z} and the definition of the stopping times \eqref{stopping time}, we observe that for a given $T>0$ we may possibly increase  $L$ so that the set $\{T_{L}\geq T\}$ satisfies $\mathbf{P}(T_{L}\geq T)>\kappa$.

Let us now define $u:=v+z$. Then $u$ is $(\mathcal{F}_t)_{t\geq0}$-adapted, solves the Navier--Stokes system \eqref{ns} and we deduce from \eqref{eq:vvv} together with \eqref{z} that   \eqref{eq:vv} holds true. To verify \eqref{eq:env}, we use \eqref{z} and apply \eqref{eq:env1}  on $ \{T_{L}\geq T\}$ to obtain
$$
\|u(T)\|_{L^{2}}\geq \|v(T)\|_{L^{2}}-\|z(T)\|_{L^{2}}> (\|v(0)\|_{L^{2}}+L)e^{LT}-\rmb{L^{1/2}/C_S}.
$$
Thus, since $u(0)=v(0)$ we may possibly increase the value of $L$ depending on $K$  and $\mathrm{Tr}(GG^{*})$ in order to conclude the desired lower bound \eqref{eq:env}. The initial value $v(0)$ is deterministic by our construction.  Finally, we set $\mathfrak{t}:=T_{L}$ which finishes the proof.
\end{proof}

To summarize the above discussion, first we fix the parameter $L$ large enough in dependence on $T,c_{R},\kappa,K$ and $\mathrm{Tr}(GG^{*})$. Then we apply Proposition~\ref{main iteration} and deduce the result of Theorem \ref{Main results1}. It remains to prove Proposition~\ref{main iteration} and to verify that the parameters $a,b,\beta$ can be appropriately chosen.

\subsection{The main iteration -- proof of Proposition \ref{main iteration}}
\label{ss:it}

The proof of Proposition  \ref{main iteration} proceeds along the lines of \cite[Section~7]{BV19}. We have to track the proof carefully to make the construction in each step $(\mathcal{F}_t)_{t\geq0}$-adapted and the initial value $v(0)$ deterministic. In the course of the proof we will need to adjust the value of the parameters $a,b,\beta$ as further conditions on these parameters will appear. The parameter $L$ is given and will be kept fixed. In addition, we have to make sure that the condition \eqref{c:a3}, which is essential in order to prove the failure of the energy inequality in Theorem \ref{Main results1}, is not violated. However, we observe that the right inequality in \eqref{c:a3} remains valid if we increase the value of $a$. In other words, given $L$ we find the minimal value of $a$ for which this inequality holds and from now on we may increase $a$ as we wish. On the other hand, increasing the value of $a$ or $b$ can in principle cause problems in the left inequality in \eqref{c:a3}, but here we may make the parameter $\beta$ smaller so that the inequality remains true. To summarize, we may freely increase $a$ or $b$ at the cost of making $\beta$ smaller.

\subsubsection{Choice of parameters}\label{s:c}

In the sequel, additional parameters will be indispensable and their value has to be carefully chosen in order to respect all the compatibility conditions appearing in the estimations below. First, for a sufficiently small  $\alpha\in (0,1)$ to be chosen below, we let $\ell\in (0,1)$ be  a small parameter  satisfying
\begin{equation}\label{ell}
\ell \lambda_q^4\leq \lambda_{q+1}^{-\alpha},\quad \ell^{-1}\leq \lambda_{q+1}^{2\alpha},\quad 4L\leq \ell^{-1}.
\end{equation}
In particular, we define
\begin{equation}\label{ell1}\ell:=\lambda_{q+1}^{-\frac{3\alpha}{2}}\lambda_q^{-2}.\end{equation}
The last condition in \eqref{ell}  together with \eqref{c:a3} leads to
$$
45\cdot (2\pi)^{3/2}<5\cdot (2\pi)^{3/2}a^{2\beta b}\leq c_{R}L\leq
c_{R}\frac{a^{4}\cdot (2\pi)^{3/2}-1}{2}.
$$
We remark that the reasoning from the beginning of Section \ref{ss:it} remains valid for this new condition: we may freely increase the value of $a$  provided we make $\beta$ smaller at the same time. In addition, we will require  $\alpha b>16$ and \rmb{$\alpha>18\beta b$}.

In order to verify the inductive estimates \eqref{inductionv} in Section~\ref{sss:v} and Section~\ref{sss:R}, it will also be necessary to absorb various expressions including  $M_{0}(t)^{1/2}$ for all $t\in[0,T_{L}]$. Since the stopping time $T_{L}$ is bounded by $L$, this reduces to absorbing $M_{0}(L)^{1/2}$ and it will be seen that the strongest such requirement is
\begin{equation}\label{c:M}
M_0(L)^{1/2}\lambda_{q+1}^{13\alpha-\frac{1}{7}}\leq \frac{c_R\delta_{q+2}}{10}
\end{equation}
needed in Section \ref{sss:R}. In other words,
$$
L^2e^{2L^2}a^{b(13\alpha-\frac{1}{7}+2b\beta)}\ll1
$$
and choosing  $b=\rmb{8\cdot 14^2L^{2}}$, $L\in \mathbb{N}$, (this choice is coming from the fact that with our choice of $\alpha$ below we want to guarantee that \rmb{$\alpha b>16$} as well as the fact that $b$ is a multiple of $7$ needed for the choice of parameters needed for the intermittent jets below, cf. Appendix \ref{s:B}) and $e^2\leq a$ leads to
$$ba^{b/14}a^{b(13\alpha-\frac{1}{7}+2b\beta)}\ll1.$$
In view of $\alpha>18\beta b$, this  can be achieved by choosing $a$ large enough and $\alpha=14^{-2}$. This choice also satisfies \rmb{$\alpha b>16$} required above and the condition \rmb{$\alpha>18\beta b$} can be achieved by choosing $\beta$ small. It is also compatible with all the other requirements needed below.

From now on, the parameters $\alpha$ and $b$ remain fixed and the  free parameters are $a$ and $\beta$ for which we already have a lower, respectively upper, bound. In the sequel, we will possibly increase $a$, and decrease $\beta$ at the same time in order to preserve all the above conditions and to fulfil further conditions appearing below.

\subsubsection{Mollification}\label{s:p}

We intend to replace $v_q$ by a mollified velocity field $v_\ell$. To this end, we extend $z_q(t)=z_q(0)$, $v_q(t)=v_q(0)$ for $t<0$ and let $\{\phi_\varepsilon\}_{\varepsilon>0}$ be a family of standard mollifiers on $\mathbb{R}^3$, and let $\{\varphi_\varepsilon\}_{\varepsilon>0}$ be a family of  standard mollifiers with support on $\mathbb{R}^+$. We define a mollification of $v_q$, $\mathring{R}_q$ and $z_q$ in space and time by convolution as follows
$$v_\ell=(v_q*_x\phi_\ell)*_t\varphi_\ell,\qquad
\mathring{R}_\ell=(\mathring{R}_q*_x\phi_\ell)*_t\varphi_\ell,\qquad
z_\ell=(\rmb{z_q}*_x\phi_\ell)*_t\varphi_\ell,$$
where $\phi_\ell=\frac{1}{\ell^3}\phi(\frac{\cdot}{\ell})$ and $\varphi_\ell=\frac{1}{\ell}\varphi(\frac{\cdot}{\ell})$.
Since the mollifier $\varphi_\ell$ is supported on $\mathbb{R}^+$, it is easy to see that $z_\ell$ is $(\mathcal{F}_t)_{t\geq0}$-adapted and so are $v_\ell$ and $\mathring{R}_\ell$. Since $\varphi_\ell$ is supported on $\mathbb{R}^+$, if the initial values $v_q(0), \mathring{R}_q(0)$ are deterministic, so are $v_\ell(0)$ and $\mathring{R}_\ell(0), \partial_t\mathring{R}_\ell(0)$. Moreover, $z_q(0)=0$ implies that $z_\ell(0)$ and ${R}_{\mathrm{com}}(0)$ given below are deterministic as well.
Then using (\ref{induction}) we obtain that $(v_\ell,\mathring{R}_\ell)$ satisfies
\begin{equation}\label{mollification}
\aligned
 \partial_tv_\ell -\Delta v_\ell+\div((v_\ell+z_\ell)\otimes (v_\ell+z_\ell))+\nabla p_\ell&=\div (\mathring{R}_\ell+R_{\textrm{com}})
\\\div v_\ell&=0,
\endaligned
\end{equation}
where
\begin{equation*}
R_{\textrm{com}}=(v_\ell+z_\ell)\mathring{\otimes}(v_\ell+z_\ell)-((v_q+\rmb{z_q})\mathring{\otimes}(v_q+\rmb{z_q}))*_x\phi_\ell*_t\varphi_\ell,
\end{equation*}
\begin{equation*}
p_\ell=(p_q*_x\phi_\ell)*_t\varphi_\ell-\frac{1}{3}(|v_\ell+z_\ell|^2-(|v_q+\rmb{z_q}|^2*_x\phi_\ell)*_t\varphi_\ell).
\end{equation*}

By using (\ref{inductionv}) and (\ref{ell}) we know for $t\in[0, T_L]$
\begin{equation}\label{error}
\|v_q-v_\ell\|_{C_tL^2}\lesssim\|v_q-v_\ell\|_{C^0_{t,x}}\lesssim \ell\|v_q\|_{C^1_{t,x}}\leq \ell\lambda_q^4 M_0(t)^{1/2}\leq M_0(t)^{1/2} \lambda_{q+1}^{-\alpha}\leq \frac{1}{4} M_0(t)^{1/2}\delta_{q+1}^{1/2},
\end{equation}
where we  used the fact that $\alpha>\beta$ and we  chose $a$ large enough in order to absorb the implicit constant.
In addition, it holds for $t\in[0, T_L]$
\begin{equation}\label{eq:vl}
\|v_\ell\|_{C_{t}L^2}\leq \|v_q\|_{C_{t}L^2}\leq  M_0(t)^{1/2} (1+\sum_{1\leq r\leq q}\delta_r^{1/2}),
\end{equation}
and for $N\geq1$\begin{equation}\label{eq:vl2}
\|v_\ell\|_{C^N_{t,x}}\lesssim \ell^{-N+1}\|v_q\|_{C^1_{t,x}}\leq \ell^{-N+1}\lambda_q^4 M_0(t)^{1/2}\leq M_0(t)^{1/2} \ell^{-N}\lambda_{q+1}^{-\alpha},
\end{equation}
where we have chosen $a$ large enough to absorb implicit constant.

\subsubsection{Construction of  $v_{q+1}$}\label{s:con}

Let us now proceed with the construction of the perturbation $w_{q+1}$ which then defines the next iteration by $v_{q+1}:=v_{\ell}+w_{q+1}$.
To this end, we make use of the construction of the intermittent jets \cite[Section 7.4]{BV19}, which we recall in Appendix~\ref{s:B}. In particular, the building blocks $W_{(\xi)}=W_{\xi,r_\perp,r_\|,\lambda,\mu}$ for $\xi\in\Lambda$ are defined in (\ref{intermittent}) and the set $\Lambda$ is introduced in Lemma \ref{geometric}.
The necessary estimates are collected  in \eqref{bounds}.  For the intermittent jets
we choose the following parameters
\begin{equation}\label{parameter}
\aligned
\lambda&=\lambda_{q+1},
\qquad
 r_\|=\lambda_{q+1}^{-4/7},
\qquad r_\perp=r_\|^{-1/4}\lambda_{q+1}^{-1}=\lambda_{q+1}^{-6/7},
\qquad
\mu=\lambda_{q+1}r_\|r_\perp^{-1}=\lambda_{q+1}^{9/7}.
\endaligned
\end{equation}
It is required that $b$ is a multiple of $7$ to ensure that $\lambda_{q+1}r_\perp= a^{(b^{q+1})/7}\in\mathbb{N}$.

In order to define the amplitude functions, let $\chi$ be a smooth function such that
$$\chi(z)=\begin{cases}
  1,& \text{if}\  0\leq z\leq1,\\
 z,&  \text{if}\  z\geq2,
\end{cases}
$$
and  $z\leq 2\chi(z)\leq 4z$ for $z\in (1,2)$. We then define for $t\in[0, T_L]$ and $\omega\in\Omega$
\begin{equation*}\rho(\omega,t,x)=4c_R\delta_{q+1}M_0(t)\chi\left((c_R\delta_{q+1}M_0(t))^{-1}|\mathring{R}_\ell(\omega,t,x)|\right),\end{equation*}
which is $(\mathcal{F}_t)_{t\geq 0}$-adapted and we have
\begin{equation*}\left|\frac{\mathring{R}_\ell(\omega,t,x)}{\rho(\omega,t,x)}\right|=\frac{1}{4}\frac{(c_R\delta_{q+1}M_0(t))^{-1}|
\mathring{R}_\ell(\omega,t,x)|}{\chi((c_R\delta_{q+1}M_0(t))^{-1}|\mathring{R}_\ell(\omega, t,x)|)}\leq\frac{1}{2}.\end{equation*}
Note that if $\mathring{R}_\ell(0,x), \partial_t\mathring{R}_\ell(0,x)$ are deterministic, so is $\rho( 0,x)$ and $\partial_t\rho( 0,x)$.
Moreover, we have for any $p\in [1,\infty]$, $t\in [0,T_L]$
\begin{equation}\label{rho}
\|\rho\|_{C_tL^p}\leq 16\left((8\pi^3)^{1/p}c_R\delta_{q+1}M_0(t)+\|\mathring{R}_\ell\|_{C_tL^p}\right).
\end{equation}
Furthermore, by mollification estimates, the embedding $W^{4,1}\subset L^\infty$ and (\ref{inductionv}) we obtain for $N\geq0$
 $t\in[0, T_L]$
 $$\|\mathring{R}_\ell\|_{C^N_{t,x}}\lesssim \ell^{-4-N}c_R\delta_{q+1}M_0(t)$$
and by a repeated application of the chain rule (see \cite[Proposition C.1]{BDLIS16}) we obtain
\begin{equation}\label{rhoN}
\aligned
\|\rho\|_{C^N_{t,x}}
&\lesssim\ell^{-4-N}c_R\delta_{q+1}M_0(t)+(c_R\delta_{q+1}M_0(t))^{-N+1}\ell^{-5N}(c_R\delta_{q+1}M_0(t))^N
\\
&\lesssim\ell^{-4-5N}c_R\delta_{q+1}M_0(t),
\endaligned
\end{equation}
where we used the fact that $\tfrac{d}{dt} M_{0}(t)=4L M_{0}(t)$ as well as $4L\leq \ell^{-1}$ and the implicit constants are independent of $\omega$.

As the next step, we define the amplitude functions
\begin{equation}\label{amplitudes}a_{(\xi)}(\omega,t,x):=a_{\xi,q+1}(\omega,t,x):=\rho(\omega,t,x)^{1/2}\gamma_\xi\left(\Id
-\frac{\mathring{R}_\ell(\omega,t,x)}{\rho(\omega,t,x)}\right)(2\pi)^{-\frac{3}{4}},\end{equation}
where $\gamma_\xi$ is introduced in  Lemma \ref{geometric}. Since $\rho$ and $\mathring{R}_\ell$ are $(\mathcal{F}_t)_{t\geq0}$-adapted, we know that also $a_{(\xi)}$ is $(\mathcal{F}_t)_{t\geq0}$-adapted. If $\mathring{R}_\ell(0,x), \partial_t\mathring{R}_\ell(0,x)$ are deterministic, so are $a_{(\xi)}(0,x)$ and $\partial_ta_{(\xi)}( 0,x)$.
By  (\ref{geometric equality}) we have
\begin{equation}\label{cancellation}(2\pi)^{\frac{3}{2}}\sum_{\xi\in\Lambda}a_{(\xi)}^2\strokedint_{\mathbb{T}^3}W_{(\xi)}\otimes W_{(\xi)}dx=\rho \Id-\mathring{R}_\ell,\end{equation}
and using (\ref{rho}) for $t\in[0, T_L]$
\begin{equation}\label{estimate a}\|a_{(\xi)}\|_{C_tL^2}\leq \|\rho\|_{C_tL^1}^{1/2}\|\gamma_\xi\|_{C^0(B_{1/2}(\Id))}\leq \frac{4c_R^{1/2}(8\pi^3+1)^{1/2}M}{8|\Lambda|(8\pi^3+1)^{1/2}}M_0(t)^{1/2}\delta_{q+1}^{1/2}\leq \frac{c_R^{1/4}M_0(t)^{1/2}\delta_{q+1}^{1/2}}{2|\Lambda|},\end{equation}
where we choose $c_R$ as a small universal constant to absorb $M$ and we use $M$ to denote the universal constant as in  Lemma~\ref{geometric}.
Furthermore, by using the fact that $\rho$ is bounded from below by $4c_R\delta_{q+1} M_0(t)$ we obtain by similar arguments as in (\ref{rhoN}) that it holds for $t\in[0, T_L]$ that
\begin{equation}\label{estimate aN}
\|a_{(\xi)}\|_{C^N_{t,x}}\leq \ell^{-2-5N}c_R^{1/4} \delta_{q+1}^{1/2}M_0(t)^{1/2},
\end{equation}
for $N\geq 0$.

With these preparations in hand,  we define the principal part $w_{q+1}^{(p)}$ of the perturbation $w_{q+1}$ as
\begin{equation}\label{principle}
w_{q+1}^{(p)}:=\sum_{\xi\in\Lambda} a_{(\xi)}W_{(\xi)}.
\end{equation}
If $\mathring{R}_\ell(0,x), \partial_t\mathring{R}_\ell(0,x)$ are deterministic, so are $w_{q+1}^{(p)}( 0,x)$ and $\partial_tw_{q+1}^{(p)}( 0, x)$.
Since the coefficients $a_{(\xi)}$ are $(\mathcal{F}_t)_{t\geq0}$-adapted and $W_{(\xi)}$ is a deterministic function we deduce that
$w_{q+1}^{(p)}$ is also $(\mathcal{F}_t)_{t\geq0}$-adapted.
Moreover, according to (\ref{cancellation}) and (\ref{Wxi}) it follows that
\begin{equation}\label{can}w_{q+1}^{(p)}\otimes w_{q+1}^{(p)}+\mathring{R}_\ell=\sum_{\xi\in \Lambda}a_{(\xi)}^2 \mathbb{P}_{\neq0}(W_{(\xi)}\otimes W_{(\xi)})+\rho \Id,
\end{equation}
where we use the notation $\mathbb{P}_{\neq0}f:=f-\mathcal{F}f(0)=f-(2\pi)^{3/2}\strokedint_{\mathbb{T}^3}f$.

We also define an incompressibility corrector by
\begin{equation}\label{incompressiblity}
w_{q+1}^{(c)}:=\sum_{\xi\in \Lambda}\textrm{curl}(\nabla a_{(\xi)}\times V_{(\xi)})+\nabla a_{(\xi)}\times \textrm{curl}V_{(\xi)}+a_{(\xi)}W_{(\xi)}^{(c)},\end{equation}
with $W_{(\xi)}^{(c)}$ and $V_{(\xi)}$ being given in (\ref{corrector}).
Since $a_{(\xi)}$ is $(\mathcal{F}_t)_{t\geq0}$-adapted and $W_{(\xi)}, W_{(\xi)}^{(c)}$ and $V_{(\xi)}$ are  deterministic functions we know that
$w_{q+1}^{(c)}$ is also $(\mathcal{F}_t)_{t\geq0}$-adapted. If $\mathring{R}_\ell(0,x), \partial_t\mathring{R}_\ell(0,x)$ are deterministic, so are $w_{q+1}^{(c)}( 0,x)$ and $\partial_tw_{q+1}^{(c)}(0, x)$.
By a direct computation we deduce that
\begin{equation*}
w_{q+1}^{(p)}+w_{q+1}^{(c)}=\sum_{\xi\in\Lambda}\textrm{curl}\,\textrm{curl}(a_{(\xi)}V_{(\xi)}),
\end{equation*}
hence
\begin{equation*}\div(w_{q+1}^{(p)}+w_{q+1}^{(c)})=0.\end{equation*}
We also introduce a temporal corrector
\begin{equation}\label{temporal}w_{q+1}^{(t)}:=-\frac{1}{\mu}\sum_{\xi\in \Lambda}\mathbb{P}\mathbb{P}_{\neq0}\left(a_{(\xi)}^2\phi_{(\xi)}^2\psi_{(\xi)}^2\xi\right),\end{equation}
where $\mathbb{P}$ is the Helmholtz projection. If $\mathring{R}_\ell(0,x), \partial_t\mathring{R}_\ell(0,x)$ are deterministic, so is $w_{q+1}^{(t)}( 0,x)$.  Similarly to above $w_{q+1}^{(t)}$ is $(\mathcal{F}_t)_{t\geq0}$-adapted and by a direct computation  we obtain
\begin{equation}\label{equation for temporal}
\aligned
&\partial_t w_{q+1}^{(t)}+\sum_{\xi\in\Lambda}\mathbb{P}_{\neq0}\left(a_{(\xi)}^2\div(W_{(\xi)}\otimes W_{(\xi)})\right)
\\
&\qquad= -\frac{1}{\mu}\sum_{\xi\in\Lambda}\mathbb{P}\mathbb{P}_{\neq0}\partial_t\left(a_{(\xi)}^2\phi_{(\xi)}^2\psi_{(\xi)}^2\xi\right)
+\frac{1}{\mu}\sum_{\xi\in\Lambda}\mathbb{P}_{\neq0}\left( a^2_{(\xi)}\partial_t(\phi^2_{(\xi)}\psi^2_{(\xi)}\xi)\right)
\\&\qquad= (\Id-\mathbb{P})\frac{1}{\mu}\sum_{\xi\in\Lambda}\mathbb{P}_{\neq0}\partial_t\left(a_{(\xi)}^2\phi_{(\xi)}^2\psi_{(\xi)}^2\xi\right)
-\frac{1}{\mu}\sum_{\xi\in\Lambda}\mathbb{P}_{\neq0}\left(\partial_t a^2_{(\xi)}(\phi^2_{(\xi)}\psi^2_{(\xi)}\xi)\right).
\endaligned
\end{equation}
Note that the first term on the right hand side can be viewed as a pressure term $\nabla p_{1}$.

Finally, the total perturbation $w_{q+1}$ is defined by
\begin{equation}\label{wq}w_{q+1}:=w_{q+1}^{(p)}+w_{q+1}^{(c)}+w_{q+1}^{(t)},\end{equation}
which is mean zero, divergence free and $(\mathcal{F}_t)_{t\geq0}$-adapted. If $\mathring{R}_\ell(0,x), \partial_t\mathring{R}_\ell(0,x)$ are deterministic, so is $w_{q+1}(0,x)$. The new velocity $v_{q+1}$ is defined as
\begin{equation}\label{vq}
v_{q+1}:=v_\ell+w_{q+1}.
\end{equation}
Thus, it is also $(\mathcal{F}_t)_{t\geq0}$-adapted. If $\mathring{R}_q(0,x), v_q(0,x)$ are deterministic, so is $v_{q+1}(0,x)$.

\subsubsection{Verification of the inductive estimates for $v_{q+1}$}
\label{sss:v}
Next, we verify the inductive estimates (\ref{inductionv}) on the level $q+1$ for $v$ and we prove (\ref{iteration}).
First, we  recall the following result from \cite[Lemma~7.4]{BV19}.

\begin{lem}\label{Lp}
Fix integers $N, \kappa\geq1$ and let $\zeta>1$ be such that
\begin{equation*}
\frac{2\pi \sqrt{3}\zeta}{\kappa}\leq\frac{1}{3}\quad \textrm{ and } \quad\zeta^4\frac{(2\pi \sqrt{3}\zeta)^N}{\kappa^N}\leq1.
\end{equation*}
Let $p\in \{1,2\}$ and let $f$ be a $\mathbb{T}^3$-periodic function such that there exists a constant $C_f>0$ such that
$$\|D^jf\|_{L^p}\leq C_f\zeta^j,$$
holds for all $0\leq j\leq N+4$. In addition, let $g$ be a $(\mathbb{T}/\kappa)^3$-periodic function. Then it holds that
$$\|fg\|_{L^p}\lesssim C_f\|g\|_{L^p},$$
where the implicit constant is universal.
\end{lem}

This result shall be used in order to bound $w_{q+1}^{(p)}$ in $L^{2}$ whereas for the other $L^{p}$-norms we apply a different approach.  By (\ref{estimate a}) and (\ref{estimate aN}) we obtain for $t\in[0, T_L]$
\begin{equation*}\|D^ja_{(\xi)}\|_{C_tL^2}\lesssim \frac{c_R^{1/4}M_0(t)^{1/2}}{2|\Lambda|}\delta_{q+1}^{1/2}\ell^{-8j},\end{equation*}
which combined with Lemma \ref{Lp} for $\zeta=\ell^{-8}$ we obtain for $t\in[0, T_L]$
\begin{equation}\label{estimate wqp}
\|w_{q+1}^{(p)}\|_{C_tL^2}\leq \sum_{\xi\in\Lambda}\frac{1}{2|\Lambda|}c_R^{1/4}M_0(t)^{1/2}\delta_{q+1}^{1/2}\|W_{(\xi)}\|_{C_tL^2}\leq \frac{1}{2}M_0(t)^{1/2}\delta_{q+1}^{1/2},
\end{equation}
where we used $c_R^{1/4}$ to absorb the universal constant and the fact that due to (\ref{intermittent}) together with the normalizations \eqref{eq:phi}, \eqref{eq:psi} we have that $\|W_{(\xi)}\|_{L^2}\simeq 1$ uniformly in all the involved parameters.

For general $L^p$ norm we apply (\ref{bounds}) and (\ref{estimate aN}) to deduce for $t\in[0, T_L]$, $p\in(1,\infty)$
\begin{equation}\label{principle est1}
\aligned
\|w_{q+1}^{(p)}\|_{C_tL^p}&\lesssim \sum_{\xi\in \Lambda}\|a_{(\xi)}\|_{C^0_{t,x}}\|W_{(\xi)}\|_{C_tL^p}\lesssim M_0(t)^{1/2}\delta_{q+1}^{1/2}\ell^{-2}r_\perp^{2/p-1}r_\|^{1/p-1/2},
\endaligned
\end{equation}
\begin{equation}\label{correction est}
\aligned
\|w_{q+1}^{(c)}\|_{C_tL^p}&\lesssim\sum_{\xi\in \Lambda}\left(\|a_{(\xi)}\|_{C^0_{t,x}}\|W_{(\xi)}^{(c)}\|_{C_tL^p}+\|a_{(\xi)}\|_{C^2_{t,x}}\|V_{(\xi)}\|_{C_tW^{1,p}}\right)
\\&\lesssim M_0(t)^{1/2}\delta_{q+1}^{1/2}\ell^{-12}r_\perp^{2/p-1}r_\|^{1/p-1/2}\left(r_\perp r_\|^{-1}+\lambda_{q+1}^{-1}\right)\lesssim M_0(t)^{1/2}\delta_{q+1}^{1/2}\ell^{-12}r_\perp^{2/p}r_\|^{1/p-3/2},
\endaligned
\end{equation}
and
\begin{equation}\label{temporal est1}
\aligned
\|w_{q+1}^{(t)}\|_{C_tL^p}&\lesssim \mu^{-1}\sum_{\xi\in\Lambda}\|a_{(\xi)}\|_{C^0_{t,x}}^2\|\phi_{(\xi)}\|_{L^{2p}}^2\|\psi_{(\xi)}\|_{C_tL^{2p}}^2
\\
&\lesssim\delta_{q+1}M_0(t) \ell^{-4}r_\perp^{2/p-1}r_\|^{1/p-2}(\mu^{-1}r_\perp^{-1}r_\|)= M_0(t)\delta_{q+1}\ell^{-4}r_\perp^{2/p-1}r_\|^{1/p-2}\lambda_{q+1}^{-1}.
\endaligned
\end{equation}
We note that for $p=2$ \eqref{principle est1} provides a worse bound than \eqref{estimate wqp} which was based on Lemma \ref{Lp}.
Since by (\ref{c:M}) $M_0(L)^{1/2}\lambda_{q+1}^{4\alpha-\frac{1}{7}}<1$ we have for $t\in[0, T_L]$
\begin{equation}\label{corr temporal}
\aligned
&\|w_{q+1}^{(c)}\|_{C_tL^p}+\|w_{q+1}^{(t)}\|_{C_tL^p}\\
&\quad\lesssim  M_0(t)^{1/2}\delta_{q+1}^{1/2}\ell^{-2}r_\perp^{2/p-1}r_\|^{1/p-1/2}\left(\ell^{-10}r_\perp r_\|^{-1}+M_0(t)^{1/2}\delta_{q+1}^{1/2}\ell^{-2}r_\|^{-3/2}\lambda_{q+1}^{-1}\right)
\\ &\quad\lesssim M_0(t)^{1/2}\delta_{q+1}^{1/2}\ell^{-2}r_\perp^{2/p-1}r_\|^{1/p-1/2},
\endaligned
\end{equation}
where we use (\ref{ell}) and the fact that $\lambda_{q+1}^{20\alpha-\frac{2}{7}}<1$ by our choice of $\alpha$. The bound \eqref{corr temporal} will be used below in the estimation of the Reynolds stress.

Combining (\ref{estimate wqp}), (\ref{correction est}) and (\ref{temporal est1}) we obtain for $t\in[0, T_L]$
\begin{equation}\label{estimate wq}
\aligned
\|w_{q+1}\|_{C_tL^2}&\leq M_0(t)^{1/2}\delta_{q+1}^{1/2}\left(\frac{1}{2}+C\ell^{-12}r_\perp r_\|^{-1}+CM_0(t)^{1/2}\delta_{q+1}^{1/2}\ell^{-4}r_\|^{-3/2}\lambda_{q+1}^{-1}\right)
\\&\leq M_0(t)^{1/2}\delta_{q+1}^{1/2}\left(\frac{1}{2}+C\lambda_{q+1}^{24\alpha-2/7}+CM_0(t)^{1/2}\delta_{q+1}^{1/2}\lambda_{q+1}^{8\alpha-1/7}\right)
\leq \frac{3}{4}M_0(t)^{1/2}\delta_{q+1}^{1/2},
\endaligned
\end{equation}
where by \eqref{c:M} we choose  $ \beta$ small enough and $a$ large enough such that
$$
C\lambda_{q+1}^{24\alpha-2/7}\leq 1/8,\quad\text{and}\quad CM_0(L)^{1/2}\delta_{q+1}^{1/2}\lambda_{q+1}^{8\alpha-1/7}\leq 1/8.
$$
The bound \eqref{estimate wq} can be directly combined with \eqref{eq:vl} and the definition of the velocity $v_{q+1}$ \eqref{vq} to deduce the first bound in \eqref{inductionv} on the level $q+1$. Indeed, for $t\in[0, T_L]$
$$
\|v_{q+1}\|_{C_{t}L^{2}}\leq \|v_{\ell}\|_{C_{t}L^{2}}+\|w_{q+1}\|_{C_{t}L^{2}} \leq M_{0}(t)^{1/2}(1+\sum_{1\leq r\leq q+1}\delta_{r}^{1/2}).
$$
In addition, \eqref{estimate wq} together with \eqref{error} yields for $t\in[0, T_L]$
$$
\|v_{q+1}-v_{q}\|_{C_{t}L^{2}}\leq \|w_{q+1}\|_{C_{t}L^{2}}+\|v_{\ell}-v_{q}\|_{C_{t}L^{2}}\leq M_{0}(t)^{1/2} \delta_{q+1}^{1/2},
$$
hence \eqref{iteration} holds.

As the next step, we shall verify the second bound in \eqref{inductionv}.
Using (\ref{estimate aN}) and (\ref{bounds}) we have for $t\in[0, T_L]$
\begin{equation}\label{principle est2}
\aligned
\|w_{q+1}^{(p)}\|_{C^1_{t,x}}&\leq \sum_{\xi\in\Lambda}\|a_{(\xi)}\|_{C^1_{t,x}}\|W_{(\xi)}\|_{C^1_{t,x}}\\
&\lesssim M_0(t)^{1/2} \ell^{-7}r_\perp^{-1}r_\|^{-1/2}\lambda_{q+1}\left(1+\frac{r_\perp \mu}{r_\|}\right)\lesssim M_0(t)^{1/2} \ell^{-7}r_\perp^{-1}r_\|^{-1/2}\lambda_{q+1}^2,
\endaligned
\end{equation}
\begin{equation}\label{correction est2}
\aligned
\|w_{q+1}^{(c)}\|_{C^1_{t,x}}
&\lesssim  \sum_{\xi\in\Lambda}\left(\|a_{(\xi)}\|_{C^1_{t,x}}\|W_{(\xi)}^{(c)}\|_{C^1_{t,x}}+\|a_{(\xi)}\|_{C^3_{t,x}}(\|V_{(\xi)}\|_{C^1_{t}C^1_x}
+\|V_{(\xi)}\|_{C_tC^2_{x}})\right)\\
&\lesssim M_0(t)^{1/2}\ell^{-17}r_\|^{-3/2}\left(\mu+\frac{r_\perp \mu\lambda_{q+1}}{r_\|}\right)\lesssim M_0(t)^{1/2}\ell^{-17}r_\|^{-3/2}\lambda_{q+1}^2,
\endaligned
\end{equation}
and
\begin{equation}\label{temporal est2}
\aligned
\|w_{q+1}^{(t)}\|_{C^1_{t,x}}&\leq \frac{1}{\mu}\sum_{\xi\in\Lambda}[\|a^2_{(\xi)}
\phi^2_{(\xi)}\psi^2_{(\xi)}\|_{C_tW^{1+\alpha,p}}+\|a^2_{(\xi)}\phi^2_{(\xi)}\psi^2_{(\xi)}\|_{C^1_tW^{\alpha,p}}]
\\
&\leq\frac{1}{\mu}\sum_{\xi\in\Lambda}
\Big(\|a_{(\xi)}\|_{C^0_{t,x}}\|a_{(\xi)}\|_{C^{1+\alpha}_{t,x}}\|\phi_{(\xi)}\|_{L^{\infty}}^2\|\psi_{(\xi)}\|_{C_tL^{\infty}}^2\\
&\quad+\|a_{(\xi)}\|_{C^1_{t,x}}\|a_{(\xi)}\|_{C^0_{t,x}}\|\phi_{(\xi)}\|_{L^{\infty}}\big(\| \phi_{(\xi)}\|_{W^{1+\alpha,\infty}}\|\psi_{(\xi)}\|_{C_tL^{\infty}}^2\\
&\quad\quad+\| \phi_{(\xi)}\|_{W^{\alpha,\infty}}\|\psi_{(\xi)}\|_{C_tL^{\infty}}\|\psi_{(\xi)}\|_{C^1_tL^{\infty}}\big)\\
&\quad+\|a_{(\xi)}\|_{C^1_{t,x}}\|a_{(\xi)}\|_{C^0_{t,x}}\|\phi_{(\xi)}\|_{L^{\infty}}^2\big(\|\psi_{(\xi)}\|_{C_tL^{\infty}}\|\psi_{(\xi)}\|_{C_tW^{1+\alpha,p}}
\\&\quad\quad+\|\psi_{(\xi)}\|_{C^1_tL^{\infty}}\|\psi_{(\xi)}\|_{C_{t}W^{\alpha,p}}+\|\psi_{(\xi)}\|_{C_tL^{\infty}}\|\psi_{(\xi)}\|_{C^1_tW^{\alpha,p}}\big)\Big)
\\&\lesssim {\frac{1}{\mu}M_0(t) \ell^{-9}r_\perp^{-2}r_\|^{-1}\lambda_{q+1}^{1+\alpha}\left(1+\frac{r_\perp \mu}{r_\|}\right)}
\lesssim M_0(t)\ell^{-9}r_\perp^{-1}r_\|^{-2}\lambda_{q+1}^{1+\alpha},
\endaligned
\end{equation}
where we chose $p$ large enough and applied the Sobolev embedding in the first inequality in \eqref{temporal est2} needed because $\mathbb{P}\mathbb{P}_{\neq0}$ is not a bounded operator on $C^0$;  in the last inequality we used interpolation and an extra $\lambda_{q+1}^\alpha$ appeared.
Combining   \eqref{eq:vl2} and \eqref{principle est2}, \eqref{correction est2}, \eqref{temporal est2} with (\ref{ell}) we obtain for $t\in[0, T_L]$
\begin{equation*}
\aligned
\|v_{q+1}\|_{C^1_{t,x}}&\leq \|v_\ell\|_{C^1_{t,x}}+\|w_{q+1}\|_{C^1_{t,x}}\\
&\leq M_0(t)^{1/2}\left(\lambda_{q+1}^\alpha+C\lambda_{q+1}^{14\alpha+22/7}+C\lambda_{q+1}^{34\alpha+20/7}+CM_0(t)^{1/2}\lambda_{q+1}^{19\alpha+3}\right)
 \leq M_0(t)^{1/2}\lambda_{q+1}^4,
\endaligned
\end{equation*}
where we used (\ref{c:M}) to have the fact that $CM_0(L)^{1/2}\leq \frac{1}{2}\lambda_{q+1}^{1-19\alpha}$. Thus, the second estimate in \eqref{inductionv} holds true on the level $q+1$.

We conclude this part with further estimates of the perturbations $w^{(p)}_{q+1}$, $w^{(c)}_{q+1}$ and $w^{(t)}_{q+1}$, which will be used below in order to bound the Reynolds stress $\mathring{R}_{q+1}$ and to establish the final estimate in \eqref{inductionv} on the level $q+1$.
By a similar approach as in \eqref{principle est1}, \eqref{correction est}, \eqref{temporal est1}, we derive  the following estimates: for $t\in[0, T_L]$ by using (\ref{ell}), (\ref{estimate aN}) and (\ref{bounds})
\begin{equation}\label{principle est22}
\aligned
\|w_{q+1}^{(p)}+w_{q+1}^{(c)}\|_{C_tW^{1,p}}&\leq\sum_{\xi\in\Lambda}
\|\textrm{curl\,}\textrm{curl}(a_{(\xi)}V_{(\xi)})\|_{C_tW^{1,p}}
\\
&\lesssim \sum_{\xi\in \Lambda}\Big(\|a_{(\xi)}\|_{C^3_{t,x}}\|V_{(\xi)}\|_{C_tL^p}+\|a_{(\xi)}\|_{C^2_{t,x}}\|V_{(\xi)}\|_{C_tW^{1,p}}\\
&\quad+\|a_{(\xi)}\|_{C^1_{t,x}}\|V_{(\xi)}\|_{C_tW^{2,p}}+\|a_{(\xi)}\|_{C^0_{t,x}}\|V_{(\xi)}\|_{C_tW^{3,p}}\Big)
\\
&\lesssim M_0(t)^{1/2}r_\perp^{2/p-1}r_\|^{1/p-1/2}\left(\ell^{-17}
\lambda_{q+1}^{-2}+\ell^{-12}\lambda_{q+1}^{-1}+\ell^{-7}+\ell^{-2}\lambda_{q+1}\right)\\
&\lesssim M_0(t)^{1/2}r_\perp^{2/p-1}r_\|^{1/p-1/2}\ell^{-2}\lambda_{q+1},
\endaligned
\end{equation}
and
\begin{equation}\label{corrector est2}
\aligned
\|w_{q+1}^{(t)}\|_{C_tW^{1,p}}
&\leq\frac{1}{\mu}\sum_{\xi\in\Lambda}
\Big(\|a_{(\xi)}\|_{C^0_{t,x}}\|a_{(\xi)}\|_{C^1_{t,x}}\|\phi_{(\xi)}\|_{L^{2p}}^2\|\psi_{(\xi)}\|_{C_tL^{2p}}^2\\
&\quad+\|a_{(\xi)}\|_{C^0_{t,x}}^2\|\phi_{(\xi)}\|_{L^{2p}}\|\nabla \phi_{(\xi)}\|_{L^{2p}}\|\psi_{(\xi)}\|_{C_tL^{2p}}^2\\
&\quad+\|a_{(\xi)}\|_{C^0_{t,x}}^2\|\phi_{(\xi)}\|_{L^{2p}}^2\|\nabla \psi_{(\xi)}\|_{C_tL^{2p}}\|\psi_{(\xi)}\|_{C_tL^{2p}}\Big)
\\
&\lesssim \frac{M_0(t)}{\mu}r_\perp^{2/p-2}r_\|^{1/p-1}\left(\ell^{-9}+\ell^{-4}\lambda_{q+1}\right)\lesssim M_0(t)r_\perp^{2/p-2}r_\|^{1/p-1}\ell^{-4}\lambda_{q+1}^{-2/7}.
\endaligned
\end{equation}

\subsubsection{Definition of the Reynolds stress $\mathring{R}_{q+1}$}\label{s:def}

Subtracting from (\ref{induction}) at level $q+1$ the system (\ref{mollification}), we obtain
\begin{equation}\label{stress}
\aligned
\div\mathring{R}_{q+1}-\nabla p_{q+1}&=\underbrace{-\Delta w_{q+1}+\partial_t(w_{q+1}^{(p)}+w_{q+1}^{(c)})+\div((v_\ell+z_\ell)\otimes w_{q+1}+w_{q+1}\otimes (v_\ell+z_\ell))}_{\div(R_{\textrm{lin}})+\nabla p_{\textrm{lin}}}
\\&\quad+\underbrace{\div\left((w_{q+1}^{(c)}+w_{q+1}^{(t)})\otimes w_{q+1}+w_{q+1}^{(p)}\otimes (w_{q+1}^{(c)}+w_{q+1}^{(t)})\right)}_{\div(R_{\textrm{cor}})+\nabla p_{\textrm{cor}}}
\\&\quad+\underbrace{\div(w_{q+1}^{(p)}\otimes w_{q+1}^{(p)}+\mathring{R}_\ell)+\partial_tw_{q+1}^{(t)}}_{\div(R_{\textrm{osc}})+\nabla p_{\textrm{osc}}}
\\&\quad+\underbrace{\div\left(v_{q+1}{\otimes}\rmb{z_{q+1}}-v_{q+1}{\otimes}z_\ell+\rmb{z_{q+1}}{\otimes}v_{q+1}-z_\ell{\otimes}v_{q+1}
+\rmb{z_{q+1}}{\otimes}\rmb{z_{q+1}}-z_\ell{\otimes}z_\ell\right)}_{\div(R_{\textrm{com}1})+\nabla p_{\textrm{com}1}}
\\&\quad+\div(R_{\textrm{com}})-\nabla p_\ell.
\endaligned
\end{equation}
We recall the inverse divergence operator $\mathcal{R}$ as in \cite[Section 5.6]{BV19}, which acts on vector fields $v$ with $\int_{\mathbb{T}^3}vdx=0$ as
\begin{equation*}
(\mathcal{R}v)^{kl}=(\partial_k\Delta^{-1}v^l+\partial_l\Delta^{-1}v^k)-\frac{1}{2}(\delta_{kl}+\partial_k\partial_l\Delta^{-1})\div\Delta^{-1}v,
\end{equation*}
for $k,l\in\{1,2,3\}$. Then $\mathcal{R}v(x)$ is a symmetric trace-free matrix for each $x\in\mathbb{T}^3$, and $\mathcal{R}$ is a right inverse of the div operator, i.e. $\div(\mathcal{R} v)=v$. By using $\mathcal{R}$ we define
\begin{equation*}
R_{\textrm{lin}}:=-\mathcal{R}\Delta w_{q+1}+\mathcal{R}\partial_t(w_{q+1}^{(p)}+w_{q+1}^{(c)})
+(v_\ell+z_\ell)\mathring\otimes w_{q+1}+w_{q+1}\mathring\otimes (v_\ell+z_\ell),
\end{equation*}
\begin{equation*}
R_{\textrm{cor}}:=(w_{q+1}^{(c)}+w_{q+1}^{(t)})\mathring{\otimes} w_{q+1}+w_{q+1}^{(p)}\mathring{\otimes} (w_{q+1}^{(c)}+w_{q+1}^{(t)}),
\end{equation*}
\begin{equation*}
R_{\textrm{com}1}:=v_{q+1}\mathring{\otimes}\rmb{z_{q+1}}-v_{q+1}\mathring{\otimes}z_\ell+\rmb{z_{q+1}}\mathring{\otimes}v_{q+1}-z_\ell\mathring{\otimes}v_{q+1}
+\rmb{z_{q+1}}\mathring{\otimes}\rmb{z_{q+1}}-z_\ell\mathring{\otimes}z_\ell.
\end{equation*}
We observe that if $\mathring{R}_q(0,x), v_q(0,x)$ are deterministic, the same is valid for the above defined error terms $R_{\mathrm{lin}}(0,x)$, $R_{\mathrm{cor}}(0,x)$, $R_{\mathrm{com1}}(0,x)$.

In order to define the remaining oscillation error from the third line in \eqref{stress}, we apply (\ref{can}) and (\ref{equation for temporal}) to obtain
\begin{align*}
&\div(w_{q+1}^{(p)}\otimes w_{q+1}^{(p)}+\mathring{R}_\ell)+\partial_tw_{q+1}^{(t)}\\
&\quad=\sum_{\xi\in\Lambda}\div\left(a^2_{(\xi)}\mathbb{P}_{\neq0}(W_{(\xi)}\otimes W_{(\xi)})\right)+\nabla \rho+\partial_tw_{q+1}^{(t)}
\\
&\quad=\sum_{\xi\in\Lambda}\mathbb{P}_{\neq0}\left(\nabla a^2_{(\xi)}\mathbb{P}_{\neq0}(W_{(\xi)}\otimes W_{(\xi)})\right)+\nabla \rho+\sum_{\xi\in\Lambda}\mathbb{P}_{\neq0}\left(a^2_{(\xi)}\div(W_{(\xi)}\otimes W_{(\xi)})\right)+\partial_tw_{q+1}^{(t)}
\\
&\quad=\sum_{\xi\in\Lambda}\mathbb{P}_{\neq0}
\left(\nabla a_{(\xi)}^2\mathbb{P}_{\neq0}(W_{(\xi)}\otimes W_{(\xi)}) \right)+\nabla \rho+\nabla p_1-\frac{1}{\mu}\sum_{\xi\in\Lambda}\mathbb{P}_{\neq0}
\left(\partial_t a_{(\xi)}^2(\phi_{(\xi)}^2\psi_{(\xi)}^2\xi) \right)
\end{align*}
Therefore,
\begin{equation*}
\aligned
 R_{\textrm{osc}}:=&\sum_{\xi\in\Lambda}\mathcal{R}
\left(\nabla a_{(\xi)}^2\mathbb{P}_{\neq0}(W_{(\xi)}\otimes W_{(\xi)}) \right)-\frac{1}{\mu}\sum_{\xi\in\Lambda}\mathcal{R}
\left(\partial_t a_{(\xi)}^2(\phi_{(\xi)}^2\psi_{(\xi)}^2\xi) \right)=:R_{\textrm{osc}}^{(x)}+R_{\textrm{osc}}^{(t)},
\endaligned
\end{equation*}
which is also deterministic at time $0$.
Finally we define the Reynolds stress on the level $q+1$ by
\begin{equation*}\aligned
 \mathring{R}_{q+1}:=R_{\textrm{lin}}+R_{\textrm{cor}}+R_{\textrm{osc}}+R_{\textrm{com}}+R_{\textrm{com}1}.
 \endaligned
 \end{equation*}
We note that by construction $\mathring{R}_{q+1}(0,x)$ is deterministic.

\subsubsection{Verification of the inductive estimate \eqref{inductionv} for $\mathring{R}_{q+1}$}
\label{sss:R}
To conclude the proof of Proposition~\ref{main iteration}, we shall verify the third estimate in \eqref{inductionv}. To this end, we estimate each term in the definition of $\mathring{R}_{q+1}$ separately.

In the following we choose $p=\frac{32}{32-7\alpha}>1$ so that it holds  in particular that $r_\perp^{2/p-2}r_\|^{1/p-1}\leq \lambda_{q+1}^\alpha$.
For the linear  error we apply \eqref{inductionv} to obtain for $t\in[0, T_L]$
\begin{equation*}
\aligned
\|R_{\textrm{lin}}\|_{C_tL^p}
&\lesssim\|\mathcal{R}\Delta w_{q+1}\|_{C_tL^p}+\|\mathcal{R}\partial_t(w_{q+1}^{(p)}+w_{q+1}^{(c)})\|_{C_tL^p}+\|(v_\ell+z_\ell)\mathring{\otimes}w_{q+1}+w_{q+1}\mathring{\otimes}(v_\ell+z_\ell)\|_{C_tL^p}
\\
&\lesssim\|w_{q+1}\|_{C_tW^{1,p}}+\sum_{\xi\in\Lambda}\|\partial_t\textrm{curl}
(a_{(\xi)}V_{(\xi)})\|_{C_tL^p}+M_0(t)^{1/2}(\lambda_{q}^4+\rmb{\lambda_{q+1}^{\frac\alpha8}})\|w_{q+1}\|_{C_tL^p},
\endaligned
\end{equation*}
where by \eqref{bounds} and \eqref{estimate aN}
\begin{equation*}
\aligned
\sum_{\xi\in\Lambda}\|\partial_t\textrm{curl}
(a_{(\xi)}V_{(\xi)})\|_{C_tL^p}&\leq \sum_{\xi\in\Lambda}\left(\|
a_{(\xi)}\|_{C_tC^1_x}\|\partial_t V_{(\xi)}\|_{C_tW^{1,p}}+\|\partial_ta_{(\xi)}\|_{C_tC^1_x}\| V_{(\xi)}\|_{C_tW^{1,p}}\right)\\
&\lesssim M_{0}(t)^{1/2} \ell^{-7}r_{\perp}^{2/p}r_{\|}^{1/p-3/2}\mu+M_{0}(t)^{1/2} \ell^{-12}r_{\perp}^{2/p-1}r_{\|}^{1/p-1/2}\lambda_{q+1}^{-1}
\endaligned
\end{equation*}
In view of (\ref{principle est22}), (\ref{corrector est2}) as well as (\ref{principle est1}), (\ref{corr temporal}),  we  deduce  for $t\in[0,T_L]$
\begin{equation*}
\aligned
\|R_{\textrm{lin}}\|_{C_tL^p}
&\lesssim M_0(t)^{1/2}\ell^{-2}r_\perp^{2/p-1}r_\|^{1/p-1/2}\lambda_{q+1}+M_0(t)\ell^{-4}r_\perp^{2/p-2}r_\|^{1/p-1}\lambda_{q+1}^{-2/7}
\\
&\ +M_0(t)^{1/2}\ell^{-7}r_\perp^{2/p}r_\|^{1/p-3/2}\mu
+M_0(t)^{1/2}\ell^{-12}r_\perp^{2/p-1}r_\|^{1/p-1/2}\lambda_{q+1}^{-1}\\&\ +M_0(t)\ell^{-2}r_\perp^{2/p-1}r_\|^{1/p-1/2}(\lambda^4_{q}+\rmb{\lambda_{q+1}^{\frac\alpha8}})
\\&\lesssim M_0(t)^{1/2}\lambda_{q+1}^{5\alpha-1/7}+M_0(t)\lambda_{q+1}^{9\alpha-2/7}+M_0(t)^{1/2}\lambda_{q+1}^{15\alpha-1/7}+M_0(t)^{1/2}\lambda_{q+1}^{25\alpha-15/7}
\\
&\leq\frac{M_0(t)c_R\delta_{q+2}}{5}.
\endaligned
\end{equation*}
Here, we have taken $a$ sufficiently large and $\beta$ sufficiently small.

The corrector error  is estimated  using \eqref{principle est1}, \eqref{correction est}, \eqref{temporal est1}, (\ref{corr temporal})  for $t\in[0, T_L]$ as
\begin{equation*}
\aligned
\|R_{\textrm{cor}}\|_{C_tL^p}
&\leq\|w_{q+1}^{(c)}+ w_{q+1}^{(t)}\|_{C_tL^{2p}}\| w_{q+1}\|_{C_tL^{2p}}+\|w_{q+1}^{(c)}+ w_{q+1}^{(t)}\|_{C_tL^{2p}}\| w_{q+1}^{(p)}\|_{C_tL^{2p}}
\\
&\lesssim M_0(t)\left(\ell^{-12}r_\perp^{1/p}r_\|^{1/(2p)-3/2}
+\ell^{-4}M_0(t)^{1/2}r_\perp^{1/p-1}r_\|^{1/(2p)-2}\lambda_{q+1}^{-1}\right)\ell^{-2}r_\perp^{1/p-1}r_\|^{1/(2p)-1/2}
\\
&\lesssim M_0(t)\left(\ell^{-14}r_\perp^{2/p-1}r_\|^{1/p-2}
+\ell^{-6}M_0(t)^{1/2}r_\perp^{2/p-2}r_\|^{1/p-5/2}\lambda_{q+1}^{-1}\right)
\\
&\lesssim M_0(t)\left(\lambda_{q+1}^{29\alpha-2/7}+M_0(t)^{1/2}\lambda_{q+1}^{13\alpha-1/7}\right)
\leq \frac{M_0(t)c_R\delta_{q+2}}{5}.
\endaligned
\end{equation*}
Here we use (\ref{c:M}) to have $M_0(L)^{1/2}\lambda_{q+1}^{13\alpha-1/7}\leq \frac{c_R\delta_{q+2}}{10}$.

Finally, we proceed with the oscillation error $R_{\textrm{osc}}$ and we focus on $R_{\textrm{osc}}^{(x)}$ first. Since $W_{(\xi)}$ is $(\mathbb{T}/(r_\perp \lambda_{q+1}))^3$ periodic, we deduce that
\begin{equation*}\mathbb{P}_{\neq0}(W_{(\xi)}\otimes W_{(\xi)})=\mathbb{P}_{\geq r_\perp \lambda_{q+1}/2}(W_{(\xi)}\otimes W_{(\xi)}),\end{equation*}
where $\mathbb{P}_{\geq r}=\Id-\mathbb{P}_{<r}$ and $\mathbb{P}_{< r}$ denotes a Fourier multiplier operator, which projects a function onto its Fourier frequencies $<r$ in absolute value.  We also recall the following results from \cite[Lemma~7.5]{BV19}.

\begin{lem}\label{lemma} Fix parameters $1\leq \zeta<\kappa, p\in (1,2]$, and assume there exists $N\in\mathbb{N}$ such that $\zeta^N\leq \kappa^{N-2}$. Let
$a\in C^N(\mathbb{T}^3)$ be such that there exists $C_a>0$ with
\begin{equation*}\|D^j a\|_{C^0}\leq C_a \zeta^j,\end{equation*}
for all $0\leq j\leq N$. Assume that $f\in L^p(\mathbb{T}^3)$ such that $\int_{\mathbb{T}^3}a(x)\mathbb{P}_{\geq\kappa}f(x)dx=0$. Then we have
\begin{equation*}\||\nabla|^{-1}(a \mathbb{P}_{\geq\kappa}f)\|_{L^p}\leq C_a \frac{\|f\|_{L^p}}{\kappa},\end{equation*}
where the implicit constant depends only on $p$ and $N$.

\end{lem}

Using Lemma \ref{lemma} with $a=\nabla a^2_{(\xi)}$ for $C_a=M_0(t)\ell^{-9}$, $\zeta=\ell^{-5}$, $\kappa=r_{\perp}\lambda_{q+1}$ and any $N\geq 3$,
we have
\begin{equation*}
\aligned
 \|R_{\textrm{osc}}^{(x)}\|_{C_{t}L^p}&\leq \sum_{\xi\in\Lambda}\big\|\mathcal{R}\big(\nabla a^2_{(\xi)}\mathbb{P}_{\geq r_\perp\lambda_{q+1}/2}(W_{(\xi)}\otimes W_{(\xi)})\big)\big\|_{C_{t}L^p}
\\
&\lesssim M_0(t)\ell^{-9}\frac{\|W_{(\xi)}\otimes W_{(\xi)}\|_{C_{t}L^p}}{r_\perp\lambda_{q+1}}\lesssim M_0(t)\ell^{-9}\frac{\|W_{(\xi)}\|^2_{C_{t}L^{2p}}}{r_\perp\lambda_{q+1}}
\\&\lesssim M_0(t)\ell^{-9}r_\perp^{2/p-2}r_\|^{1/p-1}(r_\perp^{-1}\lambda_{q+1}^{-1})\lesssim M_0(t)\ell^{-9}\lambda_{q+1}^\alpha (r_\perp^{-1}\lambda_{q+1}^{-1})\\
&\lesssim M_{0}(t)\lambda_{q+1}^{19\alpha-1/7}\leq\frac{M_0(t)c_R\delta_{q+2}}{10}.
\endaligned
\end{equation*}
For the second term $R_{\textrm{osc}}^{(t)}$ we use Fubini's theorem to integrate along the orthogonal directions of $\phi_{(\xi)}$ and $\psi_{(\xi)}$ and use (\ref{bounds}) to deduce
\begin{equation*}
\aligned
 \|R_{\textrm{osc}}^{(t)}\|_{C_{t}L^p}&\leq \mu^{-1}\sum_{\xi\in\Lambda}\|\partial_t a_{(\xi)}^2\|_{C^{0}_{t,x}}\|\phi_{(\xi)}\|_{C_{t}L^{2p}}^2\|\psi_{(\xi)}\|_{C_{t}L^{2p}}^2
\\
&\lesssim M_0(t) \mu^{-1}\ell^{-9}r_\perp^{2/p-2}r_\|^{1/p-1}\lesssim M_0(t) \lambda_{q+1}^{19\alpha-9/7}\leq \frac{M_0(t)c_R\delta_{q+2}}{10}.
\endaligned
\end{equation*}

In view of the standard mollification estimates we \rmb{use \eqref{z} to} have that for $t\in[0, T_L]$
\begin{equation*}
\aligned
\|R_{\textrm{com}}\|_{C_{t}L^1}&\lesssim \ell(\|v_q\|_{C^1_{t,x}}+\|\rmb{z_{q}}\|_{C_tC^1})(\|v_q\|_{C_tL^2}+\|\rmb{z_{q}\|_{C_tL^2}})\\
&\qquad+\ell^{\frac{1}{2}-2\delta}(\|\rmb{z_{q}}\|_{C_t^{\frac{1}{2}-2\delta}L^\infty}+\|v\|_{C_{t,x}^1})(\|v_q\|_{C_tL^2}
+\|\rmb{z_{q}}\|_{C_tL^2})\\
&\lesssim2\ell (\lambda_q^4+\rmb{\lambda_{q+1}^{\frac\alpha4}})M_0(t)+\ell^{\frac{1}{2}-2\delta}(\rmb{\lambda_{q+1}^{\frac\alpha4}}+\lambda_q^4)M_0(t)\leq \frac{M_0(t)c_R\delta_{q+2}}{5},
\endaligned
\end{equation*}
 where $\delta<\frac{1}{12}$ and  we require  that  $\ell^{\frac{1}{2}-2\delta}(\rmb{\lambda_{q+1}^{\frac\alpha4}+\lambda_q^4})<\frac{c_R\delta_{q+2}}{10}$, i.e.
 \rmb{$$\lambda_{q+1}^{2\beta b-\alpha/2}\lambda_q^{-2/3}(\lambda_{q+1}^{\frac\alpha4}+\lambda_q^4)\ll1.$$}

  With the  choice of $\ell$ in \eqref{ell1} and since we postulated that \rmb{$\alpha>18\beta b$ and $\alpha b>16$}, this can indeed be achieved by possibly increasing  $a$ and consequently decreasing $\beta$.
Finally, we \rmb{use \eqref{z} to} obtain for $t\in[0, T_L]$
\rmb{\begin{align*}
\|R_{\textrm{com}1}\|_{C_tL^1}&\lesssim (\|v_{q+1}\|_{C_tL^2}+\|z_{q+1}\|_{C_tL^2}+\|z_\ell\|_{C_tL^2})\|z_\ell-\rmb{z_{q+1}}\|_{C_tL^2}
\\&\lesssim M_0(t)^{1/2}\|z_\ell-\rmb{z_{q+1}}\|_{C_tL^2}\lesssim M_0(t)^{1/2}(\|z_\ell-\rmb{z_{q}}\|_{C_tL^2}+\|z_{q+1}-\rmb{z_{q}}\|_{C_tL^2})
\\&\lesssim M_0(t) (\ell^{\frac{1}{2}-2\delta}+\lambda_{q+1}^{-\frac\alpha8(1-\delta)})\leq \frac{M_0(t)c_R\delta_{q+2}}{5},
\end{align*}
where we use $$\lambda_{q+1}^{2b\beta-\frac\alpha8(1-\delta)}\ll1,$$
which holds by $\alpha>18\beta b$}.
Summarizing all the above estimates we obtain
\begin{equation*}
\aligned
&\|\mathring{R}_{q+1}\|_{C_{t}L^1}\leq M_0(t)c_R\delta_{q+2},
\endaligned
\end{equation*}
which is the desired last bound in \eqref{inductionv}. The proof of Proposition \ref{main iteration} is complete.

\section{Non-uniqueness in law II: the case of  a linear multiplicative noise}
\label{s:nonuniquII}

\subsection{Probabilistically weak solutions}

In the case of an additive noise, the stopping times employed in the convex integration  can be regarded as functions of the solution $u$. This does not follow a priori from their definition \eqref{stopping time}, but can be seen from \eqref{eq:Z} and \eqref{eq1}. Accordingly, it was possible to prove non-uniqueness of martingale solutions in the sense of Definition \ref{martingale solution} directly. However, the situation is rather different in the case of a linear multiplicative noise. Indeed,  the stopping times are functions of the driving noise $B$, which is not a function of $u$, and therefore it is necessary to work with the extended canonical space $\bar\Omega$  including trajectories of both the solution $u$ and the noise $B$. To this end, we define the notion of probabilistically weak solution. In the first step, we then establish  joint non-uniqueness in law: we show that the joint law of  $(u,B)$ is not unique. In the second step, we extend the finite-dimensional result of Cherny \cite{C03} to a general SPDE setting (see Appendix \ref{s:C}), proving that uniqueness in law implies joint uniqueness in law. This permits us to conclude the desired non-uniqueness of martingale solutions stated in Theorem~\ref{Main results2 li}.

To avoid confusion, we point out that the two notions of solution, i.e. martingale solution and probabilistically weak solution, are equivalent. The only reason why the proof of non-uniqueness in law from Section~\ref{s:nonuniqueI} does not apply to the case of linear multiplicative noise is the different definition of stopping times. Conversely, the proof of the present section applies to the additive noise case as well. However,  it is more complicated than the direct proof in Section~\ref{s:nonuniqueI} which does not rely on the generalization Cherny's result, Theorem~\ref{cherny}.

\begin{defn}\label{weak solution}
Let $s\geq 0$ and $x_{0}\in L^{2}_{\sigma}$, $y_0\in U_1$. A probability measure $P\in \mathscr{P}(\bar{\Omega})$ is  a probabilistically weak solution to the Navier--Stokes system (\ref{1})  with the initial value $(x_0,y_0) $ at time $s$ provided

\no(M1) $P(x(t)=x_0, y(t)=y_0,  0\leq t\leq s)=1$  and for any $n\in\mathbb{N}$
$$P\left\{(x,y)\in \bar{\Omega}: \int_0^n\|G(x(r))\|_{L_2(U;L_2^\sigma)}^2dr<+\infty\right\}=1.$$

\no(M2) Under $P$,  $y$ is a  cylindrical $(\bar{\mathcal{B}}_{t})_{t\geq s}$-Wiener process on  $U$ starting from $y_0$ at time $s$ and for every $e_i\in C^\infty(\mathbb{T}^3)\cap L^2_\sigma$, and for $t\geq s$
$$\langle x(t)-x(s),e_i\rangle+\int^t_s\langle \div(x(r)\otimes x(r))-\Delta x(r),e_i\rangle dr=\int_s^t \langle e_i, G(x(r))  dy_r\rangle.$$

\no (M3) For any $q\in \mathbb{N}$ there exists a positive real function $t\mapsto C_{t,q}$ such that  for all $t\geq s$
$$E^P\left(\sup_{r\in [0,t]}\|x(r)\|_{L^2}^{2q}+\int_{s}^t\|x(r)\|^2_{H^{\gamma}}dr\right)\leq C_{t,q}(\|x_0\|_{L^2}^{2q}+1).$$
\end{defn}

For the application to the Navier--Stokes system, we will again require a definition of probabilistically weak solutions defined up to a stopping time $\tau$. To this end, we set
$$
\bar{\Omega}_{\tau}:=\{\omega(\cdot\wedge\tau(\omega));\omega\in \bar{\Omega}\}.
$$

\begin{defn}\label{weak solution 1}
Let $s\geq 0$ and $x_{0}\in L^{2}_{\sigma}$, $y_0\in U_1 $. Let $\tau\geq s$ be a $(\bar{\mathcal{B}}_{t})_{t\geq s}$-stopping time. A probability measure $P\in \mathscr{P}(\bar{\Omega}_\tau)$ is  a probabilistically weak solution to the Navier--Stokes system (\ref{1}) on $[s,\tau]$ with the initial value $(x_0,y_0) $ at time $s$ provided

\no(M1) $P(x(t)=x_0, y(t)=y_0,  0\leq t\leq s)=1$
  and for any $n\in\mathbb{N}$
$$P\left\{(x,y)\in \bar\Omega: \int_0^{n\wedge \tau}\|G(x(r))\|_{L_2(U;L_2^\sigma)}^2dr<+\infty\right\}=1.$$

\no(M2) Under $P$, $\langle y(\cdot\wedge \tau),l\rangle_U$ is a {continuous square integrable $(\bar{\mathcal{B}}_{t})_{t\geq s}$}-martingale  starting from $y_0$ at time $s$ with quadratic variation process given by $(t\wedge \tau-s)\|l\|_U^2$ for $l\in U$.  For every $e_i\in C^\infty(\mathbb{T}^3)\cap L^2_\sigma$, and for $t\geq s$
$$\langle x(t\wedge \tau)-x(s),e_i\rangle+\int^{t\wedge \tau}_s\langle \div(x(r)\otimes x(r))-\Delta x(r),e_i\rangle dr=\int_s^{t\wedge\tau} \langle e_i, G(x(r))  dy_r\rangle.$$

\no (M3) For any $q\in \mathbb{N}$ there exists a positive real function $t\mapsto C_{t,q}$ such that  for all $t\geq s$
$$E^P\left(\sup_{r\in [0,t\wedge\tau]}\|x(r)\|_{L^2}^{2q}+\int_{s}^{t\wedge\tau}\|x(r)\|^2_{H^{\gamma}}dr\right)\leq C_{t,q}(\|x_0\|_{L^2}^{2q}+1).$$
\end{defn}

Similarly to Theorem \ref{convergence} we obtain the following existence and stability result. The proof is presented in Appendix~\ref{ap:A}.

\begin{thm}\label{convergence 1}
  For every $(s,x_0,y_0)\in [0,\infty)\times L_{\sigma}^2\times U_1 $, there exists  $P\in\mathscr{P}(\bar{\Omega})$ which is a probabilistically weak solution to the Navier--Stokes system \eqref{1} starting at time $s$ from the initial condition $(x_0,y_0)$  in the sense of Definition \ref{weak solution}. The set of all such probabilistically weak solutions  with the same implicit constant $C_{t,q}$ in  Definition \ref{weak solution} is denoted by $\mathscr{W}(s,x_0,y_0, C_{t,q})$.

   Let $(s_n,x_n,y_n)\rightarrow (s,x_0,y_0)$ in $[0,\infty)\times L_{\sigma}^2\times U_1 $ as $n\rightarrow\infty$  and let $P_n\in \mathscr{W}(s_n,x_n,y_n, C_{t,q})$. Then there exists a subsequence $n_k$ such that the sequence $(P_{n_k})_{k\in\mathbb{N}}$  converges weakly to some $P\in\mathscr{W}(s,x_0,y_0, C_{t,q})$.
\end{thm}

As in the case of additive noise, the non-uniqueness in law stated in Theorem \ref{Main results2 li} means non-uniqueness of martingale solutions in the sense of Definition \ref{martingale solution}. Non-uniqueness of probabilistically weak solutions corresponds to the joint non-uniqueness in law.

\begin{defn}\label{def:uniquelaw1}
We say that joint uniqueness in law holds for  \eqref{1} if probabilistically weak  solutions starting from the same initial distribution are unique.
\end{defn}

\subsection{General construction for  probabilistically weak solutions}
\label{ss:gen1}

The overall strategy is similar to Section~\ref{ss:gen}: in the first step, we shall extend  probabilistically weak solutions defined up a {$(\bar{\mathcal{B}}_{t})_{t\geq 0}$}-stopping time $\tau$ to the whole interval $[0,\infty)$. We denote by $\bar{\mathcal{B}}_{\tau}$ the $\sigma$-field associated to $\tau$.

\begin{prp}\label{prop:1 1}
 Let $\tau$ be a bounded $(\bar{\mathcal{B}}_{t})_{t\geq0}$-stopping time. Then for every $\omega\in \bar{\Omega}$ there exists $Q_{\omega}\in\mathscr{P}(\bar{\Omega})$ such that for $\omega\in \{x(\tau)\in L^2_\sigma\}$
\begin{equation}\label{qomega 1}
Q_\omega\big(\omega'\in\bar{\Omega}; (x,y)(t,\omega')=(x,y)(t,\omega) \textrm{ for } 0\leq t\leq \tau(\omega)\big)=1,
\end{equation}
and
\begin{equation}\label{qomega2 1}
Q_\omega(A)=R_{\tau(\omega),x(\tau(\omega),\omega),y(\tau(\omega),\omega)}(A)\qquad\text{for all}\  A\in \mathcal{B}^{\tau(\omega)}.
\end{equation}
where $R_{\tau(\omega),x(\tau(\omega),\omega),y(\tau(\omega),\omega)}\in\mathscr{P}(\bar{\Omega})$ is a probabilistically weak solution to the Navier--Stokes system \eqref{1} starting at time $\tau(\omega)$ from the initial condition $(x(\tau(\omega),\omega), y(\tau(\omega),\omega))$. Furthermore, for every $B\in\bar{\mathcal{B}}$ the mapping $\omega\mapsto Q_{\omega}(B)$ is $\bar{\mathcal{B}}_{\tau}$-measurable.
\end{prp}

\begin{proof}

The proof of this result is identical to the proof of Proposition \ref{prop:1} applied to the extended path space $\bar\Omega$ instead of $\Omega_{0}$ and making use of Theorem~\ref{convergence 1} instead of Theorem \ref{convergence}.
\end{proof}

We proceed with a result which is analogous to Proposition~\ref{prop:2}.

\begin{prp}\label{prop:2 1}
 Let $x_{0}\in L^{2}_{\sigma}$.
Let $P$ be a probabilistically weak solution to the Navier--Stokes system \eqref{1} on $[0,\tau]$ starting at the time $0$ from the initial condition $(x_{0},0)$. In addition to the assumptions of Proposition \ref{prop:1 1}, suppose that there exists a Borel  set $\mathcal{N}\subset\bar{\Omega}_{\tau}$ such that $P(\mathcal{N})=0$ and for every $\omega\in \mathcal{N}^{c}$ it holds
\begin{equation}\label{Q1 1}
\aligned
&Q_\omega\big(\omega'\in\bar{\Omega}; \tau(\omega')=
\tau(\omega)\big)=1.
\endaligned
\end{equation}
Then the  probability measure $ P\otimes_{\tau}R\in \mathscr{P}(\bar{\Omega})$ defined by
\begin{equation*}
P\otimes_{\tau}R(\cdot):=\int_{\bar{\Omega}}Q_{\omega} (\cdot)\,P(d\omega)
\end{equation*}
satisfies $P\otimes_{\tau}R= P$ on $\sigma\{x(t\wedge\tau),y(t\wedge \tau),t\geq0\}$ and
 is a probabilistically weak solution to the Navier--Stokes system \eqref{1} on $[0,\infty)$ with initial condition $(x_{0},0)$.
\end{prp}

\begin{proof}
The fact that $P\otimes_{\tau}R(A)=P(A)$ holds for every Borel set $A\in \sigma\{x(t\wedge\tau),y(t\wedge \tau),t\geq0\} $ as well as the property (M1) follows directly from the construction together with \eqref{Q1 1}.
In order to show (M3), we write
 \begin{equation*}
 \aligned
  &E^{P\otimes_{\tau}R}\left(\sup_{r\in[0,t]}\|x(r)\|_{L^2}^{2q}+\int_0^t\|x(r)\|_{H^\gamma}^2dr\right)\\
& \leq E^{P\otimes_{\tau}R}\left(\sup_{r\in[0,t\wedge \tau]}\|x(r)\|_{L^2}^{2q}+\int_0^{t\wedge\tau}\|x(r)\|_{H^\gamma}^2dr\right)
 +E^{P\otimes_{\tau}R}\left(\sup_{r\in[t\wedge\tau,t]}\|x(r)\|_{L^2}^{2q}+\int_{t\wedge\tau}^t\|x(r)\|_{H^\gamma}^2dr\right)\\
 & \leq C(\|x_{0}\|_{L^{2}}^{2q}+1)+C(E^P\|x(\tau)\|_{L^2}^{2q}+1)\leq C(\|x_{0}\|_{L^{2}}^{2q}+1),
 \endaligned
 \end{equation*}
where we used  (M3) for $P$ and for $R$, \eqref{Q1 1} and the boundedness of the stopping time $\tau$.

For (M2), we first recall that since $P$ is a probabilistically weak solution on $[0,\tau]$, the process
 $\langle y_{t\wedge\tau},l\rangle_U$ is a   continuous square integrable $(\bar{\mathcal{B}}_t)_{t\geq 0}$-martingale under $P$ with  the quadratic variation process
given by
$(t\wedge\tau)\|l\|_U^2.$
On the other hand, since for every $\omega\in \bar\Omega$, the probability measure $R_{\tau(\omega),x(\tau(\omega),\omega),y(\tau(\omega),\omega)}$ is a probabilistically weak solution  starting at the time $\tau(\omega)$ from the  initial condition $(x(\tau(\omega),\omega),y(\tau(\omega),\omega))$,  the process $\langle y_{t}-y_{t\wedge \tau(\omega)},l\rangle_U$ is a continuous square integrable $(\bar{\mathcal{B}}_{t})_{t\geq \tau(\omega)}$-martingale under $R_{\tau(\omega),x(\tau(\omega),\omega),y(\tau(\omega),\omega)}$ with  the quadratic variation process
given by
$(t-\tau(\omega))\|l\|_U^2$, $t\geq\tau(\omega)$. Then by the same arguments as in the proof of Proposition \ref{prop:2} we deduce
that under $P\otimes_\tau R$, the process $\langle y,l\rangle_{U}$ is a  continuous square integrable $(\bar{\mathcal{B}}_{t})_{t\geq 0}$-martingale  with  the quadratic variation process
given by
$t\|l\|_{U}^2$, $t\geq0$, which implies that $y$ is an  cylindrical $(\bar{\mathcal{B}}_{t})_{t\geq 0}$-Wiener process on $U $.

Furthermore, under $P$ it holds for every $e_i\in C^\infty(\mathbb{T}^3)\cap L^2_\sigma$ and for $t\geq 0$
$$M^{x,y,i}_{t\wedge\tau,0}:=\langle x(t\wedge \tau)-x(0),e_i\rangle+\int^{t\wedge \tau}_0\langle \div(x(r)\otimes x(r))-\Delta x(r),e_i\rangle dr=\int_0^{t\wedge\tau} \langle e_i, G(x(r))  dy_r\rangle.$$
On the other hand, for $\omega\in \bar\Omega$, under $R_{\tau(\omega),x(\tau(\omega),\omega),y(\tau(\omega),\omega)}$ it holds for $t\geq \tau(\omega)$
$$M^{x,y,i}_{t,t\wedge\tau}:=\langle x(t)-x(\tau(\omega)),e_i\rangle+\int^{t}_{\tau(\omega)}\langle \div(x(r)\otimes x(r))-\Delta x(r),e_i\rangle dr=\int_{\tau(\omega)}^{t} \langle e_i, G(x(r))  dy_r\rangle.$$
Therefore, we obtain
\begin{align*}
&P\otimes_\tau R\bigg\{M^{x,y,i}_{t,0}
=\int^{t}_0 \langle e_i, G(x(r))  dy_r\rangle, e_i\in C^\infty(\mathbb{T}^3)\cap L^2_\sigma, t\geq0\bigg\}
\\&\quad=\int_{\bar\Omega} dP(\omega)Q_\omega\bigg\{M^{x,y,i}_{t,t\wedge\tau(\omega)}
=\int^{t}_{t\wedge\tau(\omega)}\langle e_i, G(x(r))  dy_r\rangle,
\\&
\hspace{3.5cm}M^{x,y,i}_{t\wedge \tau(\omega),0}
=\int^{t\wedge\tau(\omega)}_0\langle e_i, G(x(r))  dy_r\rangle, e_i\in C^\infty(\mathbb{T}^3)\cap L^2_\sigma, t\geq0\bigg\}
.
\end{align*}
Now,  using (\ref{Q1 1}) and \eqref{qomega2 1} we obtain
\begin{align*}
&\int_{\bar\Omega} dP(\omega)Q_\omega\bigg\{M^{x,y,i}_{t,t\wedge\tau(\omega)}
=\int^{t}_{t\wedge\tau(\omega)}\langle e_i, G(x(r))  dy_r\rangle,e_i\in C^\infty(\mathbb{T}^3)\cap L^2_\sigma, t\geq0\bigg\}\\
& =\int_{\bar\Omega} dP(\omega)R_{\tau(\omega),x(\tau(\omega),\omega),y(\tau(\omega),\omega)}\bigg\{M^{x,y,i}_{t,t\wedge \tau(\omega)}
=\int^{t}_{t\wedge\tau(\omega)} \langle e_i, G(x(r))  dy_r\rangle, e_i\in C^\infty(\mathbb{T}^3)\cap L^2_\sigma,t\geq0\bigg\}\\
&=1,
\end{align*}
and using (\ref{Q1 1}) and (\ref{qomega 1}) we deduce
\begin{align*}
&\int_{\bar\Omega} dP(\omega)Q_\omega\bigg\{M^{x,y,i}_{t\wedge \tau(\omega),0}
=\int^{t\wedge\tau(\omega)}_0 \langle e_i, G(x(r))  dy_r\rangle, e_i\in C^\infty(\mathbb{T}^3)\cap L^2_\sigma, t\geq0\bigg\}
\\&\quad=P\bigg\{M^{x,y,i}_{t\wedge \tau,0}
=\int^{t\wedge\tau}_0 \langle e_i, G(x(r))  dy_r\rangle, e_i\in C^\infty(\mathbb{T}^3)\cap L^2_\sigma, t\geq0\bigg\}=1.
\end{align*}
In view of the elementary inequality for probability measures $Q_{\omega}(A\cap B)\geq 1-Q_{\omega}(A^{c})-Q_{\omega}(B^{c})$, we finally deduce that $P\otimes_{\tau}R$-a.s.
$$
M^{x,y,i}_{t,0}
=\int^{t}_0 \langle e_i, G(x(r))  dy_r\rangle\qquad\text{ for all }\   e_i\in C^\infty(\mathbb{T}^3)\cap L^2_\sigma,\  t\geq0,
$$
hence the condition (M2) follows.
\end{proof}

\subsection{Application to solutions obtained through Theorem \ref{Main results1 li}}
\label{s:1.4}

The general construction presented in Section \ref{ss:gen1} applies to a general infinite dimensional stochastic perturbation of the Navier--Stokes system. From now on, we restrict ourselves to the setting of a linear multiplicative noise. In particular, the driving Wiener process is real-valued and consequently $U=U_{1}=\mathbb{R}$.

For $n\in\mathbb{N}$, $L>1$ and  $\delta\in(0,1/12)$  we define
 \begin{equation*}
 \aligned\tau_L^n(\omega)&=\inf\left\{t\geq 0, |y(t,\omega)|>(L-\frac{1}{n})^{1/4}\right\}\bigwedge \inf\left\{t>0,\|y(t,\omega)\|_{C_t^{\frac{1}{2}-2\delta}}>(L-\frac{1}{n})^{1/2}\right\}\bigwedge L.
 \endaligned
 \end{equation*}
Then the sequence $\{\tau^{n}_{L}\}_{n\in\mathbb{N}}$ is nondecreasing and we define
\begin{equation}\label{eq:tauL 1}
\tau_L:=\lim_{n\rightarrow\infty}\tau_L^n.
\end{equation}
Without additional regularity of the process $y$, it holds true that $\tau^{n}_{L}(\omega)=0$.
By Lemma~\ref{time} we obtain that $\tau_L^{n}$ is $(\bar{\mathcal{B}}_t)_{t\geq0}$-stopping time and consequently also  $\tau_L$ is a $(\bar{\mathcal{B}}_t)_{t\geq 0}$-stopping time as an increasing limit of stopping times.

Now, we fix   a real-valued Wiener process $B$ defined on a probability space $(\Omega, \mathcal{F},\mathbf{P})$ and we denote by  $(\mathcal{F}_{t})_{t\geq0}$ its normal filtration. On this stochastic basis, we apply Theorem~\ref{Main results1 li} and denote by $u$ the corresponding solution to the Navier--Stokes system \eqref{ns1} on $[0,T_{L}]$, where the stopping time $T_{L}$ is defined in \eqref{stopping time li}. We recall that $u$ is adapted with respect to $(\mathcal{F}_{t})_{t\geq0}$ which is an essential property employed to prove the martingale property in the proof of  Proposition~\ref{prop:ext 1}. We denote by $P$ the law of $(u,B)$ and obtain the following result by similar arguments as in the proof of Proposition \ref{prop:ext}.

\begin{prp}\label{prop:ext 1}
The probability measure $P$ is a probabilistically weak solution to the Navier--Stokes system \eqref{ns1} on $[0,\tau_{L}]$ in the sense of Definition \ref{weak solution 1}, where $\tau_{L}$ was defined in \eqref{eq:tauL 1}.
\end{prp}

\begin{proof}
The proof is similar as the proof of Proposition \ref{prop:ext} once we note that
$$y(t,(u,B))=B(t)\quad \textrm{for } t\in [0,T_L] \quad \mathbf{P}\text{-a.s.}$$
In particular, the property (M2) in Definition \ref{weak solution 1} follows since $(u,B)$ satisfies \eqref{ns1}.
\end{proof}

\begin{prp}\label{prp:ext2 1}
The probability measure $P\otimes_{\tau_{L}}R$ is a probabilistically weak solution to the Navier--Stokes system \eqref{ns1} on $[0,\infty)$ in the sense of Definition \ref{weak solution}.
\end{prp}

\begin{proof}
In light of Proposition \ref{prop:1 1} and Proposition \ref{prop:2 1}, it only remains to establish \eqref{Q1 1}, which follows by similar arguments as in the proof of Proposition \ref{prp:ext2}.
\end{proof}

Finally, we have all in hand to conclude the proof of our main result Theorem \ref{Main results2 li}.

\begin{proof}[Proof of Theorem  \ref{Main results2 li}]
Let $T>0$ be arbitrary. Let $\kappa=1/2$ and $K=2$ and apply Theorem \ref{Main results1 li} and  Proposition \ref{prp:ext2 1}. As in the proof of Theorem~\ref{Main results2} it follows that the constructed probability measure $P\otimes_{\tau_{L}}R$ satisfies
$$
P\otimes_{\tau_{L}}R(\tau_{L}\geq T)=\mathbf{P}(T_{L}\geq T)>1/2,
$$
and consequently
$$
E^{P\otimes_{\tau_{L}}R}\big[\|x(T)\|_{L^{2}}^{2}\big]>2e^{T}\|x_{0}\|_{L^{2}}^{2}.
$$
The initial value   $x_{0}=v(0)\in L^2_\sigma$ is given through the construction in Theorem~\ref{Main results1 li}.
However, based on a Galerkin approximation one can construct a  probabilistically weak solution $\tilde P$ to (\ref{ns1}) starting from the same initial value as $P\otimes_{\tau_{L}}R$. In addition, this solution satisfies the usual energy inequality, that is,
$$
E^{\tilde P}\big[\|x(T)\|_{L^{2}}^{2}\big]\leq e^{T}\|x_{0}\|_{L^{2}}^{2}.
$$

Therefore, the two probabilistically weak solutions are distinct and as a consequence joint non-uniqueness in law, i.e. non-uniqueness of probabilistically weak solutions, holds for the Navier--Stokes system \eqref{ns1}. In view of Theorem \ref{cherny} we finally deduce that the desired non-uniqueness in law, i.e., non-uniqueness of martingale solutions, holds as well.
\end{proof}

\section{Proof of Theorem \ref{Main results1 li}}
\label{s:1.3}

As the first step,  we transform (\ref{ns1}) to a random PDE. To this end, we consider the stochastic process
$$\theta(t)=e^{B_t},\  t\geq0,$$
and define $v:=\theta^{-1} u$. Then, by It\^o's formula we obtain
\begin{equation}
\label{ns2}
\aligned
 \partial_t v+\frac{1}{2}v-\Delta v+\theta\div(v\otimes v)+\theta^{-1}\nabla P&=0,
\\\div v&=0.
\endaligned
\end{equation}

Our aim is to develop a similar induction argument as in Section \ref{s:1.1} and apply it to \eqref{ns2}.  At each step $q\in\mathbb{N}_{0}$, a pair $(v_q, \mathring{R}_q)$ is constructed solving the following system
\begin{equation}
\label{induction li}
\aligned
 \partial_tv_q+\frac{1}{2}v_q-\Delta v_q +\theta\div(v_q\otimes v_q)+\nabla p_q&=\div \mathring{R}_q,
\\
\div v_q&=0.
\endaligned
\end{equation}
We choose suitable parameters $a\in\mathbb{N}$ and $b\in\mathbb{N}$ sufficiently large and a parameter $\beta\in (0,1)$ sufficiently small and define
$$\lambda_q=a^{(b^q)}, \quad \delta_q=\lambda_q^{-2\beta}.$$
The necessary stopping times $T_{L}$ are now defined in terms of the Wiener Process $B$ as
\begin{equation}\label{stopping time li}
T_L:=\inf\{t>0, |B(t)|\geq L^{1/4}\}\wedge \inf\{t>0,\|B\|_{C_t^{1/2-2\delta}}\geq L^{1/2}\}\wedge L
\end{equation}
for  $L>1$ and $\delta\in (0,1/12)$.
As a consequence, it holds for  $t\in[0, T_L]$
\begin{equation}
\label{B}\aligned
|B(t)|\leq L^{1/4},\quad\|B\|_{C^{1/2-2\delta}_t}&\leq L^{1/2},
\endaligned
\end{equation}
which implies
\begin{equation}\label{alpha}
\|\theta\|_{C_t^{\frac{1}{2}-2\delta}}+|\theta(t)|+|\theta^{-1}(t)|\leq 3L^{1/2}e^{L^{1/4}}=:m_L^2.
\end{equation}
We also define
\begin{equation}\label{M0}
M_0(t):=e^{4Lt+2L}.
\end{equation}
For the induction, we will assume the following bounds for $(v_q,\mathring{R}_q)$ which are   valid  for $t\in[0,T_L]$
\begin{equation}
\label{inductionv li}
\aligned
\|v_q\|_{C_tL^2}\leq& m_LM_0(t)^{1/2}(1+\sum_{1\leq r\leq q}\delta_{r}^{1/2})\leq 2m_LM_0(t)^{1/2}  ,
\\ \|v_q\|_{C^1_{t,x}}\leq& m_LM_0(t)^{1/2}\lambda_q^4,
\\\|\mathring{R}_q\|_{C_tL^1}\leq & c_RM_0(t)\delta_{q+1}.
\endaligned
\end{equation}
Here  $\sum_{1\leq r\leq 0}\delta_{r}^{1/2}:=0$,  $c_R>0$ is a sufficiently small universal constant given in (\ref{estimate a li}), \eqref{estimate wqp li}  and we used the fact that $\sum_{r\geq 1}\delta_{r}^{1/2}\leq \sum_{r\geq1}a^{-rb\beta}=\frac{a^{-\beta b}}{1-a^{-\beta b}}<1/2$ and
\begin{equation}\label{cl:3}
 a^{\beta b}>3
\end{equation}
in the first inequality. The following result sets the starting point of our iteration procedure and gives the key compatibility conditions between the parameters $L,a,\beta,b$.

\begin{lem}\label{lem:v0 li}
Let $L>1$ and define
$$
v_0(t,x):=\frac{m_L e^{2Lt+L}}{(2\pi)^{\frac{3}{2}}}\left(\sin( x_3),0,0\right).
$$
Then the associated Reynolds stress is given by
\begin{equation*}
\mathring{R}_0(t,x)=\frac{m_L(2L+3/2)e^{2Lt+L}}{(2\pi)^{3/2}}\left(
\begin{array}{ccc} 0 & 0 &-\cos(x_3)
\\ 0 & 0 &0\\ -\cos(x_3) &0 &0
 \end{array}
\right).
\end{equation*}
The initial values $v_0(0,x)$ and $\mathring{R}_0(0,x)$ are deterministic.
Moreover,   all the estimates in \eqref{inductionv li} on the level $q=0$ for $(v_{0},\mathring{R}_{0})$ as well as \eqref{cl:3} are valid  provided
\begin{equation}\label{c:li1}
18\cdot (2\pi)^{3/2}\sqrt{3}<2\cdot(2\pi)^{3/2}\sqrt{3} a^{2\beta b}\leq \frac{c_{R}e^{L}}{L^{1/4}(2L+\frac32)e^{\frac12 L^{1/4}}},\qquad 4L\leq a^{4}.
\end{equation}
In particular, the minimal lower bound for $L$ is given through
\begin{equation}\label{c:li2}
18\cdot (2\pi)^{3/2}\sqrt{3}<\frac{c_{R}e^{L}}{L^{1/4}(2L+\frac32)e^{\frac12 L^{1/4}}}.
\end{equation}
\end{lem}

\begin{proof}
We observe that   for $t\in [0,T_L]$
$$\|v_0(t)\|_{L^2}=\frac{m_Le^{2Lt+L}}{\sqrt{2}}\leq m_LM_0(t)^{1/2},\qquad\|v_0\|_{C^1_{t,x}}\leq 4Lm_Le^{2Lt+L}\leq m_LM_0(t)^{1/2}\lambda_0^4,$$
 provided
 \begin{equation}\label{cl:2}4L\leq a^4.\end{equation}
The associated Reynolds stress can be directly computed
and admits the  bound
\begin{equation*}
\|\mathring{R}_0(t)\|_{L^1}\leq 2\cdot(2\pi)^{\frac{3}{2}}m_L(2L+3/2)e^{2Lt+L}\leq M_0(t)c_R\delta_1,
\end{equation*}
provided
\begin{equation}\label{cl:4}2\cdot(2\pi)^{\frac{3}{2}}\sqrt{3}L^{1/4}(2L+3/2)e^{\frac{1}{2}L^{1/4}}\leq e^Lc_Ra^{-2\beta b}.\end{equation}
Under the conditions \eqref{cl:2} and \eqref{cl:4}  all the estimates in (\ref{inductionv li}) are valid on the level $q=0$. Combining \eqref{cl:2}, \eqref{cl:4} with \eqref{cl:3} we arrive at \eqref{c:li1}, \eqref{c:li2} from the statement of the lemma.
\end{proof}

We note that the compatibility conditions \eqref{c:li1}, \eqref{c:li2} are similar in spirit to the corresponding conditions in the additive noise case, i.e. \eqref{c:a3}, \eqref{c:a2}. In other words, \eqref{c:li2} gives the minimal admissible lower bound for $L$. Then based on the second condition in \eqref{c:li1} we obtain a minimal admissible lower bound for $a$. Whenever we need to increase $a$ or $b$ in the course of the main iteration proposition below, we have to decrease the value of $\beta$ simultaneously so that the first condition in \eqref{c:li1} is not violated.

\begin{prp}\label{main iteration li} \emph{(Main iteration)}
Let  $L>1$ satisfying \eqref{c:li2} be given and let $(v_q,\mathring{R}_q)$ be an $(\mathcal{F}_t)_{t\geq0}$-adapted solution to \eqref{induction li} satisfying \eqref{inductionv li}. Then there exists a choice of parameters $a,b,\beta$ such that \eqref{c:li1} is fulfilled and  there exist $(\mathcal{F}_t)_{t\geq0}$-adapted processes
$(v_{q+1},\mathring{R}_{q+1})$ which solve \eqref{induction li}, obey \eqref{inductionv li} at level $q+1$ and  for $t\in[0,T_{L}]$ we have
\begin{equation}\label{iteration li}
\|v_{q+1}(t)-v_q(t)\|_{L^2}\leq m_LM_0(t)^{1/2}\delta_{q+1}^{1/2}.
\end{equation}
Furthermore, if $v_q(0), \mathring{R}_q(0)$ are deterministic, so are $v_{q+1}(0), \mathring{R}_{q+1}(0)$.
\end{prp}

The proof of Proposition \ref{iteration li} is presented in Section \ref{s:it li} below. Based on this result, we are able to conclude the proof of Theorem \ref{Main results1 li}.

\begin{proof}[Proof of Theorem \ref{Main results1 li}]
Starting from $(v_0,\mathring{R}_0)$ given in Lemma \ref{lem:v0 li} and using Proposition \ref{main iteration li} we obtain a sequence $(v_q, \mathring{R}_q)$ satisfying (\ref{inductionv li}) and (\ref{iteration li}).
 By interpolation, it follows for $\gamma\in (0,\frac{\beta}{4+\beta})$, $t\in [0,T_L]$
 \begin{equation*}
 \aligned
 \sum_{q\geq0}\|v_{q+1}(t)-v_q(t)\|_{H^{\gamma}}&\lesssim \sum_{q\geq0}\|v_{q+1}(t)-v_q(t)\|_{L^2}^{1-\gamma}\|v_{q+1}(t)-v_q(t)\|_{H^1}^{\gamma}
 \lesssim m_LM_0(t)^{1/2}.
 \endaligned
 \end{equation*}
Therefore, the sequence $v_{q}$ converges to a limit $v\in C([0,T_L],H^{\gamma})$ which is $(\mathcal{F}_{t})_{t\geq0}$-adapted.  Furthermore, we know that $v$ is an analytically weak solution to  (\ref{ns2}) with a deterministic initial value, since due to \eqref{inductionv li} it holds $\lim_{q\rightarrow\infty}\mathring{R}_q=0$ in $C([0,T_L];L^1)$.
According to (\ref{iteration li}) and \eqref{cl:3}, it follows  for $t\in[0,T_L]$
\begin{align*}
\|v(t)-v_0(t)\|_{L^2}&\leq \sum_{q\geq0}\|v_{q+1}(t)-v_q(t)\|_{L^2}\leq m_LM_0(t)^{1/2}\sum_{q\geq0}\delta_{q+1}^{1/2}
\leq \frac{1}{2}m_LM_0(t)^{1/2}.
\end{align*}
Now, we show that for a given $T>0$ we can choose $L=L(T)>1$ large enough so that $v$ fails the corresponding energy inequality at time $T$, namely, it holds
\begin{equation}\label{eq:env li}
\|v(T)\|_{L^{2}}>e^{2L^{1/2}}\|v(0)\|_{L^{2}}
\end{equation}
on the set $\{T_{L}\geq T\}$.
To this end, we observe that
$$e^{2L^{1/2}}\|v(0)\|_{L^2}\leq e^{2L^{1/2}}\left(\|v_0(0)\|_{L^2}+\|v(0)-v_0(0)\|_{L^2}\right)\leq e^{2L^{1/2}}\frac{3}{2}m_LM_0(0)^{1/2}.$$
On the other hand, we obtain on  $\{T_L\geq T\}$
\begin{equation*}
\aligned
&\|v(T)\|_{L^2}\geq (\|v_0(T)\|_{L^2}-\|v(T)-v_0(T)\|_{L^2})
\geq\left(\frac{1}{\sqrt{2}}-\frac{1}{2}\right)m_LM_0(T)^{1/2}>e^{2L^{1/2}}\frac{3}{2}m_LM_0(0)^{1/2}
\endaligned
\end{equation*}
provided
\begin{equation}\label{cl:1}
\left(\frac{1}{\sqrt{2}}-\frac{1}{2}\right)e^{2LT}>\frac{3}{2}e^{2L^{1/2}}.
\end{equation}
Hence \eqref{eq:env li} follows for a suitable choice of $L$ satisfying additionally \eqref{cl:1}.
Furthermore, for a given $T>0$ we could possibly increase $L$ so that
 $\mathbf{P}(T_L\geq T)>\kappa$.

To conclude the proof, we define  $u:=\theta v$ and observe that $u(0)=v(0)$. In addition, $u$ is $(\mathcal{F}_{t})_{t\geq0}$-adapted and solves the original Navier--Stokes system \eqref{ns}. Then in view of \eqref{eq:env li} and the fact that due to \eqref{B} it holds true $|\theta_{T}|\geq e^{-L^{1/4}}$ on the set  $\{T_{L}\geq T\}$, we obtain
$$
\|u(T)\|_{L^{2}}=|\theta(T)|\,\|v(T)\|_{L^{2}}> e^{L^{1/2}}\|u(0)\|_{L^{2}}
$$
on $\{T_{L}\geq T\}$.
Choosing $L$ sufficiently large in dependence on $K$ and $T$ from the statement of the theorem, the desired lower bound follows. Finally, setting $\mathfrak{t}:=T_{L}$ completes the proof.
\end{proof}

\subsection{The main iteration -- proof of Proposition \ref{main iteration li}}
\label{s:it li}

The overall strategy of the proof is similar to Section~\ref{ss:it} but modifications are required since the approximate system on the level $q$ has a different form. As in Section~\ref{ss:it}, we have to make sure that the construction is $(\mathcal{F}_{t})_{t\geq0}$-adapted at each step.

\subsubsection{Choice of parameters}
We choose a small parameter $\ell\in (0,1)$  as in Section \ref{s:c}:
for a sufficiently small  $\alpha\in (0,1)$ to be chosen below, we let $\ell\in (0,1)$ be  a small parameter defined in (\ref{ell1}) and satisfying (\ref{ell}). We note that the compatibility conditions \eqref{c:li1}, (\ref{c:li2}) as well and the last condition in (\ref{ell}) can all be fulfilled provided we make $a$ large enough and $\beta$ small enough at the same time. In addition,
we will require $\alpha b>16$ and $\alpha>8\beta b$.

In order to verify the inductive estimates (\ref{inductionv li}) we need to absorb various expressions including $m_L^{4}M_0(t)^{1/2}$ for all $t\in[0,T_L]$. To this end, we need to change  the condition (\ref{c:M}) in Section \ref{s:c} to
\begin{equation}\label{c:M li}
Cm_L^4\ell^{1/3}\lambda_q^4\leq\frac{c_R\delta_{q+2}}{5}, \quad m_L^4M_0(L)^{1/2}\lambda_{q+1}^{13\alpha-\frac{1}{7}}\leq \frac{c_R\delta_{q+2}}{10},\quad m_L\leq \ell^{-1}.
\end{equation}
 In other words, we need
 $$
 9Le^{2L^{1/4}}a^{b(-\frac{\alpha}{2}+\frac{10}{3b}+2b\beta)}\ll1,
 $$
$$
9Le^{2L^{1/4}}e^{2L^2+L}a^{b(13\alpha-\frac{1}{7}+2b\beta)}\ll1,
$$
$$\sqrt{3}L^{1/4}e^{1/2L^{1/4}}\leq a^{2+\frac{3\alpha}{2}\cdot7L^2}.$$
Choosing  $b=(7L^{2})\vee (17\cdot 14^{2})$,
in view of $\alpha>8\beta b$, \eqref{c:M li}  can be achieved by choosing $a$ large enough and $\alpha=14^{-2}$. This choice also satisfies $\alpha b>16$ required above and the condition $\alpha>8\beta b$ can be achieved by choosing $\beta$ small. It is also compatible with all the other requirements needed below.

\subsubsection{Mollification}
As the next step, we define
space-time mollifications of $v_q$ and $\mathring{R}_q$ and a time mollification of $\theta$ as follows
$$v_\ell=(v_q*_x\phi_\ell)*_t\varphi_\ell,\qquad\mathring{R}_\ell=(\mathring{R}_q*_x\phi_\ell)*_t\varphi_\ell,\qquad\theta_\ell=e^{B}*_t\varphi_\ell.$$
By choosing time mollifiers that are compactly supported in $\mathbb{R}^{+}$, the mollification preserves   $(\mathcal{F}_t)_{t\geq0}$-adaptedness. If the initial data $v_q(0), \mathring{R}_q(0)$ are deterministic, so are $v_\ell(0)$ and $\mathring{R}_\ell(0), \partial_t\mathring{R}_\ell(0)$.
Then using (\ref{induction li}) we obtain that $(v_\ell,\mathring{R}_\ell)$ satisfies
\begin{equation*}
\aligned
 \partial_tv_\ell+\frac{1}{2}v_\ell -\Delta v_\ell+\theta_\ell\div(v_\ell\otimes v_\ell)+\nabla p_\ell&=\div (\mathring{R}_\ell+R_{\textrm{com}})
\\
\div v_\ell&=0,
\endaligned
\end{equation*}
where
\begin{equation*}
R_{\textrm{com}}=\theta_\ell(v_\ell\mathring{\otimes}v_\ell)-(\theta v_q\mathring{\otimes}v_q)*_x\phi_\ell*_t\varphi_\ell,
\end{equation*}
\begin{equation*}
p_\ell=(p_q*_x\phi_\ell)*_t\varphi_\ell-\frac{1}{3}(\theta_\ell|v_\ell|^2-(\theta|v_q|^2*_x\phi_\ell)*_t\varphi_\ell).
\end{equation*}

With this setting, the counterparts of the estimates \eqref{error}, \eqref{eq:vl} and \eqref{eq:vl2} are obtained the same way only replacing $M_0(t)^{1/2}$  by $m_LM_0(t)^{1/2}$. In particular,
\begin{equation}\label{error li}
\|v_q-v_\ell\|_{C_tL^2}
\leq \frac{1}{4} m_LM_0(t)^{1/2}\delta_{q+1}^{1/2},
\end{equation}
\begin{equation}\label{eq:vl li}
\|v_\ell\|_{C_tL^2}
\leq m_L M_0(t)^{1/2} (1+\sum_{1\leq r\leq q}\delta_r^{1/2})\leq 2m_LM_0(t)^{1/2},
\end{equation}
\begin{equation}\label{eq:vl2 li}
\|v_\ell\|_{C^N_{t,x}}
\leq m_L M_0(t)^{1/2} \ell^{-N}\lambda_{q+1}^{-\alpha}.
\end{equation}

\subsubsection{Construction of $v_{q+1}$}

We recall that the intermittent jets $W_{(\xi)}$
and the corresponding estimates are summarized in Appendix \ref{s:B}. The parameters $\lambda,r_{\|},r_{\perp},\mu$ are chosen as in (\ref{parameter}) and we define  $\chi$ and $\rho$ be the same functions as in Section \ref{s:con} with $M_0(t)$ given by  (\ref{M0}).
Now, we define the modified amplitude functions
\begin{equation}
\label{amplitudes li}
\aligned
\bar{a}_{(\xi)}(\omega,t,x):=\bar{a}_{\xi,q+1}(\omega,t,x)&:=\theta_\ell^{-1/2}\rho(\omega,t,x)^{1/2}\gamma_\xi\left(\Id
-\frac{\mathring{R}_\ell(\omega,t,x)}{\rho(\omega,t,x)}\right)(2\pi)^{-\frac{3}{4}}
\\&\,=\theta_\ell^{-1/2}a_{\xi,q+1}(\omega,t,x),\endaligned\end{equation}
where $\gamma_\xi$ is introduced in  Lemma \ref{geometric} and $a_{\xi,q+1}$ is as in Section \ref{s:con} with $M_0(t)$ given in (\ref{M0}). Since $\rho, \theta_\ell$ and $\mathring{R}_\ell$ are $(\mathcal{F}_t)_{t\geq0}$-adapted, we know $\bar{a}_{(\xi)}$ is $(\mathcal{F}_t)_{t\geq0}$-adapted.
Note that since $\theta_\ell(0)$ and $\partial_t\theta_\ell(0)$ are deterministic, if $\mathring{R}_\ell(0), \partial_t\mathring{R}_\ell(0)$ are deterministic, so are $\bar{a}_\xi(0)$ and $\partial_t\bar{a}_\xi(0)$. By  (\ref{geometric equality}) we have
\begin{equation}
\label{cancellation li}
(2\pi)^{3/2}\sum_{\xi\in\Lambda}\bar{a}_{(\xi)}^2\strokedint_{\mathbb{T}^3}W_{(\xi)}\otimes W_{(\xi)}dx= \theta_\ell^{-1}(\rho\Id-\mathring{R}_\ell),
\end{equation}
and for $t\in[0, T_L]$
\begin{equation}
\label{estimate a li}
\aligned
\|\bar{a}_{(\xi)}\|_{C_tL^2}&\leq  \|\theta_\ell^{-1/2}\|_{C_t}\|\rho\|_{C_tL^1}^{1/2}\|\gamma_\xi\|_{C^0(B_{1/2}(\Id))}\\
&\leq \frac{4c_R^{1/2}(8\pi^3+1)^{1/2}M}{8|\Lambda|(8\pi^3+1)^{1/2}}m_LM_0(t)^{1/2}\delta_{q+1}^{1/2}\leq \frac{c_R^{1/4}m_LM_0(t)^{1/2}\delta_{q+1}^{1/2}}{2|\Lambda|},
\endaligned
\end{equation}
where we choose $c_R$ as a small universal constant to absorb $M$ and $M$ denotes the universal constant from  Lemma \ref{geometric} and we apply the bound $|\theta_\ell^{-1}|\leq m_L^2$.
Furthermore, since $\rho$ is bounded from below by $4c_R \delta_{q+1}M_0(t)$ we obtain for $t\in[0, T_L]$
\begin{equation}
\label{estimate aN li}
\|\bar{a}_{(\xi)}\|_{C^N_{t,x}}\lesssim \ell^{-2-5N}c_R^{1/4}m_LM_0(t)^{1/2}\delta_{q+1}^{1/2},
\end{equation}
for $N\geq0$, where we used (\ref{alpha}) and $4L\leq \ell^{-1}$ and the derivative of $\theta_\ell^{-1/2}$ gives extra $\ell^{-1}m_L^4$ and $m_L\leq \ell^{-1}$

As the next step, we define $w_{q+1}$ similarly as in Section \ref{s:con}. In particular, first
we define the principal part $w_{q+1}^{(p)}$ of $w_{q+1}$ as (\ref{principle}) with $a_{(\xi)}$ replaced by $\bar{a}_{(\xi)}$ given in (\ref{amplitudes li}).
Then it follows from  (\ref{cancellation li})
\begin{equation*}\theta_\ell w_{q+1}^{(p)}\otimes w_{q+1}^{(p)}+\mathring{R}_\ell=\theta_\ell\sum_{\xi\in \Lambda}\bar{a}_{(\xi)}^2 \mathbb{P}_{\neq0}(W_{(\xi)}\otimes W_{(\xi)})+\rho \Id.\end{equation*}
The incompressible corrector $w_{q+1}^{(c)}$ is therefore defined as in (\ref{incompressiblity}) again  with $a_{(\xi)}$ replaced by $\bar{a}_{(\xi)}$.
The temporal corrector  $w_{q+1}^{(t)}$ is now defined as in (\ref{temporal}) with  $a_{(\xi)}$ given in (\ref{amplitudes}) for $M_0(t)$ from (\ref{M0}). Note that  for the temporal corrector we use the original amplitude functions $a_{(\xi)}$ from Section~\ref{s:con} (only using  a different function  $M_{0}(t)$), since we need the extra $\theta_\ell$ to obtain a suitable cancelation.
The total velocity increment $w_{q+1}$ and the new velocity $v_{q+1}$ are then given by
\begin{equation*}
w_{q+1}:=w_{q+1}^{(p)}+w_{q+1}^{(c)}+w_{q+1}^{(t)},\qquad v_{q+1}:=v_\ell+w_{q+1}.
\end{equation*}
Both are $(\mathcal{F}_t)_{t\geq0}$-adapted, divergence free and $w_{q+1}$ is mean zero. If $v_q(0)$,  $\mathring{R}_\ell(0)$ are deterministic, so is $v_{q+1}(0)$.

\subsubsection{Verification of the inductive estimates for $v_{q+1}$}

For the counterparts of the estimates (\ref{estimate wqp})-(\ref{corrector est2}), the main difference now is the  extra $m_L$ appearing in the bounds (\ref{estimate a li}) and (\ref{estimate aN li}) for $\bar{a}_{(\xi)}$. Therefore, many of the estimates remain valid with $M_{0}(t)^{1/2}$ replaced by $m_{L}M_{0}(t)^{1/2}$, only the bounds for the temporal corrector $w_{q+1}^{(t)}$ do not change. More precisely, we obtain for $t\in[0, T_L]$
\begin{equation}
\label{estimate wqp li}
\|w_{q+1}^{(p)}\|_{C_tL^2}\leq \frac{1}{2}m_LM_0(t)^{1/2}\delta_{q+1}^{1/2},
\end{equation}
\begin{equation}
\label{principle est1 li}
\aligned
\|w_{q+1}^{(p)}\|_{C_tL^p}\lesssim m_LM_0(t)^{1/2}\delta_{q+1}^{1/2}\ell^{-2}r_\perp^{2/p-1}r_\|^{1/p-1/2},\endaligned
\end{equation}
\begin{equation}\label{correction est li}\aligned\|w_{q+1}^{(c)}\|_{C_tL^p}\lesssim m_L M_0(t)^{1/2}\delta_{q+1}^{1/2}\ell^{-12}r_\perp^{2/p}r_\|^{1/p-3/2},
\endaligned
\end{equation}
\begin{equation}\label{temporal est1 li}
\aligned
\|w_{q+1}^{(t)}\|_{C_tL^p}\lesssim M_0(t)\delta_{q+1}\ell^{-4}r_\perp^{2/p-1}r_\|^{1/p-2}\lambda_{q+1}^{-1}.
\endaligned
\end{equation}
Combining (\ref{estimate wqp li}), (\ref{correction est li}) and (\ref{temporal est1 li}) then leads to
\begin{equation}\label{estimate wq li}
\aligned
\|w_{q+1}\|_{C_tL^2}
&\leq m_L M_0(t)^{1/2}\delta_{q+1}^{1/2}\left(\frac{1}{2}+C\lambda_{q+1}^{24\alpha-2/7}+CM_0(t)^{1/2}\delta_{q+1}^{1/2}\lambda_{q+1}^{8\alpha-1/7}\right)
\leq \frac{3}{4}m_LM_0(t)^{1/2}\delta_{q+1}^{1/2},
\endaligned
\end{equation}
where we used (\ref{c:M li}) to bound $CM_0(L)^{1/2}\delta_{q+1}^{1/2}\lambda_{q+1}^{8\alpha-1/7}\leq 1/8$.

As a consequence of \eqref{estimate wq li} and \eqref{eq:vl li}, the first bound in \eqref{inductionv li} on the level $q+1$ readily follows. In addition, \eqref{estimate wq li} together with \eqref{error li} implies \eqref{iteration li} from the statement of the proposition.
In order to verify the second bound in \eqref{inductionv li}, we observe that similarly to  (\ref{principle est2})-(\ref{temporal est2}) it holds for $t\in[0, T_L]$
\begin{equation}
\label{principle est2 li}
\aligned
\|w_{q+1}^{(p)}\|_{C^1_{t,x}}\lesssim m_L M_0(t)^{1/2}  \ell^{-7}r_\perp^{-1}r_\|^{-1/2}\lambda_{q+1}^2,
\endaligned
\end{equation}
\begin{equation}
\label{correction est2 li}
\aligned
\|w_{q+1}^{(c)}\|_{C^1_{t,x}}\lesssim m_LM_0(t)^{1/2}\ell^{-17}r_\|^{-3/2}\lambda_{q+1}^2,
\endaligned
\end{equation}
\begin{equation}
\label{temporal est2 li}
\aligned\|w_{q+1}^{(t)}\|_{C^1_{t,x}}
\lesssim M_0(t)\ell^{-9}r_\perp^{-1}r_\|^{-2}\lambda_{q+1}^{1+\alpha}.
\endaligned
\end{equation}
Combining \eqref{principle est2 li}, \eqref{correction est2 li}, \eqref{temporal est2 li} with \eqref{eq:vl2 li} and taking (\ref{c:M li}) into account, the second bound in \eqref{inductionv li} follows.

In order to control the Reynolds stress below, we observe that similarly to \eqref{principle est22}, \eqref{corrector est2}, the following bounds hold true for $t\in[0,T_{L}]$, $p\in(1,\infty)$
\begin{equation}
\label{principle est22 li}
\aligned
\|w_{q+1}^{(p)}+w_{q+1}^{(c)}\|_{C_tW^{1,p}}\leq m_L M_0(t)^{1/2}r_\perp^{2/p-1}r_\|^{1/p-1/2}\ell^{-2}\lambda_{q+1},
\endaligned
\end{equation}
\begin{equation}
\label{corrector est2 li}
\aligned\|w_{q+1}^{(t)}\|_{C_tW^{1,p}}&\leq M_0(t)r_\perp^{2/p-2}r_\|^{1/p-1}\ell^{-4}\lambda_{q+1}^{-2/7}.
\endaligned
\end{equation}

\subsubsection{Definition of the Reynolds Stress $\mathring{R}_{q+1}$}

Similar as before we know
\begin{equation*}
\aligned
\div\mathring{R}_{q+1}-\nabla p_{q+1}&=\underbrace{\frac{1}{2}w_{q+1}-\Delta w_{q+1}+\partial_t(w_{q+1}^{(p)}+w_{q+1}^{(c)})+\theta_\ell\div(v_\ell\otimes w_{q+1}+w_{q+1}\otimes v_\ell)}_{\div(R_{\textrm{lin}})+\nabla p_{\textrm{lin}}}
\\&\quad+\underbrace{\theta_\ell\div\left((w_{q+1}^{(c)}+w_{q+1}^{(t)})\otimes w_{q+1}+w_{q+1}^{(p)}\otimes (w_{q+1}^{(c)}+w_{q+1}^{(t)})\right)}_{\div(R_{\textrm{cor}}+\nabla p_{\textrm{cor}})}
\\&\quad+\underbrace{\div(\theta_\ell w_{q+1}^{(p)}\otimes w_{q+1}^{(p)}+\mathring{R}_\ell)+\partial_tw_{q+1}^{(t)}}_{\div(R_{\textrm{osc}})+\nabla p_{\textrm{osc}}}
\\&\quad+\underbrace{(\theta-\theta_\ell)\div\left(v_{q+1}\mathring{\otimes} v_{q+1}\right)}_{\div(R_{\textrm{com}1})+\nabla p_{\textrm{com}1}}+\div(R_{\textrm{com}})-\nabla p_\ell.
\endaligned\end{equation*}
Therefore, applying the inverse divergence operator $\mathcal{R}$ we define
\begin{equation*}
R_{\textrm{lin}}:=\frac{1}{2}\mathcal{R}w_{q+1}-\mathcal{R}\Delta w_{q+1}+\mathcal{R}\partial_t(w_{q+1}^{(p)}+w_{q+1}^{(c)})
+\theta_\ell v_\ell\mathring\otimes w_{q+1}+\theta_\ell w_{q+1}\mathring\otimes v_\ell,
\end{equation*}
\begin{equation*}
R_{\textrm{cor}}:=\theta_\ell(w_{q+1}^{(c)}+w_{q+1}^{(t)})\mathring{\otimes} w_{q+1}+\theta_\ell w_{q+1}^{(p)}\mathring{\otimes} (w_{q+1}^{(c)}+w_{q+1}^{(t)}),
\end{equation*}
\begin{equation*}
R_{\textrm{com}1}:=(\theta_\ell-\theta)(v_{q+1}\mathring{\otimes} v_{q+1}).
\end{equation*}
And similarly to Section \ref{s:def} we have
\begin{equation*} R_{\textrm{osc}}:=\sum_{\xi\in\Lambda}\mathcal{R}
\left(\nabla a_{(\xi)}^2\mathbb{P}_{\neq0}(W_{(\xi)}\otimes W_{(\xi)}) \right)-\frac{1}{\mu}\sum_{\xi\in\Lambda}\mathcal{R}
\left(\partial_t a_{(\xi)}^2(\phi_{(\xi)}^2\psi_{(\xi)}^2\xi) \right),
\end{equation*}
with $a_{(\xi)}$ given  in (\ref{amplitudes}) for $M_0(t)$ from (\ref{M0}). Hence the bounds for $R_{\textrm{osc}}$ are the same as in Section~\ref{sss:R}.
The Reynolds stress on the level $q+1$ is then defined as
$$\mathring{R}_{q+1}:=R_{\textrm{lin}}+R_{\textrm{osc}}+R_{\textrm{cor}}+R_{\textrm{com}}+R_{\textrm{com}1}.$$

\subsubsection{Verification of the inductive estimate for $\mathring{R}_{q+1}$}

In the following we estimate the remaining  terms in $\mathring{R}_{q+1}$ separately. We choose $p=\frac{32}{32-7\alpha}>1$.

For the linear  error, have for $t\in[0, T_L]$
\begin{equation*}
\aligned
\|R_{\textrm{linear}}\|_{C_tL^p}&\lesssim\|w_{q+1}\|_{C_tW^{1,p}}+\|\mathcal{R}\partial_t(w_{q+1}^{(p)}+w_{q+1}^{(c)})\|_{C_tL^p}+m_L^2\|v_\ell\mathring{\otimes}w_{q+1}+w_{q+1}\mathring{\otimes}v_\ell\|_{C_tL^p}
\\&\lesssim\|w_{q+1}\|_{C_tW^{1,p}}+\sum_{\xi\in\Lambda}\|\partial_t\textrm{curl}
({\bar a}_{(\xi)}V_{(\xi)})\|_{C_tL^p}+\lambda_q^4m_L^3 M_0(t)^{1/2}\|w_{q+1}\|_{C_tL^p}.
\endaligned
\end{equation*}
Hence using \eqref{principle est22 li}, \eqref{corrector est2 li}, \eqref{estimate aN li}, \eqref{bounds}, \eqref{principle est1 li}, \eqref{correction est li} and \eqref{temporal est1 li} we have for $t\in[0, T_L]$
\begin{equation*}
\aligned
\|R_{\textrm{linear}}\|_{C_tL^p}
&\lesssim m_LM_0(t)^{1/2}\ell^{-2}r_\perp^{2/p-1}r_\|^{1/p-1/2}\lambda_{q+1}+M_0(t)\ell^{-4}r_\perp^{2/p-2}r_\|^{1/p-1}\lambda_{q+1}^{-2/7}
\\&\quad+m_LM_0(t)^{1/2}\ell^{-7}r_\perp^{2/p}r_\|^{1/p-3/2}\mu+m_LM_0(t)^{1/2}\ell^{-12}r_\perp^{2/p-1}r_\|^{1/p-1/2}\lambda_{q+1}^{-1}
\\&\quad+m_L^{4}M_0(t)\ell^{-2}r_\perp^{2/p-1}r_\|^{1/p-1/2}\lambda_{q}^4
\\&\lesssim m_LM_0(t)^{1/2}\lambda_{q+1}^{5\alpha-1/7}+M_0(t)\lambda_{q+1}^{9\alpha-2/7}+m_LM_0(t)^{1/2}\lambda_{q+1}^{15\alpha-1/7}
\\&\quad+m_LM_0(t)^{1/2}\lambda_{q+1}^{25\alpha-15/7}+m_L^4M_0(t)\lambda_{q+1}^{7\alpha-8/7}
\\&\leq\frac{M_0(t)c_R\delta_{q+2}}{5},
\endaligned
\end{equation*}
where we used the fact that $a$ is sufficiently large   and $\beta $ is suffiently small, in particular, \eqref{c:M li} holds.

The corrector error is estimated  by (\ref{principle est1 li}), (\ref{correction est li}) and \eqref{temporal est1 li}  for $t\in[0, T_L]$ as follows
\begin{equation*}
\aligned
\|R_{\textrm{cor}}\|_{C_tL^p}&
\leq m_L^2\|w_{q+1}^{(c)}+ w_{q+1}^{(t)}\|_{C_tL^{2p}}\| w_{q+1}\|_{C_tL^{2p}}+m_L^2\|w_{q+1}^{(c)}+ w_{q+1}^{(t)}\|_{C_tL^{2p}}\| w_{q+1}^{(p)}\|_{C_tL^{2p}}
\\
&\lesssim m_L^4M_0(t)\left(\ell^{-12}r_\perp^{1/p}r_\|^{1/(2p)-3/2}
+\ell^{-4}M_0(t)^{1/2}r_\perp^{1/p-1}r_\|^{1/(2p)-2}\lambda_{q+1}^{-1}\right)\ell^{-2}r_\perp^{1/p-1}r_\|^{1/(2p)-1/2}
\\
&\lesssim m_L^4M_0(t)\left(\ell^{-14}r_\perp^{2/p-1}r_\|^{1/p-2}
+\ell^{-6}M_0(t)^{1/2}r_\perp^{2/p-2}r_\|^{1/p-5/2}\lambda_{q+1}^{-1}\right)
\\
&\leq m_L^4M_0(t)\left(\lambda_{q+1}^{29\alpha-2/7}+M_0(t)^{1/2}\lambda_{q+1}^{13\alpha-1/7}\right)
\leq \frac{M_0(t)c_R\delta_{q+2}}{5},
\endaligned
\end{equation*}
where we used again (\ref{c:M li}) to have $m_L^4\lambda_{q+1}^{29\alpha-2/7}+m_L^4M_0(L)^{1/2}\lambda_{q+1}^{13\alpha-1/7}\leq \frac{c_R\delta_{q+2}}{5}$.

In view of a standard mollification estimate we deduce that for $t\in[0, T_L]$
\begin{equation*}
|\theta_\ell(t)-\theta(t)|\leq \ell^{1/2-2\delta}L^{1/2}e^{L^{1/2}}\leq \ell^{1/2-2\delta}m_L^2,
\end{equation*}
\begin{equation*}
\aligned
\|R_{\textrm{com}}\|_{C_tL^1}&\lesssim m_L^2\ell\|v_q\|_{C^1_{t,x}}\|v_q\|_{C_tL^2}+\ell^{\frac{1}{2}-2\delta}m_L^4M_0(t)\lambda_q^4
\\&\lesssim \ell^{\frac{1}{2}-2\delta}m_L^4M_0(t)\lambda_q^4\leq \frac{M_0(t)c_R\delta_{q+2}}{5},
\endaligned
\end{equation*}
where $\delta\in(0,1/12)$ and we choose $a$ large enough  to have
\begin{equation}\label{c:last} C\ell^{\frac{1}{2}-2\delta}m_L^4\lambda_q^4<\frac{c_R}{5}\lambda_{q+2}^{-2\beta}.\end{equation}  With the choice of $\ell$  and since we
postulated that $\alpha>8\beta b$ and $\alpha b>16$, this can indeed be achieved by possibly increasing $a$ and consequently
decreasing $\beta$.

The second commutator error can be estimated for $t\in[0, T_L]$ as follows
\begin{equation*}
\|R_{\textrm{com}1}\|_{C_tL^1}\leq \ell^{1/2-2\delta}m_L^4M_0(t)\leq \frac{M_0(t)c_R\delta_{q+2}}{5},
\end{equation*}
where we used \eqref{c:last} to have $\ell^{1/2-2\delta} m_L^4<\frac{c_R}{5}\delta_{q+2}$.

Thus, collecting  the above estimates we obtain the desired third bound in \eqref{inductionv li} and the proof of Proposition~\ref{main iteration li} is complete.

\section{Non-uniqueness in law III: the case of a nonlinear noise}
\label{s:nonuniquIII}

The treatment of a nonlinear noise requires more input coming from the driving Brownian motion. Namely, the corresponding iterated integral of $B$ against $B$ as known in the theory of rough paths. This is reflected through an additional variable $\mathbb{Y}$ included in the path space. Furthermore, we include a  variable $Z$ which is used to control the first step of the iteration scheme defined via \eqref{induction g'}, \eqref{induction g} below, namely to control $z_{0}$. This is just an auxiliary point which as a consequence of Corollary~\ref{cor:1}  does not restrict the final non-uniqueness in law result.

In what follows, we therefore use the following notations.
Let  $$\widetilde{\Omega}:=C([0,\infty);H^{-3}\times \mR^m\times \mR^{m\times m}{\times H^{-3} })\cap {L^2_{\mathrm{loc}}([0,\infty);L^2_\sigma\times \mR^m\times \mR^{m\times m}{\times L^2_\sigma })}$$  
and let $\mathscr{P}(\widetilde{\Omega})$ denote the set of all probability measures on $(\widetilde{\Omega},\widetilde{\mathcal{B}})$ with $\widetilde{\mathcal{B}}$   being the Borel $\sigma$-algebra coming from the natural topology on $\widetilde\Omega$. Let  $(x,y,\mY,{Z}):\widetilde{\Omega}\rightarrow H^{-3}\times \mR^m\times \mR^{m\times m}{\times H^{-3}}$  denote the canonical process on $\widetilde{\Omega}$ given by
$$(x_t(\omega),y_t(\omega),\mY_t(\omega),{Z_t(\omega)})=\omega(t).$$
For $t\geq 0$ we define   $\sigma$-algebra $\widetilde{\mathcal{B}}^{t}=\sigma\{ (x(s),y(s),\mY(s),{Z(s)}),s\geq t\}$.
Finally,  we define the canonical filtration  $\widetilde{\mathcal{B}}_t^0:=\sigma\{ (x(s),y(s),\mY(s),{Z(s)}),s\leq t\}$, $t\geq0$, as well as its right continuous version $\widetilde{\mathcal{B}}_t:=\cap_{s>t}\widetilde{\mathcal{B}}^0_s$, $t\geq 0$. In this section we choose $U=\mR^m$.

\subsection{Generalized probabilistically weak solutions}

Accordingly, we need to generalize our  notion of solution, taking the additional variables $\mathbb Y$ and $Z$ into account.
We fix a deterministic function $v_0\in C^1_{t,x}$, which will be chosen in Lemma \ref{lem:v0 g} below. In order to define the stopping time in the same path space, the  process $Z$ solves the following SPDE
	\begin{align}\label{eq:Z0}dZ_t-\Delta  Z_t=G(v_0+Z_t)dB_t,\quad \div Z=0.\end{align}
 By \cite[Theorem 4.2.5]{LR15}, the solution to \eqref{eq:Z0} belongs to $C([0,\infty);L^2_\sigma)$ a.s.

\begin{defn}\label{ge weak solution}
	Let $s\geq 0$ and $x_{0}\in L^{2}_{\sigma}$, $y_0\in \mR^m$, $\mY_0\in \mR^{m\times m}$, {$Z_0\in L^2_\sigma$}. A probability measure $P\in \mathscr{P}(\widetilde{\Omega})$ is  a generalized probabilistically weak solution to the Navier--Stokes system (\ref{1})  with the initial value $(x_0,y_0,\mY_0,{Z_0}) $ at time $s$ provided
	
	\no(M1) $P(x(t)=x_0, y(t)=y_0, \mY(t)=\mY_0, {Z(t)=Z_0, } 0\leq t\leq s)=1$  and for any $n\in\mathbb{N}$
	$$P\left\{(x,y,\mY,{Z})\in \widetilde{\Omega}: \int_0^n\|G(x(r))\|_{L_2(\mathbb{R}^m;L^2_\sigma)}^2dr<+\infty, {Z\in CL^2_\sigma}\right\}=1.$$
	
	\no(M2) Under $P$,  $y$ is a   $(\widetilde{\mathcal{B}}_{t})_{t\geq s}$-Brownian motion in  $\mR^{m}$ starting from $y_0$ at time $s$ and for every $e_i\in C^\infty(\mathbb{T}^3)\cap L^2_\sigma$, and for $t\geq s$
	$$\langle x(t)-x(s),e_i\rangle+\int^t_s\langle \div(x(r)\otimes x(r))-\Delta x(r),e_i\rangle dr=\int_s^t \langle e_i, G(x(r))  dy_r\rangle,$$
	and
	$$\mY(t)-\mY(s)=\int_s^ty(r)\otimes d y(r),$$
$$\langle Z(t )-Z(s),e_i\rangle-\int_s^{t}\langle\Delta  Z(r),e_i\rangle dr=\int_s^t\langle e_i,G(v_0+Z(r))d y(r)\rangle.$$
	
	\no (M3) For any $q\in \mathbb{N}$ there exists a positive real function $t\mapsto C_{t,q}$ such that  for all $t\geq s$
	$$E^P\left(\sup_{r\in [0,t]}\|x(r)\|_{L^2}^{2q}+\int_{s}^t\|x(r)\|^2_{H^{\gamma}}dr\right)\leq C_{t,q}(\|x_0\|_{L^2}^{2q}+1),$$
and
	$$E^P\left(\sup_{r\in [0,t]}\|Z(r)\|_{L^2}^{2q}+\int_{s}^t\|\nabla Z(r)\|^2_{L^2}dr\right)\leq C_{t,q}(\|Z_0\|_{L^2}^{2q}+1).$$
\end{defn}

	By the assumption \eqref{eq:nl} on $G$,  pathwise uniqueness holds for \eqref{eq:Z0}.  Hence, since $v_{0}$ is deterministic,  the law of $Z$ is uniquely determined by the driven Brownian motion $B$ according to the Yamada--Watanabe Theorem.
Under \eqref{eq:nl},  we know that the constant $C_{t,q}$ is independent of $v_0$.

Define $y_{r,t}:=y(t)-y(r)$
and $\mX_{r,t}:=\mY(t)-\mY(r)-y(r)\otimes (y(t)-y(r))$.
Note that under a generalized probabilistically weak solution $P$, the pair
$(y,\mX)$ can be viewed as a rough path, concretely, it is  the It\^o lift of an $m$-dimensional Brownian motion. In particular, the Chen's relation
\begin{equation}\label{eq:chen}
\delta\mX_{r ,\theta ,t}=\mX_{r,t}-\mX_{r,\theta}-\mX_{\theta, t}= y_{r, \theta}\otimes  y_{\theta, t},\qquad r\leq\theta\leq t,
\end{equation}
holds true.

For the application to the Navier--Stokes system, we will again require a definition of generalized probabilistically weak solutions defined up to a stopping time $\tau$. To this end, we set
$$
\widetilde{\Omega}_{\tau}:=\{\omega(\cdot\wedge\tau(\omega));\omega\in \widetilde{\Omega}\}.
$$

\begin{defn}\label{ge weak solution 1}
	Let $s\geq 0$ and $x_{0}\in L^{2}_{\sigma}$, $y_0\in \mR^m $, $\mY_0\in \mR^{m\times m}$, $Z_0\in L^2_\sigma$. Let $\tau\geq s$ be a $(\widetilde{\mathcal{B}}_{t})_{t\geq s}$-stopping time. A probability measure $P\in \mathscr{P}(\widetilde{\Omega}_\tau)$ is  a generalized probabilistically weak solution to the Navier--Stokes system (\ref{1}) on $[s,\tau]$ with the initial value $(x_0,y_0,\mY_0, {Z_0}) $ at time $s$ provided
	
	\no(M1) $P(x(t)=x_0, y(t)=y_0, \mY(t)=\mY_0, {Z(t)=Z_0,} 0\leq t\leq s)=1$
	and for any $n\in\mathbb{N}$
	$$P\left\{(x,y,\mY,{Z})\in \widetilde\Omega: \int_0^{n\wedge \tau}\|G(x(r))\|_{L_2(\mathbb{R}^m;L_2^\sigma)}^2dr<+\infty, {Z\in CL^2_\sigma}\right\}=1.$$

	\no(M2) Under $P$, $y=(y_i)_{i=1}^m$ and each component $y_i(\cdot\wedge \tau),i=1,\dots,m$, is a {continuous square integrable $(\widetilde{\mathcal{B}}_{t})_{t\geq s}$}-martingale  starting from $y^i_0$ at time $s$ with cross variation process between $y_i$ and $y_j$ given by $(t\wedge \tau-s)\delta_{i=j}$.  For every $e_i\in C^\infty(\mathbb{T}^3)\cap L^2_\sigma$, and for $t\geq s$
	$$\langle x(t\wedge \tau)-x(s),e_i\rangle+\int^{t\wedge \tau}_s\langle \div(x(r)\otimes x(r))-\Delta x(r),e_i\rangle dr=\int_s^{t\wedge\tau} \langle e_i, G(x(r))  dy_r\rangle,$$
	and
	$$\mY(t\wedge\tau)-\mY(s)=\int_s^{t\wedge \tau}y(r)\otimes dy(r),$$
$$\langle Z(t\wedge\tau)-Z(s),e_i\rangle-\int_s^{t\wedge\tau}\langle\Delta  Z(r),e_i\rangle dr=\int_s^{t\wedge \tau}\langle e_i,G(v_0+Z(r))d y(r)\rangle.$$

	\no (M3) For any $q\in \mathbb{N}$ there exists a positive real function $t\mapsto C_{t,q}$ such that  for all $t\geq s$
	$$E^P\left(\sup_{r\in [0,t\wedge\tau]}\|x(r)\|_{L^2}^{2q}+\int_{s}^{t\wedge\tau}\|x(r)\|^2_{H^{\gamma}}dr\right)\leq C_{t,q}(\|x_0\|_{L^2}^{2q}+1),$$
	and 
		$$E^P\left(\sup_{r\in [0,t\wedge \tau ]}\|Z(r)\|_{L^2}^{2q}+\int_{s}^{t\wedge\tau}\|\nabla Z(r)\|^2_{L^2}dr\right)\leq C_{t,q}(\|Z_0\|_{L^2}^{2q}+1).$$
\end{defn}

It is easy to see the following relation between the generalized probabilistically weak solutions and probabilistically weak solutions.

\begin{cor}\label{cor:1}
	Let $P$ be a generalized probabilistically weak solution starting from  $(x_{0},y_{0},0,0)$ at time $s$. Then the canonical process $\mY$ and $Z$ under $P$ is a measurable function of $y$. In other words, $P$ is fully determined by the joint probability law of $x,y$ and can be identified with a probability measure on the reduced path space $\bar\Omega$. Hence $P$ is a probabilistically weak solution  with the initial value $(x_{0},y_{0})$ at time $s$.

	Conversely, let $(x_{0},y_{0})\in L^2_\sigma\times \mR^m$ be given and let $P\in \mathscr{P}(\bar\Omega)$ be a probabilistically weak solution and define $P$-a.s. for $t\geq s$
	$$
	\mY(t):=\int_s^ty(r)\otimes dy(r),
	$$
and set $Z$ to be the unique probabilistically strong solution to  \eqref{eq:Z0} with $B$ replaced by $y$ and $Z(s)=0$.
	Let $Q$ be the joint law of $(x,y,\mY,{Z})$ under $P$. Then $Q\in \mathscr{P}(\widetilde{\Omega})$
	gives raise to a generalized probabilistically weak solution  starting from the initial value $(x_{0},y_{0},0,{0})$ at the time $s$.
\end{cor}

Similarly to Theorem \ref{convergence}, the following existence and stability result holds. We prove it in Appendix~\ref{ap:A}.

\begin{thm}\label{ge convergence 1}
	For every $(s,x_0,y_0,\mY_0,{Z_0})\in [0,\infty)\times L_{\sigma}^2\times \mR^m\times \mR^{m\times m}{\times L^2_\sigma} $, there exists  $P\in\mathscr{P}(\widetilde{\Omega})$ which is a generalized probabilistically weak solution to the Navier--Stokes system \eqref{1} starting at time $s$ from the initial condition $(x_0,y_0,\mY_0,{Z_0})$  in the sense of Definition \ref{ge weak solution}. The set of all such generalized probabilistically weak solutions  with the same implicit constant $C_{t,q}$ in  Definition \ref{ge weak solution} is denoted by $\mathscr{GW}(s,x_0,y_0, \mY_0, {Z_0}, C_{t,q})$.
	
	Let $(s_n,x_n,y_n,\mY_n,{Z_n})\rightarrow (s,x_0,y_0,\mY_0,{Z_0})$ in $[0,\infty)\times L_{\sigma}^2\times \mR^m\times \mR^{m\times m}{\times L^2_\sigma} $ as $n\rightarrow\infty$  and let $$P_n\in \mathscr{GW}(s_n,x_n,y_n,\mY_n, {Z_n}, C_{t,q}).$$ Then there exists a subsequence $n_k$ such that the sequence $(P_{n_k})_{k\in\mathbb{N}}$  converges weakly to some $P\in\mathscr{GW}(s,x_0,y_0,\mY_0, {Z_0},  C_{t,q})$.
\end{thm}

By Corollary \ref{cor:1} non-uniqueness of generalized probabilistically weak solutions from $(x_0,y_0,0,{0})$ implies  the joint non-uniqueness in law from $(x_0,y_0)$ in the sense of Definition \ref{def:uniquelaw1}.

\subsection{General construction for generalized probabilistically weak solutions}
\label{ss:gen1 ge}

The overall strategy is similar to Section~\ref{ss:gen}. In the first step, we shall extend generalized probabilistically weak solutions defined up a {$(\widetilde{\mathcal{B}}_{t})_{t\geq 0}$}-stopping time $\tau$ to the whole interval $[0,\infty)$. We denote by $\widetilde{\mathcal{B}}_{\tau}$ the $\sigma$-field associated to $\tau$.

\begin{prp}\label{prop:g1 1}
	Let $\tau$ be a bounded $(\widetilde{\mathcal{B}}_{t})_{t\geq0}$-stopping time. Then for every $\omega\in \widetilde{\Omega}$ there exists $Q_{\omega}\in\mathscr{P}(\widetilde{\Omega})$ such that for {$\omega\in \{x(\tau)\in L^2_\sigma, {Z(\tau)\in L^2_\sigma}\}$}
	\begin{equation}
	\label{qomega g1}
	Q_\omega\big(\omega'\in\widetilde{\Omega}; (x,y,\mY,{Z})(t,\omega')=(x,y,\mY,{Z})(t,\omega) \textrm{ for } 0\leq t\leq \tau(\omega)\big)=1,
	\end{equation}
	and
	\begin{equation}\label{qomega2 g1}
	Q_\omega(A)=R_{\tau(\omega),x(\tau(\omega),\omega),y(\tau(\omega),\omega),\mY(\tau(\omega),\omega),{Z}(\tau(\omega),\omega)}(A)\qquad\text{for all}\  A\in \mathcal{B}^{\tau(\omega)}.
	\end{equation}
	where $R_{\tau(\omega),x(\tau(\omega),\omega),y(\tau(\omega),\omega),\mY(\tau(\omega),\omega),{Z(\tau(\omega),\omega)}}\in\mathscr{P}(\widetilde{\Omega})$ is a generalized probabilistically weak solution to the Navier--Stokes system \eqref{1} starting at time $\tau(\omega)$ from the initial condition $$(x(\tau(\omega),\omega), y(\tau(\omega),\omega), \mY(\tau(\omega),\omega), {Z(\tau(\omega),\omega)}).$$ Furthermore, for every $B\in\widetilde{\mathcal{B}}$ the mapping $\omega\mapsto Q_{\omega}(B)$ is $\widetilde{\mathcal{B}}_{\tau}$-measurable.
\end{prp}

\begin{proof}
	
	The proof of this result is identical to the proof of Proposition \ref{prop:1} applied to the extended path space $\widetilde\Omega$ instead of $\Omega_{0}$ and making use of Theorem~\ref{ge convergence 1} instead of Theorem \ref{convergence}.
\end{proof}

We proceed with an analogue to Proposition~\ref{prop:2 1}.

\begin{prp}\label{prop:2 1 ge}
	Let $x_{0}\in L^{2}_{\sigma}$.
	Let $P$ be a generalized probabilistically weak solution to the Navier--Stokes system \eqref{1} on $[0,\tau]$ starting at the time $0$ from the initial condition $(x_{0},0,0,{0})$. In addition to the assumptions of Proposition \ref{prop:g1 1}, suppose that there exists a Borel  set $\mathcal{N}\subset\widetilde{\Omega}_{\tau}$ such that $P(\mathcal{N})=0$ and for every $\omega\in \mathcal{N}^{c}$ it holds
	\begin{equation}\label{Q1 g1}
	\aligned
	&Q_\omega\big(\omega'\in\widetilde{\Omega}; \tau(\omega')=
	\tau(\omega)\big)=1.
	\endaligned
	\end{equation}
	Then the  probability measure $ P\otimes_{\tau}R\in \mathscr{P}(\widetilde{\Omega})$ defined by
	\begin{equation*}
	P\otimes_{\tau}R(\cdot):=\int_{\widetilde{\Omega}}Q_{\omega} (\cdot)\,P(d\omega)
	\end{equation*}
	satisfies $P\otimes_{\tau}R= P$ on the $\sigma$-algebra $\sigma(x(t\wedge \tau), y(t\wedge \tau), \mY(t\wedge \tau),{Z(t\wedge \tau)},t\geq0)$ and
	is a generalized probabilistically weak solution to the Navier--Stokes system \eqref{1} on $[0,\infty)$ with initial condition $(x_{0},0,0,{0})$.
\end{prp}

\begin{proof}

Most of the  proof follows from exactly the same argument as in Proposition~\ref{prop:2 1}.
	For (M2), we consider $z$ part and similar as the proof of Proposition~\ref{prop:2 1} we obtain 
	\begin{align*}
	&P\otimes_\tau R\bigg\{\mY(t)-\mY(0)
	=\int^{t}_0 y(r) \otimes dy(r), {t\geq0}\bigg\}
	\\&\quad=\int_{\widetilde\Omega} dP(\omega)Q_\omega\bigg\{\mY(t)-\mY(t\wedge\tau(\omega))
	=\int^{t}_{t\wedge\tau(\omega)}y(r)\otimes dy(r),
	\\&
	\hspace{3.5cm}\mY(t\wedge\tau(\omega))-\mY(0)
	=\int^{t\wedge\tau(\omega)}_0y(r)\otimes dy(r), t\geq0\bigg\}=1
	,
	\end{align*}
	and 
	\begin{align*}
		P\otimes_\tau R\bigg\{&\langle Z(t)-Z(0),e_i\rangle-\int_0^t\langle Z(r),\Delta e_i\rangle dr
		\\&=\int^{t}_0 \langle e_i, G(v_0+Z(r))dy(r), e_i\in C^\infty(\mT^3)\cap L^2_\sigma, t\geq 0\bigg\}
		=1
		.
		\end{align*}
	Hence the condition (M2) follows.
\end{proof}

\subsection{Application to solutions obtained through Theorem \ref{Main resultsg1}}
\label{sg:1.4}

For $\alpha\in (0,1)$ we  denote
$$\|\mX\|_{\alpha,[0,T]}:=\sup_{0\leq r<t\leq T}\frac{|\mX_{r,t}|}{|r-t|^{\alpha}},$$
and for $n\in\mathbb{N}$, $L>1$ and  $\delta\in(0,1/12)$  we define
\begin{equation*}
\aligned\tau_L^n(\omega)&= \inf\left\{t>0,\|y(\omega)\|_{C_t^{\frac{1}{2}-2\delta}}>(\ln\ln L-\frac{1}{n})\right\}\wedge \ln\ln L
\\&\qquad \wedge \inf\left\{T\geq 0, \|\mX\|_{1-4\delta,[0,T]}>(\ln\ln L-\frac{1}{n})\right\}
\wedge \inf\left\{t\geq 0,\|Z(t)\|_{L^2}> L-\frac1n\right\}.
\endaligned
\end{equation*}
Then the sequence $\{\tau^{n}_{L}\}_{n\in\mathbb{N}}$ is nondecreasing and we define
\begin{equation}\label{eq:tauL g1}
\tau_L:=\lim_{n\rightarrow\infty}\tau_L^n.
\end{equation}
Without additional regularity of the process $y$, it holds true that $\tau^{n}_{L}(\omega)=0$.
By Lemma~\ref{time} we obtain that $\tau_L^{n}$ is $(\widetilde{\mathcal{B}}_t)_{t\geq0}$-stopping time and consequently also  $\tau_L$ is a $(\widetilde{\mathcal{B}}_t)_{t\geq 0}$-stopping time as an increasing limit of stopping times.

On a  stochastic basis $(\Omega, \mathcal{F},\mathbf{P})$, we apply Theorem~\ref{Main resultsg1} and denote by $u$ and $z_0$ the corresponding solutions to the Navier--Stokes system \eqref{ns g} and to the linear equation \eqref{induction g'} with $q=0$ on $[0,T_{L}]$, where the stopping time $T_{L}$ is defined in \eqref{stopping time g} below. We recall that $u$ is adapted with respect to $(\mathcal{F}_{t})_{t\geq0}$. The process $z_0$ is the unique probabilistically strong solution hence also adapted  to $(\mathcal{F}_{t})_{t\geq0}$. We denote by $P$ the law of $(u,B,\int_0^{\cdot} B dB,z_0)$ and obtain the following result by similar arguments as in the proof of Proposition~\ref{prop:ext}.

\begin{prp}\label{prop:ext g1}
	The probability measure $P$ is a generalized probabilistically weak solution to the Navier--Stokes system \eqref{ns g} with the initial condition $(u(0),0,0,0)$ on $[0,\tau_{L}]$ in the sense of Definition~\ref{ge weak solution 1}, where $\tau_{L}$ was defined in \eqref{eq:tauL g1}.
\end{prp}

\begin{proof}
	The proof is similar to the proof of Proposition \ref{prop:ext} once we note that from the definition of the canonical process it holds
	$$y\left(t,\left(u,B,\int BdB,{z_0} \right)\right)=B(t),\quad \mY\left(t,\left(u,B,\int BdB,{z_0} \right)\right)=\int_0^t B_rdB_r,$$
	$${Z\left(t,\left(u,B,\int BdB,{z_0} \right)\right)=z_0(t)},$$
 $$\mX_{s,t}\left(u,B,\int BdB,{z_0} \right)=\int_s^t  B_{s,r}\otimes dB_r:=\mB_{s,t}\quad \textrm{for } s,t\in [0,T_L] \quad \mathbf{P}\text{-a.s.},$$
 and that by Chen's relation \eqref{eq:chen} and the definition of $Z$
 $$t\mapsto \|B\|_{C_t^{1/2-2\delta}}\quad t\mapsto \|\mB\|_{1-4\delta, [0,t]},\quad {t\mapsto \|z_0(t)\|_{L^2}}$$ are continuous $\mathbf{P}$\text{-a.s.}
	In particular, the property (M2) in Definition \ref{ge weak solution 1} follows since $(u,B, \int BdB )$ satisfies \eqref{ns g} and $(z_0,B)$ satisfies \eqref{induction g'} with $q=0$.
\end{proof}

\begin{prp}\label{prp:ext2 1 ge}
	The probability measure $P\otimes_{\tau_{L}}R$ is a generalized probabilistically weak solution to the Navier--Stokes system \eqref{ns g} on $[0,\infty)$ in the sense of Definition \ref{ge weak solution}.
\end{prp}

\begin{proof}
	In light of Proposition \ref{prop:g1 1} and Proposition \ref{prop:2 1 ge}, it only remains to establish \eqref{Q1 g1}. By Theorem \ref{Main resultsg1} and Proposition \ref{prop:ext g1} we know that
$$P(\omega: y(\cdot\wedge \tau_L)\in C^{1/2-\delta}\mR^m, {Z(\cdot\wedge \tau_L)\in CL^2_\sigma})=1,$$
and by $\mX_{s,t\wedge \tau_L}=\int_s^{t\wedge\tau_L} y_{s,r}\otimes dy_r$ for $ t>s$
$$P(\omega: \|\mX\|_{1-2\delta,[0,\tau_L]}<\infty)=1.$$
In other words,   there exists a $P$-measurable set $\mathcal{N}\subset \widetilde{\Omega}$ such that $P(\mathcal{N})=0$ and for $\omega\in \mathcal{N}^c$
\begin{equation*}
y(\cdot\wedge \tau_L(\omega))\in C^{1/2-\delta}\mR^m,\quad \|\mX(\omega)\|_{1-2\delta,[0,\tau_L(\omega)]}<\infty,\quad {Z(\cdot\wedge \tau_L(\omega))\in CL^2_\sigma}.
\end{equation*}
Moreover, by Chen's relation it holds that for all $\omega\in\mathcal{N}^c\cap \{x(\tau_L)\in L^2_\sigma,{Z(\tau_L)\in L^2_\sigma}\},$
\begin{equation*}\aligned
& Q_\omega\left(\omega': \|\mX\|_{1-2\delta,[0,T]}<\infty, y\in C_T^{1/2-\delta}\mR^m, {Z\in C_TL^2_\sigma}, T>0 \right)
\\&= Q_\omega\left(\omega': \|\mX\|_{1-2\delta,[0,\tau_L(\omega)]}<\infty, \|\mX\|_{1-2\delta,[\tau_L(\omega)\wedge T,T]}<\infty, y\in C_T^{1/2-\delta}\mR^m,  {Z\in C_TL^2_\sigma}, T>0\right)
\\&=\delta_\omega\left(\omega': y(\cdot\wedge \tau_L(\omega))\in  C^{1/2-\delta}\mR^m,  {Z(\cdot\wedge\tau_L(\omega))\in CL^2_\sigma},\|\mX\|_{1-2\delta,[0,\tau_L(\omega)]}<\infty\right) \\
&\qquad\times  R_{\tau_L(\omega),(x,y,\mY{,Z})(\tau_L(\omega),\omega)}\Big(\omega': y-y(\cdot\wedge\tau_L(\omega))\in C^{1/2-\delta}_T\mR^m,
 {Z-Z(\cdot\wedge\tau_L(\omega))\in C_TL^2_\sigma},
\\&\qquad\qquad\qquad\qquad\qquad\qquad\|\mX\|_{1-2\delta,[\tau_L(\omega)\wedge T,T]}<\infty, T>0\Big).
\endaligned
\end{equation*}
Here the first factor on the right hand side equals to $1$ for all  $\omega\in \mathcal{N}^c$.
Since $R_{\tau_L(\omega),(x,y,\mY,Z)(\tau_L(\omega),\omega)}$ is a generalized probabilistically weak solution starting at the deterministic time $\tau_{L}(\omega)$ from the deterministic initial condition $(x,y,\mY,Z)(\tau_{L}(\omega),\omega)$, the process $\omega'\mapsto y-y(\cdot\wedge\tau_L(\omega))$ is a $(\widetilde{\mathcal{B}}_{t})_{t\geq0}$-Wiener process  starting from $\tau_L(\omega)$ and $\mX_{s,t}=\int_s^t y_{s,r}\otimes dy_r$ for $t>s\geq\tau_L(\omega)$   under the measure $R_{\tau_L(\omega),(x,y,\mY,Z)(\tau_L(\omega),\omega)}$.
Thus we deduce that also the second factor equals to $1$.
To summarize, we have proved that  for all  $\omega\in \mathcal{N}^c\cap \{x(\tau_L)\in L^2_\sigma,{Z(\tau_L)\in L^2_\sigma}\}$,
$$
Q_\omega\left(\omega': y\in  C_T^{1/2-\delta}\mR^m, {Z\in C_TL^2_\sigma}, \|\mX\|_{1-2\delta,[0,T]}<\infty,T>0 \right)=1.
$$
Therefore we know that for all  $\omega\in \mathcal{N}^c\cap \{x(\tau_L)\in L^2_\sigma, {Z(\tau_L)\in L^2_\sigma}\}$, $\|y\|_{C_T^{1/2-2\delta}}, \|\mX\|_{1-4\delta,[0,T]}$ is continuous with respect to $T$ and $Z\in CL^2_\sigma$.  The proof is completed by  the same argument as Proposition~\ref{prp:ext2}.
\end{proof}

Now, we are ready  to conclude the proof of our main result Theorem \ref{Main resultsg2}.

\begin{proof}[Proof of Theorem \ref{Main resultsg2}]
	By \eqref{eq:env g} with $K=2$ and a similar argument as in the proof of Theorem~\ref{Main results2}, we deduce that the  generalized probabilistically weak solution $P\otimes_{\tau_L}R$ does not satisfy the energy inequality.
	 The solutions obtained in Theorem \ref{ge convergence 1} satisfy the usual energy inequality by Galerkin approximation.
	 Hence the two generalized probabilistically weak solutions starting from the initial condition $(u(0),0,0,0)$ are distinct, which by Corollary \ref{cor:1} implies   joint non-uniqueness in law, i.e. non-uniqueness of probabilistically weak solutions. In view of Theorem \ref{cherny} we finally deduce that the desired non-uniqueness in law, i.e., non-uniqueness of martingale solutions, holds as well.
\end{proof}

\section{Proof of Theorem~\ref{Main resultsg1}}
\label{s:88}

In this section, we fix an $m$-dimensional Brownian motion $B$ on a probability space  $(\Omega,\mathcal{F},\mathbf{P})$ with its normal filtration $(\mathcal{F}_t)_{t\geq0}$ and assume the coefficient $G$  to satisfy \eqref{eq:nl}. The principal difference between this setting and the setting of additive or linear multiplicative noise is that no transformation of the SPDE \eqref{ns g} into a PDE with random coefficients is available. Therefore we introduce a convex integration scheme which at each step additionally solves a parametrized stochastic Stokes equation with a nonlinear It\^o noise.
To be more precise, let  $v_{-1}\equiv v_0$ be given in Lemma~\ref{lem:v0 g} below.  At each step $q\in\mathbb{N}_{0}$, we construct a triple $(z_q,v_q, \mathring{R}_q)$ solving
\begin{equation}\label{induction g'}
\aligned
dz_q-\Delta z_q dt&=G(v_{q-1}+z_q)dB,
\\
\div z_q&=0,\\
z_q(0)&=0,
\endaligned
\end{equation}
\begin{equation}\label{induction g}
\aligned
\partial_tv_q-\Delta v_q +\div((v_q+z_q)\otimes (v_q+z_q))+\nabla p_q&=\div \mathring{R}_q,
\\
\div v_q&=0.
\endaligned
\end{equation}
In order  to obtain the desired iterative estimates  \eqref{inductionv} from the random PDE \eqref{induction g}, it is necessary to have a pathwise control of each $z_{q}$. This is not possible using stochastic It\^o integration theory. Instead, we make  use of rough path theory which we present in Appendix~\ref{s:D}. As explained in Section~\ref{s:d1}, if $v_{q-1}$ is adapted, then the unique rough path solution $z_{q}$ of \eqref{induction g'}, coincides $\mathbf{P}$-a.s. with the unique stochastic solution coming from the stochastic It\^o integral theory. In particular, $z_{q}$ is an $(\mathcal{F}_{t})_{t\geq 0}$-adapted process. This in turn permits to conclude that the next iteration $v_{q}$ is $(\mathcal{F}_{t})_{t\geq 0}$-adapted as well.

As before, we consider an increasing sequence $\{\lambda_q\}_{q\in\mathbb{N}}\subset \mathbb{N}$ which diverges to $\infty$, and a sequence $\{\delta_q\}_{q\in \mathbb{N}}\subset(0,1)$  which is decreasing to $0$. For $a\in\mathbb{N}$, $b\in\mathbb{N}$,  $\beta\in (0,1)$ to be chosen below we define
$$\lambda_q=a^{(b^q)}, \quad \delta_q=\lambda_q^{-2\beta}.$$
It will be seen that $\beta$ will be chosen sufficiently small and $a$ as well as $b$ will be chosen sufficiently large. Set  $\mB_{s,t}:=\int_s^t(B_r-B_s)\otimes dB_r$ and define for  $L>1$ and $0<\delta<1/12$
\begin{equation}\label{stopping time g}\aligned
T_L&:=T_L^1\wedge T_L^2,\\
T_L^1&:=\inf\left\{t\geq0, \|B\|_{C_t^{\frac12-2\delta}}\geq \ln\ln L\right\}\wedge \inf\left\{t\geq0,\|\mB\|_{1-4\delta,[0,t]}\geq \ln\ln L\right\}\wedge \ln\ln L,
\\
T_L^2&:=\inf\left\{t\geq0, \|z_0(t)\|_{L^2}\geq L\right\}.
\endaligned
\end{equation}
By the definition of Brownian motion and the properties of the solution to the heat equation \eqref{induction g'} for $q=0$, the stopping time $T_{L}$ is $\mathbf{P}$-a.s. strictly positive and it holds that $T_{L}\uparrow\infty$ as $L\rightarrow\infty$ $\mathbf{P}$-a.s.
Moreover, by Theorem \ref{zexistence} for  $t\in[0, T_L]$ and $L$ large enough and $q\in\mathbb{N}$
\begin{equation}\label{z g}\aligned
\| z_q\|_{C_tC^1}\leq L^{1/4}(1+\|v_{q-1}\|_{C^1_{t,x}}), \quad \|z_q\|_{C_t^{\frac{1}{2}-2\delta}L^\infty}\leq L^{1/2}(1+\|v_{q-1}\|_{C^1_{t,x}}),\\
\|z_0\|_{C_tL^2}\leq L, \quad \|z_0\|_{C_tC^1}\leq L^{1/4}(1+\|v_0\|_{C^1_{t,x}}),\quad \|z_0\|_{C_t^{\frac{1}{2}-2\delta}L^\infty}\leq L^{1/2}(1+\|v_0\|_{C^1_{t,x}}).
\endaligned
\end{equation}
Let $M_{0}(t)=L^{4}e^{4Lt}$. We postulate that the  iterative bounds  \eqref{inductionv} hold for  $(v_q,\mathring{R}_q)$.

\begin{lem}\label{lem:v0 g}
	For $L>1$ define
	$$
	v_0(t,x)=\frac{L^2e^{2Lt}}{(2\pi)^{\frac{3}{2}}}\left(\sin(x_3),0,0\right),
	$$
	and let
	$z_0$ solve the  equation \eqref{induction g'} with $v_{-1}\equiv v_0$.
Then the associated Reynolds stress is given by
	\begin{equation*}
	\mathring{R}_0(t,x)=\frac{(2L+1)L^{2}e^{2Lt}}{(2\pi)^{3/2}}\left(
	\begin{array}{ccc} 0 & 0 &-\cos(x_3)
	\\ 0 & 0 &0\\ -\cos(x_3) &0 &0
	\end{array}
	\right)+v_0\mathring{\otimes} z_0+ z_0\mathring{\otimes} v_0+z_0\mathring{\otimes} z_0.
	\end{equation*}
	Moreover,   all the estimates in \eqref{inductionv} on the level $q=0$ for $(v_{0},\mathring{R}_{0})$ as well as \eqref{aaa} are valid  provided \eqref{c:a3} and \eqref{z g} hold.
	In particular, we require that \eqref{c:a2} holds.
	Furthermore, the initial values $v_0(0,x)$ and $\mathring{R}_0(0,x)$ are deterministic.
\end{lem}

\begin{proof} The proof follows from the same argument as in Lemma \ref{lem:v0}.
\end{proof}

As part of the following result we also control  the difference $z_{q+1}-z_{q}$ on the time interval $[0,t]$, $t\in [0,T_{L}]$, in a pathwise manner in the space of controlled rough paths. We refer the reader to Appendix~\ref{s:D} and in particular to Definition~\ref{controlled1} where the corresponding norm
$\|\cdot\|_{\bar B,2\alpha,\gamma}$, $\alpha\in (\frac13,\frac12)$, $\gamma\in \mR$ was defined. We denote $\|\cdot\|_{\bar B,2\alpha,\gamma,t}$ the norm in $D^{2\alpha}_{\bar{B},\gamma}([0,t])$. In general, $\|\cdot\|_{\bar B,2\alpha,\gamma}$ is the norm for the pair process $(z,z')$, where $z'$ is the Gubinelli derivative of the controlled rough path $z$. In the following we use it for $z_q$ with $z_q'=G(z_q+v_{q-1})$ and we will write $\|z_q\|_{\bar B,2\alpha,\gamma,t}$ instead of $\|z_q,z_q'\|_{\bar B,2\alpha,\gamma,t}$ in this section.

\begin{prp}\label{main iteration g}
	\emph{(Main iteration)}
	Let  $L>1$ satisfying \eqref{c:a2} be given and let $(z_q,v_q,\mathring{R}_q)$ be an $(\mathcal{F}_t)_{t\geq0}$-adapted solution to \eqref{induction g'}, \eqref{induction g} satisfying \eqref{inductionv}. Then there exists a choice of parameters $a,b,\beta$ such that \eqref{c:a3} is fulfilled and  there exist $(\mathcal{F}_t)_{t\geq0}$-adapted processes
	$(z_{q+1}, v_{q+1},\mathring{R}_{q+1})$ which solve \eqref{induction g'}, \eqref{induction g}, obey \eqref{inductionv} at level $q+1$ and  for $t\in[0,T_{L}]$ we have for  $\delta>0$, $q\in\mathbb{N}_{0}$, $\alpha_0=\frac23+\kappa$, $\kappa>0$ small enough,
	\begin{equation}\label{iteration g}
	\|v_{q+1}(t)-v_q(t)\|_{L^2}\leq M_0(t)^{1/2}\delta_{q+1}^{1/2},\quad \|v_{q+1}-v_q\|_{C^{\alpha_0}_{{t}}B^{-5-\delta}_{1,1}}\leq M_0(t)^{1/2}\delta_{q+1}^{1/2},
	\end{equation}
	and
	\begin{align}\label{ite z}
	\|z_{q+1}-z_q\|_{\bar B,\alpha_0,2\alpha_0,{t}}\leq M_0(t)^{1/2}\delta_{q+1}^{1/2}.
	\end{align}
	Furthermore, if $v_q(0), \mathring{R}_q(0)$ are deterministic, so are $v_{q+1}(0), \mathring{R}_{q+1}(0)$.
\end{prp}

Having Proposition \ref{main iteration g} at hand, we may prove Theorem \ref{Main resultsg1}.

\begin{proof}[Proof of Theorem \ref{Main resultsg1}]
	The most part of the proof follows by exactly the same argument as in the proof of Theorem \ref{Main results1}.
	Starting from $(z_0,v_0,\mathring{R}_0)$ given in Lemma \ref{lem:v0 g},   Proposition \ref{main iteration g} gives a sequence $(z_q,v_q, \mathring{R}_q)$ satisfying (\ref{inductionv}) and (\ref{iteration g}). Hence, similarly as in the proof of Theorem \ref{Main results1}
  we obtain a limiting solution $v=\lim_{q\rightarrow\infty}v_q$, which lies in $C([0,T_L],H^{\gamma})$ and \eqref{eq:vvv} holds.
    Since $v_q$ is $(\mathcal{F}_t)_{t\geq0}$-adapted for every $q\geq0$, the limit
	$v$ is $(\mathcal{F}_t)_{t\geq0}$-adapted as well. By \eqref{ite z} we obtain the convergence of $z_q$ in $\bar D_{B, 2\alpha_0}^{\alpha_0}([0,T_{L}])$ introduced in Definition \ref{controlled1}. We denote by $z$ the limit and note that it is also $(\mathcal{F}_{t})_{t\geq 0}$-adapted as a limit of adapted processes.
	
 Hence we could take limit in the equations \eqref{induction g'}, \eqref{induction g}  and conclude that $u=v+z$ satisfies the Navier--Stokes system \eqref{ns g} in the analytically  weak sense before the time $T_L$. In order to pass to the limit in the stochastic integral, we recall that by Section~\ref{s:d1} the rough integral in \eqref{induction g'} on the level $q$ coincides with the It\^o stochastic integral. By the $\mathbf{P}$-a.s. convergence of $v_{q}$ and $z_{q}$,  we may therefore pass to the limit $\lim_{q\rightarrow\infty} \int G(v_q+z_{q+1})d B=\int G(v+z)d B$ in $L^2(\Omega)$ in the It\^o formulation.  
 Moreover, the limit stochastic integral again coincides with the corresponding rough path integral.

	By the same argument as in the proof of Theorem \ref{Main results1} we obtain on $\{T_L\geq T\}$
	\begin{equation}\label{eq:env1 g}
	\|v(T)\|_{L^{2}}> (\|v(0)\|_{L^{2}}+L)e^{LT}.
	\end{equation}
	In other words,  given $T>0$ and the universal constant $c_{R}>0$, we can choose $L=L(T,c_{R})>1$ large enough so that \eqref{c:a2} as well as \eqref{c:7} holds and consequently \eqref{eq:env1 g} is satisfied. Moreover, in view of  the definition of the stopping times \eqref{stopping time g}, we observe that for a given $T>0$ we may possibly increase  $L$ so that the set $\{T_{L}\geq T\}$ satisfies $\mathbf{P}(T_{L}\geq T)>\kappa$.
	
 To verify \eqref{energy} and \eqref{eq:env g}, we use It\^{o}'s formula for $z_q$ and let $q\to \infty$ to have for $p\in \mN$
 $$\mathbf{E}\left[\sup_{t\in [0,T_L]}\|z(t)\|_{L^{2}}^{2p}+\int_0^{T_L}\|z(t)\|_{H^1}^{2}dt\right]\leq C_{t,p},$$
which combined with \eqref{eq:vvv} implies \eqref{energy}. We also have
 $$\mathbf{E}\left[\|z(T)\|_{L^{2}}^{2}\right]\leq C_GT.$$
 We then apply \eqref{eq:env1 g}  on $ \{T_{L}\geq T\}$ together with $\frac12v^2\leq z^2+u^2$ to obtain
	$$
	\mathbf{E}\left[1_{T_L\geq T}\|u(T)\|_{L^{2}}^2\right]\geq \frac12\mathbf{E}\left[1_{T_L\geq T}\|v(T)\|_{L^{2}}^2\right]-\mathbf{E}\left[\|z(T)\|_{L^{2}}^2\right]> \frac12\kappa(\|v(0)\|_{L^{2}}+L)^2e^{2LT}-C_GT.
	$$
	Thus, since $u(0)=v(0)$ we may possibly increase the value of $L$ depending on $K$  and $C_G$ in order to conclude the desired lower bound \eqref{eq:env g}. The initial value $v(0)$ is deterministic by our construction.  Finally, we set $\mathfrak{t}:=T_{L}$ which finishes the proof.
\end{proof}

\subsection{The main iteration -- proof of Proposition~\ref{main iteration g}}
\label{s:8}

\subsubsection{Choice of parameters}\label{s:c g}

Let us  summarize how the parameters need to be chosen  in order to fulfill  all the compatibility conditions in the sequel. First, for a sufficiently small  $\alpha\in (0,1)$ to be chosen below, we define $\ell\in (0,1)$ as in \eqref{ell1}.
The last condition in \eqref{ell}  together with \eqref{c:a3} leads to
$$
45\cdot (2\pi)^{3/2}<5\cdot (2\pi)^{3/2}a^{2\beta b}\leq c_{R}L\leq
c_{R}\frac{a^{4}\cdot (2\pi)^{3/2}-1}{2}.
$$
We remark that the reasoning from the beginning of Section \ref{ss:it} remains valid for this new condition: we may freely increase the value of $a$  provided we make $\beta$ smaller at the same time. In addition, we will require  $\alpha b>32$ and $\alpha>{16}\beta b^2$.

In order to verify the inductive estimates \eqref{inductionv} especially dealing with the terms with rough paths, it will also be necessary to absorb various expressions including  $M_{0}(t)^{M_0(t)^7}$ with $t=\ln\ln L$.  It will be seen that the strongest such requirement is  for $q\in\mathbb{N}$
\begin{equation}\label{c:M g}
\aligned
L^2M_{0}(\ln\ln L)^{M_0(\ln\ln L)^7}\lambda_{q}^{20\alpha-\frac{1}{42}}\leq \frac{c_R\delta_{q+2}}{10},\quad L^2M_{0}(\ln \ln L)^{M_0(\ln \ln L)^7}(\lambda_q^{-\frac{3\alpha}2}\lambda_{q-1}^{-2})^{\frac16}\lambda_{q-1}^4\leq \frac{c_R\delta_{q+2}}{10},
\endaligned
\end{equation}
needed in the estimates of $R_{q+1}$. In other words, for $\bar M_L:=\lceil L^2M_{0}(\ln \ln L)^{M_0(\ln \ln L)^7}\rceil $
$$
\bar{M}_La^{20\alpha-\frac{1}{42}+2b^2\beta}\ll1,\quad \bar{M}_La^{-\frac{b\alpha}4+\frac{11}3+2b^3\beta}\ll1,
$$
and choosing  $b=28\bar M_L\vee (33\cdot 28\cdot 80)$,  (this choice is coming from the fact that with our choice of $\alpha$ below we want to guarantee that $\alpha b>32$ as well as the fact that $b$ is a multiple of $28$ needed for the parameters  in the intermittent jets, cf. Appendix \ref{s:B}) and choosing $a$ large such that $\bar{M}_{L}<b\leq a^{\alpha/2}$ leads to
$$a^{20\alpha+\alpha/2-\frac{1}{42}+2b^2\beta}\ll1,\quad a^{\alpha/2-\frac{b\alpha}{4}+\frac{11}3+2b^3\beta}\ll1.$$
In view of $\alpha>16\beta b^2$, this  can be achieved by choosing $a$ large enough and $\alpha=80^{-1}\cdot 28^{-1}$. This choice also satisfies $\alpha b>32$ required above and the condition $\alpha>16\beta b^2$ can be achieved by choosing $\beta$ small. It is also compatible with all the other requirements needed below.

From now on, the parameters $\alpha$ and $b$ remain fixed and the  free parameters are $a$ and $\beta$ for which we already have a lower, respectively upper, bound. In the sequel, we will possibly increase $a$ and decrease $\beta$ at the same time in order to preserve all the above conditions and to fulfil further conditions appearing below.

\subsubsection{Verification of the inductive estimates for $v_{q+1}$}

Note that $v_\ell$, $z_\ell$, $\mathring{R}_\ell$ are defined the same as in Section \ref{s:p}. We could check that $v_\ell$ satisfies the following equation
\begin{equation*}
	\aligned
	\partial_tv_\ell -\Delta v_\ell+\div((v_\ell+z_\ell)\otimes (v_\ell+z_\ell))+\nabla p_\ell&=\div (\mathring{R}_\ell+R_{\textrm{com}})
	\\\div v_\ell&=0,
	\endaligned
\end{equation*}
with
\begin{equation*}
	R_{\textrm{com}}=(v_\ell+z_\ell)\mathring{\otimes}(v_\ell+z_\ell)-((v_q+z_q)\mathring{\otimes}(v_q+z_q))*_x\phi_\ell*_t\varphi_\ell.
\end{equation*}
Hence, the estimates \eqref{error}-\eqref{eq:vl2} hold in this case.

For the intermittent jets
we choose the following parameters
\begin{equation}\label{parameter g}
\aligned
\lambda&=\lambda_{q+1},
\qquad
r_\|=\lambda_{q+1}^{-4/7},
\qquad r_\perp=\lambda_{q+1}^{-{27}/{28}},
\qquad
\mu=\lambda_{q+1}^{9/7},
\endaligned
\end{equation}
where we note in particular the new value of $r_{\perp}$ as compared to \eqref{parameter}. This is needed in order to obtain the sufficient (Young)  time regularity of $v_{q+1}-v_{q}$ in \eqref{iteration g}, which in turn is employed for the rough path control of $z_{q+1}-z_{q}$ in Section~\ref{s:z} below.

Next, we define $\rho, a_{(\xi)}$, $w_{q+1}^{(p)}, w_{q+1}^{(c)}$ and $w_{q+1}^{(t)}$ as in Section \ref{s:con} with the new parameters given in \eqref{parameter g}. By exactly the same argument as in Section \ref{s:con}, all the estimates and equalities in \eqref{rho}-\eqref{vq} hold. When applying Lemma \ref{Lp} we choose $\zeta=\ell^{-8}$ with $(\ell^{-8})^5<(\lambda_{q+1}r_\perp)=\lambda_{q+1}^{1/{28}}$, where we use $\alpha=\frac1{28\cdot 80}$. Hence, \eqref{estimate wqp}-\eqref{correction est} also hold. The estimation in \eqref{temporal est1} follows the same way except for the last equality, i.e. we have
 \begin{equation*}
 \aligned
 \|w_{q+1}^{(t)}\|_{C_tL^p}\lesssim\delta_{q+1}M_0(t) \ell^{-4}r_\perp^{2/p-1}r_\|^{1/p-2}(\mu^{-1}r_\perp^{-1}r_\|).
 \endaligned
 \end{equation*}
 By the choice of the parameters, we have $M_0(\ln\ln L)^{1/2}\lambda_{q+1}^{4\alpha-1/{28}}<1$, which implies that \eqref{corr temporal} holds in this case. As a result, we obtain for $t\in[0, T_L]$
 \begin{equation*}
 \aligned
 \|w_{q+1}\|_{C_tL^2}&\leq M_0(t)^{1/2}\delta_{q+1}^{1/2}\left(\frac{1}{2}+C\lambda_{q+1}^{24\alpha-{11}/{28}}+CM_0(t)^{1/2}\delta_{q+1}^{1/2}\lambda_{q+1}^{8\alpha-1/{28}}\right)
 \leq \frac{3}{4}M_0(t)^{1/2}\delta_{q+1}^{1/2}.
 \endaligned
 \end{equation*}
Hence the first inequality in  \eqref{iteration g} holds.

\subsubsection{Estimate of $\|z_q-z_{q+1}\|_{{\bar{B},\alpha_{0},2\alpha_{0},t}}$}
\label{s:z}

In the following, we intend to estimate $\|v_{q}-v_{q+1}\|_{C_{t}^{\alpha_0}B^{-\gamma}_{1,1}}$ for some $\alpha_0>\frac23$ and $\gamma>0$. This is  required for the rough path estimate of $\|z_q-z_{q+1}\|_{{\bar{B},\alpha_{0},2\alpha_{0},t}}$ on the time interval $[0,t]$, cf. Theorem~\ref{zexistence1}.

	To this end, we first estimate $\|w_{q+1}\|_{C^{\alpha_0}_{{t}}B^{-\gamma}_{1,1}}$.
The idea is   to gain some negative power of $\lambda_{q+1}$ from $\|W_{(\xi)}\|_{C_t^{\alpha_0} B^{-\gamma}_{1,1}}$. By paraproduct estimates similar to \cite[Lemma~A.7]{GH21} and applying Lemma~\ref{l:83} we deduce for any $\delta>0$ and  $\gamma=5+\delta$
\begin{align*}
\|W_{(\xi)}\|_{C_t^{\alpha_0} B^{-\gamma}_{1,1}}&\lesssim \|\psi_{(\xi)}\|_{C_t^{\alpha_0} H^{\gamma+\delta}}\|\phi_{(\xi)}\|_{ H^{-\gamma}},
\\&\lesssim \mu^{\alpha_0}(\lambda_{q+1}r_{\perp}/r_{\parallel})^{\alpha_0+\gamma+\delta}(\lambda_{q+1})^{-\gamma}r_\perp^{-\delta}\lesssim \lambda^{{(53\alpha_0+44\delta-11\gamma)}/{28}}_{q+1}.
\end{align*}

Using the paraproduct estimates again, we have for $\gamma=5+\delta$
\begin{align*}
\|w_{q+1}^{(p)}\|_{C_t^{\alpha_0} B^{-\gamma}_{1,1}}&\lesssim \|a_{(\xi)}\|_{C_{x,t}^{\alpha_0+\gamma+\delta}}\|W_{(\xi)}\|_{C_t^{\alpha_0} B^{-\gamma}_{1,1}}
\\&\lesssim \delta^{1/2}_{q+1}M_0(t)^{1/2}\ell^{-2-5\lceil\alpha_0+\gamma+\delta\rceil}\lambda^{{(53\alpha_0+{44\delta}-11\gamma)}/{28}}_{q+1},
\end{align*}
where by $\lceil \alpha_{0}+\gamma+\delta\rceil$ we denote the  smallest integer bigger than $ \alpha_{0}+\gamma+\delta$. In view of \eqref{bounds}, we find (applying interpolation similarly to Lemma~\ref{l:83})
\begin{align*}\|w_{q+1}^{(c)}\|_{C_t^{\alpha_0} L^1}&\lesssim \|a_{(\xi)}\|_{C_{x,t}^{\alpha_0}}\|W^{(c)}_{(\xi)}\|_{C_t^{\alpha_0} L^1}+\|a_{(\xi)}\|_{C_{x,t}^{2+\alpha_0}}\|V_{(\xi)}\|_{C_t^{\alpha_0} W^{1,1}}
\\&\lesssim \delta^{1/2}_{q+1}M_0(t)^{1/2}\ell^{-17}
r_\perp r_{\|}^{1/2}(r_\perp r_{\|}^{-1}+\lambda_{q+1}^{-1})\Big(\frac{r_\perp \lambda_{q+1}\mu}{r_{\|}}\Big)^{\alpha_0}
\\&\lesssim \delta^{1/2}_{q+1}M_0(t)^{1/2}\ell^{-17}
\lambda_{q+1}^{{53\alpha_0}/{28}-{23}/{14}}.
\end{align*}
We choose $p=\frac{35}{35-14\alpha}>1$ so that it holds  in particular that $r_\perp^{2/p-2}r_\|^{1/p-1}\leq \lambda_{q+1}^\alpha$. Then we obtain for $\alpha_0=2/3+\kappa$ with $\kappa>0$ small
\begin{align*}
\|w_{q+1}^{(t)}\|_{C_t^{\alpha_0} L^p}&\lesssim \mu^{-1}\|a_{(\xi)}\|_{C_{x,t}^{\alpha_0}}\|a_{(\xi)}\|_{C_{x,t}}\|\phi_{(\xi)}\|_{ L^{2p}}^2\|\psi_{(\xi)}\|_{C_t^{\alpha_0}L^{2p}}\|\psi_{(\xi)}\|_{C_tL^{2p}}
\\&\lesssim \delta_{q+1}M_0(t)\ell^{-9}r_{\perp}^{2/p-2}r_{\|}^{1/p-1}\mu^{-1}\Big(\frac{r_\perp \lambda_{q+1}\mu}{r_{\|}}\Big)^{\alpha_0}
\\&\lesssim\delta_{q+1}M_0(t)\ell^{-9}\lambda_{q+1}^{{53}\alpha_0/{28}+\alpha -9/7}
\end{align*}
and then using $\alpha_{0}=2/3+\kappa$ we get $53\alpha_{0}/28-9/7= -1/42+53\kappa/28$ so finally
\begin{align*}
\|w_{q+1}^{(t)}\|_{C_t^{\alpha_0} L^p}&\lesssim\delta_{q+1}M_0(t)\ell^{-9}\lambda_{q+1}^{-1/42+53\kappa/28+\alpha}.
\end{align*}
We also have
\begin{align*}
\|v_q-v_\ell\|_{C_t^{\alpha_0}L^2}\lesssim \|v_q-v_\ell\|_{C_tL^2}^{1-\alpha_0} \|v_q-v_\ell\|_{C^1_tL^2}^{\alpha_0}\lesssim\ell^{1-\alpha_0}\|v_q\|_{C^1_{t,x}}\leq  \ell^{1-\alpha_0}\lambda_q^4M_0(t)^{1/2}.
\end{align*}
Combining the above estimates we obtain for $\alpha_0=2/3+\kappa$, $\gamma=5+\delta$
\begin{align*}
&\|v_q-v_{q+1}\|_{C_t^{\alpha_0}B^{-\gamma}_{1,1}}\\
&\quad\lesssim M_0(t)^{1/2}(\ell^{1-\alpha_0}\lambda_q^4+\ell^{-17}\lambda_{q+1}^{{53\alpha_0}/{28}-{23}/{14}}+M_0(t)^{1/2}\ell^{-9}\lambda_{q+1}^{-1/{42}+{53\kappa/28+\alpha}}\\
&\qquad\qquad+\ell^{-2-5\lceil\alpha_0+\gamma+\delta\rceil}\lambda_{q+1}^{{(53\alpha_0+44\delta-11\gamma)}/{28}})
\\&\quad\lesssim M_0(t)^{1/2}(\ell^{1-\alpha_0}\lambda_q^4+M_0(t)^{1/2}\ell^{-9}\lambda_{q+1}^{-1/{42}+{53\kappa/28+\alpha}})\leq M_0(t)^{1/2}\delta_{q+1}^{1/2},
\end{align*}
since by the definition of $\ell$ and $1-\alpha_0> 1/6$ we have
$\ell^{1-\alpha_0}\lambda_q^4\leq(\lambda_{q+1}^{-\frac{3\alpha}2}\lambda_q^{-2})^{1/6}\lambda_q^4\leq \lambda_{q+1}^{-\beta b}$, which is satisfied by the second inequality in \eqref{c:M g} by replacing $q$ by $q+1$.
  Then the second inequality in \eqref{iteration g} holds.

By the choice of $v_0$ in Lemma~\ref{lem:v0 g} we obtain
\begin{align*}
\|v_q\|_{C_t^{\alpha_0}B^{-5-\delta}_{1,1}}\lesssim\|v_0\|_{C_t^{\alpha_0}B^{-5-\delta}_{1,1}}+\sum_{k\geq0}\|v_k-v_{k+1}\|_{C_t^{\alpha_0}B^{-5-\delta}_{1,1}}&\leq (2L+2)M_0(t)^{1/2}.
\end{align*}
Moreover, by Theorem \ref{zexistence1} we have
\begin{align}
\label{bd:z2}
\|z_q\|_{C_tL^2}\leq\|z_{q}\|_{\bar{B},\alpha_{0},2\alpha_{0},t} \leq (2L+2)LM_0(t)^{1/2},
\end{align}
We intend to combine the above two  estimates with the  last inequality in Theorem \ref{zexistence1}. To this end, we observe  that $\bar{N}=M_0(\ln\ln L)$ in Theorem \ref{zexistence1} and  the right hand side of the estimate can be controlled by
$\bar{N}^{(6.5)T\bar{N}^{6.5}}$. Then we could choose $L$ large enough to have it be controlled by $\bar{N}^{\bar{N}^{7}}$. As a consequence, we choose $\kappa$ small satisfying $53\kappa/28\leq \alpha/2$ and \eqref{c:M g} to see for $q\in\mathbb{N}$ that 
\begin{align}\label{bd:z1}
\|z_q-z_{q+1}\|_{\bar B,\alpha_0,2\alpha_0,t}&\lesssim M_0(t)^{1/2}(\ell_0^{1-\alpha_0}\lambda_{q-1}^4+\ell_0^{-9}\lambda_{q}^{-1/{42}+{53\kappa/28+\alpha}}M_0(t)^{1/2})M_0(\ln\ln L)^{M_0(\ln\ln L)^7}\\
&\leq M_0(t)^{1/2}\delta_{q+1}^{1/2},\nonumber
\end{align}
with $\ell_0=\lambda_q^{-\frac{3\alpha}{2}}\lambda_{q-1}^{-2}$. For $q=0$ nothing needs to be proven since $z_1=z_0$.
Hence, \eqref{ite z} holds.

\subsubsection{Conclusion}

\begin{proof}[Proof of Proposition \ref{main iteration g}]

Recall that we  changed the parameter $r_\perp$ compared to before and more precisely $r_\perp$ becomes smaller. Hence,  \eqref{principle est2}-\eqref{temporal est2} hold, which implies
for $t\in[0, T_L]$
\begin{equation*}
\aligned
\|v_{q+1}\|_{C^1_{t,x}}&\leq \|v_\ell\|_{C^1_{t,x}}+\|w_{q+1}\|_{C^1_{t,x}}\\
&\leq M_0(t)^{1/2}\left(\lambda_{q+1}^\alpha+C\lambda_{q+1}^{14\alpha+3+{1}/{4}}+C\lambda_{q+1}^{34\alpha+20/7}+CM_0(t)^{1/2}\lambda_{q+1}^{19\alpha+3+3/{28}}\right)
\leq M_0(t)^{1/2}\lambda_{q+1}^4,
\endaligned
\end{equation*}
where we use $CM_0(\ln\ln L)^{1/2}\leq \frac{1}{2}\lambda_{q+1}^{{25}/{28}-19\alpha}$. Thus, the second estimate in \eqref{inductionv} holds true on the level $q+1$.
Moreover, the estimates \eqref{principle est22}-\eqref{corrector est2} hold.

In the following we control $\mathring{R}_{q+1}$.
We choose $p=\frac{35}{35-14\alpha}>1$ so that  $r_\perp^{2/p-2}r_\|^{1/p-1}\leq \lambda_{q+1}^\alpha$.
First, we recall that $r_\perp$ becomes smaller and $p$ is close to $1$. Hence $r_\perp^{2/p-1}$ becomes smaller and we use $r_\perp^{2/p-2}r_\|^{1/p-1}\leq \lambda_{q+1}^\alpha$. As a result, the bounds for   $R_{\textrm{osc}}^{(t)}$ and the terms not involving $z$ in $R_{\textrm{lin}}$ do not change.
For $R_{\textrm{cor}}$ we have
\begin{equation*}
\aligned
&\|R_{\textrm{cor}}\|_{C_tL^p}
\\&\lesssim M_0(t)\left(\ell^{-12}r_\perp^{1/p}r_\|^{1/(2p)-3/2}
+\ell^{-4}M_0(t)^{1/2}r_\perp^{1/p-1}r_\|^{1/(2p)-1/2}\mu^{-1}r_\perp^{-1}r_{\|}^{-1/2}\right)\ell^{-2}r_\perp^{1/p-1}r_\|^{1/(2p)-1/2}
\\
&\lesssim M_0(t)\left(\ell^{-14}r_\perp^{2/p-1}r_\|^{1/p-2}
+\ell^{-6}M_0(t)^{1/2}r_\perp^{2/p-3}r_\|^{1/p-3/2}\mu^{-1}\right)
\\
&\lesssim M_0(t)\left(\lambda_{q+1}^{29\alpha-\frac{11}{28}}+M_0(t)^{1/2}\lambda_{q+1}^{13\alpha-\frac1{28}}\right)
\leq \frac{M_0(t)c_R\delta_{q+2}}{5}.
\endaligned
\end{equation*}

Similarly as before,
we also have
\begin{equation*}
\aligned
\|R_{\textrm{osc}}^{(x)}\|_{C_{t}L^p}
&\lesssim M_0(t)\ell^{-9}r_\perp^{2/p-2}r_\|^{1/p-1}(r_\perp^{-1}\lambda_{q+1}^{-1})\lesssim M_0(t)\ell^{-9}\lambda_{q+1}^\alpha (r_\perp^{-1}\lambda_{q+1}^{-1})\\
&\lesssim M_{0}(t)\lambda_{q+1}^{19\alpha-\frac1{28}}\leq\frac{M_0(t)c_R\delta_{q+2}}{10}.
\endaligned
\end{equation*}

 In the following it suffices to consider the terms containing $z$ in $R_{\textrm{lin}}$, $R_{\textrm{com}}$ and $R_{\textrm{com}1}$. Recall that  we have
$$z_\ell=(z_q*_x\phi_\ell)*_t\varphi_\ell,$$
\begin{equation*}
R_{\textrm{com}}=(v_\ell+z_\ell)\mathring{\otimes}(v_\ell+z_\ell)-((v_q+z_q)\mathring{\otimes}(v_q+z_q))*_x\phi_\ell*_t\varphi_\ell,
\end{equation*}
\begin{equation*}
R_{\textrm{com}1}:=v_{q+1}\mathring{\otimes}z_{q+1}-v_{q+1}\mathring{\otimes}z_\ell+z_{q+1}\mathring{\otimes}v_{q+1}-z_\ell\mathring{\otimes}v_{q+1}
+z_{q+1}\mathring{\otimes}z_{q+1}-z_\ell\mathring{\otimes}z_\ell.
\end{equation*}
In order to estimate the following remaining  term in $R_{\textrm{lin}}$,
\begin{equation*}
R_{\textrm{lin}}^1:=(v_\ell+z_{\ell})\mathring\otimes w_{q+1}+w_{q+1}\mathring\otimes (v_\ell+z_{\ell}),
\end{equation*}
we use \eqref{z g} as well as \eqref{ell} to obtain for $q\in\mathbb{N}$
\begin{equation*}
\aligned
\|R_{\textrm{lin}}^1\|_{C_tL^p}
&\lesssim\|(v_\ell+z_\ell)\mathring{\otimes}w_{q+1}+w_{q+1}\mathring{\otimes}(v_\ell+z_\ell)\|_{C_tL^p}
\lesssim M_0(t)^{1/2}(\lambda_{q}^4+\lambda^4_{q-1}L^{1/4})\|w_{q+1}\|_{C_tL^p}\\
&\lesssim M_0(t)\ell^{-2}r_\perp^{2/p-1}r_\|^{1/p-1/2}(\lambda^4_{q}+\lambda^4_{q-1}L^{1/4})\lesssim M_0(t)\lambda_{q+1}^{6\alpha-5/4}\lesssim\frac{M_0(t)c_R\delta_{q+2}}{10},
\endaligned
\end{equation*}
where we used $\lambda_{q-1}^{4}L^{1/4}\leq \lambda_{q}^{4}$. For $q=0$ we have
\begin{equation*}
\aligned
\|R_{\textrm{lin}}^1\|_{C_tL^p}
&\lesssim M_0(t)\ell^{-2}r_\perp^{2/p-1}r_\|^{1/p-1/2}(\lambda^4_{0}+\lambda^4_{0}L^{1/4})\lesssim M_0(t)\lambda_0^4L^{1/4}\lambda_{1}^{5\alpha-5/4}\leq\frac{M_0(t)c_R\delta_{q+2}}{10},
\endaligned
\end{equation*}
where the last inequality follows from  $\alpha b>32$ and $\alpha>16 \beta b^2$.

In view of the standard mollification estimates we deduce for $t\in[0, T_L]$ and  $q=0$
\begin{align*}
\|R_{\textrm{com}}\|_{C_{t}L^1}\lesssim\ell^{\frac{1}{2}-2\delta}\lambda_0^4L^{1/4}L^2M_0(t)\leq  \frac{M_0(t)c_R\delta_{2}}{5},
\end{align*}
where we used  $\alpha b>32$ and $\alpha>16 \beta b^2$ in the last inequality. For $q\in\mathbb{N}$ it holds
\begin{equation*}
\aligned
\|R_{\textrm{com}}\|_{C_{t}L^1}&\lesssim \ell(\|v_q\|_{C^1_{t,x}}+\|z_q\|_{C_tC^1})(\|v_q\|_{C_tL^2}+\|z_q\|_{C_tL^2})\\
&\qquad+\ell^{\frac{1}{2}-2\delta}(\|z_q\|_{C_t^{\frac{1}{2}-2\delta}L^\infty}+\|v_q\|_{C_{t,x}^1})(\|v_q\|_{C_tL^2}
+\|z_q\|_{C_t{L^2}})\\
&\lesssim\ell  \lambda_q^4L^2M_0(t)+\ell^{{1}/{2}-2\delta}\lambda_q^4L^2M_0(t)\leq \frac{M_0(t)c_R\delta_{q+2}}{5},
\endaligned
\end{equation*}
where we used \eqref{bd:z2} and  $\lambda_{q-1}^{4}L^{1/2}\leq \lambda_{q}^{4}$ and $\delta<\frac{1}{12}$ and  we require    $\ell^{{1}/{2}-2\delta}\lambda_q^4L^2<\frac{c_R\delta_{q+2}}{10}$, i.e.
\begin{align*}
{L^2}\lambda_{q+1}^{-\alpha/2}\lambda_q^{-2/3}\lambda_q^4<\lambda_{q+1}^{-2\beta b},
\end{align*}
with the  choice of $\ell$ in \eqref{ell1} and the exponents were obtained with the choice $\delta=1/12$.  Since we postulated that  $\alpha b>32$, this can indeed be achieved by possibly increasing  $a$ and consequently decreasing $\beta$.

Finally, we use \eqref{z g} and  \eqref{c:M g}  to
obtain for $t\in[0, T_L]$ and $q\in\mathbb{N}$
\begin{align*}
&\|R_{\textrm{com}1}\|_{C_tL^1}\lesssim \|v_{q+1}\|_{C_tL^2}\|z_{q+1}-z_\ell\|_{C_tL^2}+(\|z_{q+1}\|_{C_tL^2}+\|z_{\ell}\|_{C_tL^2})\|z_{q+1}-z_\ell\|_{C_tL^2}
\\&\lesssim M_0(t)(\ell_0^{1-\alpha_0}\lambda_{q-1}^4+M_0(t)^{1/2}\ell_0^{-9}\lambda_{q}^{-1/{42}+53\kappa/28+\alpha}+\ell^{1/2-2\delta}\lambda_{q-1}^4){M_0(\ln\ln L)^{M_0(\ln\ln L)^7}L^2}
\\&\lesssim\frac{M_0(t)c_R\delta_{q+2}}{5},
\end{align*}
with $\ell_0=\lambda_{q}^{-3\alpha/2}\lambda_{q-1}^{-2}$ and where
 we used \eqref{bd:z1} yielding for $q\in\mathbb{N}$
$$
\begin{aligned}
&\|z_{q+1}-z_{\ell}\|_{C_{t}L^{2}}\leq \|z_{q+1}-z_{q}\|_{C_{t}L^{2}}+\|z_{q}-z_{\ell}\|_{C_{t}L^{2}}\\
&\quad\lesssim M_{0}(t)^{1/2}(\ell_{0}^{1-\alpha_{0}}\lambda_{q-1}^{4}+\ell^{-9}\lambda_{q}^{-1/42+53\kappa/28+\alpha}M_0(t)^{1/2}+\ell^{1/2-2\delta}\lambda_{q-1}^4){M_{0}(
\ln\ln L)^{M_{0}(\ln\ln L)^{7}}}.
\end{aligned}
$$
For $q=0$ we use \eqref{z g} to have
\begin{align*}
\|z_1-z_\ell\|_{C_tL^2}\leq \|z_1-z_0\|_{C_tL^2}+\|z_0-z_\ell\|_{C_tL^2}\leq \ell^{1/2-2\delta}L^{1/2}(1+\|v_0\|_{C^1_{t,x}})
\leq\ell^{1/2-2\delta}L^{1/2}\lambda_0^4M_0(t)^{1/2},
\end{align*}
which combined with \eqref{bd:z2} implies for $t\in[0, T_L]$ and $q=0$
\begin{align*}
\|R_{\textrm{com}1}\|_{C_tL^1}&\lesssim M_0(t)\ell^{1/2-2\delta}\lambda_{0}^4L^{5/2}
\lesssim\frac{M_0(t)c_R\delta_{2}}{5}.
\end{align*}
The proof is complete.
\end{proof}

\color{black}

 \appendix
  \renewcommand{\appendixname}{Appendix~\Alph{section}}
  \section{Proof of Theorem \ref{convergence}, Theorem \ref{convergence 1} and Theorem \ref{ge convergence 1}}
  \label{ap:A}

Let us begin with  the following tightness result.

 \begin{lem}\label{tightness}
Let $\{(s_{n},x_{n})\}_{n\in\mathbb{N}}\subset [0,\infty)\times L^{2}_{\sigma}$ such that $(s_{n},x_{n})\to (s,x_{0})$. Let $\{P_n\}_{n\in\mathbb{N}}$ be a family of probability measures on $\Omega_0$ satisfying for all $n\in\mathbb{N}$
 \begin{equation}\label{est:Pn1}
 P_n(x(t)=x_n, 0\leq t\leq s_n)=1
\end{equation}
and  for some $\gamma,\kappa>0$ and any $T>0$
 \begin{equation}\label{est:Pn}
 \sup_{n\in\mathbb{N}} E^{P_n}\left(\sup_{t\in[0,T]}\|x(t)\|_{L^2}+\sup_{r\neq t\in[0,T]}\frac{\|x(t)-x(r)\|_{H^{-3}}}{|t-r|^\kappa}+\int_{s_n}^T\|x(r)\|_{H^\gamma}^2dr\right)<\infty.
 \end{equation}
Then $\{P_n\}_{n\in\mathbb{N}}$ is tight in $\mathbb{S}:=C_{\mathrm{loc}}([0,\infty);H^{-3})\cap L^2_{\mathrm{loc}}([0,\infty);L^2_\sigma)$.
 \end{lem}

 \begin{proof}
In view of the uniform bound \eqref{est:Pn}, the canonical process under the measure $P_{n}$ is bounded in $L^{\infty}_{\mathrm{loc}}([0,\infty);L^{2})\cap C^{\kappa}_{\mathrm{loc}}([0,\infty);H^{-3})\cap L^{2}_{\mathrm{loc}}([s_{n},\infty);H^{\gamma})$ and the bounds are  uniform in $n$. We recall that a set $K\subset \mathbb{S}$ is compact provided
$$
K_{T}:=\{f|_{[0,T]};f\in K\}\subset C([0,T];H^{-3})\cap L^{2}(0,T;L^2_\sigma)
$$
is compact for every $T>0$. In addition, for every $T>0$, the following embedding
$$
L^{\infty}(0,T;L^{2})\cap C^{\kappa}([0,T];H^{-3})\cap L^{2}([0,T];H^{\gamma})\subset C([0,T];H^{-3})\cap L^{2}(0,T;L^{2}_{\sigma})
$$
is compact, see e.g. \cite[Section 1.8.2]{BFH18}. This implies that also the embedding of the local-in-time spaces
$$
L^{\infty}_{\mathrm{loc}}([0,\infty);L^{2})\cap C^{\kappa}_{\mathrm{loc}}([0,\infty);H^{-3})\cap L^{2}_{\mathrm{loc}}([0,\infty);H^{\gamma})
\subset \mathbb{S}
$$
is compact as well.
This result, however, cannot be applied directly in order to prove the claim of the lemma due to the fact that the uniform $H^{\gamma}$ regularity in \eqref{est:Pn} only holds on the respective time intervals $[s_{n},T]$. The idea is instead  to use \eqref{est:Pn1} which says that under each measure $P_{n}$  the canonical process is constant  on $[0,s_{n}]$  and its value equals to $x_{n}$. Together with the fact that $(s_{n},x_{n})\to (s,x_{0})$ in $[0,\infty)\times L^{2}_{\sigma}$,  the desired compactness then follows.

To be more precise,
we fix $\epsilon>0$ and any $k\in\mathbb{N}$, $k\geq k_{0}:=\sup_{n\in\mathbb{N}}s_{n}$, we may choose $R_k>0$ sufficiently large such that
$$
P_n\left(x\in\Omega_0: \sup_{t\in[0,k]}\|x(t)\|_{L^2}+\sup_{r\neq t\in[0,k]}\frac{\|x(t)-x(r)\|_{H^{-3}}}{|t-r|^\kappa}+\int_{s_n}^k\|x(r)\|_{H^\gamma}^2dr>R_k\right)\leq \epsilon/2^k.
$$
Now, we set
 $\Omega_n:=\{x\in \Omega_0; x(t)=x_n, 0\leq t\leq s_n\}$ and define
\begin{equation}\label{eq:K}
K:=\bigcup_{n\in \mathbb{N}}\bigcap_{\substack{k\in\mathbb{N}\\k\geq k_{0}}}\left\{x\in \Omega_n;\sup_{t\in[0,k]}\|x(t)\|_{L^2}+\sup_{r\neq t\in[0,k]}\frac{\|x(t)-x(r)\|_{H^{-3}}}{|t-r|^\kappa}+\int_{s_n}^k\|x(r)\|_{H^\gamma}^2dr\leq R_k  \right\}.
\end{equation}
By Chebyshev's inequality together with \eqref{est:Pn}, it follows that
\begin{equation*}
\sup_{n\in\mathbb{N}} P_n(\overline{K}^c)\leq\sup_{n\in\mathbb{N}} P_n(K^c)\leq \epsilon,
\end{equation*}
so it only remains to show that the set $\overline{K}$ is a compact set in $\mathbb{S}$. As mentioned above, it is sufficient to prove that for every $k\in\mathbb{N}$, $k\geq k_{0}$, the restriction of functions in $K$ to $[0,k]$ is relatively compact in  $\mathbb{S}_k:=C([0,k],H^{-3})\cap L^2(0,k;L^2_{\sigma})$.

To this end, let $\{x_m\}_{m\in \mathbb{N}}$ be a sequence in $K$.
If there exists $N\in\mathbb{N}$ so that for   infinitely many $m$ it holds $x_m\in \Omega_N$, the result can be obtained by a standard argument based on the compact embedding discussed above. If this is not true, we may assume  without loss of generality  that $x_m\in \Omega_m$. The compactness in $C([0,k];H^{-3})$ is a direct consequence of the bound
$$\sup_{t\in[0,k]}\|x_m(t)\|_{L^2}+\sup_{r\neq t\in[0,k]}\frac{\|x_m(t)-x_m(r)\|_{H^{-3}}}{|t-r|^\kappa}\leq R_k$$
and the compact embedding
$$
L^{\infty}(0,k;L^{2})\cap C^{\kappa}([0,k];H^{-3})\subset C([0,k];H^{-3}).
$$
Consequently, we can find a subsequence $x_{m_l}$ such that
\begin{equation}
\label{H-3}
\lim_{l,n\rightarrow\infty}\sup_{t\in[0,k]}\|x_{m_l}(t)-x_{m_n}(t)\|_{H^{-3}}=0.
\end{equation}
With this in hand, we deduce
$$
\aligned\int_0^k\|x_{m_l}(t)-x_{m_n}(t)\|_{L^2}^2dt
&\leq
\int_0^{s_{m_l}\wedge s_{m_n}}\|x_{m_l}(t)-x_{m_n}(t)\|_{L^2}^2dt
\\
&\quad+\int_{s_{m_l}\wedge s_{m_n}}^{s_{m_l}\vee s_{m_n}}\|x_{m_l}(t)-x_{m_n}(t)\|_{L^2}^2dt+\int_{s_{m_l}\vee s_{m_n}}^k\|x_{m_l}(t)-x_{m_n}(t)\|_{L^2}^2dt
\\
&\leq k \|x_{m_l}(0)-x_{m_n}(0)\|_{L^2}^2+4R_k^2(s_{m_l}\vee s_{m_n}-s_{m_l}\wedge s_{n_m})\\
&\quad+\varepsilon \int_{s_{m_l}\vee s_{m_n}}^k\|x_{m_l}(t)-x_{m_n}(t)\|_{H^\gamma}^2dt+C_\varepsilon k\sup_{t\in[0,k]}\|x_{m_l}(t)-x_{m_n}(t)\|_{H^{-3}}^2
\\
&\leq k\|x_{m_l}(0)-x_{m_n}(0)\|_{L^2}^2+4R_k^2(s_{m_l}\vee s_{m_n}-s_{m_l}\wedge s_{m_n})\\
&\quad+4\varepsilon R_k+C_\varepsilon k\sup_{t\in[0,k]}\|x_{m_l}(t)-x_{m_n}(t)\|_{H^{-3}}^2\rightarrow0,
\endaligned
$$
as $m_l, m_n\rightarrow\infty$ and we used interpolation and Young's inequality in the second step and we used (\ref{H-3}) in the last step. Now the proof is complete.
\end{proof}

\begin{proof}[Proof of Theorem \ref{convergence}]  The first result giving existence of a martingale solution can be easily deduced by Galerkin approximation and the same arguments as in \cite{FR08, GRZ09}.
 The second result giving the stability of martingale solutions with respect to the initial time and initial condition will be proved in the sequel based on Lemma \ref{tightness}.

First, we prove that the set $\{P_n\}_{n\in\mathbb{N}}$ is tight in $\mathbb{S}:=C_{\rm{loc}}([0,\infty);H^{-3})\cap L^2_{\textrm{loc}}([0,\infty);L^2_\sigma)$.
  To this end, we denote  $F(x):=-\mathbb{P}\div(x\otimes x)+\Delta x$. Since for every $n\in\mathbb{N}$, the measure $P_{n}$ is a martingale solution to \eqref{1} starting from the initial condition $x_{n}$ at  time $s_{n}$ in the sense of Definition \ref{martingale solution},
 we know that for $t\in[ s_n,\infty)$
\begin{equation*}
x(t)=x_n+\int_{s_n}^tF(x(r))dr+M^{x}_{t,s_n}, \quad P_n\text{-a.s.,}
\end{equation*}
where $t\mapsto M_{t,s_{n}}^{x,i}=\langle M^{x}_{t,s_{n}},e_{i}\rangle$, $x\in\Omega_{0}$, is a continuous square integrable  martingale with respect to $(\mathcal{B}_{t})_{t\geq s_{n}}$ with the quadratic variation process given by $t\mapsto \int_{s_n}^t\|G(x(r))^* e_{i}\|_{U}^{2}dr$.
Moreover, according to (M3) it holds  for every $p>1$
\begin{equation*}
\aligned
 E^{P_n}\left[\sup_{r\neq t\in [s_n,T]}\frac{\|\int_r^tF(x(l))dl\|_{H^{-3}}^p}{|t-r|^{p-1}}\right]
 &\leq
E^{P_n}\left[\int_{s_n}^t\|F(x(r))\|_{H^{-3}}^pdr\right]
\\
&\lesssim \|x_n\|_{L^2}^{2p}+1,
\endaligned
\end{equation*}
where the implicit constant is universal and therefore independent of $n$ since all $P_n$ share the same $C_{t,q}$.
By the condition on $G$ we have for every $p>1$
\begin{align*}
&E^{P_n}\|M_{t,s_n}-M_{r,s_n}\|_{L^2}^{2p}\leq C_pE^{P_n}
\left(\int_r^t\|G(x(l))\|_{L_2(U,L^2_\sigma)}^2dl\right)^p
\\&\leq C_p|t-r|^{p-1}E^{P_n}
\int_r^t\|G(x(l))\|_{L_2(U,L^2_\sigma)}^{2p}dl
\\&\leq C_p|t-r|^{p-1}E^{P_n}
\int_r^t(\|x(l)\|_{L^2}^{2p}+1)dl\leq C_p|t-r|^{p-1}(\|x_n\|_{L^2}^{2p}+1).
\end{align*}
By Kolmogorov's criterion, for any $\alpha\in(0,\frac{p-1}{2p})$ we get
$$
\aligned
E^{P_n}\left[\sup_{r\neq t\in [0,T]}\frac{\|M_{t,s_n}-M_{r,s_n}\|_{L^2}}
{|t-r|^{p\alpha}}\right]\leq C_p(\|x_n\|_{L^2}^{2p}+1).
\endaligned
$$
Combining the above estimates, we conclude for all $\kappa\in (0,1/2)$ that
\begin{equation}
\label{tight}
\aligned
 \sup_{n\in\mathbb{N}}E^{P_n}\left[\sup_{r\neq t\in [0,T]}\frac{\|x(t)-x(r)\|_{H^{-3}}}{|t-r|^{\kappa}}\right]<\infty.
 \endaligned
 \end{equation}
Combining (\ref{tight}), (M3) and using  Lemma \ref{tightness} it follows that the set $\{P_n\}_{n\in\mathbb{N}}$ is tight in $\mathbb{S}$.

Without loss of generality, we may assume that $P_n$  converges weakly to some probability measure $P\in \mathscr{P}(\Omega_{0})$. It remains to prove that $P\in \mathscr{C}(s,x_0, C_{t,q})$.
By Skorohod's representation theorem, there exists a probability space $(\tilde{\Omega},\tilde{\mathcal{F}},\tilde{P})$ and $\mathbb{S}$-valued
random variables $\tilde{x}_n$ and $\tilde{x}$ such that
\begin{enumerate}
\item[(i)] $\tilde{x}_n$ has the law $P_n$ for each $n\in\mathbb{N}$,
\item[(ii)] $\tilde{x}_n\rightarrow \tilde{x}$ in $\mathbb{S}$ $\tilde{P}$-a.s., and $\tilde{x}$ has the law $P$.
\end{enumerate}
Since the initial conditions $x_{n}$ as well as the initial times $s_{n}$ are deterministic, we obtain by (i), (ii), and (M1) applied to  $P_n$ that
$$
\aligned
 P(x(t)=x_0,0\leq t\leq s)&=\tilde{P}(\tilde{x}(t)=x_0,0\leq t\leq s)=\lim_{n\rightarrow\infty}\tilde{P}(\tilde{x}_n(t)=x_n,0\leq t\leq s_n)\\&=\lim_{n\rightarrow\infty}P_{n}(x(t)=x_n,0\leq t\leq s_n)=1.
 \endaligned
 $$
As the next step,  we verify (M2) for $P$.  We know that under $\tilde{P}$ it holds according to the convergence in  (ii) that for every $e_i\in C^\infty(\mathbb{T}^3)$
$$
\langle \tilde{x}_n(t),e_i\rangle\rightarrow  \langle \tilde{x}(t),e_i\rangle,\quad \int_{s_n}^t\langle F(\tilde{x}_n(r)),e_i\rangle dr\rightarrow  \int_s^t\langle F(\tilde{x}(r)),e_i\rangle dr\qquad\tilde{P}\text{-a.s.}
$$
This implies for every $t\in[ s,\infty)$ and every $p>1$
\begin{equation*}
\sup_{n\in\mathbb{N}}E^{\tilde{P}}
\big[|M_{t,s_n}^{\tilde{x}_n,i}|^{2p}\big]
\leq  C\sup_{n\in\mathbb{N}}E^{P_{n}}\left[\left(\int_{s_n}^t\|G( x(r))\|^2_{L_2(U,L_{\sigma}^2)}ds\right)^{p}\right]<\infty,
\end{equation*}
\begin{equation}
\label{martingale1}
\lim_{n\rightarrow\infty} E^{\tilde{P}}\big[|M_{t,s_n}^{\tilde{x}_n,i}-M_{t,s}^{\tilde{x},i}|\big]=0.
\end{equation}
Let $t>r\geq s$ and $g$ be any bounded and real-valued $\mathcal{B}_r$-measurable continuous function on $\mathbb{S}$. Using (\ref{martingale1})
we know
$$
\aligned  E^{P}\big[(M_{t,s}^{x,i}-M_{r,s}^{x,i})g(x)\big]&=E^{\tilde{P}}\big[(M_{t,s}^{\tilde{x},i}-M_{r,s}^{\tilde{x},i})g(\tilde{x})\big]
=\lim_{n\rightarrow\infty}E^{\tilde{P}}\big[(M_{t,s_n}^{\tilde{x}_n,i}-M_{r,s_n}^{\tilde{x}_n,i})g(\tilde{x}_n)\big]
\\&=\lim_{n\rightarrow\infty}E^{{P_n}}\big[(M_{t,s_n}^{{x},i}-M_{r,s_n}^{{x},i})g({x})\big]=0.
\endaligned
$$
Consequently, we deduce
$$
E^P\big[M_{t,s}^{x,i}|\mathcal{B}_r\big]=M_{r,s}^{x,i}
$$
hence $t\mapsto M^{i}_{t,s}$ is a $(\mathcal{B}_{t})_{t\geq s}$-martingale under $P$.
Similarly, we have
\begin{equation*}
\lim_{n\rightarrow\infty} E^{\tilde{P}}\big[|M_{t,s_n}^{\tilde{x}_n,i}-M_{t,s}^{\tilde{x},i}|^2\big]=0,
\end{equation*}
which gives
$$
E^P\big[(M_{t,s}^{x,i})^2-\int_s^t\|G(x(l))^*e_i\|_{U}^2dl|\mathcal{B}_r\big]=(M_{r,s}^{x,i})^2-\int_r^t\|G(x(l))^*e_i\|_{U}^2dl|
$$
and accordingly (M2) follows.

Finally, we verify (M3). Define
$$S(t,s,x):=\sup_{r\in [0,t]}\|x(r)\|_{L^2}^{2q}+\int_s^t\|x(r)\|^2_{H^\gamma}dr,$$
It is easy to see that $x\mapsto S(t,s,x)$ is lower semicontinuous on $\mathbb{S}$. Hence, by Fatou's lemma
$$
\aligned
 E^P[S(t,s,x)]=E^{\tilde{P}}[S(t,s,\tilde{x})]\leq \liminf_{n\rightarrow\infty}E^{\tilde{P}}[S(t,s_n,\tilde{x}_n)]\leq C_{t,q}\liminf_{n\rightarrow\infty}(\|x_n\|_{L^2}^{2q}+1)<\infty.
\endaligned
$$
The proof is complete.
\end{proof}

\begin{proof}[Proof of Theorem \ref{convergence 1}]  The existence of a probabilistically weak solution can be easily deduced from Theorem \ref{convergence} and the martingale representation theorem, see \cite{DPZ92}.
 The stability of weak solutions with respect to the initial time and initial condition will be proved in a similar way as in Theorem \ref{convergence}.
First, we prove that the set $(P_n)_{n\in\mathbb{N}}$ is tight in
$$
\bar{\mathbb{S}}:=C_{\rm{loc}}([0,\infty);H^{-3}\times U_1)\cap L^2_{\textrm{loc}}([0,\infty);L^2_\sigma\times U_1).
$$
To this end,
we denote  $F(x):=-\mathbb{P}\div(x\otimes x)+\Delta x$ and recall that
for every $n\in\mathbb{N}$, the measure $P_{n}$ is a probabilistically weak solution to \eqref{1} starting from the initial condition $x_{n}$ at  time $s_{n}$ in the sense of Definition \ref{weak solution}. Thus, for $t\in[ s_n,\infty)$
\begin{equation*}
x(t)=x_n+\int_{s_n}^tF(x(r))dr+\int_{s_n}^tG(x_r)dy_r, \quad P_n\text{-a.s.}
\end{equation*}
where under $P_{n}$ the process $y$ is a cylindrical Wiener process on $U$ starting from $y_{n}$ at time $s_{n}$. In other words, under  $P_{n}$ the process $t\mapsto y(t +s_{n})-y_{n}$ is a cylindrical Wiener process  on $U$  starting at time $0$ from the initial value $0$. Since the law of the Wiener process is unique and tight, for a given $\epsilon>0$ there exists a compact set $K_{1}\subset C([0,\infty);U_{1})\cap L^{2}_{\rm{loc}}([0,\infty);U_{1})$ such that
\begin{equation*}
\sup_{n\in\mathbb{N}}P_{n}\big(y(\cdot+s_{n})-y_{n}\in K_{1}^{c}\big)\leq \epsilon.
\end{equation*}

Let us now define
\begin{align*}
K_{2}&:=\bigcup_{n\in\mathbb{N}}\big\{y\in C([0,\infty);U_{1});\\
&\qquad\qquad y(t+s_{n})-y_{n}\in K_{1}\text{ for }t\in[0,\infty),\ y(t)=y_{n}\text{ for }t\in[0,s_{n}]\big\}.
\end{align*}
Then
\begin{equation}\label{eq:K2}
\sup_{n\in\mathbb{N}}P_{n}(\overline{K_{2}}^{c})\leq\sup_{n\in\mathbb{N}}P_{n}\big(y(\cdot+s_{n})-y_{n}\in K_{1}^{c}\big)\leq \epsilon
\end{equation}
and we claim that $K_{2}$ is relatively compact in $C([0,\infty);U_{1})\subset L^{2}_{\rm{loc}}([0,\infty);U_{1})$. Indeed, let $\{y^{m}\}_{m\in\mathbb{N}}$ be a sequence in $K_{2}$. Then for every $m\in\mathbb{N}$ there exists $n_{m}\in\mathbb{N}$ and ${y}^{m,n_{m}}\in K_{1}$ so that
$$
y^{m}(t+s_{n_{m}})-y_{n_{m}}=y^{m,n_{m}}(t)\quad\text{for}\ t\in[0,\infty),\qquad\qquad y^{m}(t)=y_{n_{m}}\quad\text{for}\ t\in[0,s_{n_{m}}].
$$
If there exists $N\in\mathbb{N}$ such that $n_{m}=N$ for infinitely many $m\in\mathbb{N}$ then the relative compactness of $\{y^{m}\}_{m\in\mathbb{N}}$ follows directly from the fact that the corresponding sequence $\{y^{m,n_{m}}\}_{m\in\mathbb{N}}$ is relatively compact due to compactness of $K_{1}$. If such an $N$ does not exist, then by passing to a subsequence and relabelling we can assume without loss of generality  that $n_{m}=m$. In addition, it holds for $t\in[s_{m},\infty)$
$$
y^{m}(t)=y^{m,m}(t-s_{m})+y_{m}.
$$
Hence using the relative compactness of
$$
\{y^{m,m}\}_{m\in\mathbb{N}}\subset K_{1}\quad\text{and}\quad\{(s_{m},y_{m})\}_{m\in\mathbb{N}}\subset [0,\infty)\times U_{1},
$$
we finally deduce that the given sequence $\{y^{m}\}_{m\in\mathbb{N}}$ is relatively compact.

Now, we recall that the set $K$ defined  in the course of the proof of Theorem~\ref{convergence} in \eqref{eq:K}  is relatively compact in $C_{\rm{loc}}([0,\infty);H^{-3})\cap L^{2}_{\rm{loc}}([0,\infty);L^{2}_{\sigma})$. Chebyshev's inequality again shows that
\begin{equation}\label{eq:Ka}
\sup_{n\in\mathbb{N}} P_n(\overline{K}^c)\leq\sup_{n\in\mathbb{N}} P_n(K^c)\leq \epsilon.
\end{equation}
Hence  the set $K\times K_{2}$ is relatively compact in  $\bar{\mathbb{S}}$ and  the desired tightness follows from  \eqref{eq:K2}, \eqref{eq:Ka}.

Without loss of generality, we may assume that $P_n$  converges weakly to some probability measure $P$. It remains to prove that $P\in \mathscr{W}(s,x_0,y_0,C_{t,q})$.
By Skorokhod's representation theorem, there exists a probability space $(\tilde{\Omega},\tilde{\mathcal{F}},\tilde{P})$ and $\bar{\mathbb{S}}$-valued
random variables $(\tilde{x}_n,\tilde{y}_n)$ and $(\tilde{x},\tilde{y})$ such that
\begin{enumerate}
\item[(i)] $(\tilde{x}_n,\tilde{y}_n)$ has the law $P_n$ for each $n\in\mathbb{N}$,
\item[(ii)] $(\tilde{x}_n,\tilde{y}_n)\rightarrow (\tilde{x},\tilde{y})$ in $\bar{\mathbb{S}}$ $\tilde{P}$-a.s., and $(\tilde{x},\tilde{y})$ has the law $P$.
\end{enumerate}
Let $(\tilde{\mathcal{F}}_t)_{t\geq0}$ be the $\tilde{P}$-augmented canonical filtration of the process $(\tilde{x},\tilde{y})$. Then it is easy to see that $\tilde{y}$ is a cylindrical Wiener process on $U$ with respect to $(\tilde{\mathcal{F}}_t)_{t\geq0}$. In fact,
 let $t>s$ and $g$ be any bounded and real valued $\bar{\cB}_s$-measurable continuous functions on $\bar\Omega$ we have
\begin{align*}
&E^{P}\Big[(y(t)-y(s))g(x,y)\Big]
=E^{ \tilde P}\Big[\Big(\tilde y(t)-\tilde y(s)\Big)g(\tilde x,\tilde y)\Big]
\\=&\lim_{n\to\infty}E^{\tilde P}\Big[\Big(\tilde y_n(t)-\tilde y_n(s)\Big)g(\tilde x_n,\tilde y_n)\Big]=0,
\end{align*}
and similarly for $y_i=\langle y,l_i \rangle$ with $\{l_i\}$ orthonormal basis in $U$
\begin{align*}
&E^{P}\Big[(y_i(t)y_j(t)-y_i(s)y_j(s)-\delta_{i=j}(t-s))g(x,y)\Big]
=0,
\end{align*}
We then obtain that $y$ is a cylindrical Wiener process on $U$ with respect to $(\bar{\cB}_t)_{t\geq0}$ under $P$.

 The conditions (M1) and (M3) follow similarly as in the proof of Theorem \ref{convergence}.
Finally,  we shall verify (M2) for $P$.  We know that under $\tilde{P}$ it holds according to the convergence in  (ii) that for every $e_i\in C^\infty(\mathbb{T}^3)$
$$
\langle \tilde{x}_n(t),e_i\rangle\rightarrow  \langle \tilde{x}(t),e_i\rangle,\quad \int_{s_n}^t\langle F(\tilde{x}_n(r)),e_i\rangle dr\rightarrow  \int_s^t\langle F(\tilde{x}(r)),e_i\rangle dr\qquad\tilde{P}\text{-a.s.}
$$
Define $$M_{t,s}^{x,i}=\langle x(t)-x(s)-\int_s^tF(x(r))dr,e_i\rangle.$$
Then we have for every $t\in[ s,\infty)$ and every $p\in(1,\infty)$
\begin{equation*}
\sup_{n\in\mathbb{N}}E^{\tilde{P}}\big[|M_{t,s_n}^{\tilde{x}_n,i}|^{2p}\big]
\leq C\sup_{n\in\mathbb{N}}E^{P_{n}}\left[\left(\int_{s_n}^t\|G( x(r))\|^2_{L_2(U,L_{\sigma}^2)}ds\right)^{p}\right]<\infty,
\end{equation*}
\begin{equation}
\label{martingale2}
\lim_{n\rightarrow\infty} E^{\tilde{P}}\big[|M_{t,s_n}^{\tilde{x}_n,i}-M_{t,s}^{\tilde{x},i}|^{2}\big]=0.
\end{equation}
Let $t>r\geq s$ and $g$ be any bounded continuous function on $\bar{\mathbb{S}}$. Using (\ref{martingale2})
we know
$$
\aligned  &E^{P}\big[(M_{t,s}^{x,i}-M_{r,s}^{x,i})g(x|_{[0,r]},y|_{[0,r]})\big]=E^{\tilde{P}}\big[(M_{t,s}^{\tilde{x},i}-M_{r,s}^{\tilde{x},i})g(\tilde{x}|_{[0,r]},\tilde{y}|_{[0,r]})\big]
\\&=\lim_{n\rightarrow\infty}E^{\tilde{P}}\big[(M_{t,s_n}^{\tilde{x}_n,i}-M_{r,s_n}^{\tilde{x}_n,i})g(\tilde{x}_n|_{[0,r]},\tilde{y}_n|_{[0,r]})\big]
=\lim_{n\rightarrow\infty}E^{{P_n}}\big[(M_{t,s_n}^{x,i}-M_{r,s_n}^{x,i})g(x|_{[0,r]},y|_{[0,r]})\big]=0.
\endaligned
$$
Consequently, we deduce that
$t\mapsto M^{i}_{t,s}$ is a $(\bar{\mathcal{B}}_{t})_{t\geq s}$-martingale under $P$.
Similarly, we obtain
$$
E^P\left[(M_{t,s}^{x,i})^2-\int_s^t\|G(x)^*e_i\|_{U}^2dl\Big|\bar{\mathcal{B}}_r\right]=(M_{r,s}^{x,i})^2-\int_s^r\|G(x)^*e_i\|_{U}^2dl,
$$
which identifies the quadratic variation of $t\mapsto M^{i}_{t,s}$. It remains to identify the cross variation of this process with the cylindrical Wiener process $y$ under $P$. To this end, we let $\{l_j\}_{j\in\mathbb{N}}$ be an orthonormal basis of $U$ and define $y_j=\langle y,l_j\rangle_U$. Then we deduce that
$$
E^P\left[M_{t,s}^{x,i}(y_j(t)-y_j(s))-\int_s^t\langle G^*(x)e_i,l_j\rangle_U dl\Big|\bar{\mathcal{B}}_r\right]=M_{r,s}^{x,i}(y_j(r)-y_j(s))-\int_s^r\langle G^*(x)e_i,l_j\rangle_Udl.
$$
Thus, the quadratic variation process of $ M_{t,s}^{x,i}-\int_s^t \langle e_i, G(x)dy\rangle$ is $0$
which implies (M2). The proof is complete.
\end{proof}

\begin{proof}[Proof of Theorem \ref{ge convergence 1}]
	The existence of a generalized probabilistically weak solution follows from a similar argument as in the proof of Theorem \ref{convergence 1} and defining $\mY(t)=\mY_0+\int_s^ty(r)\otimes dy(r)$. By \cite[Theorem 4.2.5]{LR15} we have
		$P(Z\in CL^2_\sigma)=1$. 
 The stability of solutions with respect to the initial time and initial condition follows in a similar way as in Theorem \ref{convergence 1}.
	First, it follows by a similar argument as in Theorem \ref{convergence 1} that the set $(P_n)_{n\in\mathbb{N}}$ is tight in
	$$
	\widetilde{\mathbb{S}}:=C_{\rm{loc}}([0,\infty);H^{-3}\times \mR^m\times \mR^{m\times m}\times H^{-3})\cap L^2_{\textrm{loc}}([0,\infty);L^2_\sigma\times \mR^m\times \mR^{m\times m}\times L^2_\sigma).
	$$
	Without loss of generality, we may assume that $P_n$  converges weakly to some probability measure $P$. It remains to prove that $P\in \mathscr{GW}(s,x_0,y_0,\mY_0,Z_0,C_{t,q})$.
	By Skorokhod's representation theorem, there exists a probability space $({\Omega'},{\mathcal{F}'},{P'})$ and $\widetilde{\mathbb{S}}$-valued
	random variables $(\tilde{x}_n,\tilde{y}_n,\tilde{\mY}_n,\tilde{Z}_n)$ and $(\tilde{x},\tilde{y},\tilde{\mY},{\tilde{Z}})$ such that
	\begin{enumerate}
		\item[(i)] $(\tilde{x}_n,\tilde{y}_n,\tilde{\mY}_n,{\tilde{Z}_n})$ has the law $P_n$ for each $n\in\mathbb{N}$,
		\item[(ii)] $(\tilde{x}_n,\tilde{y}_n,\tilde{\mY}_n,{\tilde{Z}_n})\rightarrow (\tilde{x},\tilde{y},\tilde{\mY},{\tilde{Z}})$ in $\widetilde{\mathbb{S}}$ ${P'}$-a.s., and $(\tilde{x},\tilde{y},\tilde{\mY},{\tilde{Z}})$ has the law $P$.
	\end{enumerate}
In the following we verify (M1)-(M3) for $P$.

	 For (M1), using the  convergence in (i) above, we have
\begin{align*}
	&P(x(t)=x_0,y(t)=y_0,\mY(t)=\mY_0,{Z(t)=Z_0}, 0\leq t\leq s)
	\\&\qquad= P'(\tilde x(t)=x_0,\tilde y(t)=y_0,\tilde \mY(t)=\mY_0, \tilde Z(t)=Z_0, 0\leq t\leq s)=1.
\end{align*}
	By the condition on $G$ \eqref{eq:nl}, for every $T>0$ $P(\int_0^T\|G(x(r))\|_{L_2(\mR^m;L^2_\sigma)}^{2}<\infty)=1$. (M3) follows by similar arguments as in the proof of Theorem  \ref{convergence}. Then by \cite[Theorem 4.2.5]{LR15} we have
		$P(Z\in CL^2_\sigma)=1$. 

	For (M2), we write  $ \mY=( \mY_{ij})$, $y=(y_i)$ and $\tilde \mY=(\tilde \mY_{ij})$, $\tilde y=(\tilde y_i)$ and $\tilde \mY_n=(\tilde \mY_{n,ij})$, $\tilde y_n=(\tilde y_{n,i})$.
	Similarly as in the proof of Theorem \ref{convergence 1} we  obtain that $y$ is $(\widetilde \cB_t)_{t\geq s}$- $\mR^m$-valued Brownian motion.
	
	Then, we shall  prove
	\begin{equation}\label{eq:bm}P\Big(\mY(t)-\mY(s)=\int_s^ty_r\otimes dy_r\Big)=1.\end{equation}
	To this end, we need to verify
	\begin{itemize}
		\item The quadratic variation process of $\mY_{ij}$ is given by $\int_s^t y_i^2(r) dr $.
		\item The cross variation of $\mY_{ij}$ with  $y_k$ is given by $\int_s^ty_i(r)dr\delta_{j=k}$.
	\end{itemize}
	Let $t>r\geq s$ and $g$ be any bounded and real valued $\widetilde{\cB}_r$-measurable continuous functions on $\widetilde\Omega$. It holds
	\begin{align*}
	&E^{P}\Big[\Big(\mY_{ij}^2(t)-\mY_{ij}^2(r)-\int_r^t y_i^2(l)dl\Big)g(x,y,\mY)\Big]
	\\&=E^{P'}\Big[\Big(\tilde \mY_{ij}^2(t)-\tilde \mY_{ij}^2(r)-\int_r^t \tilde y_i^2(l)dl\Big)g(\tilde x,\tilde y,\tilde \mY)\Big]
	\\&=\lim_{n\to\infty}E^{P'}\Big[\Big(\tilde \mY_{n,ij}^2(t)-\tilde \mY_{n,ij}^2(r)-\int_{r}^t \tilde y_{n,i}^2(r)dr\Big)g(\tilde x_n,\tilde y_n,\tilde \mY_n)\Big]=0,
	\end{align*}
	and
	\begin{align*}
	&E^{P}\Big[\Big(\mY_{ij}(t) y_k(t)-\mY_{ij}(r)y_k(r)-\delta_{k=j}\int_r^t y_i(l)dl\Big)g(x,y,\mY)\Big]
	\\&=E^{P'}\Big[\Big(\tilde \mY_{ij}(t)\tilde y_k(t)-\tilde \mY_{ij}(r)\tilde y_k(r)-\delta_{k=j}\int_r^t \tilde y_i(l)dl\Big)g(\tilde x,\tilde y,\tilde \mY)\Big]
	\\&=\lim_{n\to\infty}E^{P'}\Big[\Big(\tilde \mY_{n,ij}(t)\tilde y_{n,k}(t)-\tilde \mY_{n,ij}(r)\tilde y_{n,k}(r)-\delta_{k=j}\int_{r}^t \tilde y_{n,i}(l)dl\Big)g(\tilde x_n,\tilde y_n,\tilde \mY_n)\Big]=0.
	\end{align*}
	Hence the quadratic variation process of $\mY(t)-\mY(s)-\int_s^ty_r\otimes dy_r$ is zero and \eqref{eq:bm} follows. Similarly, 
	$$P\Big(Z(t)-Z(s)-\int_s^t \Delta Z(r)dr=\int_s^t G(v_0+Z(r))dy_r\Big)=1.$$
	The rest follows by the same argument as in the proof of Theorem \ref{convergence 1}.
	\end{proof}

\color{black}

  \section{Intermittent jets}
  \label{s:B}

In this  part we recall the construction of intermittent jets from \cite[Section 7.4]{BV19} \rmb{and derive a new estimate in Lemma~\ref{l:83}.}
We point out that the construction is entirely deterministic, that is, none of the functions below depends on $\omega$.
Let us begin with  the following geometric lemma which can be found in  \cite[Lemma 6.6]{BV19}.

\begin{lem}\label{geometric}
Denote by $\overline{B_{1/2}}(\mathrm{Id})$ the closed ball of radius $1/2$ around the identity matrix $\mathrm{Id}$, in the space of $3\times 3$ symmetric matrices. There
exists $\Lambda\subset \mathbb{S}^2\cap \mathbb{Q}^3$ such that for each $\xi\in \Lambda$ there exists a  $C^\infty$-function $\gamma_\xi:\overline{B_{1/2}}(\mathrm{Id})\rightarrow\mathbb{R}$ such that
\begin{equation*}
R=\sum_{\xi\in\Lambda}\gamma_\xi^2(R)(\xi\otimes \xi)
\end{equation*}
for every symmetric matrix satisfying $|R-\mathrm{Id}|\leq 1/2$.
For $C_\Lambda=8|\Lambda|(1+8\pi^3)^{1/2}$, where $|\Lambda|$ is the cardinality of the set $\Lambda$, we define
the constant
\begin{equation*}
M=C_\Lambda\sup_{\xi\in \Lambda}(\|\gamma_\xi\|_{C^0}+\sum_{|j|\leq N}\|D^j\gamma_\xi\|_{C^0}).
\end{equation*}
 For each $\xi\in \Lambda$ let us define $A_\xi\in \mathbb{S}^2\cap \mathbb{Q}^3$ to be an orthogonal vector to $\xi$. Then for each $\xi\in\Lambda$ we have that $\{\xi, A_\xi, \xi\times A_\xi\}\subset \mathbb{S}^2\cap \mathbb{Q}^3$ form an orthonormal basis for $\mathbb{R}^3$.
We label by $n_*$ the smallest natural such that
\begin{equation*}\{n_*\xi, n_*A_\xi, n_*\xi\times A_\xi\}\subset \mathbb{Z}^3\end{equation*}
for every $\xi\in \Lambda$.
\end{lem}

\rmb{Let $\chi:\mathbb{R}^2\rightarrow\mathbb{R}$ be a smooth function with support in a ball of radius $1$. We define $\Phi:= -\Delta\chi$ and $\phi:=-\Delta\Phi =(-\Delta)^{2}\chi$ and we normalize so that
\begin{equation}\label{eq:phi}
\frac{1}{4\pi^2}\int_{\mathbb{R}^2}\phi^2(x_1,x_2)dx_1dx_2=1.
\end{equation}
This particular form of $\phi$ given through the function $\chi$ did not appear in \cite{BV19}. It is needed below in Lemma~\ref{l:83} which we apply in Section~\ref{s:8} with $\gamma=5+\delta$, i.e. $l=2$.}
By definition we know $\int_{\mathbb{R}^2}\phi dx=0$. Define $\psi:\mathbb{R}\rightarrow\mathbb{R}$ to be a smooth, mean zero function with support in the ball of radius $1$ satisfying
\begin{equation}\label{eq:psi}
\frac{1}{2\pi}\int_{\mathbb{R}}\psi^2(x_3)dx_3=1.
\end{equation}
For parameters $r_\perp, r_\|>0$ such that
\begin{equation*}r_\perp\ll r_\|\ll1,\end{equation*}
we define the rescaled cut-off functions
\begin{equation*}\phi_{r_\perp}(x_1,x_2)=\frac{1}{r_\perp}\phi\left(\frac{x_1}{r_\perp},\frac{x_2}{r_\perp}\right),\quad
\Phi_{r_\perp}(x_1,x_2)=\frac{1}{r_\perp}\Phi\left(\frac{x_1}{r_\perp},\frac{x_2}{r_\perp}\right),\quad \psi_{r_\|}(x_3)=\frac{1}{r_\|^{1/2}}\psi\left(\frac{x_3}{r_\|}\right).\end{equation*}
We periodize $\phi_{r_\perp}, \Phi_{r_\perp}$ and $\psi_{r_\|}$ so that they are viewed as periodic functions on $\mathbb{T}^2, \mathbb{T}^2$ and $\mathbb{T}$ respectively.

Consider a large real number $\lambda$ such that $\lambda r_\perp\in\mathbb{N}$, and a large time oscillation parameter $\mu>0$. For every $\xi\in \Lambda$ we introduce
\begin{equation*}\aligned
\psi_{(\xi)}(t,x)&:=\psi_{\xi,r_\perp,r_\|,\lambda,\mu}(t,x):=\psi_{r_{\|}}(n_*r_\perp\lambda(x\cdot \xi+\mu t))
\\ \Phi_{(\xi)}(x)&:=\Phi_{\xi,r_\perp,\lambda}(x):=\Phi_{r_{\perp}}(n_*r_\perp\lambda(x-\alpha_\xi)\cdot A_\xi, n_*r_\perp\lambda(x-\alpha_\xi)\cdot(\xi\times A_\xi))\\
\phi_{(\xi)}(x)&:=\phi_{\xi,r_\perp,\lambda}(x):=\phi_{r_{\perp}}(n_*r_\perp\lambda(x-\alpha_\xi)\cdot A_\xi, n_*r_\perp\lambda(x-\alpha_\xi)\cdot(\xi\times A_\xi)),
\endaligned\end{equation*}
 where $\alpha_\xi\in\mathbb{R}^3$ are shifts to ensure that $\{\Phi_{(\xi)}\}_{\xi\in\Lambda}$ have mutually disjoint support.

The intermittent jets $W_{(\xi)}:\mathbb{T}^3\times \mathbb{R}\rightarrow\mathbb{R}^3$ are defined as in \cite[Section 7.4]{BV19}.
\begin{equation}\label{intermittent}W_{(\xi)}(t,x):=W_{\xi,r_\perp,r_\|,\lambda,\mu}(t,x):=\xi\psi_{(\xi)}(t,x)\phi_{(\xi)}(x).\end{equation}
By the choice of $\alpha_\xi$ we have that
\begin{equation}\label{Wxi}
W_{(\xi)}\otimes W_{(\xi')}\equiv0, \textrm{ for } \xi\neq \xi'\in\Lambda,
\end{equation}
and by the normalizations \eqref{eq:phi} and \eqref{eq:psi} we obtain
$$
\strokedint_{\mathbb{T}^3}W_{(\xi)}(t,x)\otimes W_{(\xi)}(t,x)dx=\xi\otimes\xi.
$$
These facts combined with Lemma \ref{geometric} imply that
\begin{equation}\label{geometric equality}
\sum_{\xi\in\Lambda}\gamma_\xi^2(R)\strokedint_{\mathbb{T}^3}W_{(\xi)}(t,x)\otimes W_{(\xi)}(t,x)dx=R,
\end{equation}
for every symmetric matrix $R$ satisfying $|R-\textrm{Id}|\leq 1/2$. Since $W_{(\xi)}$ are not divergence free, we  introduce the corrector term
\begin{equation}\label{corrector}
W_{(\xi)}^{(c)}:=\frac{1}{n_*^2\lambda^2}\nabla \psi_{(\xi)}\times \textrm{curl}(\Phi_{(\xi)}\xi)
=\textrm{curl\,curl\,} V_{(\xi)}-W_{(\xi)}.
\end{equation}
with
\begin{equation*}
V_{(\xi)}(t,x):=\frac{1}{n_*^2\lambda^2}\xi\psi_{(\xi)}(t,x)\Phi_{(\xi)}(x).
\end{equation*}
Thus we have
\begin{equation*}
\div\left(W_{(\xi)}+W_{(\xi)}^{(c)}\right)\equiv0.
\end{equation*}

Next, we  recall the key   bounds from \cite[Section 7.4]{BV19}. For $N, M\geq0$ and $p\in [1,\infty]$ the following holds provided $r_{\|}^{-1}\ll r_{\perp}^{-1}\ll \lambda$
\begin{equation}\label{bounds}\aligned
\|\nabla^N\partial_t^M\psi_{(\xi)}\|_{C_tL^p}&\lesssim r^{1/p-1/2}_\|\left(\frac{r_\perp\lambda}{r_\|}\right)^N
\left(\frac{r_\perp\lambda \mu}{r_\|}\right)^M,\\
\|\nabla^N\phi_{(\xi)}\|_{L^p}+\|\nabla^N\Phi_{(\xi)}\|_{L^p}&\lesssim r^{2/p-1}_\perp\lambda^N,\\
\|\nabla^N\partial_t^MW_{(\xi)}\|_{C_tL^p}+\frac{r_\|}{r_\perp}\|\nabla^N\partial_t^MW_{(\xi)}^{(c)}\|_{C_tL^p}+\lambda^2\|\nabla^N\partial_t^MV_{(\xi)}
\|_{C_tL^p}&\lesssim r^{2/p-1}_\perp r^{1/p-1/2}_\|\lambda^N\left(\frac{r_\perp\lambda\mu}{r_\|}\right)^M,
\endaligned\end{equation}
where the implicit constants may depend on $p, N$ and $M$, but are independent of $\lambda, r_\perp, r_\|, \mu$.

Finally, we establish two additional estimates employed in Section~\ref{s:8}.

\begin{lem} \label{l:83}
Let $\alpha_{0}\in [0,1]$, $\gamma>0$, $\delta> 0$. Suppose that for $l=\frac{\gamma-1-\delta}2\in \mN$ and $\phi=(-\Delta)^l \chi$ for a smooth  function $\chi$ with support in a ball of radius $1$. Then it holds true
$$
\| \psi_{(\xi)} \|_{C^{\alpha_0}_T H^{\gamma + \delta}} \lesssim \mu^{\alpha_0} \left( \frac{r_{\perp} \lambda}{r_{\|}}
   \right)^{\alpha_0 + \gamma + \delta},
$$
$$
\|\phi_{(\xi)}\|_{H^{-\gamma}}\lesssim \lambda^{-\gamma}r_\perp^{-\delta}.
$$
\end{lem}

\begin{proof}
The first bound is a consequence of \eqref{bounds} and interpolation
$$ \| \psi_{(\xi)} \|_{H^{\gamma + \delta}} \lesssim \| \psi_{(\xi)} \|^{1 - (\gamma + \delta)
   / {\lceil \gamma + \delta \rceil}}_{L^2} \| \psi_{(\xi)} \|_{H^{{\lceil \gamma + \delta \rceil}}}^{(\gamma + \delta) /
   {\lceil \gamma + \delta \rceil}}\lesssim \left( \frac{r_{\perp} \lambda}{r_{\|}} \right)^{\gamma
   	+ \delta}, $$
   hence using interpolation again leads to
$$ \| \psi_{(\xi)} \|_{C^{\alpha_0}_T H^{\gamma + \delta}} \lesssim \| \psi_{(\xi)} \|^{1 -
   \alpha_0}_{C_T H^{\gamma + \delta}} \| \psi_{(\xi)} \|^{\alpha_0}_{C^1_T H^{\gamma
   + \delta}} 
 \lesssim \mu^{\alpha_0} \left( \frac{r_{\perp} \lambda}{r_{\|}}
   \right)^{\alpha_0 + \gamma + \delta} . $$

Let us now show the estimate for the  $H^{-\gamma}$-norm of $\phi_{(\xi)}$.
We view $\phi_{(\xi)}$ as periodic function on $\mR^3$ and have
    \begin{align*}
    \int_{\mathbb{R}^3}\phi_{(\xi)}(x)e^{-i k\cdot x}dx&=e^{-i k\cdot \alpha_\xi}\int_{\mathbb{R}^3}\phi_{r_\perp}(n_*\lambda r_\perp u_1,n_*\lambda r_\perp u_2)e^{-i Ak\cdot u}du\nonumber
    \\&=\delta_{0}((Ak)_3)e^{-i k\cdot \alpha_\xi}\int_{\mathbb{R}^2}\phi_{r_\perp}(n_*\lambda r_\perp u_1,n_*\lambda r_\perp u_2)e^{-i [(Ak)_1u_1+(Ak)_2u_2]}du_1du_2\nonumber
    \\&=\delta_{0}((Ak)_3)e^{-i k\cdot \alpha_\xi}(n_*\lambda r_\perp)^{-2}\int_{\mathbb{R}^2}\phi_{r_\perp}( u_1, u_2)e^{-i [\frac{(Ak)_1}{n_*\lambda r_\perp}u_1+\frac{(Ak)_2}{n_*\lambda r_\perp}u_2]}du_1du_2\nonumber
      \\&=\delta_{0}((Ak)_3)e^{-i k\cdot \alpha_\xi}\sum_{m\in 2\pi n_*\lambda r_\perp\mathbb{Z}^2}\hat{\phi}_{r_\perp}\left(\frac{(Ak)_1}{n_*\lambda r_\perp},\frac{(Ak)_2}{n_*\lambda r_\perp} \right)\delta_{0}(((Ak)_1,(Ak)_2)+m),
    \end{align*}
    where in the first equality we used $u=(u_1,u_2,u_3)=(x\cdot A_\xi, x\cdot(\xi\times A_\xi),x\cdot \xi):=Ax$ with $A$ being an orthonormal matrix and $\hat{\phi}_{r_\perp}(k)=\int_{\mathbb{T}^2}\phi_{r_\perp}(x)e^{-i k\cdot x}dx$. Then
    \begin{align*}(1-\Delta)^{-\gamma/2} \phi_{(\xi)}(x)=\sum_{m\in 2\pi n_*\lambda r_\perp\mathbb{Z}^2}e^{-i A^*(m,0)\cdot \alpha_\xi}(1+|m|^2)^{-\gamma/2}\hat{\phi}_{r_\perp}\left(\frac{m}{n_*\lambda r_\perp}\right)e^{iA^*(m,0)\cdot x}.
    \end{align*}
 Thus, we have
    \begin{align*}
    \|\phi_{(\xi)}\|_{H^{-\gamma}}^2&\lesssim \sum_{m\in 2\pi n_*\lambda r_\perp\mathbb{Z}^2}(1+|m|^2)^{-\gamma}\left|\hat{\phi}_{r_\perp}\left(\frac{m}{n_*\lambda r_\perp}\right)\right|^2
   \\ &=\sum_{m\in 2\pi n_*\lambda r_\perp\mathbb{Z}^2\setminus \{0\}}(1+|m|^2)^{-\gamma}\left|\hat{\phi}_{r_\perp}\left(\frac{m}{n_*\lambda r_\perp}\right)\right|^2
   \\&\lesssim(\lambda r_\perp)^{-2\gamma}\sum_{k\in 2\pi\mathbb{Z}^2\setminus \{0\}}|k|^{-2\gamma}\left|\hat{\phi}_{r_\perp}(k)\right|^2.
    \end{align*}
Here, in the equality we used the fact that $\phi_{r_\perp}$ has zero mean. Moreover, we have for $l=\frac{\gamma-1-\delta}2\in \mN$ and $\phi=(-\Delta)^l \chi$ for a smooth compactly supported function $\chi$ that
$$
\hat{\phi}_{r_\perp}(k)=r_\perp\int_{\mathbb{R}^2} \phi(x)e^{-i kr_\perp\cdot x} dx=r_\perp \int_{\mathbb{R}^2}  \chi(-\Delta)^le^{-ikr_\perp\cdot x}dx,
$$
which implies  for $l=\frac{\gamma-1-\delta}2$, $\delta>0$
\begin{align*}\|\phi_{(\xi)}\|_{H^{-\gamma}}^2&\lesssim(\lambda r_\perp)^{-2\gamma}r_\perp^2\sum_{k\in 2\pi\mathbb{Z}^2\setminus \{0\}}|k|^{-2\gamma}(\|\chi\|_{L^1}(|k|r_\perp)^{2l})^2
\\&\lesssim(\lambda)^{-2\gamma}(r_\perp)^{-2\delta}\sum_{k\in 2\pi\mathbb{Z}^2\setminus \{0\}}|k|^{-2-2\delta}\|\chi\|_{L^1}^2
\\&\lesssim(\lambda)^{-2\gamma}(r_\perp)^{-2\delta}.
    \end{align*}
  \end{proof}

\color{black}

\section{Uniqueness in law implies joint uniqueness in law}
  \label{s:C}

In this part we will extend the result of Cherny \cite{C03} to a general infinite dimensional setting. A generalization to a semigroup framework in Banach spaces was proved by Ondrej\'at in \cite{On04}. Let $U,U_{1},H$ and $H_1$ be  separable Hilbert spaces and suppose that the embedding $U\subset U_{1}$ is Hilbert--Schmidt and the embedding $H\subset H_1$ is continuous.
Consider the SPDE of the form
\begin{equation}
\label{spde}
dX=F(X)dt+G(X)dB,\qquad X(0)=x\in H,
\end{equation}
where   $F:H\rightarrow H_1$ is $\mathcal{B}(H)/\mathcal{B}(H_1)$ measurable and $B$ is a cylindrical Wiener process on the Hilbert space $U$ which is defined on  a stochastic basis $(\Omega,\mathcal{F},(\mathcal{F}_t)_{t\geq0},P)$. In other words, $B$ can be viewed as a continuous process taking values in $U_1$ and we assume that for $x\in H$, $G(x)$ is an Hilbert--Schmidt operator from $U$ to $H$. Solutions to \eqref{spde} are then understood in the following sense.

\begin{defn}
A pair $(X,B)$ is a solution to \eqref{spde} provided there exists  a stochastic basis $(\Omega,\mathcal{F},(\mathcal{F}_t)_{t\geq0},P)$ such that\\
\no(H1) $B$ is a cylindrical $(\mathcal{F}_t)_{t\geq0}$-Wiener process on  $U$;\\
\no(H2) $X$ is an $(\mathcal{F}_t)_{t\geq0}$-adapted process in $C([0,\infty);H_{1})$ $P$-a.s.;\\
\no(H3) $F(X)\in L^{1}_{\rm{loc}}([0,\infty);H_{1})$ and $G(X)\in L^{2}_{\rm{loc}}([0,\infty);L_{2}(U,H))$ $P$-a.s.;\\
\no(H4) $P$-a.s. it holds for all $t\in[0,\infty)$
$$
X_{t}=x+\int_{0}^{t}F(X_{s})ds+\int_{0}^{t}G(X_{s})dB_{s}.
$$
\end{defn}

Let us now recall the definition of uniqueness in law as well as joint uniqueness in law.

\begin{defn}
We say that uniqueness in law holds for \eqref{spde} if for any two solutions $(X,B)$ and $(\tilde X,\tilde
B)$ starting from the same initial distribution, one has $\mathrm{Law}(X)=\mathrm{Law}(\tilde X)$.
We say that joint uniqueness in law holds for \eqref{spde} if for any two solutions $(X,B)$ and $(\tilde X,\tilde
B)$ starting from the same initial distribution, one has $\mathrm{Law}(X,B)=\mathrm{Law}(\tilde X,\tilde B)$.
\end{defn}

Clearly, joint uniqueness in law implies uniqueness in law. The following result shows that the two notions are in fact equivalent for  SPDEs of the form \eqref{spde}.

\begin{thm} \label{cherny}
Suppose that uniqueness in law holds for \eqref{spde}. Then joint uniqueness in law holds for \eqref{spde}.
\end{thm}

Set $E=L_2(U,H)$. Since $E$ is separable, it follows that $C([0,t],E)$ is dense in $L^2([0,t],E)$. By the same argument as  Lemma 3.2 in \cite{C03}, we can prove the following result.

\begin{lem}\label{con}
Let $t>0$ and $f\in L^2([0,t],E)$. For $k\in \mathbb{N}$, set
$$
f^{(k)}(s)=\begin{cases}
0, & \text{if }s\in [0,\frac{t}{k}],\\
\frac{k}{t}\int_{(i-1)t/k}^{it/k}f(r)dr,& \text{if } s\in (\frac{it}{k},\frac{(i+1)t}{k}], \qquad (i=1,\dots,k-1).
\end{cases}
$$
Then $f^{(k)}\rightarrow f$ in $L^2([0,t],E)$.
\end{lem}

By Lemma \ref{con} and the same argument as in Lemma 3.3 in \cite{C03}, we obtain  the following.

\begin{lem}\label{condition}
Let $(X,B)$ be a  solution to \eqref{spde} defined on a stochastic basis $(\Omega,\mathcal{F},(\mathcal{F}_t)_{t\geq0},P)$. Let $(Q_\omega)_{\omega\in\Omega}$ be a conditional probability distribution of $(X,B)$ given $\mathcal{F}_0$\footnote{Here, we consider $(X,B)$ as a $C([0,\infty);H_{1}\times U_{1})$-valued process.}. Let $Y$ be the coordinate process with values in $H_{1}$ and let $Z$ be the coordinate process with values in $U_{1}$. Let $(\mathcal{H}_t)_{t\geq0}$ be the canonical filtration on $C([0,\infty), H_{1}\times U_{1})$ and denote $\mathcal{H}=\bigvee_{{t\geq0}} \mathcal{H}_t$. Then for $P$-a.e. $\omega\in \Omega$ the pair $(Y,Z)$ is a solution to \eqref{spde}  on the stochastic basis $(C([0,\infty);H_{1}\times U_{1}),\mathcal{H},(\mathcal{H}_t)_{t\geq0},Q_\omega)$.
\end{lem}

\begin{proof}[Proof of Theorem \ref{cherny}]
Let $(X,B)$ be a solution to \eqref{spde} on  a stochastic basis $(\Omega,\mathcal{F},(\mathcal{F}_t)_{t\geq0},P)$. Let $\{\beta^k\}_{k\in\mathbb{N}}$ and $\{\bar{\beta}^k\}_{k\in\mathbb{N}}$ be two families of  independent real-valued Wiener processes defined on another stochastic basis $(\Omega',\mathcal{F}',(\mathcal{F}'_t)_{t\geq0},P')$ and set
$$(\tilde{\Omega},\tilde{\mathcal{F}},(\tilde{\mathcal{F}_t})_{t\geq0},\tilde{P})=\big(\Omega\times\Omega',\mathcal{F}\otimes\mathcal{F}',(\mathcal{F}_t\otimes \mathcal{F}'_t)_{t\geq0},P\otimes P'\big).$$
All the processes $X,B,\beta^k,\bar{\beta}^k$, $k\in\mathbb{N}$, can be defined on $\tilde{\Omega}$ in an obvious way. Assume that the cylindrical Wiener process $B$ admits the decomposition $B=\sum_{k=1}^{\infty}\alpha^kl_k$, where $\{\alpha^k\}_{k\in\mathbb{N}}$ is a family of independent real-valued Wiener processes and $\{l_k\}_{k\in\mathbb{N}}$ is an orthonormal basis in  $U$. Let $\varphi(x)$ be the orthogonal projection from $U$ to $(\ker G(x))^\bot$ and $\psi(x)$ be the orthogonal projection from $U$ to $\ker G(x)$. Then set
$$\varphi_s:=\varphi(X_s),\quad \psi_s:=\psi(X_s),$$
$$V_t=\sum_{k=1}^{\infty}\left[\int_0^t\varphi_sd\alpha_s^k\,l_k+\int_0^t\psi_sd\beta^k_s\,l_k\right],\qquad\bar{V}_t=\sum_{k=1}^{\infty}\left[\int_0^t\varphi_sd\bar{\beta}^k_s\,l_k+\int_0^t\psi_sd\alpha^k_s\,l_k\right].$$
In the following, $\langle\!\!\langle \cdot,\cdot\rangle\!\!\rangle_{t}$ denotes the cross-variation process at time $t$. We obtain
$$
\aligned\langle\!\!\langle \langle V,l_i\rangle_U,\langle V,l_j\rangle_U\rangle\!\!\rangle_t
&=\sum_{k=1}^{\infty}\left[\int_0^t\langle\varphi_sl_k,l_i\rangle_U \langle\varphi_sl_k,l_j\rangle_U ds+\int_0^t\langle\psi_sl_k,l_i\rangle_U \langle\psi_sl_k,l_j\rangle_U ds\right]
\\&=\int_0^t\left[\langle \varphi_sl_i,\varphi_sl_j\rangle_U+\langle \psi_sl_i,\psi_sl_j\rangle_U\right]ds
=\int_0^t\langle (\varphi_s+\psi_s)l_i,(\varphi_s+\psi_s)l_j\rangle_Uds
\\&=\int_0^t\langle l_i,l_j\rangle_Uds=\delta_{ij}t.
\endaligned
$$
Similarly, we obtain
$$
\langle\!\!\langle \langle V,l_i\rangle_U,\langle \bar{V},l_j\rangle_U\rangle\!\!\rangle_t=0,\qquad\langle\!\!\langle \langle {\bar V},l_i\rangle_U,\langle {\bar V},l_j\rangle_U\rangle\!\!\rangle_t=\delta_{ij}t.
$$
As a consequence,  under $\tilde P$ the process $(V,\bar{V})$ is an $(\tilde{\mathcal{F}_t})_{t\geq0}$-cylindrical Wiener process on $U\times U$.
Moreover, for any $t\geq0$, we have
$$\int_0^tG(X_s)dB_s=\int_0^tG(X_s)\varphi_sdB_s=\int_0^tG(X_s)dV_s.$$
Hence $(X,V)$ is a solution to \eqref{spde} on $(\tilde{\Omega},\tilde{\mathcal{F}},(\tilde{\mathcal{F}_t})_{t\geq0},\tilde{P})$.

Consider now the filtration
$$\mathcal{G}_s=\tilde{\mathcal{F}}_s\vee \sigma\big(\bar{V}_t;t\geq0\big)=\tilde{\mathcal{F}}_s\vee \sigma\big(\bar{V}_t-\bar{V}_s;t\geq s\big),\quad s\geq0.$$
Since $\tilde{\mathcal{F}}_s$ and $\sigma(V_t-V_s;t\geq s)\vee \sigma(\bar{V}_t-\bar{V}_s;t\geq s)$ are independent, the process $V$ is a cylindrical $(\mathcal{G}_t)_{t\geq0}$-Wiener process on $U$ under $\tilde{P}$. Thus $(X,V)$ is a solution to \eqref{spde} on $(\tilde{\Omega},\tilde{\mathcal{F}},(\tilde{\mathcal{G}_t})_{t\geq0},\tilde{P})$.

Let $(Q_{\tilde{\omega}})_{\tilde{\omega}\in \tilde{\Omega}}$ be a conditional probability distribution of $(X,V)$ given $\mathcal{G}_0$. By Lemma~\ref{condition},  for $\tilde{P}$-a.e. $\tilde\omega\in\tilde\Omega$, the pair $(Y,Z)$ is a solution to \eqref{spde} on $(C([0,\infty);H\times U),\mathcal{H},(\mathcal{H}_t)_{t\geq0},Q_{\tilde{\omega}})$. As the uniqueness in law holds for \eqref{spde}, the probability law induced by $Y$ on each of these stochastic bases, i.e. $Q_{\tilde\omega}\circ Y^{-1}$,  is the same for $\tilde{P}$-a.e. $\tilde{\omega}\in\tilde\Omega$. Since this is the conditional probability distribution of $X$ given $\mathcal{G}_{0}$, it follows that  the process $X$ is independent of $\mathcal{G}_0$. In particular, we deduce that $X$ and $\bar{V}$ are independent. Let $\chi(x)$ be the pseudo-inverse  of $G(x)$ (see e.g. \cite[Appendix C]{LR15} for more details), then  $\chi(x)G(x)=\varphi(x)$. Set $\chi_s:=\chi(X_s)$. Thus,
$$
\int_0^t\varphi_sdB_s=\int_0^t\chi_sG(X_s)dB_s=\int_0^t\chi_sdM_s,
$$
where
$$
M_t=\int_0^tG(X_s)dB_s=X_t-x-\int_0^tF(X_s)ds.
$$
Accordingly, we obtain
$$
B_t=\int_0^t\varphi_sdB_s+\int_0^t\psi_sdB_s=\int_0^t\chi_sdM_s+\int_0^t\psi_sd\bar{V}_s.
$$
The process $M$ is a measurable functional of $X$ while $\bar{V}$ is independent of $X$. Thus the distribution  $\mathrm{Law}(X,B)$ is unique.
\end{proof}

\section{Analysis of  rough partial differential equations}
\label{s:D}

In this section, we employ the theory of rough paths to derive estimates for the following rough partial differential equation. Assume that $v\in C^{1}_{t,x}$ and $z$ solves the system
\begin{align}\label{eq:new}
\begin{aligned}
 dz&=\Delta z dt +G(v+z) dB,\\
\textrm{div} z&=0,\\
z(0)&=z_0,
\end{aligned}
\end{align}
with $\div z_0=0$.
Then we have
$$
z(t)=P_{t}z_{0}+\int_0^tP_{t-s} G(v+z)dB_s,
$$
where
 $P_t=e^{t\Delta}$ is the heat semigroup.
The nonlinearity $G$ in \eqref{eq:new} is defined through
\begin{equation}\label{G}
G(u)=\Big(g_{ij}\big(\cdot, \langle u,\varphi^{ij}_1\rangle,..., \langle u,\varphi^{ij}_{k_{ij}}\rangle\big)\Big)
\end{equation}
with $g_{ij}\in C_b^3(\mT^3\times \mR^{k_{ij}}),$ $\varphi^{ij}_{\ell}\in {C}^{\infty}(\mT^{3})$, $i=1,\dots,3,$ $ j=1,\dots,m$, $\ell=1,\dots,k_{ij}$, i.e. the functions $g_{i j}$ as well as their derivatives up to order $3$ are bounded and $g_{\cdot j}$ is divergence free with respect to the  spatial variable in $\mT^3$.

The driving process $B$ is an $m$-dimensional Brownian motion and we view it as a rough path. To this end, fix $\alpha\in(\frac13,\frac12)$. We use $\rho_\alpha(B)$ to denote its $\alpha$-H\"older rough path seminorm which is given by
$$\rho_\alpha(B)=\sup_{0\leq s<t\leq T}\frac{|B_{t}-B_{s}|}{|t-s|^\alpha}+\sup_{0\leq s<t\leq T}\frac{|\int_s^t(B_{r}-B_{s})\otimes dB_r|}{|t-s|^{2\alpha}}.$$
The first component of the rough path is denoted by $B_{s,t}:=B_{t}-B_{s}$ and we understand the iterated integral $\mathbb{B}_{s,t}:=\int_s^t(B_r-B_{s})\otimes dB_r$ in the It\^{o} sense. However, the results of this section apply mutatis mutandis to other rough paths lifts of the Brownian motion  as well as general rough paths.

Let $C^\beta$ denote  the  closure of smooth functions from  with respect to the usual H\"{o}lder norm.
	We also use the H\"older-Besov space $\mathcal{C}^\beta, \beta\in \mR,$  defined by the closure of  smooth functions with respect to the $B^{\beta}_{\infty,\infty}$-norm
	$$
	\|f\|_{\cC^\beta}:=\|f\|_{B^{\beta}_{\infty,\infty}}=\sup_{j\in\mN_{0}\cup\{-1\}} 2^{\beta j}\|\Delta_j f\|_{L^\infty}
	$$
	with $\Delta_j,$ $j\in\mN_{0}\cup\{-1\}$, being the usual Littlewood-Paley blocks.  For any $0<\beta\notin\mN$  it is well known that
	(see  \cite[page 99]{BCD11})
	$
	\|f\|_{C^\beta}\asymp \|f\|_{\cC^\beta}.
	$
 For a path $h$ defined on $[0,T]$, we denote its increment $h_{t}-h_{s}$ by $h_{s,t}$.

 We also recall the following smoothing effect  from heat semigroup (see e.g. \cite[Lemma~2.8]{ZZZ20}), which is used in the following proof.
\begin{lem}\label{lem:2.8}
Let  $T>0$.
\begin{enumerate}[(i)]
	\item For any $\theta>0$ and $\alpha\in\mR$, it holds that for $t\in [0,T]$
	\begin{align}\label{E1}
	\|P_t f\|_{\cC^{\theta+\alpha}}\lesssim t^{-\theta/2}\|f\|_{\cC^\alpha},\qquad \|P_t f\|_{H^{\theta+\alpha}}\lesssim t^{-\theta/2}\|f\|_{H^\alpha},
	\end{align}
		with the proportional constant independent of $f$.
	\item For any $0<\theta<2$ and $t\in [0,T]$, it holds that
	\begin{align}\label{E2}
	\|P_t f-f\|_{L^\infty}\lesssim t^{\theta/2}\|f\|_{\cC^{\theta}},\qquad 	\|P_t f-f\|_{L^2}\lesssim t^{\theta/2}\|f\|_{H^{\theta}},
	\end{align}
	with the proportional constant independent of $f$.
\end{enumerate}
\end{lem}

 Now, we introduce the definition of controlled rough path adapted to our purposes (see also \cite{Gub04}).

\begin{defn}\label{controlled}
Let $\gamma\in\R$.  We call a pair $(y,y')$ a controlled rough path in the $\cC^{\gamma}$-scale provided
	$(y,y')\in C_T\cC^{\gamma}\times (C_T\cC^{\gamma-2\alpha}\cap C^\alpha_TL^\infty)$ and the remainder
\begin{equation}\label{eq:R}
	(s,t)\mapsto R^y_{s,t}:=y_{s,t}-y'_sB_{s,t}
	\end{equation}
belongs to $C_{2,T}^{2\alpha}L^{\infty}$, the space of $2$-index maps on $[0,T]^{2}$ with values in $L^{\infty}$ so that
$$\|R^y\|_{2\alpha,L^\infty}=\sup_{0\leq s<t\leq T}\frac{\|R^y_{s,t}\|_{L^\infty}}{|t-s|^{2\alpha}}<\infty.$$
\end{defn}

The space of controlled rough paths in the $\cC^{{\gamma}}$-scale is denoted by $D^{2\alpha}_{B,\gamma}$ and endowed with the norm
$$
\|y,y'\|_{B,2\alpha,\gamma}=\|y\|_{C_T\cC^{\gamma}}+\|y'\|_{C_T\cC^{\gamma-2\alpha}}+\|y'\|_{C^\alpha_TL^\infty}+\|R^y\|_{2\alpha,L^\infty}.$$

We also present the corresponding definition with ${\cC^{\gamma}}$ replaced by $H^{{\gamma}}$.

\begin{defn}\label{controlled1}
Let $\gamma\in\R$. We call a pair $(y,y')$ a controlled rough path in the $H^{\gamma}$-scale provided
	$(y,y')\in C_TH^{\gamma}\times (C_TH^{\gamma-2\alpha}\cap C^\alpha_TL^2)$ and the remainder
\begin{equation}\label{eq:R1}
	(s,t)\mapsto R^y_{s,t}:=y_{s,t}-y'_sB_{s,t}
	\end{equation}
	belongs to $C_{2,T}^{2\alpha}L^2$, the space of $2$-index maps on $[0,T]^{2}$ with values in $L^{2}$ so that
	$$\|R^y\|_{2\alpha,L^2}=\sup_{0\leq s<t\leq T}\frac{\|R^y_{s,t}\|_{L^2}}{|t-s|^{2\alpha}}<\infty.$$
\end{defn}

The space of controlled rough paths in the $H^{{\gamma}}$-scale is denoted by $\bar{D}^{2\alpha}_{B,\gamma}$ and endowed with the norm
$$\|y,y'\|_{\bar{B},2\alpha,\gamma}=\|y\|_{C_TH^{\gamma}}+\|y'\|_{C_TH^{\gamma-2\alpha}}+\|y'\|_{C^\alpha_TL^2}+\|R^y\|_{2\alpha,L^2}.$$

The following integration lemma is a version of \cite[Theorem 4.5]{GHN19} adapted to our setting.

\begin{lem}\label{l:int}
Let  $\sigma\in[0,\alpha)$ and $(y,y')\in D^{2\alpha}_{B,4\alpha-2\sigma}$. Then the integral
$$ \int_0^t P (t - r) y_r d B_r := \lim_{| \pi | \rightarrow 0} \sum_{[s,
   r] \in \pi} P (t - s) (y_s B_{s, r} + y'_s \mathbb{B}_{s, r})
   $$
exists as an element of $\cC^{-2\kappa}$ for $\kappa >1-3\alpha+\sigma$ where the limit is taken over
partitions $\pi$ of $[0, t]$ with vanishing mesh size. Moreover, for every  $0\leq\theta<1$ it holds
\begin{align*}
\begin{aligned}
&\left\|\int_{s}^{t}P(t-r)y_{r}dB_{r}-P(t-s)y_{s}B_{s,t}-P(t-s)y'_{s}\mathbb{B}_{s,t}\right\|_{{\cC}^{4\alpha-2\theta}}\\
&\qquad\lesssim \|y,y'\|_{B,2\alpha,4\alpha-2\sigma}
|t-s|^{(\alpha-\sigma+\theta)\wedge (3\alpha)}\rho_\alpha(B).
\end{aligned}\end{align*}
Here the implicit constant is independent of $y, \rho_\alpha(B)$.
\end{lem}

\begin{proof}
The proof follows the ideas of  the usual sewing lemma which has already appeared in many variants, see e.g. \cite[Lemma 4.2]{FH14}, \cite[Theorem 2.4]{GH19}. The key computation is the following.

 Let
$$
 \xi_{s, t} := y_s B_{s, t} + y'_s \mathbb{B}_{s, t}=:\xi_{s,t}^1+ \xi_{s,t}^2,
 $$
which gives
$$ \delta \xi_{s, u, t} := \xi_{s,t}-\xi_{s,u} -\xi_{u,t} = -R^y_{s, u} B_{u, t} - y'_{s,
   u} \mathbb{B}_{u, t}=:h_{s,u,t}^1+ h_{s,u,t}^2,
   $$
where the first equality is the definition of increment  of a two-index map $\xi$.
Consider dyadic partitions $\pi_k=\{s=t_0<t_1<...<t_{2^k}=t\},$ with $t_i=s+2^{-k}i(t-s)$,  let
$$
I_{k}:=\sum_{[u,v]\in\pi_{k}}P_{t-u}\xi_{u,v},
$$
and denote $m=(u+v)/2$.
Then we have
\begin{align*}
I_k-I_{k+1}&=\sum_{[u,v]\in \pi_k}P_{t-u}\delta \xi_{u,m, v}+P_{t-m}(P_{m-u}-I)\xi_{m,v}
\\&=\sum_{[u,v]\in \pi_k}P_{t-u}h^1_{u,m, v}+\sum_{[u,v]\in \pi_k}P_{t-u}h^2_{u,m, v}
\\&\quad+\sum_{[u,v]\in \pi_k}P_{t-m}(P_{m-u}-I)\xi^1_{m,v}+\sum_{[u,v]\in \pi_k}P_{t-m}(P_{m-u}-I)\xi^2_{m,v}
\\&=\sum_{i=1}^4J_i.
\end{align*}
We have by \eqref{E1} and \eqref{E2} for $2\alpha-1<\beta<3\alpha-\sigma-1$ and $\beta\leq 2\alpha-\theta$
$$\begin{aligned}
\|J_4\|_{{\cC}^{4\alpha-2\theta}}&\lesssim \sum_{[u,v]\in \pi_k} (t-m)^{\beta-2\alpha+\theta}\|(P_{m-u}-I)\xi^2_{m,v}\|_{C^{2\beta}}
\\&\lesssim \sum_{[u,v]\in \pi_k} (t-m)^{\beta-2\alpha+\theta}(m-u)^{\alpha-\sigma-\beta}\|y'\|_{C_T\cC^{2\alpha-2\sigma}}(v-m)^{2\alpha}\rho_\alpha(B)
\\&\lesssim \|y'\|_{C_T\cC^{2\alpha-2\sigma}}\rho_\alpha(B)2^{-k(3\alpha-\beta-1-\sigma)}|t-s|^{3\alpha-\beta-1-\sigma}\sum_{[u,v]\in \pi_k} (t-m)^{\beta-2\alpha+\theta}(m-u)
\\&\lesssim \|y'\|_{C_T\cC^{2\alpha-2\sigma}}\rho_\alpha(B)2^{-k(3\alpha-\beta-1-\sigma)}|t-s|^{\alpha-\sigma+\theta},
\end{aligned}
$$
where in the last inequality above, in view of  the condition $\beta-2\alpha+\theta> -1$, we estimated the Riemann sum by the corresponding integral (using convexity of the integrand) and integrated.
Similarly we have for $2\alpha-1<\beta<3\alpha-\sigma-1$ and $\beta\leq 2\alpha-\theta$
$$\begin{aligned}
\|J_3\|_{{\cC}^{4\alpha-2\theta}}&\lesssim \sum_{[u,v]\in \pi_k} (t-m)^{\beta-2\alpha+\theta}(m-u)^{2\alpha-\sigma-\beta}\|y\|_{C_T\cC^{4\alpha-2\sigma}}\rho_\alpha(B)(v-m)^{\alpha}
\\&\lesssim \|y\|_{C_T\cC^{4\alpha-2\sigma}}\rho_\alpha(B)2^{-k(3\alpha-\beta-1-\sigma)}|t-s|^{3\alpha-\beta-1-\sigma}\sum_{[u,v]\in \pi_k} (t-m)^{\beta-2\alpha+\theta}(m-u)
\\&\lesssim \|y\|_{C_T\cC^{4\alpha-2\sigma}}\rho_\alpha(B)2^{-k(3\alpha-\beta-1-\sigma)}|t-s|^{\alpha-\sigma+\theta}.
\end{aligned}
$$
Moreover, we have  by \eqref{E1} and \eqref{E2}
$$\begin{aligned}
\|J_1\|_{{\cC}^{4\alpha-2\theta}}&\lesssim \sum_{[u,v]\in \pi_k} (t-{u})^{{(-2\alpha+\theta)\wedge 0}}\|h^1_{u,m,v}\|_{L^\infty}
\\&\lesssim \sum_{[u,v]\in \pi_k} (t-u)^{(-2\alpha+\theta)\wedge 0}(m-u)^{2\alpha}\|R^y\|_{2\alpha,L^\infty}(v-m)^{\alpha}\rho_\alpha(B)
\\&\lesssim \|R^y\|_{2\alpha,L^\infty}2^{-k(3\alpha-1)}|t-s|^{3\alpha-1}\sum_{[u,v]\in \pi_k} (t-m)^{(-2\alpha+\theta)\wedge 0}(m-u)\rho_\alpha(B)
\\&\lesssim \|R^y\|_{2\alpha,L^\infty}2^{-k(3\alpha-1)}|t-s|^{{(\alpha+\theta)\wedge (3\alpha)}}\rho_\alpha(B).
\end{aligned}
$$
Similarly, we have
$$\begin{aligned}\|J_2\|_{{\cC}^{4\alpha-2\theta}}&\lesssim \sum_{[u,v]\in \pi_k} (t-{u})^{(-2\alpha+\theta)\wedge 0}\|h^2_{u,m,v}\|_{L^\infty}
\\&\lesssim \sum_{[u,v]\in \pi_k} (t-u)^{(-2\alpha+\theta)\wedge 0}(m-u)^{\alpha}\|y'\|_{C^\alpha_TL^\infty}\rho_\alpha(B)(v-m)^{2\alpha}
\\&\lesssim \|y'\|_{C^\alpha_TL^\infty}\rho_\alpha(B)2^{-k(3\alpha-1)}|t-s|^{(\alpha+\theta)\wedge (3\alpha)}.\end{aligned}
$$

Thus the result then follows by summing over $k$ and taking limit. In particular the lower bound for $\kappa$ from the statement of the lemma is coming from the requirement that $\alpha-\sigma+\theta>1$.
\end{proof}

As the corresponding semigroup estimates remain the same in the $H^{{\gamma}}$-scale, we obtain also the following result.

\begin{lem}\label{l:int1}
Let $\sigma\in[0,\alpha)$ and  $(y,y')\in \bar{D}^{2\alpha}_{B,4\alpha-2\sigma}$. Then the integral
$$ \int_0^t P (t - r) y_r d B_r := \lim_{| \pi | \rightarrow 0} \sum_{[s,
   r] \in \pi} P (t - s) (y_s B_{s, r} + y'_s \mathbb{B}_{s, r})
   $$
exists as an element of  $H^{-2\kappa}$ for $\kappa>1-3\alpha-\sigma$,  where the limit is taken over
partitions $\pi$ of $[0, t]$ with vanishing mesh size. Moreover, for every $0\leq\theta<1$ it holds
\begin{equation*}
\begin{aligned}
&\left\|\int_{s}^{t}P(t-r)y_{r}dB_{r}-P(t-s)y_{s}B_{s,t}-P(t-s)y'_{s}\mathbb{B}_{s,t}\right\|_{H^{4\alpha-2\theta}}\\
&\qquad\lesssim \|y,y'\|_{\bar{B},2\alpha,4\alpha-2\sigma}
|t-s|^{(\alpha-\sigma+\theta)\wedge (3\alpha)}\rho_\alpha(B).
\end{aligned}
\end{equation*}
Here the implicit constant is independent of $y, \rho_\alpha(B)$.
\end{lem}

By a similar argument as \cite[Lemma 3.5]{HN21} we obtain the following result:

\begin{lem}\label{lem:G2}
Let $T\in(0,1]$. Let $\sigma\in[0,\alpha)$ and $(y,y')\in D^{2\alpha}_{B,4\alpha-2\sigma}.$  Then
	$$(z,z')=\left(\int_0^\cdot P(\cdot-s)y_sdB_s,y\right)\in D^{2\alpha}_{B,4\alpha}$$
	and
	$$\|z,z'\|_{B,2\alpha,4\alpha}\lesssim (1+\rho_\alpha(B))(\|y_0\|_{\cC^{{4\alpha-2\sigma}}}+\|y'_0\|_{\cC^{{2\alpha-2\sigma}}}+T^{\alpha(\alpha-\sigma)/(2\alpha-\sigma)}\|y,y'\|_{B,2\alpha,4\alpha-2\sigma}).
	$$
Here the implicit constant is independent of $y, \rho_\alpha(B)$.	
\end{lem}

\begin{proof}
By \eqref{eq:R} we first have
\begin{align*}
\|z'\|_{C^{\alpha}_TL^\infty}&=\|y\|_{C^{\alpha}_TL^\infty}\lesssim (\|y'\|_{C_TL^\infty}+\|y_0\|_{L^\infty})\rho_\alpha(B)+\|R^y\|_{\alpha,L^\infty}
\\&\lesssim(1+\rho_\alpha(B))( \|y'_0\|_{L^\infty}+\|y_0\|_{L^\infty}+T^\alpha\|y,y'\|_{B,2\alpha,4\alpha-2\sigma})
\end{align*}
The desired bound for the Gubinelli derivative $z'=y$ in $C_{T}\cC^{2\alpha}$ follows from
$$
\|z'\|_{C^{\alpha(\alpha-\sigma)/(2\alpha-\sigma)}_T\cC^{2\alpha}}\lesssim \|y\|^{\alpha/(2\alpha-\sigma)}_{C_T\cC^{4\alpha-2\sigma}}\|y\|^{(\alpha-\sigma)/(2\alpha-\sigma)}_{C^{\alpha}_TL^\infty}.
$$
In order to bound $z$ in $C_T\cC^{4\alpha}$, we write
\begin{align*}
\begin{aligned}
z_{t}=\bigg(\int_{0}^{t}P(t-s)y_{s}dB_{s}-P(t)y_{0}B_{0,t}-P(t)y'_{0}\mathbb{B}_{0,t}\bigg)+P(t)y_{0}B_{0,t}+P(t)y'_{0}\mathbb{B}_{0,t}.
\end{aligned}
\end{align*}
We  apply Lemma~\ref{l:int} with $\theta=0$ to control the first term {for $0\leq t\leq T$}
$$
\bigg\|\int_{0}^{t}P(t-s)y_{s}dB_{s}-P(t)y_{0}B_{0,t}-P(t)y'_{0}\mathbb{B}_{0,t}\bigg\|_{\cC^{4\alpha}}\lesssim  T^{\alpha-\sigma}\| y, y' \|_{B, 2 \alpha, 4 \alpha-2\sigma}\rho_\alpha(B),
$$
and for the remaining two we estimate as follows by \eqref{E1}
$$
\|P_{t}y'_{0}\mathbb{B}_{0,t}\|_{{\cC}^{4\alpha}}\lesssim t^{-\alpha-\sigma+2\alpha}\|y'_{0}\|_{\cC^{2\alpha-2\sigma}}\rho_\alpha(B)\lesssim \|y'_{0}\|_{\cC^{2\alpha-2\sigma}}\rho_\alpha(B),
$$
$$
\|P_{t}y_{0}B_{0,t}\|_{{\cC}^{4\alpha}}\lesssim t^{-\sigma+\alpha}\|y_{0}\|_{\cC^{4\alpha-2\sigma}}\rho_\alpha(B)\lesssim \|y_{0}\|_{\cC^{4\alpha-2\sigma}}\rho_\alpha(B).
$$

 It remains to control the $2\alpha$-H\"older norm of $R^{z}$ in $L^{\infty}$. It holds
$$
R^{z}_{s,t}=\left(\int_s^t P (t - r) y_r d B_r - P (t - s) y_s B_{s, t} - P (t - s) y'_s
\mathbb{B}_{s, t}\right)
$$
$$ + (P (t - s) - \mathrm{Id}) y_s B_{s, t} + (P (t - s) - \mathrm{Id}) \int_0^s P
   (s - r) y_r d B_r + P (t - s) y'_s \mathbb{B}_{s, t} = I_1 + \cdots + I_4.
   $$
Applying Lemma \ref{l:int} with $\theta=2\alpha-\kappa$, $\kappa>0$ small enough  we obtain {for $0\leq t\leq T$}
$$ \| I_1 \|_{L^\infty}\lesssim \| I_1 \|_{\cC^{2\kappa}}\lesssim
   \| y, y' \|_{B, 2 \alpha, 4 \alpha-2\sigma} | t - s |^{3\alpha-\sigma-\kappa}\rho_\alpha(B)\lesssim
   T^{\alpha-\sigma-\kappa}\| y, y' \|_{B, 2 \alpha, 4 \alpha-2\sigma} | t - s |^{2\alpha}\rho_\alpha(B),
   $$
   whereas the remaining terms are estimated as follows by \eqref{E1} and \eqref{E2}
$$ \| I_2 \|_{L^\infty}\lesssim | t - s |^{3 \alpha-\sigma} \| y_s \|_{\mathcal C^{4
   \alpha-2\sigma}}\rho_\alpha(B)\lesssim
   T^{\alpha-\sigma}\| y, y' \|_{B, 2 \alpha, 4 \alpha-2\sigma} | t - s |^{2\alpha}\rho_\alpha(B),
   $$

$$ \| I_3 \|_{L^\infty} \lesssim | t - s |^{2 \alpha} \left\| \int_0^s P (s
   - r) y_r d B_r \right\|_{\mathcal C^{4 \alpha}} =\|z_{s}\|_{\mathcal{C}^{4\alpha}}|t-s|^{2\alpha},
   $$
   which combined with the above estimate for $z\in C_{T}C^{4\alpha}$ yields the desired bound for $I_{3}$, and finally
$$ \| I_4 \|_{L^\infty} \lesssim | t - s |^{2 \alpha} \| y'_s
   \|_{L^\infty} \rho_\alpha(B)\lesssim (T^{\alpha}\| y, y' \|_{B, 2 \alpha, 4 \alpha-2\sigma}+\|y_0'\|_{L^\infty}) | t - s |^{2\alpha}.$$
   The claim follows.
\end{proof}

By the same arguments, we deduce the $H^{{\gamma}}$-counterpart of Lemma \ref{lem:G2}.

\begin{lem}
Let $T\in (0,1]$. Let $ \sigma \in[0,\alpha)$ and $(y,y')\in \bar{D}^{2\alpha}_{B,4\alpha-2\sigma}.$  Then
	$$(z,z')=\left(\int_0^\cdot P(\cdot-s)y_sdB_s,y\right)\in \bar{D}^{2\alpha}_{B,4\alpha}$$
	and
	$$\|z,z'\|_{\bar{B},2\alpha,4\alpha}\lesssim (1+\rho_\alpha(B))(\|y_0\|_{H^{4\alpha-2\sigma}}+\|y'_0\|_{H^{2\alpha-2\sigma}}+T^{\alpha(\alpha-\sigma)/(2\alpha-\sigma)}\|y,y'\|_{\bar{B},2\alpha,4\alpha-2\sigma}).
	$$
	Here the implicit constant is independent of $y, \rho_\alpha(B)$.
\end{lem}

By a similar argument as \cite[Lemma 3.6]{HN21} we obtain the following result.

\begin{lem} \label{l:D5}
Let $G$ satisfy assumption \eqref{G} and $(y,  G(v+y))\in D^{2\alpha}_{B,4\alpha}$. Then for $\sigma\in [0,\alpha)$ it holds $( G(v+y),  DG(v+y) G(v+y))\in D^{2\alpha}_{B,4\alpha-2\sigma}$
	and
	$$\| G(v+y),  DG(v+y) G(v+y)\|_{B,2\alpha,4\alpha-2\sigma}\lesssim (1+\|y,  G(v+y)\|_{B,2\alpha,4\alpha}+\|v\|_{C^1_{T,x}})(1+\rho_\alpha(B))^2.$$
Moreover, if
$(\tilde{y},  G(v+\tilde{y}))\in D^{2\alpha}_{B,4\alpha}$ then
	\begin{align*}
	&\|G(v+y)- G(v+\tilde{y}),  DG(v+y) G(v+y)- DG(v+\tilde{y}) G(v+\tilde{y})\|_{B,2\alpha,4\alpha-2\sigma}
\\& \lesssim (1+\|y,  G(v+y)\|_{B,2\alpha,4\alpha}+\|\tilde{y},  G(v+\tilde{y})\|_{B,2\alpha,4\alpha}+\|v\|_{C^{1}})(1+\rho_\alpha(B))^2
\\&\qquad\times (\|y-\tilde{y}, G(v+y)- G(v+\tilde{y})\|_{B,2\alpha,4\alpha}).
\end{align*}
\end{lem}

\begin{rem}\label{r:d7}
It will be seen in the proof below that due to the definition of the coefficient $G$ in \eqref{G}, the spatial regularity of the controlled rough path $(G(v+y),  DG(v+y) G(v+y))$ actually only depends on the spatial regularity of the functions $g_{ij}$ and not on the spatial regularity of $v$, $y$. Consequently, the claimed space regularity of order $4\alpha-2\sigma$ was only taken for convenience in order to follow more easily the arguments of \cite{HN21}.
\end{rem}

\begin{proof}[Proof of Lemma \ref{l:D5}]
 For simplicity we only concentrate on the case  $G(y)=g(\cdot,\langle y,\varphi\rangle)$ with $\varphi$ smooth and $g\in C^{3}_{b}$. First, we observe that as a consequence of \eqref{eq:R} with $y'= G(v+y)$ it holds
 \begin{equation*}
 \|y\|_{C^{\alpha}_{T}L^\infty}\lesssim (\|y'\|_{C_T\cC^{2\alpha}}+\|y_0\|_{L^\infty}+
\|R^{y}\|_{\alpha,L^\infty})(1+\rho_\alpha(B)).
\end{equation*}
Then, since the spatial dependence of $G(v+y)$ only depends on the spatial dependence of $g$, we get
$$
\| G(v+y)\|_{C_T\cC^{4\alpha-2\sigma}}\lesssim 1.
$$
	and since $ DG(v+y) G(v+y)= \partial g(\cdot,\langle v+y,\varphi\rangle)\langle g(\cdot,\langle v+y,\varphi\rangle),\varphi\rangle$ where $\partial$ denotes the derivative with respect to the variable in place of the inner product
	$$
	\| DG(v+y) G(v+y)\|_{C^\alpha_TL^\infty}\lesssim 1+\|y\|_{C^\alpha_TL^\infty}+\|v\|_{C^1_{T,x}}\lesssim (1+\|y, G(v+y)\|_{B,2\alpha,4\alpha}+\|v\|_{C^1_{T,x}})(1+\rho_\alpha(B)),
	$$
and
$$
\|DG(v+y) G(v+y)\|_{C_T\cC^{2\alpha-2\sigma}}\lesssim 1.
$$
	Moreover,
	\begin{align}
	\label{eq:Rst}
	\begin{aligned}
	R^{G}_{s,t}&= [g(\langle v_t+y_t,\varphi\rangle)-g(\langle v_s+y_s,\varphi\rangle)- \partial g(\langle v_s+y_s,\varphi\rangle)\langle  g(\langle v_s+y_s,\varphi\rangle)B_{s,t},\varphi\rangle]
	\\&=\int_0^1   [\partial g(\langle v_s+y_s+r(v_{s,t}+y_{s,t}),\varphi\rangle)[\langle v_{s,t}+y_{s,t},\varphi\rangle]dr \\
	&\qquad- \partial g(\langle v_s+y_s,\varphi\rangle)]\langle  g(\langle v_s+y_s,\varphi\rangle)B_{s,t},\varphi\rangle
	\\&=\int_0^1   \partial g(\langle v_s+y_s+r(v_{s,t}+y_{s,t}),\varphi\rangle)\langle v_{s,t},\varphi \rangle dr \\
	&\qquad+\int_0^1  [ \partial g(\langle v_s+y_s+r(v_{s,t}+y_{s,t}),\varphi\rangle)- \partial g(\langle v_s+y_s,\varphi\rangle)]dr\langle  g(\langle v_s+y_s,\varphi\rangle)B_{s,t},\varphi\rangle
	\\
	&\qquad+\int_0^1   \partial g(\langle v_s+y_s+r(v_{s,t}+y_{s,t}),\varphi\rangle)\langle R^y_{s,t},\varphi\rangle dr.
	\end{aligned}
	\end{align}
	Consequently,  we deduce
	$$
	\begin{aligned}
	\|R^G_{s,t}\|_{2\alpha, L^\infty}&\lesssim (\|v\|_{C^1_{T,x}}+\|y\|_{C^\alpha_TL^\infty})(1+\rho_\alpha(B))+\|R^y\|_{2\alpha, L^\infty}\\
	&\lesssim (1+\|v\|_{C^1_{T,x}}+\|y,  G(v+y)\|_{B,2\alpha, 4\alpha})(1+\rho_\alpha(B))^2 .
	\end{aligned}
	$$
Thus the proof of the first result is complete. The second one is a simpler version of the argument in the proof of Lemma \ref{lem:G1} below, so we leave it to the reader.
\end{proof}

\begin{lem}\label{lem:G1}
Let $G$ satisfy assumption \eqref{G} and $(y,  G(v+y))\in \bar{D}^{2\alpha}_{B,4\alpha}$. Then for $\sigma\in[0,\alpha)$, $( G(v+y),  DG(v+y) G(v+y))\in \bar{D}^{2\alpha}_{B,4\alpha-2\sigma}$
	and for $\gamma>0$
	$$\| G(v+y),  DG(v+y) G(v+y)\|_{\bar{B},2\alpha,4\alpha-2\sigma}\lesssim (1+\|y,  G(v+y)\|_{\bar{B},2\alpha,4\alpha}+\|v\|_{C^{2\alpha}_{T}H^{-\gamma}})(1+\rho_\alpha(B))^2.$$
Moreover, if
$(\tilde{y},  G(\tilde{v}+\tilde{y}))\in \bar{D}^{2\alpha}_{B,4\alpha}$. Then
	 for $\gamma>0$
	\begin{align*}
	&\| G(v+y)- G(\tilde{v}+\tilde{y}),  DG(v+y) G(v+y)- DG(\tilde{v}+\tilde{y}) G(\tilde{v}+\tilde{y})\|_{\bar{B},2\alpha,4\alpha-2\sigma}
\\&\lesssim (1+\|y,  G(v+y)\|_{\bar{B},2\alpha,4\alpha}+\|\tilde{y},  G(\tilde{v}+\tilde{y})\|_{\bar{B},2\alpha,4\alpha}+\|v\|_{C^{2\alpha}_{T}H^{-\gamma}}+\|\tilde{v}\|_{C^{2\alpha}_{T}H^{-\gamma}})(1+\rho_\alpha(B))^2
\\&\qquad\times (\|y-\tilde{y}, G(v+y)- G(\tilde{v}+\tilde{y})\|_{\bar{B},2\alpha,4\alpha}+\|v-\tilde{v}\|_{C^{2\alpha}_{T}H^{-\gamma}}).
\end{align*}
\end{lem}

\begin{proof}
 For notational simplicity we again focus only on the  case  $G(y)=g(\cdot,\langle y,\varphi\rangle)$ with $\varphi$ smooth and {$g\in C^{3}_{b}$}. The general case follows the same argument. The first estimate is similar as in Lemma \ref{l:D5}. Now we prove the second one. First, we observe that as a consequence of \eqref{eq:R1} with $y'= G(v+y)$ it holds
 \begin{equation*}
 \|y-\tilde{y}\|_{C^{\alpha}_{T}L^{2}}\lesssim \|y-\tilde{y},  G(v+y)- G(\tilde{v}+\tilde{y})\|_{\bar{B},2\alpha,4\alpha}(1+\rho_\alpha(B)).
\end{equation*}
Then, we have (as in Remark \ref{r:d7}, also here the spatial regularity of $v,y$ does not influence the estimate)
	$$\| G(y+v)- G(\tilde{y}+\tilde{v})\|_{C_TH^{4\alpha-2\sigma}}\lesssim \|y-\tilde{y}\|_{C_TH^{4\alpha-2\sigma}}+\|v-\tilde{v}\|_{C_{T}H^{-\gamma}},$$
and
	$$
	\| DG(v+y) G(v+y)- DG(\tilde{v}+\tilde{y}) G(\tilde{v}+\tilde{y})\|_{C_TH^{2\alpha-2\sigma}}\lesssim \|y-\tilde{y}\|_{C_TH^{2\alpha-2\sigma}}+\|v-\tilde{v}\|_{C_{T}H^{-\gamma}}.
	$$
Moreover, we have
\begin{align*}
\begin{aligned}
&( DG(v+y) G(v+y)- DG(\tilde{v}+\tilde{y}) G(\tilde{v}+\tilde{y}))_{s,t}
\\&=\int_0^1 (\partial^2g(\langle y_s+v_s+r(y+v)_{s,t},\varphi\rangle)-\partial^2g(\langle \tilde{y}_s+\tilde{v}_s+r(\tilde{y}+\tilde{v})_{s,t},\varphi\rangle))dr\\
&\qquad\qquad\times\langle (y+v)_{s,t},\varphi\rangle \langle  G_t(v+y),\varphi\rangle
\\&\quad+\int_0^1 \partial^2g(\langle \tilde{y}_s+\tilde{v}_s+r(\tilde{y}+\tilde{v})_{s,t},\varphi\rangle)dr\langle [(y+v)_{s,t}-(\tilde{y}+\tilde{v})_{s,t}],\varphi\rangle \langle  G_t(v+y),\varphi\rangle
\\&\quad+\int_0^1 \partial^2g(\langle \tilde{y}_s+\tilde{v}_s+r(\tilde{y}+\tilde{v})_{s,t},\varphi\rangle)dr\langle (\tilde{y}+\tilde{v})_{s,t},\varphi\rangle \langle  G_t(v+y)-G_t(\tilde{v}+\tilde{y}),\varphi\rangle
\\&\quad+ (\partial g(\langle y_s+v_s,\varphi\rangle)-\partial g(\langle \tilde{y}_s+\tilde{v}_s,\varphi\rangle)) \langle G(v+y)_{s,t},\varphi\rangle\\
&\quad+ \partial g(\langle \tilde{y}_s+\tilde{v}_s,\varphi\rangle) \langle  G(v+y)_{s,t}-G(\tilde{v}+\tilde{y})_{s,t},\varphi\rangle
\end{aligned}
\end{align*}
which implies
\begin{align*}
	&\| DG(v+y) G(v+y)- DG(\tilde{v}+\tilde{y}) G(\tilde{v}+\tilde{y})\|_{C^\alpha_TL^2}
\\&\lesssim(1+\|y,  G(v+y)\|_{\bar{B},2\alpha,4\alpha}+\|\tilde{y},  G(\tilde{v}+\tilde{y})\|_{\bar{B},2\alpha,4\alpha}+\|v\|_{C^{2\alpha}_{T}H^{-\gamma}}+\|\tilde{v}\|_{C^{2\alpha}_{T}H^{-\gamma}})(1+\rho_\alpha(B))^2
\\&\quad\quad\times (\|y-\tilde{y},  G(v+y)- G(\tilde{v}+\tilde{y})\|_{\bar{B},2\alpha,4\alpha}+\|v-\tilde{v}\|_{C^{2\alpha}_{T}H^{-\gamma}}).
	\end{align*}
		
	Furthermore, for $ R^{G}_{s,t}$ in \eqref{eq:Rst} we have $R^{G}_{s,t}-\tilde{R}^{G}_{s,t}=I_1+I_2+I_3,$ with $I_i$ corresponding to the difference of last three lines:
	\begin{align*} I_1+I_3&=\int_0^1 ( \partial g(\langle v_s+y_s+r(v_{s,t}+y_{s,t}),\varphi\rangle)-\partial g(\langle \tilde{v}_s+\tilde{y}_s+r(\tilde{v}_{s,t}+\tilde{y}_{s,t}),\varphi\rangle))\langle v_{s,t}+R^y_{s,t},\varphi \rangle dr
\\&\quad+\int_0^1 \partial g(\langle \tilde{v}_s+\tilde{y}_s+r(\tilde{v}_{s,t}+\tilde{y}_{s,t}),\varphi\rangle)\langle v_{s,t}-\tilde{v}_{s,t}+R^y_{s,t}-R^{\tilde{y}}_{s,t},\varphi \rangle dr
\end{align*}

\begin{align*}
I_2&=\int_0^1\int_0^1  \partial^2 g(\langle v_s+y_s+r\theta(v_{s,t}+y_{s,t}),\varphi\rangle)d\theta dr\langle (v_{s,t}+y_{s,t}),\varphi\rangle\langle  g(\langle v_s+y_s,\varphi\rangle)B_{s,t},\varphi\rangle
\\&\quad-\int_0^1\int_0^1  \partial^2 g(\langle \tilde{v}_s+\tilde{y}_s+r\theta(\tilde{v}_{s,t}+\tilde{y}_{s,t}),\varphi\rangle)d\theta dr\langle (\tilde{v}_{s,t}+\tilde{y}_{s,t}),\varphi\rangle\langle  g(\langle \tilde{v}_s+\tilde{y}_s,\varphi\rangle)B_{s,t},\varphi\rangle.
	\end{align*}
Then we have
	\begin{align*}
	\|R^G-\tilde{R}^G\|_{2\alpha, L^2}&\lesssim(1+\|y,  G(v+y)\|_{\bar{B},2\alpha,4\alpha}+\|\tilde{y},  G(\tilde{v}+\tilde{y})\|_{\bar{B},2\alpha,4\alpha}+\|v\|_{C^{2\alpha}_{T}H^{-\gamma}}+\|\tilde{v}\|_{C^{2\alpha}_{T}H^{-\gamma}})
\\&\qquad\times (1+\rho_\alpha(B))^2(\|y-\tilde{y},  G(v+y)- G(\tilde{v}+\tilde{y})\|_{\bar{B},2\alpha,4\alpha}+\|v-\tilde{v}\|_{C^{2\alpha}_{T}H^{-\gamma}}).\end{align*}
Thus, the proof is complete.
\end{proof}

Thus, combining Lemma \ref{lem:G2}, Lemma~\ref{l:D5}   and a similar argument as in \cite[Lemma 3.8, Theorem 3.9]{HN21} we obtain the following result.

\begin{thm}\label{zexistence} Let $T>0$ and $G$ satisfy the assumption \eqref{G}. Then there exists a unique global solution $(z, G(z+v))\in D^{2\alpha}_{B,4\alpha}([0,T])$ to \eqref{eq:new}.
	Moreover, the solution satisfies for $\sigma\in [0,\alpha)$
	$$\|z\|_{C_T\cC^{4\alpha}}\leq (\|z_0\|_{\cC^{4\alpha}}+1+\|v\|_{C^1_{T,x}})N^{TN^{\frac{2\alpha-\sigma}{\alpha(\alpha-\sigma)}}},$$
and%
	$$\|z\|_{C^\alpha_TL^\infty}\leq (\|z_0\|_{C^{4\alpha}}+1+\|v\|_{C^1_{T,x}})TN^{N^{\frac{2\alpha-\sigma}{\alpha(\alpha-\sigma)}}T+\frac{2\alpha-\sigma}{\alpha(\alpha-\sigma)}}.$$
	where $N=C(1+\rho_\alpha(B))^{{3}}$ for some constant $C$ independent of $B$, $z$, $v$.
\end{thm}

\begin{proof} 
By Lemma \ref{lem:G2}, Lemma~\ref{l:D5}, we have for $\theta=\frac{\alpha(\alpha-\sigma)}{2\alpha-\sigma}$ and $z'=G(z+v)$
$$
\|z,z'\|_{B,2\alpha,4\alpha}\leq C(1+\rho_\alpha(B))^3(\|z_0\|_{\mathcal{C}^{4\alpha}}+1+T^\theta \| z,z'\|_{B,2\alpha,4\alpha}+T^\theta \|v\|_{C^1_{T,x}}).$$
Then we can choose $\bar T$ such that $1/4<C(1+\rho_\alpha(B))^3\bar T^{\theta} \leq 1/2$. Thus we have
$$
\|z,z'\|_{B,2\alpha,4\alpha,[0,\bar T]}\leq 2C(1+\rho_\alpha(B))^3(\|z_0\|_{\mathcal{C}^{4\alpha}}+1)+ \|v\|_{C^1_{T,x}}.$$
Here $\|z,z'\|_{B,2\alpha,4\alpha,[s,t]}$ is the norm for the space $D^{2\alpha}_{B,4\alpha}$ on  time interval $[s,t]$.
Starting from $\bar T$ we obtain
$$\|z,z'\|_{B,2\alpha,4\alpha,[\bar T,2\bar T]}\leq N(N(\|z_0\|_{\mathcal{C}^{4\alpha}}+1)+ \|v\|_{C^1_{T,x}}+1)+\|v\|_{C^1_{T,x}}.$$
Then we know we have at most $l<(4N)^{1/\theta} T$ times  and by iteration we obtain
$$\|z\|_{C_T\cC^{4\alpha}}\lesssim (\|z_0\|_{\cC^{4\alpha}}+1+\|v\|_{C^1_{T,x}})N^{TN^{1/\theta}},$$
and $$\|z,z'\|_{B,2\alpha,4\alpha,[0,T]} \leq (\|z_0\|_{\cC^{4\alpha}}+1+\|v\|_{C^1_{T,x}})TN^{N^{1/\theta}T+1/\theta}.$$
Here we may increase the constant in $N$.
\end{proof}

Analogously,  we obtain the result in the $H^{{\gamma}}$-scale.

\begin{thm}\label{zexistence1}
Let $T>0$ and $G$ satisfy the assumption \eqref{G}. Then there exists a unique global solution $(z, G(z+v))\in \bar{D}^{2\alpha}_{B,4\alpha}([0,T])$ to \eqref{eq:new}.
	Moreover, the solution satisfies for $\sigma\in [0,\alpha)$
	$$\|z\|_{C_TH^{4\alpha}}\leq (\|z_0\|_{H^{4\alpha}}+1+\|v\|_{C^{2\alpha}_{T}H^{-\gamma}})N^{TN^{\frac{2\alpha-\sigma}{\alpha(\alpha-\sigma)}}},$$
	and
	$$\|z,z'\|_{\bar{B},2\alpha,4\alpha}\leq (\|z_0\|_{H^{4\alpha}}+1+\|v\|_{C^{2\alpha}_{T}H^{-\gamma}})TN^{N^{\frac{2\alpha-\sigma}{\alpha(\alpha-\sigma)}}T+\frac{2\alpha-\sigma}{\alpha(\alpha-\sigma)}}.$$
	where $z'=G(z+v)$, $N=C(1+\rho_\alpha(B))^3$ for some constant $C$ independent of $B$, $z$, $v$. Furthermore, let $z,\tilde{z}$ be two solutions corresponding to $v,\tilde{v}$ respectively. Then we have
	$$\|z-\tilde{z}\|_{C_TH^{4\alpha}}\lesssim \|v-\tilde v\|_{C^{2\alpha}_{T}H^{-\gamma}}\bar{N}^{\bar{N}^{\frac{2\alpha-\sigma}{\alpha(\alpha-\sigma)}}T},$$
	and
	$$\|z-\tilde{z},z'-\tilde{z}'\|_{\bar B,2\alpha,4\alpha}\lesssim \|v-\tilde v\|_{C^{2\alpha}_{T}H^{-\gamma}}T\bar{N}^{T\bar{N}^{\frac{2\alpha-\sigma}{\alpha(\alpha-\sigma)}}+\frac{2\alpha-\sigma}{\alpha(\alpha-\sigma)}}.$$
Here $\bar N=C(1+\|z, z'\|_{\bar{B},2\alpha,4\alpha}+\|\tilde{z}, \tilde z'\|_{\bar{B},2\alpha,4\alpha}+\|v\|_{C^{2\alpha}_{T}H^{-\gamma}}+\|\tilde{v}\|_{C^{2\alpha}_{T}H^{-\gamma}})(1+\rho_\alpha(B))^3$.
\end{thm}

\subsection{Back to the It\^o stochastic integration}\label{s:d1}
If $v$ was $(\mathcal{F}_{t})_{t\geq 0}$-adapted, then the equation \eqref{eq:new} can be solved in the It\^o sense as well. As expected, the solutions $z$ obtained from these two approaches then coincide $\mathbf{P}$-a.s., which can be seen as follows. The equation \eqref{eq:new} possesses a unique stochastic solution $z^{\rm{sto}}$ adapted to the filtration $(\mathcal{F}_{t})_{t\geq 0}$. In view of \cite[Proposition 5.1]{FH14}, $z^{\rm{sto}}$ solves \eqref{eq:new} also in the rough path sense $\mathbf{P}$-a.s. On the other hand, for  $\omega$ from the set of full probability where the rough path lift $(B,\mB)$ is constructed, we obtain a rough path solution $z(\omega)$. By uniqueness for the rough path formulation of \eqref{eq:new}, $z(\omega)=z^{\rm{sto}}(\omega)$ on a set of full probability. As a consequence, the rough path solution may be regarded as an adapted stochastic process.

\color{black}

\end{document}